\documentclass[11pt]{amsart}
\usepackage{amsthm,amsmath,amssymb}

\usepackage{geometry}
\usepackage{graphicx, cite}
\usepackage{color}

\numberwithin{equation}{section}
\numberwithin{figure}{section}
\theoremstyle{plain}
\newtheorem{thm}{Theorem}[section]
\newtheorem{lem}[thm]{Lemma}

\newtheorem{con}[thm]{Conjecture}
\newtheorem{claim}[thm]{Claim}

\theoremstyle{remark}
\newtheorem{rmk}[thm]{Remark}

\newcommand{\M}{\operatorname{M}}
\newcommand{\Hf}{\operatorname{H}}

\begin{document}

\title{Lozenge tilings of hexagons with central holes and dents}

\author{Tri Lai}
\address{Department of Mathematics, University of Nebraska -- Lincoln, Lincoln, NE 68588}
\email{tlai3@unl.edu}
\thanks{This research was supported in part  by Simons Foundation Collaboration Grant (\# 585923).}



\date{\today}

\dedicatory{}

\begin{abstract}
Ciucu showed that the number of lozenge tilings of a hexagon in which a chain of equilateral triangles of alternating orientations, called a `\emph{fern}', has been removed in the center is given by a simple product formula (\emph{Adv. Math. 2017}).  In this paper, we present a multi-parameter generalization of  this work by giving an explicit tiling enumeration for a hexagon with three ferns removed, besides the middle fern located in the center as in Ciucu's region, we remove two additional ferns from two sides of the hexagon. Our result also implies a counterpart of MacMahon's classical formula of boxed plane partitions, corresponding the \emph{exterior} of the union of three disjoint concave polygons obtained by turning 120 degrees after drawing each side.  
\end{abstract}

\maketitle
\tableofcontents
\section{Introduction}
 MacMahon's classical theorem \cite{Mac} on plane partitions fitting in a given box (see \cite{Mac}, and \cite{Stanley}, \cite{Andrews}, \cite{Kup}, \cite{Stem}, \cite{Zeil},\cite{CK}, \cite{Ciu1}, \cite{Tri1}, \cite{Tri2}, \cite{CL} for more recent developments) is equivalent to the fact that the number of lozenge tilings of a centrally symmetric hexagon of side-lengths $a,b,c,a,b,c$ (in clockwise order, starting from the north side) on the triangular lattice is equal to
\begin{equation}\label{MacMahoneq}
\mathbf{P}(a,b,c):=\frac{\Hf(a)\Hf(b)\Hf(c)\Hf(a+b+c)}{\Hf(a+b)\Hf(b+c)\Hf(c+a)},
\end{equation}
where the \emph{hyperfactorial} function $\Hf(n)$ is defined as $\Hf(n):=0!\cdot 1!\cdot 2!\cdots(n-1)!$. Here a \emph{lozenge} is union of any two equilateral triangles sharing an edge; a \emph{lozenge tiling} of a region on the triangular lattice is a covering of the region by lozenges without gaps or overlaps.

The striking formula of MacMahon motivates us to find similar ones. In particular, we would like to investigate lozenge tilings of hexagons with certain `defects', and the most popular defect is a removal of a collection of one or more equilateral triangles. Strictly speaking, one would like to classify this defect based on the position where the triangle has been removed as follows.  If a triangle is removed inside the hexagon, we call it a \emph{(triangular) hole}; if the triangle is removed along the boundary of the hexagon, we call it a \emph{(triangular) dent}.

The tale of tiling enumerations of  hexagons with holes (also known as  `\emph{holey hexagons}') originally came from an (ex-)open problem posed by James Propp.  In 1999, James Propp published an article \cite{Propp} tracking the progress on a list of 20 open problems in the field of exact enumeration of tilings, which he presented
in a lecture in 1996, as part of the special program on algebraic combinatorics
organized at MSRI. The article also presented
a list of 12 new open problems. Problem 2 on the list asks for a tiling formula for a hexagon of side-lengths\footnote{From now on, we always list the side-lengths of a hexagon in the clockwise order, starting from the north side.} $n,
n+ 1, n, n+ 1, n, n+ 1$ with  the central unit triangle removed (see Figure \ref{fig:centralhole}(a)). Ciucu \cite{Ciu2} solved and  generalized this  problem to $(a, b+ 1, b, a+ 1, b, b+ 1)$-hexagons with the central unit triangle
removed (shown in Figure \ref{fig:centralhole}(b)). Gessel and Helfgott later obtained this result independently by a different method \cite{HG}.  S. Okada and
C. Krattenthaler \cite{OK} have solved an even more general problem for a 3-parameter family of holey hexagons,  $(a, b+ 1, c, a+ 1, b, c+ 1)$-hexagons with the central unit
triangle removed (illustrated in Figure \ref{fig:centralhole}(c)).

One readily sees that, in the above results, the central triangular holes have all size $1$.  A milestone in this line of work is when Ciucu, Eisenk\"{o}lbl, Krattenthaler and Zare \cite{CEKZ} showed that the tilings of a hexagon are still enumerated by a simple product formula  if we remove a triangle of an arbitrary side-length in the `center'\footnote{Strictly speaking, the triangular hole is only located exactly at the center when all sides of the hexagon have the same parity; in the other cases, it is $1/2$ unit off the center.}, called a `\emph{cored hexagon}' (see Figure \ref{fig:centralhole}(d) for an example).  In 2013, Ciucu and Krattenthaler \cite{CK} extended the structure of the central triangular hole in the cored hexagons to a cluster of four triangular holes, called a `\emph{shamrock hole}'. The explicit enumeration of a hexagon with a shamrock hole in the center (called a `\emph{$S$-cored hexagon}' or a `\emph{shamrock-cored hexagon}') also yields a nice asymptotic result that they mentioned as a `dual' of MacMahon's theorem (see Figure \ref{fig:centralhole}(e) for a $S$-cored hexagon). Ciucu \cite{Ciu1} later considered a new structure, called a `\emph{fern}', a string of triangles with  alternating orientations, and a hexagon with a fern removed in the center is called a `\emph{$F$-cored hexagon}' or a `\emph{fern-cored hexagon}' (illustrated in Figure \ref{fig:centralhole}(f)). This new structure also yields a nice tiling formula and another dual of MacMahon's theorem. We refer the reader to \cite{CL, Halfhex1, Halfhex2, Halfhex3} for more recent work on the fern structure.

\begin{figure}\centering
\includegraphics[width=13cm]{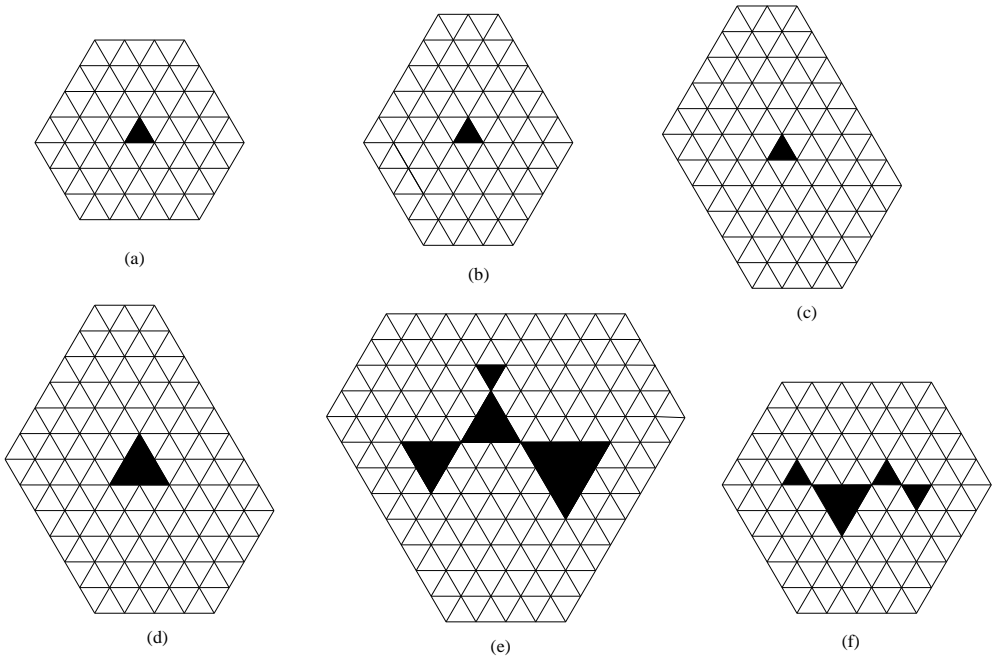}
\caption{Several hexagons with central holes whose tilings are enumerated by simple product formulas (ordered in chronological order). The black triangles indicates the triangular holes.} \label{fig:centralhole}
\end{figure}

For a sequence $\textbf{a}:=(a_i)_{i=1}^{m}$, we denote $o_a:=\sum_{\text{$i$ odd}} a_i$ and $e_a:=\sum_{\text{$i$ even}}a_i$.  Let $S(a_1,a_2,\dotsc,a_m)$ denote the upper half of a hexagon of side-lengths $e_a,o_a,o_a,e_a,o_a,o_a$ in which $k:=\lfloor\frac{m}{2}\rfloor$ triangles of side-lengths $a_1,a_3,a_5,\dots,a_{2k+1}$ removed from the base, such that the distance between the $i$-th and the $(i+1)$-th removed triangles is $a_{2i}$ (see Figure \ref{fig:semihex} for an example).  We call  the region $S(a_1,a_2,\dotsc,a_m)$ a \emph{dented semihexagon}. Cohn, Larsen and Propp \cite{CLP} interpreted semi-strict Gelfand--Tsetlin patterns as lozenge tilings of  the dented semihexagon $S(a_1,a_2,\dotsc,a_m)$, and obtained the following tiling formula
\begin{align}\label{semieq}
s(a_1,a_2,\dots,a_{2l-1})&=s(a_1,a_2,\dots,a_{2l})\\
&=\dfrac{1}{\Hf(a_1+a_{3}+a_{5}+\dotsc+a_{2l-1})}\dfrac{\prod_{\substack{1\leq i<j\leq 2l-1\\
                  \text{$j-i$ odd}}}\Hf(a_i+a_{i+1}+\dotsc+a_{j})}{\prod_{\substack{1\leq i<j\leq 2l-1\\
                  \text{$j-i$ even}}}\Hf(a_i+a_{i+1}+\dotsc+a_{j})},
\end{align}
where $s(a_1,a_2,\dotsc,a_m)$ denotes the number of tilings\footnote{We include here the original formula of Cohn--Larsen--Propp for convenience. Let $T_{m,n}(x_1,\dotsc,x_n)$ be the region obtained from the semihexagon of side-lengths $m$, $n$, $m+n$, $n$ (clockwise from the top) by removing the $n$ up-pointing unit triangles from its bottom that are in the positions $x_1,x_2,\dotsc,x_n$ as counted from left to right. Then the number tilings of the resulting region is given by
\begin{equation*}
\prod_{1\leq i<j\leq n}\frac{x_j-x_i}{j-i}.
\end{equation*}} of $S(a_1,a_2,\dotsc,a_m)$.

\begin{figure}\centering
\setlength{\unitlength}{3947sp}%
\begingroup\makeatletter\ifx\SetFigFont\undefined%
\gdef\SetFigFont#1#2#3#4#5{%
  \reset@font\fontsize{#1}{#2pt}%
  \fontfamily{#3}\fontseries{#4}\fontshape{#5}%
  \selectfont}%
\fi\endgroup%
\resizebox{8cm}{!}{
\begin{picture}(0,0)%
\includegraphics{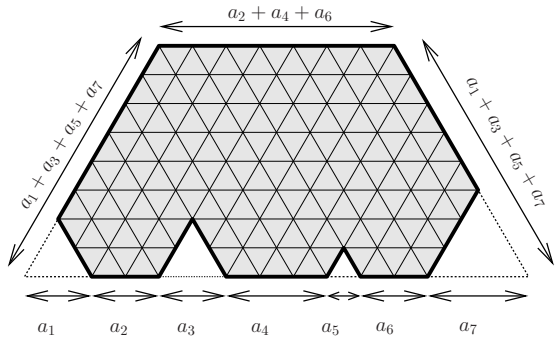}%
\end{picture}%
%
%

\begin{picture}(6994,4016)(1923,-5559)
\put(4502,-1877){\makebox(0,0)[lb]{\smash{{\SetFigFont{14}{16.8}{\rmdefault}{\mddefault}{\itdefault}{$a_2+a_4+a_6$}%
}}}}
\put(2178,-4122){\rotatebox{60.0}{\makebox(0,0)[lb]{\smash{{\SetFigFont{14}{16.8}{\rmdefault}{\mddefault}{\itdefault}{$a_1+a_3+a_5+a_7$}%
}}}}}
\put(2296,-5544){\makebox(0,0)[lb]{\smash{{\SetFigFont{14}{16.8}{\rmdefault}{\mddefault}{\itdefault}{$a_1$}%
}}}}
\put(3136,-5544){\makebox(0,0)[lb]{\smash{{\SetFigFont{14}{16.8}{\rmdefault}{\mddefault}{\itdefault}{$a_2$}%
}}}}
\put(3916,-5544){\makebox(0,0)[lb]{\smash{{\SetFigFont{14}{16.8}{\rmdefault}{\mddefault}{\itdefault}{$a_3$}%
}}}}
\put(4771,-5544){\makebox(0,0)[lb]{\smash{{\SetFigFont{14}{16.8}{\rmdefault}{\mddefault}{\itdefault}{$a_4$}%
}}}}
\put(5589,-5544){\makebox(0,0)[lb]{\smash{{\SetFigFont{14}{16.8}{\rmdefault}{\mddefault}{\itdefault}{$a_5$}%
}}}}
\put(6219,-5529){\makebox(0,0)[lb]{\smash{{\SetFigFont{14}{16.8}{\rmdefault}{\mddefault}{\itdefault}{$a_6$}%
}}}}
\put(7186,-5536){\makebox(0,0)[lb]{\smash{{\SetFigFont{14}{16.8}{\rmdefault}{\mddefault}{\itdefault}{$a_7$}%
}}}}
\put(7263,-2603){\rotatebox{300.0}{\makebox(0,0)[lb]{\smash{{\SetFigFont{14}{16.8}{\rmdefault}{\mddefault}{\itdefault}{$a_1+a_3+a_5+a_7$}%
}}}}}
\end{picture}}
\caption{The dented semihexagon $S(2,2,2,3,1,2,4)$.}
\label{fig:semihex}
\end{figure}

Even though there are a number of elegant enumerations for hexagons with holes and for hexagons with dents, to the best of the author's knowledge, there are \emph{not} any known results for hexagons in which \emph{both} holes and dents are apparent. In this paper, we consider a number of such `rare' families of regions. In particular, our region can be considered as a \emph{multi-parameter} generalization of Ciucu's $F$-cored hexagons in \cite{Ciu1}, besides a fern removed in the center of the hexagon, we remove two more ferns at the same level from two sides of the hexagons.  See Figure \ref{fig:threefern} for two examples. The precise definition of our regions will be given in the next section. We would like to emphasize that the side-lengths and the number of triangles in each of the three ferns are all \emph{arbitrary}.  Our main theorems (Theorem \ref{main1}--\ref{mainQ4}) show that the numbers of tilings of our new regions are always given by a certain product of the tiling number of a cored hexagon, the tiling numbers of two dented semihexagons determined by the ferns, and a simple multiplicative factor. When the two side ferns vanish, our work implies exactly Ciucu's main result in \cite[Theorem 2.1]{Ciu1}. As a consequence, our results generalize almost all known enumerations of hexagons with central holes listed above (the only exception is the enumeration of the $S$-cored hexagons in \cite{CK}). Especially, our theorems also imply a new `dual' of MacMahon's classical theorem on plane partitions, that generalizes the dual of Ciucu in \cite[Theorem 1.1]{Ciu1}.



\begin{figure}\centering
\includegraphics[width=14cm]{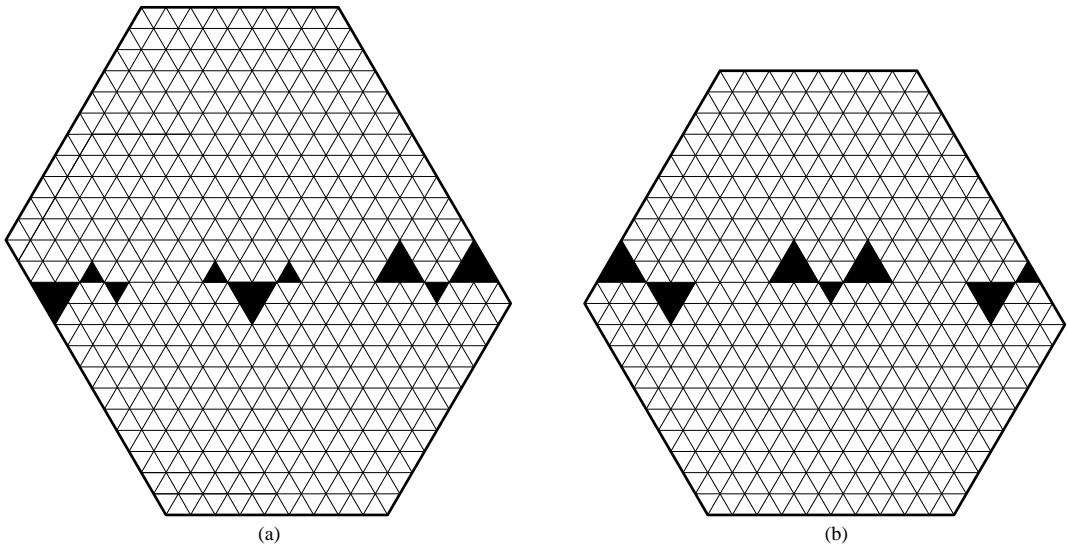}
\caption{Two hexagons with three ferns removed. The black triangles indicate the ones removed.} \label{fig:threefern}
\end{figure}

The rest of this paper is organized as follows. In Section \ref{Statement1}, we give exact tiling enumerations for eight families of regions corresponding to the `central case'. The new dual of MacMahon's theorem is also presented in the same section.  The proofs of these theorems are provided in Section \ref{sec:proof1}.  We wrap up the paper by several remarks and open questions in Section \ref{sec:remark}.

\section{Precise statements of the main results}\label{Statement1}

\subsection{Cored hexagons and Ciucu--Eisenk\"{o}lbl--Krattenthaler--Zare's Theorems}
Continuing the line of work about hexagons with a unit triangle removed in the center in \cite{Ciu2}, \cite{HG} and \cite{OK}, Ciucu, Eisenk\"{o}lbl, Krattenthaler and Zare enumerated the \emph{cored hexagon} (or \emph{punctured hexagon}) $C_{x,y,z}(m)$ that are obtained by removing the central equilateral triangle of side-length $m$ from the hexagon $H$ of side-lengths $x, y+m, z,x+m,y,z+m$ (see Figure \ref{fig:core} for examples). We define this region in detail in the next paragraph.

\begin{figure}\centering
\includegraphics[width=12cm]{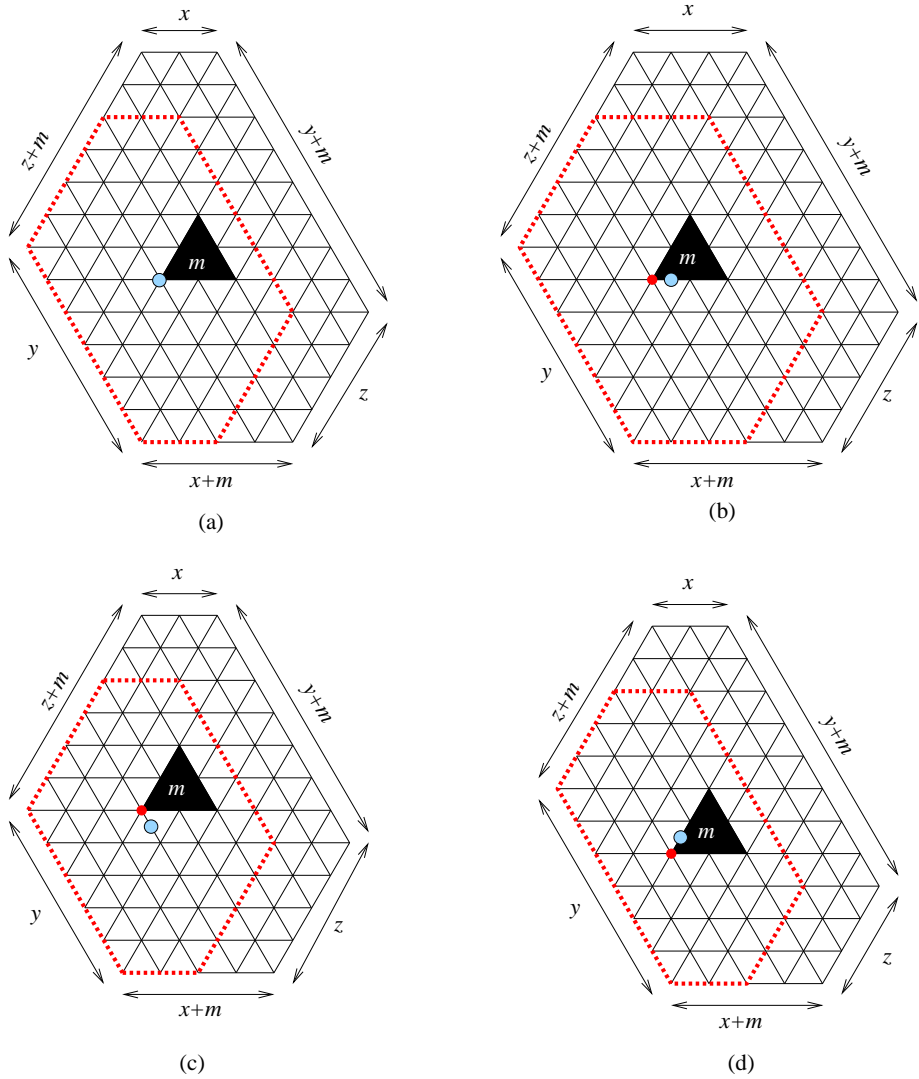}
\caption{(a) The cored hexagon $C_{2,6,4}(2)$. (b) The cored hexagon $C_{3,6,4}(2)$. (c) The cored hexagon $C_{2,5,4}(2)$. (d) The cored hexagon $C_{2,6,3}(2)$. }\label{fig:core}
\end{figure}

We start with an \emph{auxiliary hexagon} $H_0$ with side-lengths $x,y,z,x,y,z$ (indicated by the hexagons with the dashed boundary in Figure \ref{fig:core}). Next, we push the north, the northeast, and the southeast sides of the hexagon $m$ units outward, and keep other sides staying put. This way we get a larger hexagon $H$, called the \emph{base hexagon}, of side-lengths $x, y+m, z,x+m,y,z+m$.  Finally, we remove an up-pointing $m$-triangle such that its left vertex is located at the closest lattice point to the center of the auxiliary hexagon $H_0$. Precisely, there are four cases to distinguish based on the parities of $x,y,z$.  When $x$, $y$ and $z$ have the same parity, the center of the hexagon is a lattice vertex and our removed triangle has the left vertex at the center. One readily sees that, in this case, the triangular hole stays evenly between each pair of parallel sides of the hexagon $H$. In particular, the distance between the north side of the hexagon and the top of the triangular hole and the distance between the base of the triangular hole and the south side of the hexagon are both $\frac{y+z}{2}$; the distances corresponding to the northeast and southwest sides of the hexagon are both $\frac{x+z}{2}$; the distances corresponding to the northwest and southeast sides of the hexagon are both $\frac{x+y}{2}$ (see Figure \ref{fig:core}(a); the hexagon wit the dashed boundary indicates the auxiliary hexagon).  Next, we consider the case when $x$ has parity different from that of $y$ and $z$. In this case, the center of the auxiliary hexagon $H_0$ is \emph{not} a lattice vertex anymore. It is the middle point of a horizontal unit lattice interval. We now place the triangular hole such that its leftmost is $1/2$ unit to the left of the center of the auxiliary hexagon (illustrated in Figure \ref{fig:core}(b); the larger shaded dot indicates the center of the auxiliary hexagon). Similarly, if $y$ has the opposite parity  to $x$ and $z$, we place the triangular hole $1/2$ unit to the northwest of the center of the auxiliary hexagon $H_0$ (shown in Figure \ref{fig:core}(c)). Finally, if $z$ has parity different from that of $x$ and $y$, the hole is located $1/2$ unit to the southwest of the center of $H_0$ (see Figure \ref{fig:core}(d)).

Next, we extend  the definition of the hyperfactorial function to the case of half-integers:
\begin{equation}\label{hyper2}
\Hf(n)=\begin{cases}
\prod_{k=0}^{n-1}\Gamma(k+1) & \text{for $n$ a positive integer;}\\
\prod_{k=0}^{n-\frac{1}{2}}\Gamma(k+\frac{1}{2}) & \text{for $n$ a positive half-integer.}
\end{cases}
\end{equation}
where $\Gamma$ denotes the classical gamma function. Recall that $\Gamma(n+1)=n!$ and $\Gamma(n+\frac{1}{2})=\frac{(2n)!}{4^nn!}\sqrt{\pi}$, for a nonnegative integer $n$.

We can combine Theorems 1 and 2 in \cite{CEKZ} as follows:
\begin{thm}[Ciucu-Eisenk\"{o}lbl-Krattenthaler-Zare \cite{CEKZ}]\label{corethm}
Assume that $m,x,y,z$ are nonnegative integers, such that $y$ and $z$ have the same parity. Then the number of lozenge tilings of the cored hexagon $C_{x,y,z}(m)$ is given by
\begin{align}\label{coreeqx}
\M(C_{x,y,z}(m))=&\frac{\Hf(m+x)\Hf(m+y)\Hf(m+z)\Hf(m+x+y+z)}{\Hf(m+x+y)\Hf(m+y+z)\Hf(m+z+x)}\notag\\
&\times \frac{\Hf(m+\left\lfloor\frac{x+y+z}{2}\right\rfloor )\Hf(m+\left\lceil\frac{x+y+z}{2}\right\rceil)}{\Hf(m+\left\lceil\frac{x+y}{2}\right\rceil)\Hf(m+\frac{y+z}{2})\Hf(m+\left\lfloor\frac{z+x}{2}\right\rfloor)}\notag\\
&\times \frac{\Hf(\frac{m}{2})^2\Hf(\left\lfloor\frac{x}{2}\right\rfloor)\Hf(\left\lceil\frac{x}{2}\right\rceil)\Hf(\left\lfloor\frac{y}{2}\right\rfloor)\Hf(\left\lceil\frac{y}{2}\right\rceil)\Hf(\left\lfloor\frac{z}{2}\right\rfloor)\Hf(\left\lceil\frac{z}{2}\right\rceil)}{\Hf(\frac{m}{2}+\left\lfloor\frac{x}{2}\right\rfloor)\Hf(\frac{m}{2}+\left\lceil\frac{x}{2}\right\rceil)\Hf(\frac{m}{2}+\left\lfloor\frac{y}{2}\right\rfloor)\Hf(\frac{m}{2}+\left\lceil\frac{y}{2}\right\rceil)\Hf(\frac{m}{2}+\left\lfloor\frac{z}{2}\right\rfloor)\Hf(\frac{m}{2}+\left\lceil\frac{z}{2}\right\rceil)}\notag\\
&\times \frac{\Hf(\frac{m}{2}+\left\lfloor\frac{x+y}{2}\right\rfloor)\Hf(\frac{m}{2}+\left\lceil\frac{x+y}{2}\right\rceil)\Hf(\frac{m}{2}+\frac{y+z}{2})^2\Hf(\frac{m}{2}+\left\lfloor\frac{z+x}{2}\right\rfloor)\Hf(\frac{m}{2}+\left\lceil\frac{z+x}{2}\right\rceil)}{\Hf(\frac{m}{2}+\left\lfloor\frac{x+y+z}{2}\right\rfloor)\Hf(\frac{m}{2}+\left\lceil\frac{x+y+z}{2}\right\rceil)\Hf(\left\lfloor\frac{x+y}{2}\right\rfloor)\Hf(\frac{y+z}{2})\Hf(\left\lceil\frac{z+x}{2}\right\rceil)}.
\end{align}
Here we use the notation $\M(R)$ for the number of lozenge tilings of the region $R$.
\end{thm}
By the symmetry, if $x$ and $y$ have the same parity, then the number of tilings of $C_{x,y,z}(m)$ is exactly the expression on the right-hand side of (\ref{coreeqx}) above with $x$ replaced by $z$, $y$ replaced by $x$, and $z$ replaced by $y$.  Similarly, if $x$ and $z$ have the same parity, then the number of tilings of $C_{x,y,z}(m)$ is exactly the expression on the right-hand side of (\ref{coreeqx}) with $x$ replaced by $y$, $y$ replaced by $z$, and $z$ replaced by $x$.

Inspired by the cored hexagons above, we will define our eight families of hexagons with  three collinear ferns removed in the next subsection. Depending on the height of the common horizontal lattice line $\ell$  along which our three ferns are lined up, there are two cases to distinguish: $\ell$ leaves the west and east vertices of the hexagon on opposite sides (see Figure \ref{fig:threefern}(a) for an example) or $\ell$ leaves the two vertices on the same side (see Figure \ref{fig:threefern}(b)). By the symmetry, we can assume from now on that the east vertex of the hexagon is always below the line $\ell$.

\subsection{The case $\ell$ separates the west and east vertices of the hexagon}

We now define our four families of hexagons with three collinear ferns removed, in the case when the horizontal lattice line $\ell$ separates the east and west vertices of the hexagon, as follows.   Our definitions are illustrated by Figures \ref{fig:construct1}--\ref{fig:construct4}. However, we ignore the inner hexagons and the arrows in these figures at the moment (these details will be used later in the alternative definitions of our regions in Subsection 2.4). We call them \emph{$R$-families}. Assume that $x,y,z$ are nonnegative integers and that $\textbf{a}=(a_1,a_2,\dotsc,a_m)$, $\textbf{b}:=(b_1,b_2,\dotsc,b_n)$, $\textbf{c}=(c_1,c_2,\dotsc,c_k)$ are three (may be empty) sequences  of nonnegative integers. The three sequences $\textbf{a}, \textbf{b}, \textbf{c}$ determine the side-lengths of triangles in the left, the right, and the central ferns, respectively.  Set
\begin{align}
e_a:=\sum_{i\ even}a_i, &\quad o_a:=\sum_{i \ odd} a_i,\notag\\
e_b:=\sum_{j\ even}b_j, &\quad o_b:=\sum_{j \ odd} b_j,\notag\\
e_c:=\sum_{t\ even}c_t, &\quad o_c:=\sum_{t \ odd} c_t,
\end{align}
and $a:=a_1+a_2+\cdots+a_m$, $b:=b_1+b_2+\cdots+b_n$, $c:=c_1+c_2+\cdots+c_k$.

\medskip

Assume that $x$ and $z$ have the same parity, we definition of our first $R$-family of regions $R^{\odot}_{x,y,z}(\textbf{a};\textbf{c};\textbf{b})$ in the next paragraph.

We start with the base hexagon $H$ of side-lengths $x+o_a+e_b+e_c,$  $2y+z+e_a+o_b+e_c+ |a-b|$,  $z+o_a+e_b+e_c,$ $x+e_a+o_b+e_c$, $2y+z+o_a+e_b+e_c+ |a-b|,$ $z+e_a+o_b+e_c$, in which $x$ and $z$ have the same parity (see the outermost hexagon in Figure \ref{fig:construct1}).  Suppose first that the total length $a$  of the left fern  is not greater than the total length $b$ of the right fern. Next, we remove at the level $y$ above the east vertex of the hexagon $H$ three ferns as follows. The left fern consists of $m$ triangles of alternating orientations with side-lengths $a_1,a_2,\dotsc,a_m$ as they appear from left to right, and starts with a down-pointing triangle. The right fern consists of $n$ alternating-oriented triangles of side-lengths $b_1, b_2,\dotsc,b_n$ as they appear from \emph{right to left}, and starts with an \emph{up-pointing} triangle. It is easy to see that the distance between the rightmost of the left fern and the leftmost of the right ferns is $c+x+z$. The middle fern of length $c$ consists of alternating-oriented triangles of side-lengths $c_1,c_2,\dots,c_k$ and starts with an up-pointing triangle. We next put this fern equally between the left and the right ferns as indicated by three strings of black triangles in Figure \ref{fig:construct1} (i.e. the distances between two consecutive ferns are both $\frac{x+z}{2}$, which is an integer as $x+z$ is even in this case). If $a>b$, we define the region similarly, the only difference is that we now remove the three ferns at the level $y+(a-b)$ above the east vertex of the hexagon (as opposed to removing at the level $y$ as in the previous case).

Next, we define the second $R$-family consisting of regions $R^{\leftarrow}_{x,y,z}(\textbf{a};\textbf{c};\textbf{b})$, in the case when $x$ has different parity from that of $z$. We follow the same process as in the case of the $R^{\odot}$-type regions above, the only difference is that, since $x+z$ is now odd, we place the middle fern $1$ unit closer to the left fern than the right one, that is the distance between the left and the middle ferns is $\left\lfloor\frac{x+z}{2}\right\rfloor$ and the distance between the middle and the right ferns is  $\left\lceil\frac{x+z}{2}\right\rceil$ (see Figure \ref{fig:construct2} for an example in the case $a>b$).

\medskip

The third and the fourth $R$-families are defined little differently, as we allow $y$ to take the  value $-1$ in certain situations.
\medskip

Our third $R$-family is for the case when $x$ and $z$ have the same parity and is defined as follows. We now start with a slightly different base hexagon, that has side-lengths $x+o_a+e_b+e_c,$  $2y+z+e_a+o_b+e_c+ |a-b|+1$,  $z+o_a+e_b+e_c,$ $x+e_a+o_b+e_c$, $2y+z+o_a+e_b+e_c+ |a-b|+1,$ $z+e_a+o_b+e_c$ (indicated by the outermost hexagon in Figure \ref{fig:construct3}). Next, we repeat the process in the definition of the first $R$-family, the only difference is that we are now removing the three ferns at the level $y+1$ above the east vertex of the hexagon if $a< b$, and at the level $y+(a-b)+1$ if $a\geq b$. Moreover, in the case $a<b$, the parameter $y$ may take the  value $-1$. Denote by $R^{\nwarrow}_{x,y,z}(\textbf{a};\textbf{c};\textbf{b})$ the resulting region.

Our fourth $R$-family consists of the regions $R^{\swarrow}_{x,y,z}(\textbf{a};\textbf{c};\textbf{b})$ when $x$ has different parity from that of $z$. In this case, our base hexagon $H$ is the same as that in the third $R$-family, however, we now remove the our ferns in the same way as in the second family. In particular, we remove the three ferns at the level $y$ or $y+a-b$, depending on whether $a\leq b$ or $a>b$, such that  the distance between the left and middle ferns is $\left\lfloor\frac{x+z}{2}\right\rfloor$ (illustrated in Figure \ref{fig:construct4}). Similar to the third $R$-family, we allow $y$ to take the value $-1$ when $a>b$.

\begin{figure}\centering
\setlength{\unitlength}{3947sp}%
\begingroup\makeatletter\ifx\SetFigFont\undefined%
\gdef\SetFigFont#1#2#3#4#5{%
  \reset@font\fontsize{#1}{#2pt}%
  \fontfamily{#3}\fontseries{#4}\fontshape{#5}%
  \selectfont}%
\fi\endgroup%
\resizebox{15cm}{!}{
\begin{picture}(0,0)%
\includegraphics{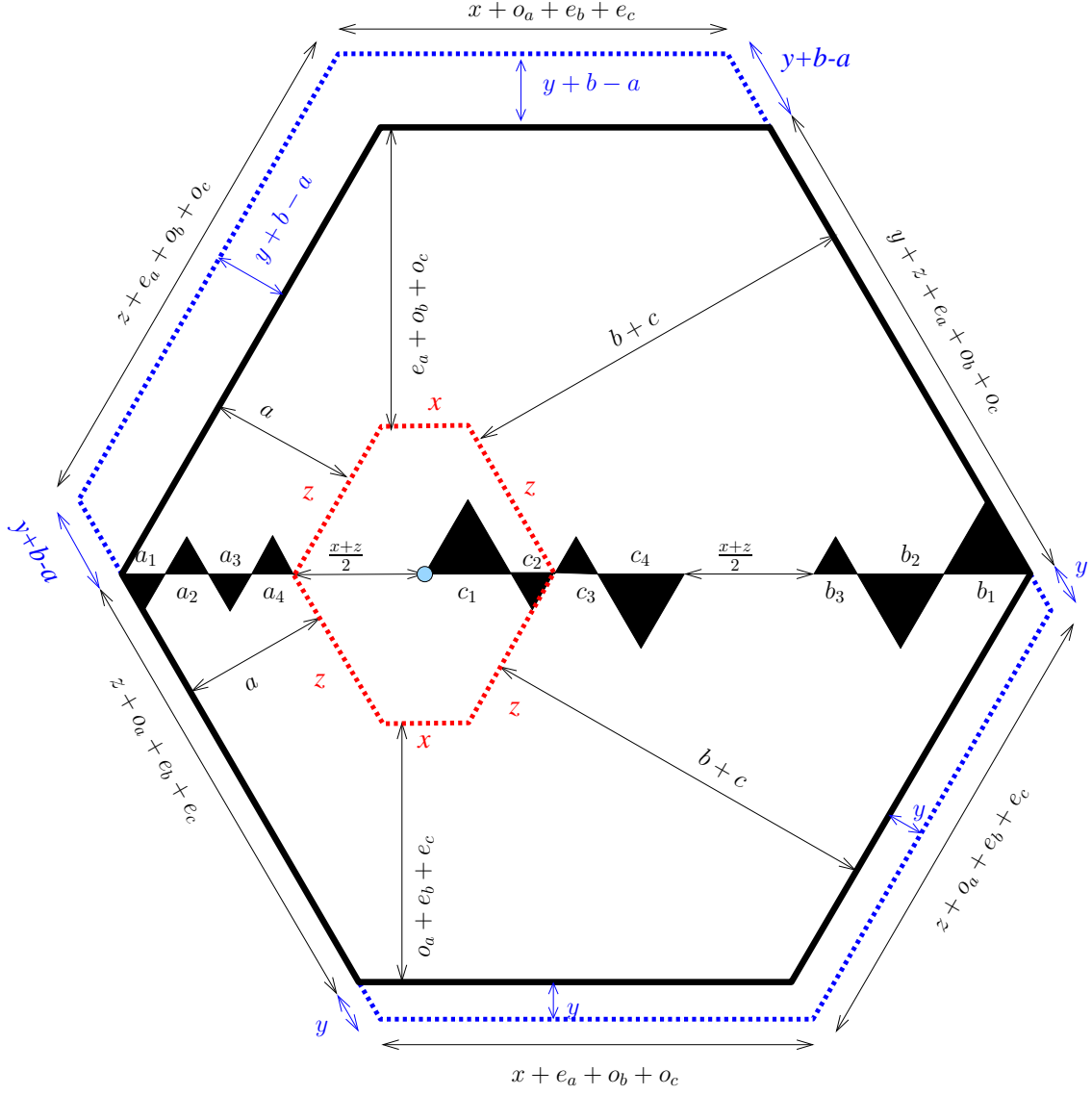}%
\end{picture}%
%
%

\begin{picture}(10318,10458)(1347,-10639)
\put(6673,-9981){\makebox(0,0)[lb]{\smash{{\SetFigFont{14}{16.8}{\rmdefault}{\mddefault}{\itdefault}{\color[rgb]{0,0,1}$y$}%
}}}}
\put(9946,-8091){\makebox(0,0)[lb]{\smash{{\SetFigFont{14}{16.8}{\rmdefault}{\mddefault}{\itdefault}{\color[rgb]{0,0,1}$y$}%
}}}}
\put(5638,-6018){\makebox(0,0)[lb]{\smash{{\SetFigFont{14}{16.8}{\rmdefault}{\mddefault}{\itdefault}{\color[rgb]{0,0,0}$c_1$}%
}}}}
\put(6252,-5700){\makebox(0,0)[lb]{\smash{{\SetFigFont{14}{16.8}{\rmdefault}{\mddefault}{\itdefault}{\color[rgb]{0,0,0}$c_2$}%
}}}}
\put(6763,-6018){\makebox(0,0)[lb]{\smash{{\SetFigFont{14}{16.8}{\rmdefault}{\mddefault}{\itdefault}{\color[rgb]{0,0,0}$c_3$}%
}}}}
\put(7274,-5663){\makebox(0,0)[lb]{\smash{{\SetFigFont{14}{16.8}{\rmdefault}{\mddefault}{\itdefault}{\color[rgb]{0,0,0}$c_4$}%
}}}}
\put(10548,-6018){\makebox(0,0)[lb]{\smash{{\SetFigFont{14}{16.8}{\rmdefault}{\mddefault}{\itdefault}{\color[rgb]{0,0,0}$b_1$}%
}}}}
\put(9832,-5663){\makebox(0,0)[lb]{\smash{{\SetFigFont{14}{16.8}{\rmdefault}{\mddefault}{\itdefault}{\color[rgb]{0,0,0}$b_2$}%
}}}}
\put(9116,-6018){\makebox(0,0)[lb]{\smash{{\SetFigFont{14}{16.8}{\rmdefault}{\mddefault}{\itdefault}{\color[rgb]{0,0,0}$b_3$}%
}}}}
\put(2589,-5663){\makebox(0,0)[lb]{\smash{{\SetFigFont{14}{16.8}{\rmdefault}{\mddefault}{\itdefault}{\color[rgb]{0,0,0}$a_1$}%
}}}}
\put(2978,-6018){\makebox(0,0)[lb]{\smash{{\SetFigFont{14}{16.8}{\rmdefault}{\mddefault}{\itdefault}{\color[rgb]{0,0,0}$a_2$}%
}}}}
\put(3388,-5663){\makebox(0,0)[lb]{\smash{{\SetFigFont{14}{16.8}{\rmdefault}{\mddefault}{\itdefault}{\color[rgb]{0,0,0}$a_3$}%
}}}}
\put(3797,-6018){\makebox(0,0)[lb]{\smash{{\SetFigFont{14}{16.8}{\rmdefault}{\mddefault}{\itdefault}{\color[rgb]{0,0,0}$a_4$}%
}}}}
\put(4410,-5663){\makebox(0,0)[lb]{\smash{{\SetFigFont{14}{16.8}{\rmdefault}{\mddefault}{\itdefault}{\color[rgb]{0,0,0}$\frac{x+z}{2}$}%
}}}}
\put(8093,-5663){\makebox(0,0)[lb]{\smash{{\SetFigFont{14}{16.8}{\rmdefault}{\mddefault}{\itdefault}{\color[rgb]{0,0,0}$\frac{x+z}{2}$}%
}}}}
\put(7134,-3624){\rotatebox{30.0}{\makebox(0,0)[lb]{\smash{{\SetFigFont{14}{16.8}{\rmdefault}{\mddefault}{\itdefault}{\color[rgb]{0,0,0}$b+c$}%
}}}}}
\put(5320,-3885){\rotatebox{90.0}{\makebox(0,0)[lb]{\smash{{\SetFigFont{14}{16.8}{\rmdefault}{\mddefault}{\itdefault}{\color[rgb]{0,0,0}$e_a+o_b+o_c$}%
}}}}}
\put(3759,-4246){\rotatebox{330.0}{\makebox(0,0)[lb]{\smash{{\SetFigFont{14}{16.8}{\rmdefault}{\mddefault}{\itdefault}{\color[rgb]{0,0,0}$a$}%
}}}}}
\put(3663,-6910){\rotatebox{30.0}{\makebox(0,0)[lb]{\smash{{\SetFigFont{14}{16.8}{\rmdefault}{\mddefault}{\itdefault}{\color[rgb]{0,0,0}$a$}%
}}}}}
\put(5378,-9437){\rotatebox{90.0}{\makebox(0,0)[lb]{\smash{{\SetFigFont{14}{16.8}{\rmdefault}{\mddefault}{\itdefault}{\color[rgb]{0,0,0}$o_a+e_b+e_c$}%
}}}}}
\put(7904,-7580){\rotatebox{330.0}{\makebox(0,0)[lb]{\smash{{\SetFigFont{14}{16.8}{\rmdefault}{\mddefault}{\itdefault}{\color[rgb]{0,0,0}$b+c$}%
}}}}}
\put(6445,-1169){\makebox(0,0)[lb]{\smash{{\SetFigFont{14}{16.8}{\rmdefault}{\mddefault}{\itdefault}{\color[rgb]{0,0,1}$y+b-a$}%
}}}}
\put(3809,-2822){\rotatebox{60.0}{\makebox(0,0)[lb]{\smash{{\SetFigFont{14}{16.8}{\rmdefault}{\mddefault}{\itdefault}{\color[rgb]{0,0,1}$y+b-a$}%
}}}}}
\put(5740,-467){\makebox(0,0)[lb]{\smash{{\SetFigFont{14}{16.8}{\rmdefault}{\mddefault}{\itdefault}{\color[rgb]{0,0,0}$x+o_a+e_b+e_c$}%
}}}}
\put(6149,-10624){\makebox(0,0)[lb]{\smash{{\SetFigFont{14}{16.8}{\rmdefault}{\mddefault}{\itdefault}{\color[rgb]{0,0,0}$x+e_a+o_b+o_c$}%
}}}}
\put(10241,-9207){\rotatebox{60.0}{\makebox(0,0)[lb]{\smash{{\SetFigFont{14}{16.8}{\rmdefault}{\mddefault}{\itdefault}{\color[rgb]{0,0,0}$z+o_a+e_b+e_c$}%
}}}}}
\put(9729,-2534){\rotatebox{300.0}{\makebox(0,0)[lb]{\smash{{\SetFigFont{14}{16.8}{\rmdefault}{\mddefault}{\itdefault}{\color[rgb]{0,0,0}$y+z+e_a+o_b+o_c$}%
}}}}}
\put(2496,-3419){\rotatebox{60.0}{\makebox(0,0)[lb]{\smash{{\SetFigFont{14}{16.8}{\rmdefault}{\mddefault}{\itdefault}{\color[rgb]{0,0,0}$z+e_a+o_b+o_c$}%
}}}}}
\put(2350,-6799){\rotatebox{300.0}{\makebox(0,0)[lb]{\smash{{\SetFigFont{14}{16.8}{\rmdefault}{\mddefault}{\itdefault}{\color[rgb]{0,0,0}$z+o_a+e_b+e_c$}%
}}}}}
\put(4297,-10145){\makebox(0,0)[lb]{\smash{{\SetFigFont{14}{16.8}{\rmdefault}{\mddefault}{\itdefault}{\color[rgb]{0,0,1}$y$}%
}}}}
\put(11481,-5769){\makebox(0,0)[lb]{\smash{{\SetFigFont{14}{16.8}{\rmdefault}{\mddefault}{\itdefault}{\color[rgb]{0,0,1}$y$}%
}}}}
\end{picture}}
\caption{Construction of the hexagon with 3 ferns removed $R^{\odot}_{2,1,4}(1,1,1,1;\ 2,2,1; \ 2,1,1,2)$.}\label{fig:construct1}
\end{figure}

\begin{figure}\centering
\setlength{\unitlength}{3947sp}%
\begingroup\makeatletter\ifx\SetFigFont\undefined%
\gdef\SetFigFont#1#2#3#4#5{%
  \reset@font\fontsize{#1}{#2pt}%
  \fontfamily{#3}\fontseries{#4}\fontshape{#5}%
  \selectfont}%
\fi\endgroup%
\resizebox{15cm}{!}{
\begin{picture}(0,0)%
\includegraphics{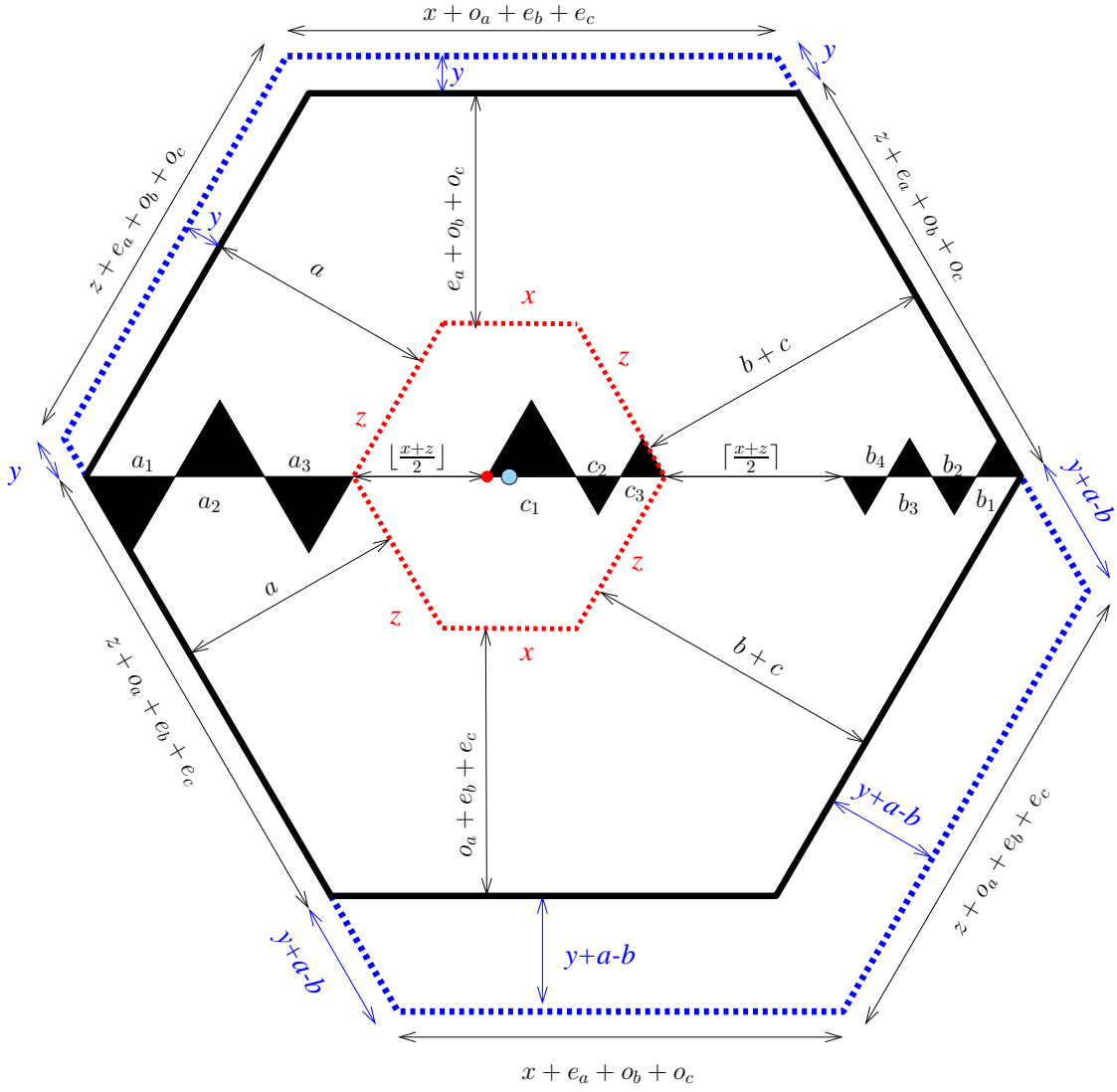}%
\end{picture}%
%
%

\begin{picture}(10369,10132)(2543,-10209)
\put(6706,-2889){\rotatebox{90.0}{\makebox(0,0)[lb]{\smash{{\SetFigFont{14}{16.8}{\rmdefault}{\mddefault}{\itdefault}{\color[rgb]{0,0,0}$e_a+o_b+o_c$}%
}}}}}
\put(9341,-3657){\rotatebox{30.0}{\makebox(0,0)[lb]{\smash{{\SetFigFont{14}{16.8}{\rmdefault}{\mddefault}{\itdefault}{\color[rgb]{0,0,0}$b+c$}%
}}}}}
\put(9211,-6217){\rotatebox{330.0}{\makebox(0,0)[lb]{\smash{{\SetFigFont{14}{16.8}{\rmdefault}{\mddefault}{\itdefault}{\color[rgb]{0,0,0}$b+c$}%
}}}}}
\put(6817,-8138){\rotatebox{90.0}{\makebox(0,0)[lb]{\smash{{\SetFigFont{14}{16.8}{\rmdefault}{\mddefault}{\itdefault}{\color[rgb]{0,0,0}$o_a+e_b+e_c$}%
}}}}}
\put(4946,-5718){\rotatebox{30.0}{\makebox(0,0)[lb]{\smash{{\SetFigFont{14}{16.8}{\rmdefault}{\mddefault}{\itdefault}{\color[rgb]{0,0,0}$a$}%
}}}}}
\put(5320,-2702){\rotatebox{330.0}{\makebox(0,0)[lb]{\smash{{\SetFigFont{14}{16.8}{\rmdefault}{\mddefault}{\itdefault}{\color[rgb]{0,0,0}$a$}%
}}}}}
\put(3691,-4524){\makebox(0,0)[lb]{\smash{{\SetFigFont{14}{16.8}{\rmdefault}{\mddefault}{\itdefault}{\color[rgb]{0,0,0}$a_1$}%
}}}}
\put(4301,-4902){\makebox(0,0)[lb]{\smash{{\SetFigFont{14}{16.8}{\rmdefault}{\mddefault}{\itdefault}{\color[rgb]{0,0,0}$a_2$}%
}}}}
\put(5133,-4520){\makebox(0,0)[lb]{\smash{{\SetFigFont{14}{16.8}{\rmdefault}{\mddefault}{\itdefault}{\color[rgb]{0,0,0}$a_3$}%
}}}}
\put(7262,-4922){\makebox(0,0)[lb]{\smash{{\SetFigFont{14}{16.8}{\rmdefault}{\mddefault}{\itdefault}{\color[rgb]{0,0,0}$c_1$}%
}}}}
\put(7876,-4568){\makebox(0,0)[lb]{\smash{{\SetFigFont{14}{16.8}{\rmdefault}{\mddefault}{\itdefault}{\color[rgb]{0,0,0}$c_2$}%
}}}}
\put(8223,-4797){\makebox(0,0)[lb]{\smash{{\SetFigFont{14}{16.8}{\rmdefault}{\mddefault}{\itdefault}{\color[rgb]{0,0,0}$c_3$}%
}}}}
\put(11454,-4891){\makebox(0,0)[lb]{\smash{{\SetFigFont{14}{16.8}{\rmdefault}{\mddefault}{\itdefault}{\color[rgb]{0,0,0}$b_1$}%
}}}}
\put(11149,-4568){\makebox(0,0)[lb]{\smash{{\SetFigFont{14}{16.8}{\rmdefault}{\mddefault}{\itdefault}{\color[rgb]{0,0,0}$b_2$}%
}}}}
\put(10740,-4922){\makebox(0,0)[lb]{\smash{{\SetFigFont{14}{16.8}{\rmdefault}{\mddefault}{\itdefault}{\color[rgb]{0,0,0}$b_3$}%
}}}}
\put(10436,-4486){\makebox(0,0)[lb]{\smash{{\SetFigFont{14}{16.8}{\rmdefault}{\mddefault}{\itdefault}{\color[rgb]{0,0,0}$b_4$}%
}}}}
\put(6037,-4486){\makebox(0,0)[lb]{\smash{{\SetFigFont{14}{16.8}{\rmdefault}{\mddefault}{\itdefault}{\color[rgb]{0,0,0}$\lfloor\frac{x+z}{2}\rfloor$}%
}}}}
\put(9106,-4486){\makebox(0,0)[lb]{\smash{{\SetFigFont{14}{16.8}{\rmdefault}{\mddefault}{\itdefault}{\color[rgb]{0,0,0}$\lceil\frac{x+z}{2}\rceil$}%
}}}}
\put(6369,-369){\makebox(0,0)[lb]{\smash{{\SetFigFont{14}{16.8}{\rmdefault}{\mddefault}{\itdefault}{\color[rgb]{0,0,0}$x+o_a+e_b+e_c$}%
}}}}
\put(10511,-1430){\rotatebox{300.0}{\makebox(0,0)[lb]{\smash{{\SetFigFont{14}{16.8}{\rmdefault}{\mddefault}{\itdefault}{\color[rgb]{0,0,0}$z+e_a+o_b+o_c$}%
}}}}}
\put(11326,-8874){\rotatebox{60.0}{\makebox(0,0)[lb]{\smash{{\SetFigFont{14}{16.8}{\rmdefault}{\mddefault}{\itdefault}{\color[rgb]{0,0,0}$z+o_a+e_b+e_c$}%
}}}}}
\put(7276,-10194){\makebox(0,0)[lb]{\smash{{\SetFigFont{14}{16.8}{\rmdefault}{\mddefault}{\itdefault}{\color[rgb]{0,0,0}$x+e_a+o_b+o_c$}%
}}}}
\put(3446,-6112){\rotatebox{300.0}{\makebox(0,0)[lb]{\smash{{\SetFigFont{14}{16.8}{\rmdefault}{\mddefault}{\itdefault}{\color[rgb]{0,0,0}$z+o_a+e_b+e_c$}%
}}}}}
\put(3379,-2923){\rotatebox{60.0}{\makebox(0,0)[lb]{\smash{{\SetFigFont{14}{16.8}{\rmdefault}{\mddefault}{\itdefault}{\color[rgb]{0,0,0}$z+e_a+o_b+o_c$}%
}}}}}
\end{picture}%
}
\caption{Construction of the hexagon with 3 ferns removed $R^{\leftarrow}_{3,1,4}(2,2,2;\ 1,1,1,1; \ 2,1,1)$.}\label{fig:construct2}
\end{figure}

\begin{figure}\centering
\setlength{\unitlength}{3947sp}%
\begingroup\makeatletter\ifx\SetFigFont\undefined%
\gdef\SetFigFont#1#2#3#4#5{%
  \reset@font\fontsize{#1}{#2pt}%
  \fontfamily{#3}\fontseries{#4}\fontshape{#5}%
  \selectfont}%
\fi\endgroup%
\resizebox{15cm}{!}{
\begin{picture}(0,0)%
\includegraphics{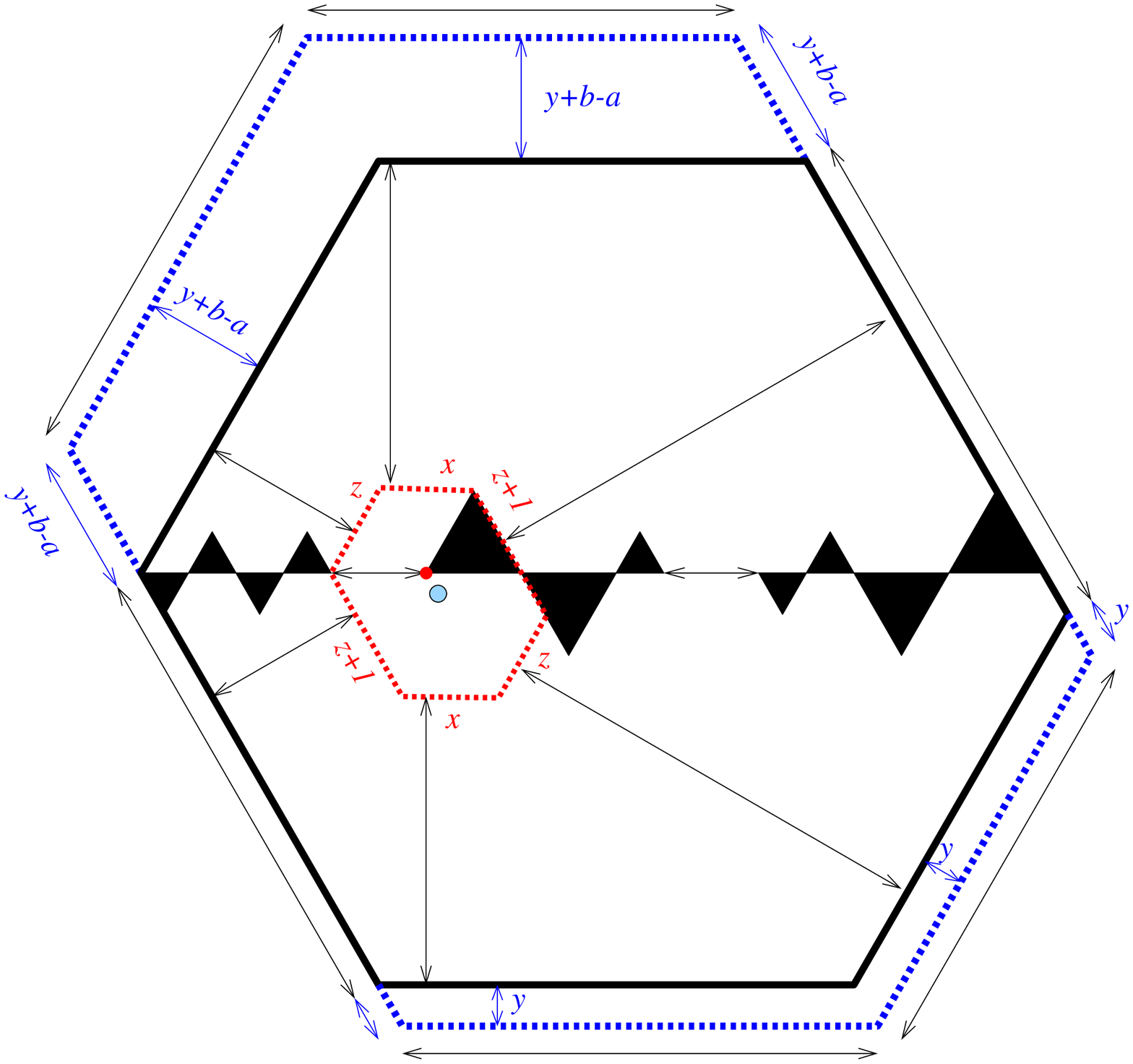}%
\end{picture}%
%
%

\begin{picture}(9784,9637)(2880,-10235)
\put(6616,-890){\makebox(0,0)[lb]{\smash{{\SetFigFont{14}{16.8}{\rmdefault}{\mddefault}{\itdefault}{\color[rgb]{0,0,0}$x+o_a+e_b+e_c$}%
}}}}
\put(3786,-3503){\rotatebox{60.0}{\makebox(0,0)[lb]{\smash{{\SetFigFont{14}{16.8}{\rmdefault}{\mddefault}{\itdefault}{\color[rgb]{0,0,0}$z+e_a+o_b+o_c$}%
}}}}}
\put(4059,-6811){\rotatebox{300.0}{\makebox(0,0)[lb]{\smash{{\SetFigFont{14}{16.8}{\rmdefault}{\mddefault}{\itdefault}{\color[rgb]{0,0,0}$z+o_a+e_b+e_c+1$}%
}}}}}
\put(7237,-10220){\makebox(0,0)[lb]{\smash{{\SetFigFont{14}{16.8}{\rmdefault}{\mddefault}{\itdefault}{\color[rgb]{0,0,0}$x+e_a+o_b+o_c$}%
}}}}
\put(11282,-9200){\rotatebox{60.0}{\makebox(0,0)[lb]{\smash{{\SetFigFont{14}{16.8}{\rmdefault}{\mddefault}{\itdefault}{\color[rgb]{0,0,0}$z+o_a+e_b+e_c$}%
}}}}}
\put(10507,-2651){\rotatebox{300.0}{\makebox(0,0)[lb]{\smash{{\SetFigFont{14}{16.8}{\rmdefault}{\mddefault}{\itdefault}{\color[rgb]{0,0,0}$z+e_a+o_b+o_c+1$}%
}}}}}
\put(8604,-4599){\rotatebox{30.0}{\makebox(0,0)[lb]{\smash{{\SetFigFont{14}{16.8}{\rmdefault}{\mddefault}{\itdefault}{\color[rgb]{0,0,0}$b+c$}%
}}}}}
\put(6445,-4358){\rotatebox{90.0}{\makebox(0,0)[lb]{\smash{{\SetFigFont{14}{16.8}{\rmdefault}{\mddefault}{\itdefault}{\color[rgb]{0,0,0}$e_a+o_b+o_c$}%
}}}}}
\put(8926,-7387){\rotatebox{330.0}{\makebox(0,0)[lb]{\smash{{\SetFigFont{14}{16.8}{\rmdefault}{\mddefault}{\itdefault}{\color[rgb]{0,0,0}$b+c$}%
}}}}}
\put(6817,-8846){\rotatebox{90.0}{\makebox(0,0)[lb]{\smash{{\SetFigFont{14}{16.8}{\rmdefault}{\mddefault}{\itdefault}{\color[rgb]{0,0,0}$o_a+e_b+e_c$}%
}}}}}
\put(5279,-6720){\rotatebox{30.0}{\makebox(0,0)[lb]{\smash{{\SetFigFont{14}{16.8}{\rmdefault}{\mddefault}{\itdefault}{\color[rgb]{0,0,0}$a$}%
}}}}}
\put(5218,-4946){\rotatebox{330.0}{\makebox(0,0)[lb]{\smash{{\SetFigFont{14}{16.8}{\rmdefault}{\mddefault}{\itdefault}{\color[rgb]{0,0,0}$a$}%
}}}}}
\put(5971,-5686){\makebox(0,0)[lb]{\smash{{\SetFigFont{14}{16.8}{\rmdefault}{\mddefault}{\itdefault}{\color[rgb]{0,0,0}$\frac{x+z}{2}$}%
}}}}
\put(8746,-5611){\makebox(0,0)[lb]{\smash{{\SetFigFont{14}{16.8}{\rmdefault}{\mddefault}{\itdefault}{\color[rgb]{0,0,0}$\frac{x+z}{2}$}%
}}}}
\put(4297,-5691){\makebox(0,0)[lb]{\smash{{\SetFigFont{14}{16.8}{\rmdefault}{\mddefault}{\itdefault}{\color[rgb]{0,0,0}$a_1$}%
}}}}
\put(4576,-6046){\makebox(0,0)[lb]{\smash{{\SetFigFont{14}{16.8}{\rmdefault}{\mddefault}{\itdefault}{\color[rgb]{0,0,0}$a_2$}%
}}}}
\put(5011,-5686){\makebox(0,0)[lb]{\smash{{\SetFigFont{14}{16.8}{\rmdefault}{\mddefault}{\itdefault}{\color[rgb]{0,0,0}$a_3$}%
}}}}
\put(5416,-6046){\makebox(0,0)[lb]{\smash{{\SetFigFont{14}{16.8}{\rmdefault}{\mddefault}{\itdefault}{\color[rgb]{0,0,0}$a_4$}%
}}}}
\put(6886,-6031){\makebox(0,0)[lb]{\smash{{\SetFigFont{14}{16.8}{\rmdefault}{\mddefault}{\itdefault}{\color[rgb]{0,0,0}$c_1$}%
}}}}
\put(7696,-5701){\makebox(0,0)[lb]{\smash{{\SetFigFont{14}{16.8}{\rmdefault}{\mddefault}{\itdefault}{\color[rgb]{0,0,0}$c_2$}%
}}}}
\put(8184,-6011){\makebox(0,0)[lb]{\smash{{\SetFigFont{14}{16.8}{\rmdefault}{\mddefault}{\itdefault}{\color[rgb]{0,0,0}$c_3$}%
}}}}
\put(11431,-6091){\makebox(0,0)[lb]{\smash{{\SetFigFont{14}{16.8}{\rmdefault}{\mddefault}{\itdefault}{\color[rgb]{0,0,0}$b_1$}%
}}}}
\put(10546,-5641){\makebox(0,0)[lb]{\smash{{\SetFigFont{14}{16.8}{\rmdefault}{\mddefault}{\itdefault}{\color[rgb]{0,0,0}$b_2$}%
}}}}
\put(9961,-6031){\makebox(0,0)[lb]{\smash{{\SetFigFont{14}{16.8}{\rmdefault}{\mddefault}{\itdefault}{\color[rgb]{0,0,0}$b_3$}%
}}}}
\put(9571,-5716){\makebox(0,0)[lb]{\smash{{\SetFigFont{14}{16.8}{\rmdefault}{\mddefault}{\itdefault}{\color[rgb]{0,0,0}$b_4$}%
}}}}
\end{picture}%
}
\caption{Construction of the region $R^{\nwarrow}_{2,1,2}(1,1,1,1;\ 2,2,1;\ 2,2,1,1)$.}\label{fig:construct3}
\end{figure}

\begin{figure}\centering
\setlength{\unitlength}{3947sp}%
\begingroup\makeatletter\ifx\SetFigFont\undefined%
\gdef\SetFigFont#1#2#3#4#5{%
  \reset@font\fontsize{#1}{#2pt}%
  \fontfamily{#3}\fontseries{#4}\fontshape{#5}%
  \selectfont}%
\fi\endgroup%
\resizebox{15cm}{!}{
\begin{picture}(0,0)%
\includegraphics{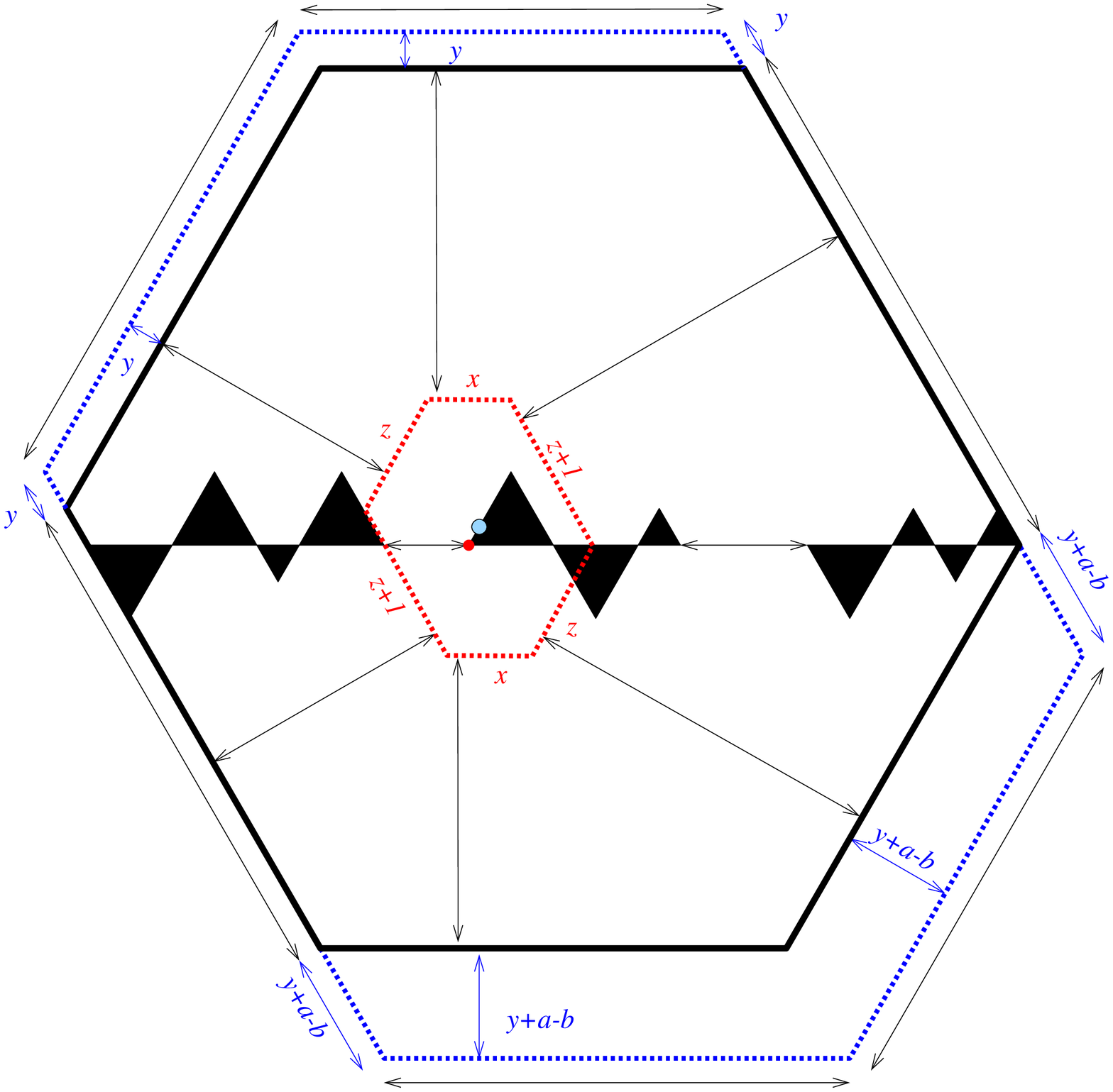}%
\end{picture}%
%
%

\begin{picture}(10857,11232)(15557,-11825)
\put(22197,-7899){\rotatebox{330.0}{\makebox(0,0)[lb]{\smash{{\SetFigFont{14}{16.8}{\rmdefault}{\mddefault}{\itdefault}{\color[rgb]{0,0,0}$b+c$}%
}}}}}
\put(20259,-9961){\rotatebox{90.0}{\makebox(0,0)[lb]{\smash{{\SetFigFont{14}{16.8}{\rmdefault}{\mddefault}{\itdefault}{\color[rgb]{0,0,0}$o_a+e_b+e_c$}%
}}}}}
\put(25021,-6495){\makebox(0,0)[lb]{\smash{{\SetFigFont{14}{16.8}{\rmdefault}{\mddefault}{\itdefault}{\color[rgb]{0,0,0}$b_1$}%
}}}}
\put(24509,-6081){\makebox(0,0)[lb]{\smash{{\SetFigFont{14}{16.8}{\rmdefault}{\mddefault}{\itdefault}{\color[rgb]{0,0,0}$b_2$}%
}}}}
\put(24100,-6672){\makebox(0,0)[lb]{\smash{{\SetFigFont{14}{16.8}{\rmdefault}{\mddefault}{\itdefault}{\color[rgb]{0,0,0}$b_3$}%
}}}}
\put(23486,-6200){\makebox(0,0)[lb]{\smash{{\SetFigFont{14}{16.8}{\rmdefault}{\mddefault}{\itdefault}{\color[rgb]{0,0,0}$b_4$}%
}}}}
\put(20228,-6554){\makebox(0,0)[lb]{\smash{{\SetFigFont{14}{16.8}{\rmdefault}{\mddefault}{\itdefault}{\color[rgb]{0,0,0}$c_1$}%
}}}}
\put(21047,-6200){\makebox(0,0)[lb]{\smash{{\SetFigFont{14}{16.8}{\rmdefault}{\mddefault}{\itdefault}{\color[rgb]{0,0,0}$c_2$}%
}}}}
\put(21763,-6613){\makebox(0,0)[lb]{\smash{{\SetFigFont{14}{16.8}{\rmdefault}{\mddefault}{\itdefault}{\color[rgb]{0,0,0}$c_3$}%
}}}}
\put(16541,-6200){\makebox(0,0)[lb]{\smash{{\SetFigFont{14}{16.8}{\rmdefault}{\mddefault}{\itdefault}{\color[rgb]{0,0,0}$a_1$}%
}}}}
\put(17359,-6554){\makebox(0,0)[lb]{\smash{{\SetFigFont{14}{16.8}{\rmdefault}{\mddefault}{\itdefault}{\color[rgb]{0,0,0}$a_2$}%
}}}}
\put(18075,-6140){\makebox(0,0)[lb]{\smash{{\SetFigFont{14}{16.8}{\rmdefault}{\mddefault}{\itdefault}{\color[rgb]{0,0,0}$a_3$}%
}}}}
\put(18587,-6554){\makebox(0,0)[lb]{\smash{{\SetFigFont{14}{16.8}{\rmdefault}{\mddefault}{\itdefault}{\color[rgb]{0,0,0}$a_4$}%
}}}}
\put(21699,-4224){\rotatebox{30.0}{\makebox(0,0)[lb]{\smash{{\SetFigFont{14}{16.8}{\rmdefault}{\mddefault}{\itdefault}{\color[rgb]{0,0,0}$b+c$}%
}}}}}
\put(20017,-3791){\rotatebox{90.0}{\makebox(0,0)[lb]{\smash{{\SetFigFont{14}{16.8}{\rmdefault}{\mddefault}{\itdefault}{\color[rgb]{0,0,0}$e_a+o_b+o_c$}%
}}}}}
\put(17904,-4681){\rotatebox{330.0}{\makebox(0,0)[lb]{\smash{{\SetFigFont{14}{16.8}{\rmdefault}{\mddefault}{\itdefault}{\color[rgb]{0,0,0}$a$}%
}}}}}
\put(18324,-7784){\rotatebox{30.0}{\makebox(0,0)[lb]{\smash{{\SetFigFont{14}{16.8}{\rmdefault}{\mddefault}{\itdefault}{\color[rgb]{0,0,0}$a$}%
}}}}}
\put(16580,-7819){\rotatebox{300.0}{\makebox(0,0)[lb]{\smash{{\SetFigFont{14}{16.8}{\rmdefault}{\mddefault}{\itdefault}{\color[rgb]{0,0,0}$z+_a+e_b+e_c+1$}%
}}}}}
\put(20284,-11810){\makebox(0,0)[lb]{\smash{{\SetFigFont{14}{16.8}{\rmdefault}{\mddefault}{\itdefault}{\color[rgb]{0,0,0}$x+e_a+o_b+o_c$}%
}}}}
\put(24679,-10770){\rotatebox{60.0}{\makebox(0,0)[lb]{\smash{{\SetFigFont{14}{16.8}{\rmdefault}{\mddefault}{\itdefault}{\color[rgb]{0,0,0}$z+o_a+e_b+e_c$}%
}}}}}
\put(23942,-2854){\rotatebox{300.0}{\makebox(0,0)[lb]{\smash{{\SetFigFont{14}{16.8}{\rmdefault}{\mddefault}{\itdefault}{\color[rgb]{0,0,0}$z+e_a+o_b+o_c+1$}%
}}}}}
\put(19164,-885){\makebox(0,0)[lb]{\smash{{\SetFigFont{14}{16.8}{\rmdefault}{\mddefault}{\itdefault}{\color[rgb]{0,0,0}$x+o_a+e_b+e_c$}%
}}}}
\put(16289,-3925){\rotatebox{60.0}{\makebox(0,0)[lb]{\smash{{\SetFigFont{14}{16.8}{\rmdefault}{\mddefault}{\itdefault}{\color[rgb]{0,0,0}$z+e_a+o_b+o_c$}%
}}}}}
\put(19396,-6031){\makebox(0,0)[lb]{\smash{{\SetFigFont{14}{16.8}{\rmdefault}{\mddefault}{\updefault}{\color[rgb]{0,0,0}$\lfloor\frac{x+z}{2}\rfloor$}%
}}}}
\put(22351,-6084){\makebox(0,0)[lb]{\smash{{\SetFigFont{14}{16.8}{\rmdefault}{\mddefault}{\updefault}{\color[rgb]{0,0,0}$\lceil\frac{x+z}{2}\rceil$}%
}}}}
\end{picture}%
}
\caption{The region $R^{\swarrow}_{2,1,3}(2,2,1,2; \ 1,1,1,2; 2,2,1)$.}\label{fig:construct4}
\end{figure}

The very special case of our regions when  $\textbf{a}=\textbf{b}=\emptyset$ gives exactly the $F$-cored hexagons in \cite{Ciu1}, and if we specialize further with $\textbf{c}=(m)$, we get the cored hexagons in \cite{CEKZ}. This  is visually apparent when the $y$-parameter of the $F$-cored hexagon (or cored hexagon) is greater than or equal to the $z$-parameter. In the other case, when the $y$-parameter less than the $z$-parameter, we get back the $F$-cored hexagons $F^{\odot}_{x,z,y+2z}(\textbf{c}), F^{\leftarrow}_{x,z,y+2z}(\textbf{c})$,  $F^{\nwarrow}_{x,z,2y+z+1}(\textbf{c})$ and  $F^{\swarrow}_{x,z,2y+1}(\textbf{c})$  (as denoted in \cite{Ciu1}) by reflecting the region $R^{\odot}_{x,y,z}(\emptyset;{}^{0}\textbf{c};\emptyset), R^{\leftarrow}_{x,y,z}(\emptyset;{}^{0}\textbf{c};\emptyset)$,  $R^{\swarrow}_{x,y,z}(\emptyset;{}^{0}\textbf{c};\emptyset)$ and  $R^{\nwarrow}_{x,y,z}(\emptyset;{}^{0}\textbf{c};\emptyset)$ over a horizontal line, respectively. Here we denote ${}^{0}\textbf{s}$ the sequence obtained by including a $0$ term in front of the sequence $\textbf{s}$, i.e ${}^{0}\textbf{s}=(0,s_1,s_2,\dots,s_n)$ if $\textbf{s}=(s_1,s_2,\dots,s_n)$. 


\begin{thm}\label{main1}
Assume that $\textbf{a}=(a_1,a_2,\dotsc,a_m)$, $\textbf{b}:=(b_1,b_2,\dotsc,b_n)$, $\textbf{c}=(c_1,c_2,\dotsc,c_k)$ are three sequences  of nonnegative integers and that $x,y,z$ are three nonnegative integers, such that $x$ and $z$ have the same parity.

(a) If $a\geq b$, then
\begin{align}\label{maineq1a}
\M&(R^{\odot}_{x,y,z}(\textbf{a};\textbf{c};\textbf{b}))=\M(C_{x,2y+z+2a,z}(c))\notag\\
&\times s\left(y,a_1,\dotsc, a_{m},\frac{x+z}{2},c_1,\dotsc,c_{k}+\frac{x+z}{2}+b_n,b_{n-1},\dotsc,b_1\right)\notag\\
&\times s\left(a_1,\dotsc, a_{m-1},a_{m}+\frac{x+z}{2}+c_1,\dotsc,c_{k},\frac{x+z}{2},b_n,\dotsc,b_1,y+a-b\right)\notag\\
&\times \frac{\Hf(c+\frac{x+z}{2})}{\Hf(c)\Hf(\frac{x+z}{2})}\frac{\Hf(a+y+\frac{x+z}{2})}{\Hf(a+c+y+\frac{x+z}{2})}\notag\\
&\times \frac{\Hf(a+y+z)\Hf(a+c+y+z)}{\Hf(e_a+o_b+o_c+y+z)\Hf(a+o_a-o_b+e_c+y+z)}\notag\\
&\times \frac{\Hf(e_a+o_b+o_c+y)\Hf(a+o_a-o_b+e_c+y)}{\Hf(a+y)^2}
\end{align}
if $m,n,k$ are even. The other cases, when one or more numbers among $m,n,k$ are odd, can be reduced to the even case by including an empty triangle at the end of the corresponding ferns.

(b)  If $a< b$, then
\begin{align}\label{maineq1b}
\M&(R^{\odot}_{x,y,z}(\textbf{a};\textbf{c};\textbf{b}))=\M(C_{x,2y+z+2b,z}(c))\notag\\
&\times s\left(y+b-a,a_1,\dotsc, a_{m},\frac{x+z}{2},c_1,\dotsc,c_{k}+\frac{x+z}{2}+b_n,b
_{n-1},\dotsc,b_1\right)\notag\\
&\times s\left(a_1,\dotsc, a_{m-1},a_{m}+\frac{x+z}{2}+c_1,\dotsc,c_{k},\frac{x+z}{2},b_n,\dotsc,b_1,y\right)\notag\\
&\times \frac{\Hf(c+\frac{x+z}{2})}{\Hf(c)\Hf(\frac{x+z}{2})}\frac{\Hf(b+y+\frac{x+z}{2})}{\Hf(b+c+y+\frac{x+z}{2})}\notag\\
&\times \frac{\Hf(b+y+z)\Hf(b+c+y+z)}{\Hf(b+o_b-o_a+o_c+y+z)\Hf(o_a+e_b+e_c+y+z)}\notag\\
&\times \frac{\Hf(b+o_b-o_a+o_c+y)\Hf(o_a+e_b+e_c+y)}{\Hf(b+y)^2},
\end{align}
for even $m,n,k$. The other cases follow the even case in the same way as in part (a).

\end{thm}

The formulas in Theorem \ref{main1} can be combined into a \emph{single} formula as follows:
\begin{align}\label{maineq1c}
\M&(R^{\odot}_{x,y,z}(\textbf{a};\textbf{c};\textbf{b}))=\M(C_{x,2y+z+2\max(a,b),z}(c))\notag\\
&\times s\left(y+b-\min(a,b),a_1,\dotsc, a_{m},\frac{x+z}{2},c_1,\dotsc,c_{k}+\frac{x+z}{2}+b_n,b_{n-1},\dotsc,b_1\right)\notag\\
&\times s\left(a_1,\dotsc, a_{m-1},a_{m}+\frac{x+z}{2}+c_1,\dotsc,c_{k},\frac{x+z}{2},b_n,\dotsc,b_1,y+a-\min(a,b)\right)\notag\\
&\times \frac{\Hf(c+\frac{x+z}{2})}{\Hf(c)\Hf(\frac{x+z}{2})}\frac{\Hf(\max(a,b)+y+\frac{x+z}{2})}{\Hf(\max(a,b)+c+y+\frac{x+z}{2})}\notag\\
&\times \frac{\Hf(\max(a,b)+y+z)\Hf(\max(a,b)+c+y+z)}{\Hf(\max(a,b)-o_a+o_b+o_c+y+z)\Hf(\max(a,b)+o_a-o_b+e_c+y+z)}\notag\\
&\times \frac{\Hf(\max(a,b)-o_a+o_b+o_c+y)\Hf(\max(a,b)+o_a-o_b+e_c+y)}{\Hf(\max(a,b)+y)^2}.
\end{align}

For the sake of brevity, we use similar combined formulas in our next main theorems.

\begin{thm}\label{main2}
Assume that $\textbf{a}=(a_1,a_2,\dotsc,a_m)$, $\textbf{b}:=(b_1,b_2,\dotsc,b_n)$, $\textbf{c}=(c_1,c_2,\dotsc,c_k)$ are three sequences  of nonnegative integers and that $x,y,z$ are three nonnegative integers, such that $x$ has parity opposite to $z$. Then
\begin{align}\label{maineq2a}
\M&(R^{\leftarrow}_{x,y,z}(\textbf{a};\textbf{c};\textbf{b}))=\M(C_{x,2y+z+2\max(a,b),z}(c))\notag\\
&\times s\left(y+b-\min(a,b),a_1,\dotsc, a_{m},\left\lfloor\frac{x+z}{2}\right\rfloor,c_1,\dotsc,c_{k}+\left\lceil\frac{x+z}{2}\right\rceil+b_n,b_{n-1},\dotsc,b_1\right)\notag\\
&\times s\left(a_1,\dotsc, a_{m-1},a_{m}+\left\lfloor\frac{x+z}{2}\right\rfloor+c_1,\dotsc,c_{k},\left\lceil\frac{x+z}{2}\right\rceil,b_n,\dotsc,b_1,y+a-\min(a,b)\right)\notag\\
&\times\frac{\Hf(c+\left\lfloor\frac{x+z}{2}\right\rfloor)}{\Hf(c)\Hf(\left\lfloor\frac{x+z}{2}\right\rfloor)}\frac{\Hf(\max(a,b)+y+\left\lfloor\frac{x+z}{2}\right\rfloor)}{\Hf(\max(a,b)+c+y+\left\lfloor\frac{x+z}{2}\right\rfloor)}\notag\\
&\times \frac{\Hf(\max(a,b)+y+z)\Hf(\max(a,b)+c+y+z)}{\Hf(\max(a,b)-o_a+o_b+o_c+y+z)\Hf(\max(a,b)+o_a-o_b+e_c+y+z)}\notag\\
&\times \frac{\Hf(\max(a,b)-o_a+o_b+o_c+y)\Hf(\max(a,b)+o_a-o_b+e_c+y)}{\Hf(\max(a,b)+y)^2}
\end{align}
if $m,n,k$ are even. The other cases, when one or more numbers among $m,n,k$ are odd, can be reduced to the even case as in Theorem \ref{main1}.
\end{thm}

\begin{thm}\label{main3}
Assume that $\textbf{a}=(a_1,a_2,\dotsc,a_m)$, $\textbf{b}:=(b_1,b_2,\dotsc,b_n)$, $\textbf{c}=(c_1,c_2,\dotsc,c_k)$ are three sequences  of nonnegative integers and that $x,z$ are two nonnegative integers, such that $x$ and $z$ have the same parity. Assume  in addition that $y$ is an integer, such that $y\geq 0$ when $b\leq a$ and $y\geq -1$ when $b>a$. Then
\begin{align}\label{maineq3a}
\M&(R^{\nwarrow}_{x,y,z}(\textbf{a};\textbf{c};\textbf{b}))=\M(C_{x,2y+z+2\max(a,b)+1,z}(c))\notag\\
&\times s\left(y+b-\min(a,b),a_1,\dotsc, a_{m},\frac{x+z}{2},c_1,\dotsc,c_{k}+\frac{x+z}{2}+b_n,b_{n-1},\dotsc,b_1\right)\notag\\
&\times s\left(a_1,\dotsc, a_{m-1},a_{m}+\frac{x+z}{2}+c_1,\dotsc,c_{k},\frac{x+z}{2},b_n,\dotsc,b_1,y+a+1-\min(a,b)\right)\notag\\
&\times\frac{\Hf(c+\frac{x+z}{2})}{\Hf(c)\Hf(\frac{x+z}{2})}\frac{\Hf(\max(a,b)+y+\frac{x+z}{2})}{\Hf(\max(a,b)+c+y+\frac{x+z}{2})}\notag\\
&\times \frac{\Hf(\max(a,b)+y+z+1)\Hf(\max(a,b)+c+y+z)}{\Hf(\max(a,b)-o_a+o_b+o_c+y+z)\Hf(\max(a,b)+o_a-o_b+e_c+y+z+1)}\notag\\
&\times \frac{\Hf(\max(a,b)-o_a+o_b+o_c+y)\Hf(\max(a,b)+o_a-o_b+e_c+y+1)}{\Hf(\max(a,b)+y)\Hf(\max(a,b)+y+1)},
\end{align}
for even $m,n,k$. The other cases, when one or more numbers among $m,n,k$ are odd, can be reduced to the even case as in Theorem \ref{main1}.

\end{thm}

\begin{thm}\label{main4}
Assume that $\textbf{a}=(a_1,a_2,\dotsc,a_m)$, $\textbf{b}:=(b_1,b_2,\dotsc,b_n)$, $\textbf{c}=(c_1,c_2,\dotsc,c_k)$ are three sequences  of nonnegative integers and that $x,z$ are two nonnegative integers, such that $x$ and $z$ have different parities. Assume  in addition that $y$ is an integer, such that $y\geq 0$ when $a\leq b$ and $y\geq -1$ when $a>b$, and that $m,n,k$ are all even (the cases, when at least one of $m,n,k$ is odd, follow by including a $0$-triangle to the end of the ferns if needed). Then

%

\begin{align}\label{maineq4a}
\M&(R^{\swarrow}_{x,y,z}(\textbf{a};\textbf{c};\textbf{b}))=\M(C_{x,2y+z+2\max(a,b)+1,z}(c))\notag\\
&\times s\left(y+1+b-\min(a,b),a_1,\dotsc, a_{m},\left\lfloor\frac{x+z}{2}\right\rfloor,c_1,\dotsc,c_{k}+\left\lceil\frac{x+z}{2}\right\rceil+b_n,b_{n-1},\dotsc,b_1\right)\notag\\
&\times s\left(a_1,\dotsc, a_{m-1},a_{m}+\left\lfloor\frac{x+z}{2}\right\rfloor+c_1,\dotsc,c_{k},\left\lceil\frac{x+z}{2}\right\rceil,b_n,\dotsc,b_1,y+a-\min(a,b)\right)\notag\\
&\times\frac{\Hf(c+\left\lfloor\frac{x+z}{2}\right\rfloor)}{\Hf(c)\Hf(\left\lfloor\frac{x+z}{2}\right\rfloor)}\frac{\Hf(\max(a,b)+y+\left\lceil\frac{x+z}{2}\right\rceil)}{\Hf(\max(a,b)+c+y+\left\lceil\frac{x+z}{2}\right\rceil)}\notag\\
&\times \frac{\Hf(\max(a,b)+y+z)\Hf(\max(a,b)+c+y+z+1)}{\Hf(\max(a,b)-o_a+o_b+o_c+y+z+1)\Hf(\max(a,b)+o_a-o_b+e_c+y+z)}\notag\\
&\times \frac{\Hf(\max(a,b)-o_a+o_b+o_c+y+1)\Hf(\max(a,b)+o_a-o_b+e_c+y)}{\Hf(\max(a,b)+y)\Hf(\max(a,b)+y+1)}.
\end{align}

\end{thm}

\subsection{The case when the west and east vertices of the hexagon are both below $\ell$}

Besides the above four `$R$-families', we have four more `$Q$-families' of regions in which the line $\ell$ containing three ferns stays above the west and the east vertices of the hexagon (as opposed to separating these two vertices as in the case of the $R^{\odot}$-, $R^{\leftarrow}$-, $R^{\nwarrow}$-, $R^{\swarrow}$-type regions).

The definitions of our $Q$-families are illustrated in Figures \ref{fig:constructQ1}--\ref{fig:constructQ4}. For the purpose of our definitions, we ignore all the inner hexagons and the arrows in these figures in the moment. These details will be used later in the alternative definitions of the regions in Subsection 2.4.

Assume that $x,y,z$ are nonnegative integers and that $\textbf{a}=(a_1,\dotsc,a_m)$, $\textbf{b}=(b_1,\dotsc,b_n)$, $\textbf{c}=(c_1,\dotsc,c_k)$ are three sequences of nonnegative integers as usual.

Our  first  $Q$-family is obtained from the  base hexagon $H$ of side-lengths $x+e_a+e_b+e_c, y+z+o_a+o_b+o_c+\max(a-b,0) ,y+z+e_a+e_b+e_c+\max(b-a,0),x+o_a+o_b+o_c,y+z+e_a+e_b+e_c+\max(a-b,0),y+z+o_a+o_b+o_c+\max(b-a,0)$, in which $x$ and $z$ have the same parity (see the outermost hexagon in  Figure \ref{fig:constructQ1}). We remove  at the level $y+\max(a-b,0)$ above the east vertex of the hexagon $H$ three ferns with sequences of side-lengths $\textbf{a}, \textbf{c}, \textbf{b}$ as in the case of the $R^{\odot}$-type regions. The only difference here is that \emph{all} three ferns have now the first triangle up-pointing (note that the right fern still  runs in the opposite direction to the left and the middle ferns, i.e. from right to left). We still arrange the  three ferns so that the left and the right ferns touch the northwest and the northeast sides of the hexagon, respectively, and the  middle fern is located evenly between of the latter ones.  Denote this region by $Q^{\odot}_{x,y,z}(\textbf{a};\ \textbf{c};\  \textbf{b})$.

The second $Q$-family, consisting of the regions $Q^{\leftarrow}_{x,y,z}(\textbf{a};\ \textbf{c};\  \textbf{b})$, is similar to the first one, the only differences are $x$ and $z$ have different parities and the middle fern is now 1-unit closer to the left fern (see Figure \ref{fig:constructQ2}).

\begin{figure}
\setlength{\unitlength}{3947sp}%
\begingroup\makeatletter\ifx\SetFigFont\undefined%
\gdef\SetFigFont#1#2#3#4#5{%
  \reset@font\fontsize{#1}{#2pt}%
  \fontfamily{#3}\fontseries{#4}\fontshape{#5}%
  \selectfont}%
\fi\endgroup%
\resizebox{13cm}{!}{
\begin{picture}(0,0)%
\includegraphics{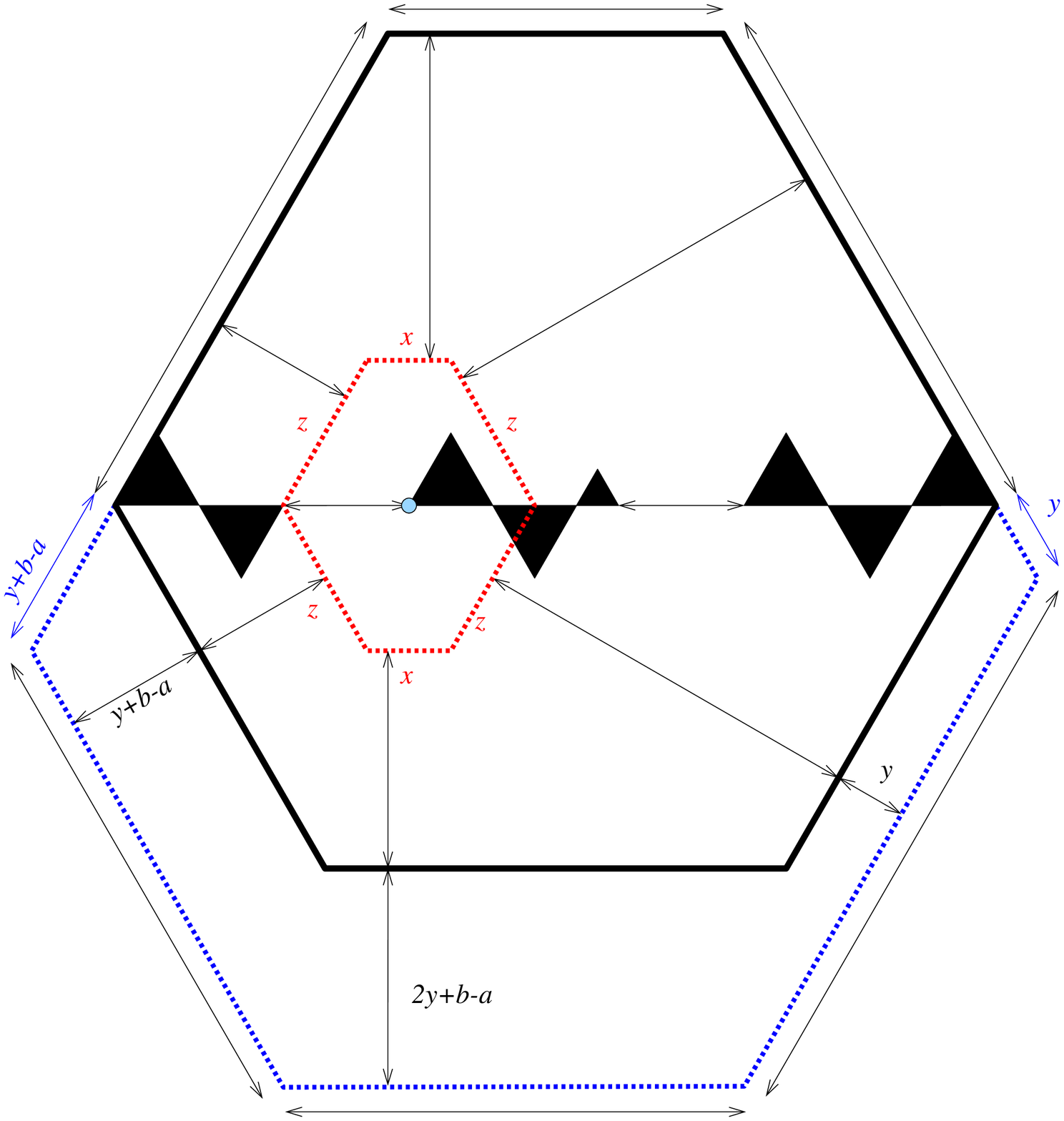}%
\end{picture}%
%
%

\begin{picture}(10619,11535)(2904,-18631)
\put(12170,-12708){\makebox(0,0)[lb]{\smash{{\SetFigFont{14}{16.8}{\rmdefault}{\mddefault}{\itdefault}{$b_1$}%
}}}}
\put(11764,-16464){\rotatebox{60.0}{\makebox(0,0)[lb]{\smash{{\SetFigFont{14}{16.8}{\rmdefault}{\mddefault}{\itdefault}{$y+z+b-a+e_a+e_b+e_c$}%
}}}}}
\put(7201,-18616){\makebox(0,0)[lb]{\smash{{\SetFigFont{14}{16.8}{\rmdefault}{\mddefault}{\itdefault}{$x+o_a+o_b+o_c$}%
}}}}
\put(3376,-14929){\rotatebox{300.0}{\makebox(0,0)[lb]{\smash{{\SetFigFont{14}{16.8}{\rmdefault}{\mddefault}{\itdefault}{$y+z+e_a+e_b+e_c$}%
}}}}}
\put(4424,-10972){\rotatebox{60.0}{\makebox(0,0)[lb]{\smash{{\SetFigFont{14}{16.8}{\rmdefault}{\mddefault}{\itdefault}{$z+o_a+o_b+o_c$}%
}}}}}
\put(4382,-12708){\makebox(0,0)[lb]{\smash{{\SetFigFont{14}{16.8}{\rmdefault}{\mddefault}{\itdefault}{$a_1$}%
}}}}
\put(5206,-12316){\makebox(0,0)[lb]{\smash{{\SetFigFont{14}{16.8}{\rmdefault}{\mddefault}{\itdefault}{$a_2$}%
}}}}
\put(7246,-12646){\makebox(0,0)[lb]{\smash{{\SetFigFont{14}{16.8}{\rmdefault}{\mddefault}{\itdefault}{$c_1$}%
}}}}
\put(8181,-12294){\makebox(0,0)[lb]{\smash{{\SetFigFont{14}{16.8}{\rmdefault}{\mddefault}{\itdefault}{$c_2$}%
}}}}
\put(8693,-12589){\makebox(0,0)[lb]{\smash{{\SetFigFont{14}{16.8}{\rmdefault}{\mddefault}{\itdefault}{$c_3$}%
}}}}
\put(11352,-12294){\makebox(0,0)[lb]{\smash{{\SetFigFont{14}{16.8}{\rmdefault}{\mddefault}{\itdefault}{$b_2$}%
}}}}
\put(10534,-12708){\makebox(0,0)[lb]{\smash{{\SetFigFont{14}{16.8}{\rmdefault}{\mddefault}{\itdefault}{$b_3$}%
}}}}
\put(9380,-13884){\rotatebox{330.0}{\makebox(0,0)[lb]{\smash{{\SetFigFont{14}{16.8}{\rmdefault}{\mddefault}{\itdefault}{$b+c$}%
}}}}}
\put(8747,-10270){\rotatebox{30.0}{\makebox(0,0)[lb]{\smash{{\SetFigFont{14}{16.8}{\rmdefault}{\mddefault}{\itdefault}{$b+c$}%
}}}}}
\put(7468,-7370){\makebox(0,0)[lb]{\smash{{\SetFigFont{14}{16.8}{\rmdefault}{\mddefault}{\itdefault}{$x+e_a+e_b+e_c$}%
}}}}
\put(7366,-10145){\rotatebox{90.0}{\makebox(0,0)[lb]{\smash{{\SetFigFont{14}{16.8}{\rmdefault}{\mddefault}{\itdefault}{$o_a+o_b+o_c$}%
}}}}}
\put(6956,-15460){\rotatebox{90.0}{\makebox(0,0)[lb]{\smash{{\SetFigFont{14}{16.8}{\rmdefault}{\mddefault}{\itdefault}{$e_a+e_b+e_c$}%
}}}}}
\put(5836,-10854){\rotatebox{1.0}{\makebox(0,0)[lb]{\smash{{\SetFigFont{14}{16.8}{\rmdefault}{\mddefault}{\itdefault}{$a$}%
}}}}}
\put(5529,-13688){\rotatebox{1.0}{\makebox(0,0)[lb]{\smash{{\SetFigFont{14}{16.8}{\rmdefault}{\mddefault}{\itdefault}{$a$}%
}}}}}
\put(11355,-9332){\rotatebox{300.0}{\makebox(0,0)[lb]{\smash{{\SetFigFont{14}{16.8}{\rmdefault}{\mddefault}{\itdefault}{$z+o_a+o_b+o_c$}%
}}}}}
\put(6136,-12226){\makebox(0,0)[lb]{\smash{{\SetFigFont{14}{16.8}{\rmdefault}{\mddefault}{\itdefault}{$\frac{x+z}{2}$}%
}}}}
\put(9286,-12316){\makebox(0,0)[lb]{\smash{{\SetFigFont{14}{16.8}{\rmdefault}{\mddefault}{\itdefault}{$\frac{x+z}{2}$}%
}}}}
\end{picture}%
}
  \caption{How to construct the region $Q^{\odot}_{2,2,4}(2,2;\ 2,2,1; \ 2,2,2)$.}\label{fig:constructQ1}
\end{figure}
\begin{figure}
  \centering
  \setlength{\unitlength}{3947sp}%
\begingroup\makeatletter\ifx\SetFigFont\undefined%
\gdef\SetFigFont#1#2#3#4#5{%
  \reset@font\fontsize{#1}{#2pt}%
  \fontfamily{#3}\fontseries{#4}\fontshape{#5}%
  \selectfont}%
\fi\endgroup%
\resizebox{13cm}{!}{
  \begin{picture}(0,0)%
\includegraphics{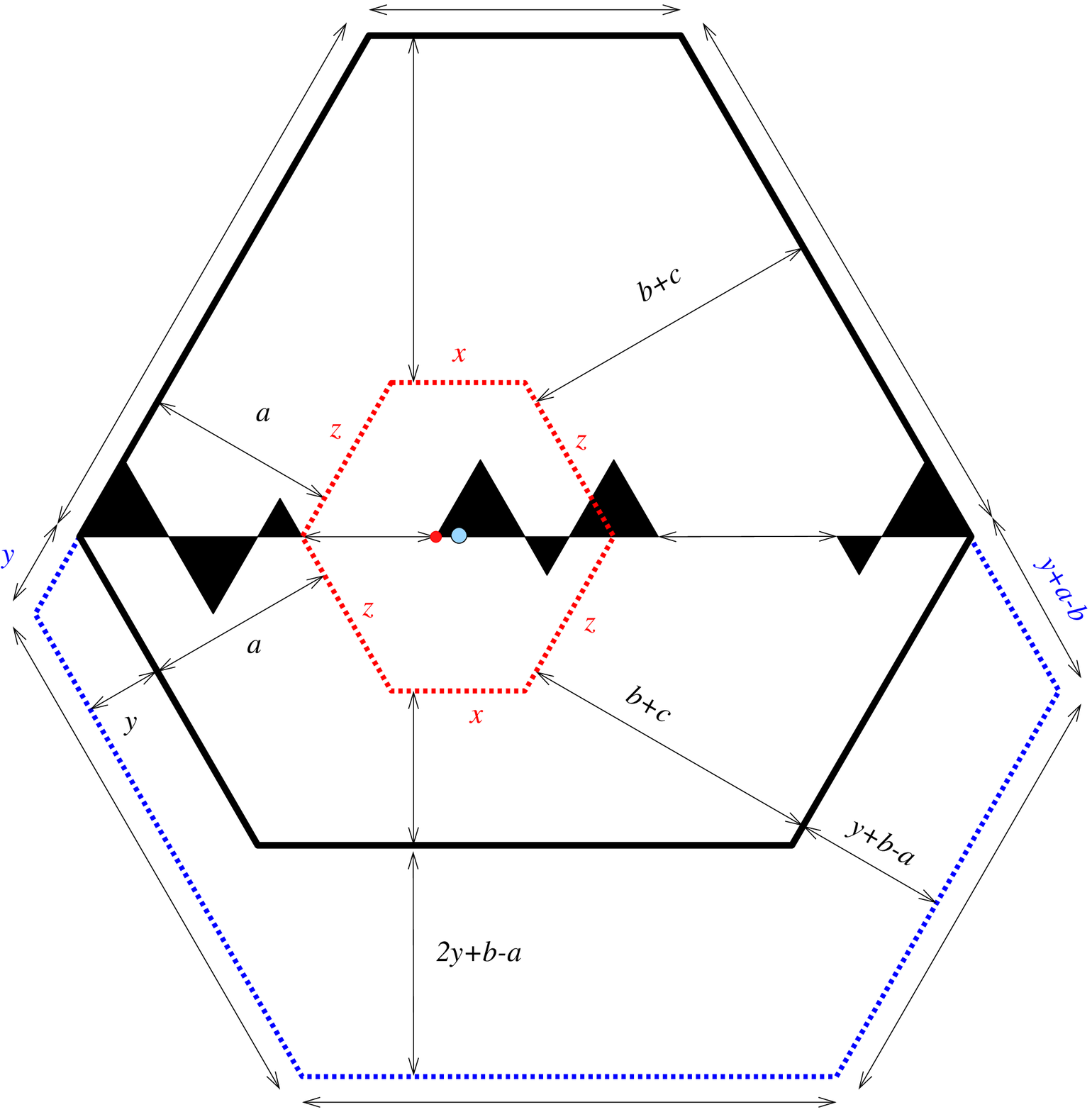}%
\end{picture}%
%
%

\begin{picture}(10178,10924)(15840,-17691)
\put(20746,-17676){\makebox(0,0)[lb]{\smash{{\SetFigFont{14}{16.8}{\rmdefault}{\mddefault}{\itdefault}{$x+o_a+o_b+o_c$}%
}}}}
\put(24489,-16482){\rotatebox{60.0}{\makebox(0,0)[lb]{\smash{{\SetFigFont{14}{16.8}{\rmdefault}{\mddefault}{\itdefault}{$y+z+e_a+e_b+e_c$}%
}}}}}
\put(16143,-14021){\rotatebox{300.0}{\makebox(0,0)[lb]{\smash{{\SetFigFont{14}{16.8}{\rmdefault}{\mddefault}{\itdefault}{$y+z+a-b+e_a+e_b+e_c$}%
}}}}}
\put(19821,-7056){\makebox(0,0)[lb]{\smash{{\SetFigFont{14}{16.8}{\rmdefault}{\mddefault}{\itdefault}{$x+e_a+e_b+e_c$}%
}}}}
\put(23444,-8737){\rotatebox{300.0}{\makebox(0,0)[lb]{\smash{{\SetFigFont{14}{16.8}{\rmdefault}{\mddefault}{\itdefault}{$z+o_a+o_b+o_c$}%
}}}}}
\put(17114,-10180){\rotatebox{60.0}{\makebox(0,0)[lb]{\smash{{\SetFigFont{14}{16.8}{\rmdefault}{\mddefault}{\itdefault}{$z+o_a+o_b+o_c$}%
}}}}}
\put(22391,-11961){\makebox(0,0)[lb]{\smash{{\SetFigFont{14}{16.8}{\rmdefault}{\mddefault}{\itdefault}{$\lceil\frac{x+z}{2}\rceil$}%
}}}}
\put(19081,-11931){\makebox(0,0)[lb]{\smash{{\SetFigFont{14}{16.8}{\rmdefault}{\mddefault}{\itdefault}{$\lfloor\frac{x+z}{2}\rfloor$}%
}}}}
\put(19936,-9766){\rotatebox{90.0}{\makebox(0,0)[lb]{\smash{{\SetFigFont{14}{16.8}{\rmdefault}{\mddefault}{\itdefault}{$o_a+o_b+o_c$}%
}}}}}
\put(19916,-14701){\rotatebox{90.0}{\makebox(0,0)[lb]{\smash{{\SetFigFont{14}{16.8}{\rmdefault}{\mddefault}{\itdefault}{$e_a+e_b+e_c$}%
}}}}}
\put(16921,-12321){\makebox(0,0)[lb]{\smash{{\SetFigFont{14}{16.8}{\rmdefault}{\mddefault}{\itdefault}{$a_1$}%
}}}}
\put(17691,-11976){\makebox(0,0)[lb]{\smash{{\SetFigFont{14}{16.8}{\rmdefault}{\mddefault}{\itdefault}{$a_2$}%
}}}}
\put(18351,-12271){\makebox(0,0)[lb]{\smash{{\SetFigFont{14}{16.8}{\rmdefault}{\mddefault}{\itdefault}{$a_3$}%
}}}}
\put(20186,-12336){\makebox(0,0)[lb]{\smash{{\SetFigFont{14}{16.8}{\rmdefault}{\mddefault}{\itdefault}{$c_1$}%
}}}}
\put(20766,-11966){\makebox(0,0)[lb]{\smash{{\SetFigFont{14}{16.8}{\rmdefault}{\mddefault}{\itdefault}{$c_2$}%
}}}}
\put(21546,-12316){\makebox(0,0)[lb]{\smash{{\SetFigFont{14}{16.8}{\rmdefault}{\mddefault}{\itdefault}{$c_3$}%
}}}}
\put(24246,-12346){\makebox(0,0)[lb]{\smash{{\SetFigFont{14}{16.8}{\rmdefault}{\mddefault}{\itdefault}{$b_1$}%
}}}}
\put(23666,-11956){\makebox(0,0)[lb]{\smash{{\SetFigFont{14}{16.8}{\rmdefault}{\mddefault}{\itdefault}{$b_2$}%
}}}}
\end{picture}}
  \caption{How to construct the region $Q^{\leftarrow}_{3,2,4}(2,2,1;\ 2,1,2;\ 2,1)$.}\label{fig:constructQ2}
\end{figure}
\begin{figure}
\setlength{\unitlength}{3947sp}%
\begingroup\makeatletter\ifx\SetFigFont\undefined%
\gdef\SetFigFont#1#2#3#4#5{%
  \reset@font\fontsize{#1}{#2pt}%
  \fontfamily{#3}\fontseries{#4}\fontshape{#5}%
  \selectfont}%
\fi\endgroup%
\resizebox{13cm}{!}{
  \begin{picture}(0,0)%
\includegraphics{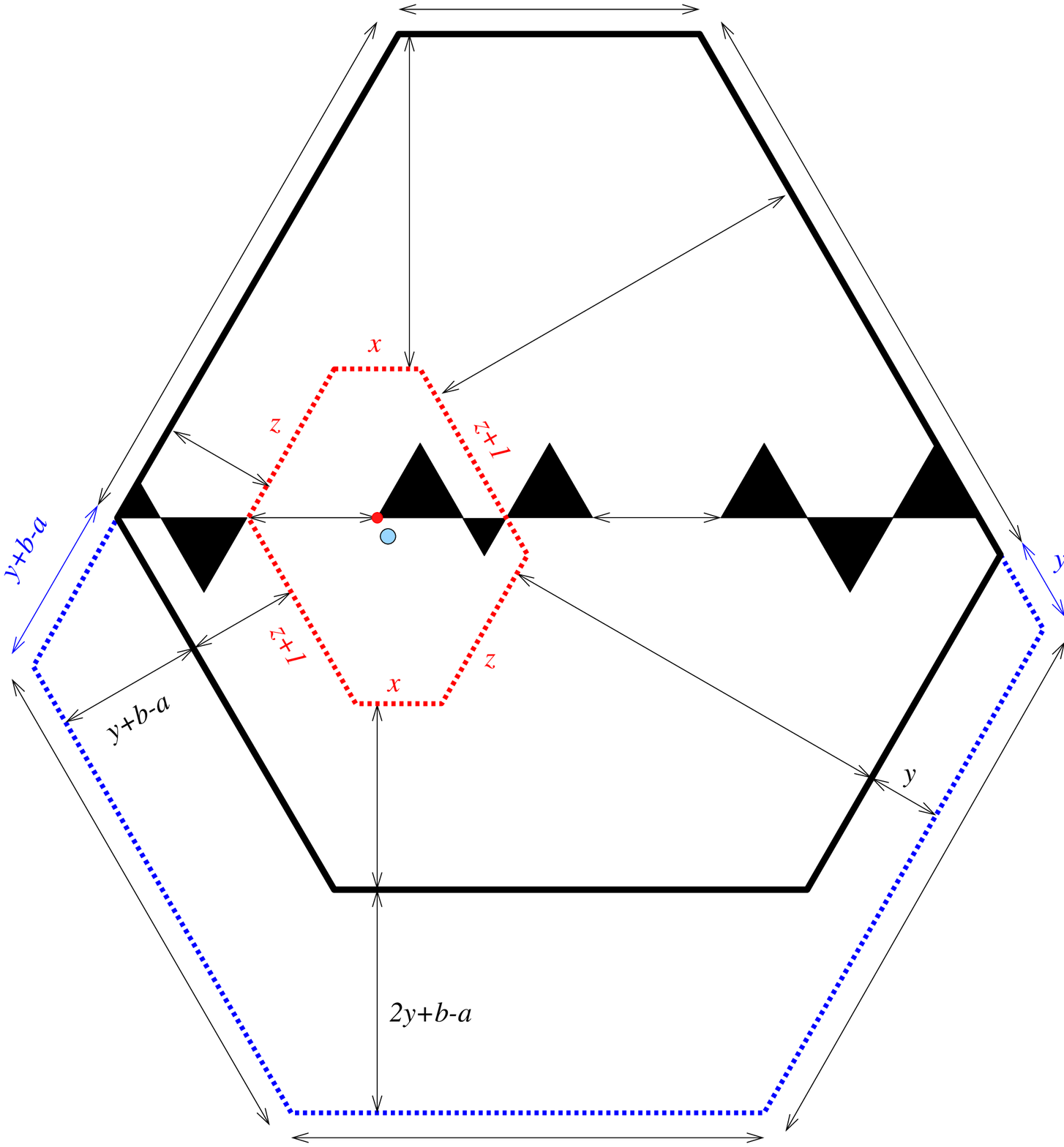}%
\end{picture}%
%
%

\begin{picture}(10391,11528)(3908,-19507)
\put(11951,-13101){\makebox(0,0)[lb]{\smash{{\SetFigFont{14}{16.8}{\rmdefault}{\mddefault}{\itdefault}{$b_2$}%
}}}}
\put(12761,-13521){\makebox(0,0)[lb]{\smash{{\SetFigFont{14}{16.8}{\rmdefault}{\mddefault}{\itdefault}{$b_1$}%
}}}}
\put(11141,-13531){\makebox(0,0)[lb]{\smash{{\SetFigFont{14}{16.8}{\rmdefault}{\mddefault}{\itdefault}{$b_3$}%
}}}}
\put(7850,-16533){\rotatebox{90.0}{\makebox(0,0)[lb]{\smash{{\SetFigFont{14}{16.8}{\rmdefault}{\mddefault}{\itdefault}{$e_a+e_b+e_c$}%
}}}}}
\put(10167,-14389){\rotatebox{330.0}{\makebox(0,0)[lb]{\smash{{\SetFigFont{14}{16.8}{\rmdefault}{\mddefault}{\itdefault}{$b+c$}%
}}}}}
\put(9322,-11111){\rotatebox{30.0}{\makebox(0,0)[lb]{\smash{{\SetFigFont{14}{16.8}{\rmdefault}{\mddefault}{\itdefault}{$b+c$}%
}}}}}
\put(8184,-10736){\rotatebox{90.0}{\makebox(0,0)[lb]{\smash{{\SetFigFont{14}{16.8}{\rmdefault}{\mddefault}{\itdefault}{$o_a+o_b+o_c$}%
}}}}}
\put(5628,-11535){\rotatebox{60.0}{\makebox(0,0)[lb]{\smash{{\SetFigFont{14}{16.8}{\rmdefault}{\mddefault}{\itdefault}{$z+o_a+o_b+o_c$}%
}}}}}
\put(11839,-9776){\rotatebox{300.0}{\makebox(0,0)[lb]{\smash{{\SetFigFont{14}{16.8}{\rmdefault}{\mddefault}{\itdefault}{$z+1+o_a+o_b+o_c$}%
}}}}}
\put(8362,-8271){\makebox(0,0)[lb]{\smash{{\SetFigFont{14}{16.8}{\rmdefault}{\mddefault}{\itdefault}{$x+e_a+e_b+e_c$}%
}}}}
\put(12509,-17822){\rotatebox{60.0}{\makebox(0,0)[lb]{\smash{{\SetFigFont{14}{16.8}{\rmdefault}{\mddefault}{\itdefault}{$y+z+b-a+e_a+e_b+e_c$}%
}}}}}
\put(4195,-15832){\rotatebox{300.0}{\makebox(0,0)[lb]{\smash{{\SetFigFont{14}{16.8}{\rmdefault}{\mddefault}{\itdefault}{$y+z+1+e_a+e_b+e_c$}%
}}}}}
\put(7945,-19492){\makebox(0,0)[lb]{\smash{{\SetFigFont{14}{16.8}{\rmdefault}{\mddefault}{\itdefault}{$x+o_a+o_b+o_c$}%
}}}}
\put(5934,-12541){\makebox(0,0)[lb]{\smash{{\SetFigFont{14}{16.8}{\rmdefault}{\mddefault}{\itdefault}{$a$}%
}}}}
\put(6160,-14059){\makebox(0,0)[lb]{\smash{{\SetFigFont{14}{16.8}{\rmdefault}{\mddefault}{\itdefault}{$a$}%
}}}}
\put(6752,-13098){\makebox(0,0)[lb]{\smash{{\SetFigFont{14}{16.8}{\rmdefault}{\mddefault}{\itdefault}{$\frac{x+z}{2}$}%
}}}}
\put(9921,-13114){\makebox(0,0)[lb]{\smash{{\SetFigFont{14}{16.8}{\rmdefault}{\mddefault}{\itdefault}{$\frac{x+z}{2}$}%
}}}}
\put(5291,-13501){\makebox(0,0)[lb]{\smash{{\SetFigFont{14}{16.8}{\rmdefault}{\mddefault}{\itdefault}{$a_1$}%
}}}}
\put(5811,-13091){\makebox(0,0)[lb]{\smash{{\SetFigFont{14}{16.8}{\rmdefault}{\mddefault}{\itdefault}{$a_2$}%
}}}}
\put(7971,-13461){\makebox(0,0)[lb]{\smash{{\SetFigFont{14}{16.8}{\rmdefault}{\mddefault}{\itdefault}{$c_1$}%
}}}}
\put(8481,-13161){\makebox(0,0)[lb]{\smash{{\SetFigFont{14}{16.8}{\rmdefault}{\mddefault}{\itdefault}{$c_2$}%
}}}}
\put(9161,-13461){\makebox(0,0)[lb]{\smash{{\SetFigFont{14}{16.8}{\rmdefault}{\mddefault}{\itdefault}{$c_3$}%
}}}}
\end{picture}}
  \caption{How to construct the region $Q^{\nwarrow}_{2,2,4}(1,2;\ 2,1,2;\ 2,2,2)$.}\label{fig:constructQ3}
\end{figure}
\begin{figure}\centering
\setlength{\unitlength}{3947sp}%
\begingroup\makeatletter\ifx\SetFigFont\undefined%
\gdef\SetFigFont#1#2#3#4#5{%
  \reset@font\fontsize{#1}{#2pt}%
  \fontfamily{#3}\fontseries{#4}\fontshape{#5}%
  \selectfont}%
\fi\endgroup%
\resizebox{13cm}{!}{
\begin{picture}(0,0)%
\includegraphics{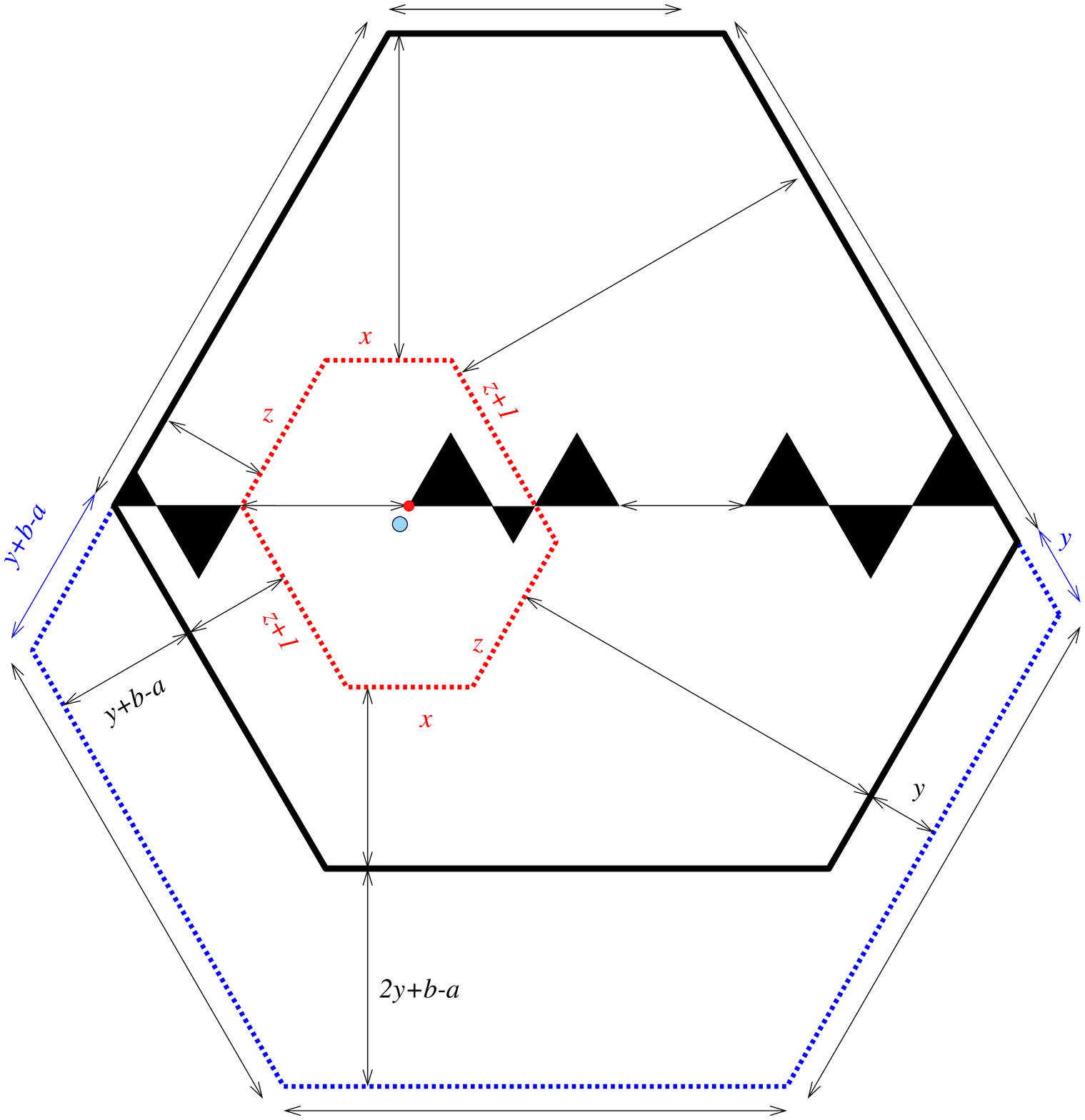}%
\end{picture}%
%
%

\begin{picture}(10630,11528)(3908,-19507)
\put(13169,-13521){\makebox(0,0)[lb]{\smash{{\SetFigFont{14}{16.8}{\rmdefault}{\mddefault}{\itdefault}{$b_1$}%
}}}}
\put(12359,-13101){\makebox(0,0)[lb]{\smash{{\SetFigFont{14}{16.8}{\rmdefault}{\mddefault}{\itdefault}{$b_2$}%
}}}}
\put(11549,-13531){\makebox(0,0)[lb]{\smash{{\SetFigFont{14}{16.8}{\rmdefault}{\mddefault}{\itdefault}{$b_3$}%
}}}}
\put(7850,-16533){\rotatebox{90.0}{\makebox(0,0)[lb]{\smash{{\SetFigFont{14}{16.8}{\rmdefault}{\mddefault}{\itdefault}{$e_a+e_b+e_c$}%
}}}}}
\put(10127,-14513){\rotatebox{330.0}{\makebox(0,0)[lb]{\smash{{\SetFigFont{14}{16.8}{\rmdefault}{\mddefault}{\itdefault}{$b+c$}%
}}}}}
\put(9322,-11111){\rotatebox{30.0}{\makebox(0,0)[lb]{\smash{{\SetFigFont{14}{16.8}{\rmdefault}{\mddefault}{\itdefault}{$b+c$}%
}}}}}
\put(8184,-10736){\rotatebox{90.0}{\makebox(0,0)[lb]{\smash{{\SetFigFont{14}{16.8}{\rmdefault}{\mddefault}{\itdefault}{$o_a+o_b+o_c$}%
}}}}}
\put(5628,-11535){\rotatebox{60.0}{\makebox(0,0)[lb]{\smash{{\SetFigFont{14}{16.8}{\rmdefault}{\mddefault}{\itdefault}{$z+o_a+o_b+o_c$}%
}}}}}
\put(12071,-9808){\rotatebox{300.0}{\makebox(0,0)[lb]{\smash{{\SetFigFont{14}{16.8}{\rmdefault}{\mddefault}{\itdefault}{$z+1+o_a+o_b+o_c$}%
}}}}}
\put(8362,-8271){\makebox(0,0)[lb]{\smash{{\SetFigFont{14}{16.8}{\rmdefault}{\mddefault}{\itdefault}{$x+e_a+e_b+e_c$}%
}}}}
\put(12714,-18058){\rotatebox{60.0}{\makebox(0,0)[lb]{\smash{{\SetFigFont{14}{16.8}{\rmdefault}{\mddefault}{\itdefault}{$y+z+b-a+e_a+e_b+e_c$}%
}}}}}
\put(4195,-15832){\rotatebox{300.0}{\makebox(0,0)[lb]{\smash{{\SetFigFont{14}{16.8}{\rmdefault}{\mddefault}{\itdefault}{$y+z+1+e_a+e_b+e_c$}%
}}}}}
\put(7945,-19492){\makebox(0,0)[lb]{\smash{{\SetFigFont{14}{16.8}{\rmdefault}{\mddefault}{\itdefault}{$x+o_a+o_b+o_c$}%
}}}}
\put(6036,-12541){\makebox(0,0)[lb]{\smash{{\SetFigFont{14}{16.8}{\rmdefault}{\mddefault}{\itdefault}{$a$}%
}}}}
\put(6138,-14161){\makebox(0,0)[lb]{\smash{{\SetFigFont{14}{16.8}{\rmdefault}{\mddefault}{\itdefault}{$a$}%
}}}}
\put(6952,-13098){\makebox(0,0)[lb]{\smash{{\SetFigFont{14}{16.8}{\rmdefault}{\mddefault}{\itdefault}{$\lceil\frac{x+z}{2}\rceil$}%
}}}}
\put(10321,-13114){\makebox(0,0)[lb]{\smash{{\SetFigFont{14}{16.8}{\rmdefault}{\mddefault}{\itdefault}{$\lfloor\frac{x+z}{2}\rfloor$}%
}}}}
\put(5291,-13501){\makebox(0,0)[lb]{\smash{{\SetFigFont{14}{16.8}{\rmdefault}{\mddefault}{\itdefault}{$a_1$}%
}}}}
\put(5811,-13091){\makebox(0,0)[lb]{\smash{{\SetFigFont{14}{16.8}{\rmdefault}{\mddefault}{\itdefault}{$a_2$}%
}}}}
\put(8371,-13461){\makebox(0,0)[lb]{\smash{{\SetFigFont{14}{16.8}{\rmdefault}{\mddefault}{\itdefault}{$c_1$}%
}}}}
\put(8881,-13161){\makebox(0,0)[lb]{\smash{{\SetFigFont{14}{16.8}{\rmdefault}{\mddefault}{\itdefault}{$c_2$}%
}}}}
\put(9402,-13452){\makebox(0,0)[lb]{\smash{{\SetFigFont{14}{16.8}{\rmdefault}{\mddefault}{\itdefault}{$c_3$}%
}}}}
\end{picture}}
  \centering
  \caption{How to construct the region $Q^{\nearrow}_{3,2,4}(1,2;\ 2,1,2;\ 2,2,2)$.}\label{fig:constructQ4}
\end{figure}

We next define the third and the fourth $Q$-families, in which the parameter $y$ is taking the value $-1$  when $b>a$.

To define the third  $Q$-family, we start with  a slightly different base hexagon of side-lengths $x+e_a+e_b+e_c, y+z+o_a+o_b+o_c+\max(a-b,0)+1,y+z+e_a+e_b+e_c+\max(b-a,0),x+o_a+o_b+o_c,y+z+e_a+e_b+e_c+\max(a-b,0)+1,y+z+o_a+o_b+o_c+\max(a-b,0)$, in which $zx$ and $z$ have the same parity (see the outermost hexagons in Figure \ref{fig:constructQ3}). We now remove our three ferns in the same way as in the first $Q$-family at the level $y+\max(b-a,0)$ above the east vertex of the hexagon. Denote by  $Q^{\nwarrow}_{x,y,z}(\textbf{a};\ \textbf{c};\  \textbf{b})$ by the newly defined region.  When $x$ and $z$ have different parities, the fourth $Q$-family is obtained similarly by removing the three ferns from the same base hexagon as in the definition of the $Q^{\nwarrow}$-type regions. However, we now remove the middle fern 1-unit closer to the \emph{right} fern. Denote by $Q^{\nearrow}_{x,y,z}(\textbf{a};\ \textbf{c};\  \textbf{b})$ the resulting region (illustrated in Figure \ref{fig:constructQ4}).

\begin{thm}\label{mainQ1}
Assume that $x,y,z$ are nonnegative integers  and that $\textbf{a}=(a_1,\dotsc,a_m),$  $\textbf{b}=(b_1,\dotsc,b_n)$, $\textbf{c}=(c_1,\dotsc,c_k)$ are sequences of nonnegative integers. Assume in addition that $x$ and $z$ have the same parity and that $m,n,k$ are all even\footnote{Similar to Theorems \ref{main1}--\ref{main4}, for the next enumerations in this paper, we can assume that each of our ferns consists of an even number of triangles (as other cases can be reduced to this case by appending a $0$-triangle to the ferns if needed).}. Then
\begin{align}\label{maineqQ1}
\M&(Q^{\odot}_{x,y,z}(\textbf{a};\textbf{c};\textbf{b}))=\M(C_{x,2y+z+2\max(a,b),z}(c))\notag\\
&\times s\left(a_1,\dotsc, a_{m}+\frac{x+z}{2},c_1,\dotsc,c_{k}+\frac{x+z}{2}+b_n,b_{n-1},\dotsc,b_1\right)\notag\\
&\times s\left(y+b-\min(a,b), a_1,\dotsc, a_{m-1},a_{m},\frac{x+z}{2}+c_1,\dotsc,c_{k},\frac{x+z}{2},b_n,\dotsc,b_1,y+a-\min(a,b)\right)\notag\\
&\times \frac{\Hf(c+\frac{x+z}{2})}{\Hf(c)\Hf(\frac{x+z}{2})}\frac{\Hf(\max(a,b)+y+\frac{x+z}{2})}{\Hf(\max(a,b)+c+y+\frac{x+z}{2})}\notag\\
&\times \frac{\Hf(\max(a,b)+y+z)\Hf(\max(a,b)+c+y+z)}{\Hf(o_a+o_b+o_c+z)\Hf(|a-b|+e_a+e_b+e_c+2y+z)}\notag\\
&\times \frac{\Hf(o_a+o_b+o_c)\Hf(|a-b|+e_a+e_b+e_c+2y)}{\Hf(\max(a,b)+y)^2}.
\end{align}
\end{thm}

\begin{thm}\label{mainQ2}
Assume that $x,y,z$ are nonnegative integers and that $\textbf{a}=(a_1,\dotsc,a_m),$  $\textbf{b}=(b_1,\dotsc,b_n)$, $\textbf{c}=(c_1,\dotsc,c_k)$ are sequences of nonnegative integers. We also assume that $x$ and $z$ have different parities, and that $m,n,k$ are all even. Then
\begin{align}\label{maineqQ2}
\M&(Q^{\leftarrow}_{x,y,z}(\textbf{a};\textbf{c};\textbf{b}))=\M(C_{x,2y+z+2\max(a,b),z}(c))\notag\\
&\times s\left(a_1,\dotsc, a_{m}+\left\lfloor\frac{x+z}{2}\right\rfloor,c_1,\dotsc,c_{k}+\left\lceil\frac{x+z}{2}\right\rceil+b_n,b_{n-1},\dotsc,b_1\right)\notag\\
&\times s\left(y+b-\min(a,b),a_1,\dotsc, ,a_{m},\left\lfloor\frac{x+z}{2}\right\rfloor+c_1,\dotsc,c_{k},\left\lceil\frac{x+z}{2}\right\rceil,b_n,\dotsc,b_1,y+a-\min(a,b)\right)\notag\\
&\times\frac{\Hf(c+\left\lfloor\frac{x+z}{2}\right\rfloor)}{\Hf(c)\Hf(\left\lfloor\frac{x+z}{2}\right\rfloor)}\frac{\Hf(\max(a,b)+y+\left\lfloor\frac{x+z}{2}\right\rfloor)}{\Hf(\max(a,b)+c+y+\left\lfloor\frac{x+z}{2}\right\rfloor)}\notag\\
&\times \frac{\Hf(\max(a,b)+y+z)\Hf(\max(a,b)+c+y+z)}{\Hf(o_a+o_b+o_c+z)\Hf(|a-b|+e_a+e_b+e_c+2y+z)}\notag\\
&\times \frac{\Hf(o_a+o_b+o_c)\Hf(|a-b|+e_a+e_b+e_c+2y)}{\Hf(\max(a,b)+y)^2}.
\end{align}
\end{thm}

\begin{thm}\label{mainQ3}
Assume that $x,z$ are nonnegative integers of the same parity, $y$ is an integer at least $-1$, and $y$ can only take the value $-1$ when $a<b$. Assume in addition that $\textbf{a}=(a_1,\dotsc,a_m),$  $\textbf{b}=(b_1,\dotsc,b_n)$, $\textbf{c}=(c_1,\dotsc,c_k)$ are sequences of an even number of nonnegative integers. Then
\begin{align}\label{maineq3a}
\M&(Q^{\nwarrow}_{x,y,z}(\textbf{a};\textbf{c};\textbf{b}))=\M(C_{x,2y+z+2\max(a,b)+1,z}(c))\notag\\
&\times s\left(a_1,\dotsc, a_{m}+\frac{x+z}{2},c_1,\dotsc,c_{k}+\frac{x+z}{2}+b_n,b_{n-1},\dotsc,b_1\right)\notag\\
&\times s\left(y+b-\min(a,b),a_1,\dotsc,a_{m},\frac{x+z}{2}+c_1,\dotsc,c_{k},\frac{x+z}{2},b_n,\dotsc,b_1,y+a+1-\min(a,b)\right)\notag\\
&\times\frac{\Hf(c+\frac{x+z}{2})}{\Hf(c)\Hf(\frac{x+z}{2})}\frac{\Hf(\max(a,b)+y+\frac{x+z}{2})}{\Hf(\max(a,b)+c+y+\frac{x+z}{2})}\notag\\
&\times \frac{\Hf(\max(a,b)+y+z+1)\Hf(\max(a,b)+c+y+z)}{\Hf(o_a+o_b+o_c+z)\Hf(|a-b|+e_a+e_b+e_c+2y+z+1)}\notag\\
&\times \frac{\Hf(o_a+o_b+o_c)\Hf(|a-b|+e_a+e_b+e_c+2y+1)}{\Hf(\max(a,b)+y)\Hf(\max(a,b)+y+1)}.
\end{align}
\end{thm}

\begin{thm}\label{mainQ4}
Assume that $x,z$ are nonnegative integers of different parities, $y$ is an integer at least $-1$, and $y$ can only take the value $-1$ when $a<b$. Assume in addition that $\textbf{a}=(a_1,\dotsc,a_m),$  $\textbf{b}=(b_1,\dotsc,b_n)$, $\textbf{c}=(c_1,\dotsc,c_k)$ are sequences of an even number of nonnegative integers. Then
\begin{align}\label{maineq3a}
\M&(Q^{\nearrow}_{x,y,z}(\textbf{a};\textbf{c};\textbf{b}))=\M(C_{x,2y+z+2\max(a,b)+1,z}(c))\notag\\
&\times s\left(a_1,\dotsc, a_{m}+\left\lceil\frac{x+z}{2}\right\rceil,c_1,\dotsc,c_{k}+\left\lfloor\frac{x+z}{2}\right\rfloor+b_n,b_{n-1},\dotsc,b_1\right)\notag\\
&\times s\left(y+b-\min(a,b),a_1,\dotsc, a_{m},\left\lceil\frac{x+z}{2}\right\rceil+c_1,\dotsc,c_{k},\left\lfloor\frac{x+z}{2}\right\rfloor,b_n,\dotsc,b_1,y+a-\min(a,b)+1\right)\notag\\
&\times\frac{\Hf(c+\left\lfloor\frac{x+z}{2}\right\rfloor)}{\Hf(c)\Hf(\left\lfloor\frac{x+z}{2}\right\rfloor)}\frac{\Hf(\max(a,b)+y+\left\lceil\frac{x+z}{2}\right\rceil)}{\Hf(\max(a,b)+c+y+\left\lceil\frac{x+z}{2}\right\rceil)}\notag\\
&\times \frac{\Hf(\max(a,b)+y+z)\Hf(\max(a,b)+c+y+z+1)}{\Hf(o_a+o_b+o_c+z)\Hf(|a-b|+e_a+e_b+e_c+2y+z+1)}\notag\\
&\times \frac{\Hf(o_a+o_b+o_c)\Hf(|a-b|+e_a+e_b+e_c+2y+1)}{\Hf(\max(a,b)+y)\Hf(\max(a,b)+y+1)}.
\end{align}
\end{thm}

\medskip

One readily sees that when the middle fern is empty, then our eight regions (4 $R$-regions and 4 $Q$-regions) become special cases of the `\emph{doubly-intruded hexagons}'  in \cite{CL}. More precisely, the regions in \cite{CL} depend on four parameters $x,y,z,t$, besides the two ferns, and the our regions here only depend on three parameters $x,y,z$. Moreover, the $q$-enumeration in \cite{CL} does not appear in our regions.

\subsection{Alternative definitions of the $R$-and $Q$-families}

The above direct definitions of the $R$- and $Q$-families are straightforward, however,  to see more clearly that our regions are  common generalizations of the cored hexagons in \cite{CEKZ} and the $F$-cored hexagons  in \cite{Ciu1}, we give an equivalent constructive definition as follows.

We start with an auxiliary hexagon $H_0$ of side-lengths $x,z,z,x,z,z$ (see the inner hexagon with the dashed contour in Figures \ref{fig:construct1} and \ref{fig:construct2}). Next, we push out all six sides of $H_0$ as follows. We push the north, northeast, southeast, south, southwest, and northwest sides  of $H_0$ outward $e_a+o_b+o_c$, $b+c$, $b+c$, $o_a+e_b+e_c$, $a,$ $a$ units, respectively. We obtain the hexagon $H_1$ with side-lengths $x+o_a+e_b+e_c,$  $z+e_a+o_b+e_c$,  $z+o_a+e_b+e_c,$ $x+e_a+o_b+e_c$, $z+o_a+e_b+e_c,$ $z+e_a+o_b+e_c$ (indicated by the hexagon with the solid bold contour in the above figures).

 If the total length of  the left fern is greater than or equal to the total length of the right fern, i.e. $a> b$, we push in addition the south, southeast, north, and northwest sides of the hexagon $H_1$  respectively $y+a-b$, $y+a-b$, $y$, and $y$ units outward; otherwise, if $a\leq b$, we push out these sides $y$, $y$, $y+b-a$, and $y+b-a$ units, respectively. This way the hexagon $H_1$ is extended to the base hexagon $H$ of side-lengths $x+o_a+e_b+e_c,$  $y+e_a+o_b+e_c+ |a-b|$,  $z+o_a+e_b+e_c,$ $x+e_a+o_b+e_c$, $y+o_a+e_b+e_c+ |a-b|,$ $z+e_a+o_b+e_c$ as in the direct definition of the regions above (the extension of $H_1$ is indicated by the portion with the dashed boundary in  Figures \ref{fig:construct1} and \ref{fig:construct2}).

Finally, we remove the middle fern, consisting of triangles of side-lengths $c_i$'s, such the leftmost of the fern is exactly at the center of the of the auxiliary hexagon $H_0$ if $x$ and $z$ have the same parity, or is $1/2$ unit to the left of the center of $H_0$ in the case when $x$ and $z$ have opposite parities. The left fern and the right fern are removed on the same level as the middle fern, such that the leftmost of the left fern and the rightmost of the right fern touch the boundary of the hexagon. This gives us the regions $R_{x,y,z}^{\odot}(\textbf{a};\textbf{c};\textbf{b})$ and $R_{x,y,z}^{\leftarrow}(\textbf{a};\textbf{c};\textbf{b})$, respectively.

 To define the third and the fourth $R$-families, we start instead with an auxiliary hexagon of side-lengths $x,z+1,z,x,z+1,z$
  (see the inner hexagons in Figures \ref{fig:construct3} and \ref{fig:construct4}). We still perform the above 2-stage pushing process to obtain the base hexagon of side-lengths
  $x+o_a+e_b+e_c,$  $y+e_a+o_b+e_c+ |a-b|+1$,  $z+o_a+e_b+e_c,$ $x+e_a+o_b+e_c$, $y+o_a+e_b+e_c+ |a-b|+1,$ $z+e_a+o_b+e_c$. As mentioned in the direct definitions in Subsection 2.2, in the case when the region $R_{x,y,z}^{\nwarrow}(\textbf{a};\textbf{c};\textbf{b})$ has $b>a$
   and in the case when the region $R_{x,y,z}^{\swarrow}(\textbf{a};\textbf{c};\textbf{b})$ has $b<a$, we allow $y$ to take the \emph{negative} value $-1$. Here, we understand that pushing outward `$-1$ unit' is equivalent to pushing inward $1$ unit.
   We obtain the region  $R_{x,y,z}^{\nwarrow}(\textbf{a};\textbf{c};\textbf{b})$ or the region $R_{x,y,z}^{\swarrow}(\textbf{a};\textbf{c};\textbf{b})$ if the middle fern is placed $1/2$ unit to the northwest or $1/2$ unit to the southwest of the center of the auxiliary hexagon $H_0$
   (corresponding to the case when $x$ and $z$ have the same parity or the case when they have different parities).

   We note that this constructive definition of our regions also explains the use of the super scripts $\odot, \leftarrow, \nwarrow, \swarrow$ in our notations. These super scripts clarify the relative position of the leftmost of the  middle fern to the center of the auxiliary hexagon $H_0$.  We have adopted these notations  in \cite{Ciu1}.
   
  \begin{rmk}\label{rmk1}
In the above constructive definition of the second, the third and the fourth $R$-families, there are actually three more families of regions corresponding to the case when the leftmost of the middle fern is located $1/2$ unit to the east, the southeast, or the northeast of the center of the auxiliary hexagon $H_0$. However, we do \emph{not} consider in detail these regions here, as they can be viewed as $180^{\circ}$ rotations of our three $R$-families.
   \end{rmk}
   
   \medskip

Next, we provide the constructive definitions for the four $Q$-families.

The construction of the regions $Q^{\odot}_{x,y,z}(\textbf{a};\ \textbf{c};\  \textbf{b})$ and $Q^{\leftarrow}_{x,y,z}(\textbf{a};\ \textbf{c};\  \textbf{b})$ are shown in Figures \ref{fig:constructQ1} and \ref{fig:constructQ2}, respectively. We start with an auxiliary hexagon $H_0$ of side-lengths $x,z,z,x,z,z$ (illustrated by the inner hexagons with the dashed boundary), and we push out all the sides (in clockwise order from the north side) of this hexagon  by $o_a+o_b+o_c, b+c,b+c,e_a+e_b+e_c,a,a$ units, respectively. This way, we get a larger hexagon $H_1$ of side-lengths $x+e_a+e_b+e_c, z+o_a+o_b+o_c, z+e_a+e_b+e_c, x+o_a+o_b+o_c,z+e_a+e_b+e_c, z+o_a+o_b+o_c$ (shown as the hexagon with the bold solid boundary). The second pushing depends on whether $a\geq b$ or $b> a$. If $a\leq b$, we push out the southeast, the south and the southwest sides of the hexagon $H_1$ respectively $y, y+b-a,y+b-a$ units; otherwise we push these sides respectively $y+a-b,y+a-b,y$ units (these are indicated by the portion with the dashed boundary outside the solid contour). If $x$ and $z$ have the same parity, i.e. the center of the auxiliary hexagon $H_0$ is a lattice vertex, we arrange the middle fern so that its leftmost point is exactly at the center, the  left and the right ferns are located at the same level, such that they touch the northwest and the northeast sides of the hexagon, respectively (see Figure  \ref{fig:constructQ1}). The resulting region is exactly the region
$Q^{\odot}_{x,y,z}(\textbf{a};\ \textbf{c};\  \textbf{b})$ defined above. In the case when $x$ has parity opposite to $z$, we arrange the middle fern $1/2$ unit to the left of the center of the auxiliary hexagon $H_0$  (the left and right ferns are still lined up in the same way as in the definition of the $R$-type regions) and obtain the region $Q^{\leftarrow}_{x,y,z}(\textbf{a};\ \textbf{c};\  \textbf{b})$ (see Figure  \ref{fig:constructQ2}).

Next, the construction of the regions $Q^{\nwarrow}_{x,y,z}(\textbf{a};\ \textbf{c};\  \textbf{b})$ and $Q^{\nearrow}_{x,y,z}(\textbf{a};\ \textbf{c};\  \textbf{b})$ are shown in Figures \ref{fig:constructQ3} and \ref{fig:constructQ4}, respectively. We are allowing $y$ to take the value $-1$  when $b>a$ with the convention that: pushing outward `$-1$ unit' is exactly pushing inward $1$ unit. We now start with a different auxiliary hexagon $H_0$ of side-lengths $x,z+1,z,z+1,z$ . We still perform the same 2-stage pushing process as above to obtain the base hexagon $H$. We now choose the middle fern, such that its leftmost point is $1/2$ unit to the northwest of the center of the auxiliary hexagon if $x$ and $z$ have the same parity; otherwise, we put the middle fern $1/2$ unit to the northeast of the center of the auxiliary hexagon $H_0$ (the other two ferns are still chosen in the same way as in the $Q^{\odot}$- and $Q^{\leftarrow}$-type regions above). We have here the regions $Q^{\nwarrow}_{x,y,z}(\textbf{a};\ \textbf{c};\  \textbf{b})$ and $Q^{\nearrow}_{x,y,z}(\textbf{a};\ \textbf{c};\  \textbf{b})$, respectively.

\subsection{Dual of MacMahon's theorem on plane partitions}

\begin{figure}\centering
\includegraphics[width=10cm]{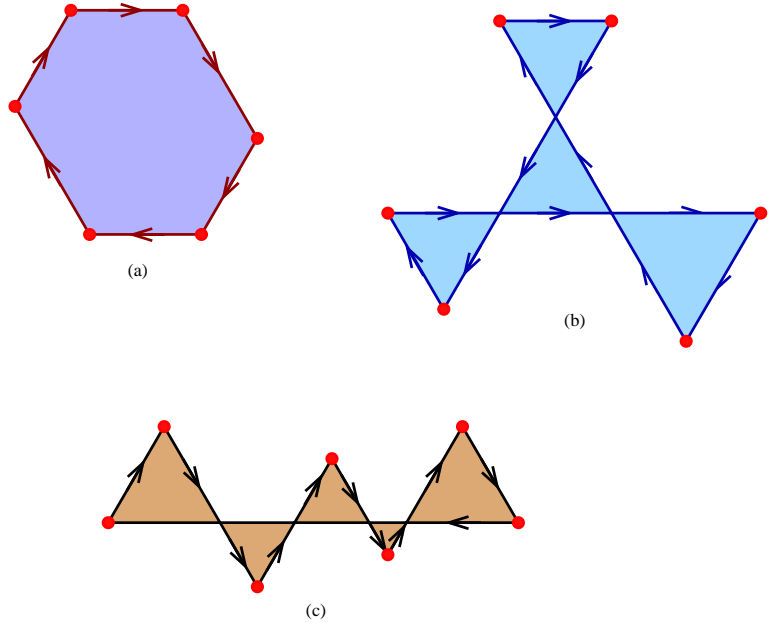}
\caption{(a) The boundary of the hexagon in MacMahon's theorem \cite{Mac}. (b) The boundary of the concave hexagon in the dual of MacMahon's theorem \cite{CK}. (c) The boundary of the concave polygon in the dual of MacMahon's theorem \cite{Ciu1}.}\label{ContourDual}
\end{figure}

MacMahon's classical theorem on boxed plane partitions \cite{Mac} yields the beautiful product formula (\ref{MacMahoneq}) for the number of lozenge tilings of the interior  of a convex hexagon on the triangular lattice obtained by turning $60^{\circ}$ after drawing each side (see Figure \ref{ContourDual}(a)). In \cite{CK}, Ciucu and Krattenthaler considered a counterpart of MacMahon's theorem, corresponding to the \emph{exterior} of a concave hexagon by turning $120^{\circ}$ after drawing each side. In particular, they consider the asymptotic behavior of the ratio between tiling number of a regular $S$-cored hexagon and tiling number of a normalized version of the $S$-cored hexagon (see Figure \ref{ContourDual}(b)). Based on their explicit tiling formula of a $S$-cored hexagon, Ciucu and Krattenthaler showed that the later ratio tends to a product of two instances of MacMahon's product (\ref{MacMahoneq}) (see Theorem 1.1 in \cite{CK}). They called this striking asymptotic result a `\emph{dual}' of MacMahon's theorem. Ciucu later obtained another dual of MacMahon's theorem (see Theorem 1.1 in \cite{Ciu1}), corresponding to the exterior of a concave polygon obtained by turning $120^{\circ}$ after drawing each side (see Figure \ref{ContourDual}(c)). More precisely, using his explicit tiling enumeration for $F$-cored hexagons, Ciucu showed that the ratio between tiling numbers of a $F$-cored hexagon and a normalized version of this $F$-cored hexagon tends to a nice product formula. Interestingly, this formula is a product of two instances of Cohn--Larsen--Propp's product formula (\ref{semieq}), which in turn can be considered as a generalization of MacMahon's formula (\ref{MacMahoneq}).

In this subsection, we use our tiling formulas for the $R^{\odot}$- and $R^{\leftarrow}$-type regions above to obtain a new dual of MacMahon's theorem. Our dual corresponds to the exterior of the union of three concave polygons that are similar to that in Ciucu's dual in \cite{Ciu1}. 

Let $x,z$ be fixed positive real numbers, and let $\textbf{a}=(a_1,\dotsc,a_m)$, $\textbf{c}=(c_1,\dotsc,c_k)$, $\textbf{b}=(b_1,\dotsc,b_n)$ be three fixed sequences of nonnegative integers, such that $a=\sum_{i}a_i=\sum_{j}b_j=b$. We consider the behavior of the ratio between the numbers of tilings of the two $R$-regions $R_{\lfloor xN\rfloor ,0,\lfloor zN\rfloor}(\textbf{a}; \textbf{c}; \textbf{b})$ and $R_{\lfloor xN\rfloor,0,\lfloor zN\rfloor}(e_a,o_a;\  e_c, o_c;\  e_b,o_b)$,
where
\begin{equation}
R_{x,y,z}(\textbf{a}; \textbf{c}; \textbf{b}):=
\begin{cases}
R^{\odot}_{x,y,z}(\textbf{a};\textbf{c};\textbf{b}) &\text{if $x$ and $z$ have the same parity}\\
R^{\leftarrow}_{x,y,z}(\textbf{a};\textbf{c};\textbf{b}) &\text{if $x$ has parity opposite to $z$.}
\end{cases}
\end{equation}
We show that this ratio tends to a product of \emph{six} instances of Cohn--Larsen--Propp's product formula, as $N$ gets large.

\begin{thm}\label{asymthm}
For three given sequences of nonnegative integers $\textbf{a}=(a_1,\dotsc,a_m)$, $\textbf{c}=(c_1,\dotsc,c_k)$, $\textbf{b}=(b_1,\dotsc,b_n)$, such that $a=b$,  and for two positive numbers $x,z$, we have
\begin{align}\label{dualeq}
\lim_{N\to \infty}\frac{\M(R_{\lfloor xN\rfloor ,0,\lfloor zN\rfloor}(\textbf{a}; \textbf{c}; \textbf{b}))}{\M(R_{\lfloor xN\rfloor,0,\lfloor zN\rfloor}(e_a,o_a;\  e_c, o_c;\  e_b,o_b))}=&s(a_1,\dotsc,a_{m-1})s(a_2,\dotsc,a_m)s(b_1,\dotsc,b_{n-1})s(b_2,\dotsc,b_{n})\notag\\
&\times s(c_1,\dotsc,c_{k-1})s(c_2,\dotsc,c_{k}).
\end{align}
Recall that $s(a_1,\dotsc,a_{n})$ denotes the tiling number of the dented semihexagon $S(a_1,\dotsc,a_{n})$ defined in (\ref{semieq}). The above theorem can be visualized as in the Figure \ref{fig:geointer}.
\end{thm}

\begin{figure}\centering
\setlength{\unitlength}{3947sp}%
\begingroup\makeatletter\ifx\SetFigFont\undefined%
\gdef\SetFigFont#1#2#3#4#5{%
  \reset@font\fontsize{#1}{#2pt}%
  \fontfamily{#3}\fontseries{#4}\fontshape{#5}%
  \selectfont}%
\fi\endgroup%
\resizebox{8cm}{!}{
\begin{picture}(0,0)%
\includegraphics{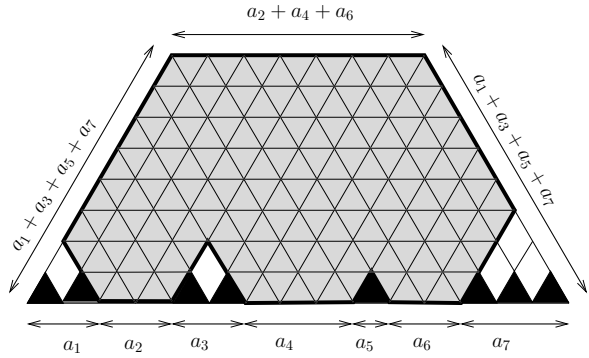}%
\end{picture}%
%
%

\begin{picture}(6833,4101)(1915,-5731)
\put(4741,-1911){\makebox(0,0)[lb]{\smash{{\SetFigFont{14}{16.8}{\rmdefault}{\mddefault}{\updefault}{\color[rgb]{0,0,0}$a_2+a_4+a_6$}%
}}}}
\put(7303,-2601){\rotatebox{300.0}{\makebox(0,0)[lb]{\smash{{\SetFigFont{14}{16.8}{\rmdefault}{\mddefault}{\updefault}{\color[rgb]{0,0,0}$a_1+a_3+a_5+a_7$}%
}}}}}
\put(2161,-4601){\rotatebox{60.0}{\makebox(0,0)[lb]{\smash{{\SetFigFont{14}{16.8}{\rmdefault}{\mddefault}{\updefault}{\color[rgb]{0,0,0}$a_1+a_3+a_5+a_7$}%
}}}}}
\put(2660,-5716){\makebox(0,0)[lb]{\smash{{\SetFigFont{14}{16.8}{\rmdefault}{\mddefault}{\updefault}{\color[rgb]{0,0,0}$a_1$}%
}}}}
\put(3351,-5711){\makebox(0,0)[lb]{\smash{{\SetFigFont{14}{16.8}{\rmdefault}{\mddefault}{\updefault}{\color[rgb]{0,0,0}$a_2$}%
}}}}
\put(4101,-5681){\makebox(0,0)[lb]{\smash{{\SetFigFont{14}{16.8}{\rmdefault}{\mddefault}{\updefault}{\color[rgb]{0,0,0}$a_3$}%
}}}}
\put(5061,-5681){\makebox(0,0)[lb]{\smash{{\SetFigFont{14}{16.8}{\rmdefault}{\mddefault}{\updefault}{\color[rgb]{0,0,0}$a_4$}%
}}}}
\put(5971,-5691){\makebox(0,0)[lb]{\smash{{\SetFigFont{14}{16.8}{\rmdefault}{\mddefault}{\updefault}{\color[rgb]{0,0,0}$a_5$}%
}}}}
\put(6621,-5681){\makebox(0,0)[lb]{\smash{{\SetFigFont{14}{16.8}{\rmdefault}{\mddefault}{\updefault}{\color[rgb]{0,0,0}$a_6$}%
}}}}
\put(7521,-5681){\makebox(0,0)[lb]{\smash{{\SetFigFont{14}{16.8}{\rmdefault}{\mddefault}{\updefault}{\color[rgb]{0,0,0}$a_7$}%
}}}}
\end{picture}%
}
\caption{Obtaining the region $S(2,2,2,3,1,2,4)$ (the shaded region with the bold contour) from the region $T_{7,8}(1,2,5,6,10,13,14,15)$ by removing several vertical forced lozenges; the black triangles indicate the unit triangle removed in the region $T_{7,8}(1,2,5,6,10,13,14,15)$.}\label{semihexmultiple2}
\end{figure}

\begin{proof}
First, by Theorems \ref{main1} and \ref{main2}, we have
\begin{align}\label{dualeq2}
\lim_{N\to \infty}\frac{\M(R_{\lfloor xN\rfloor ,0,\lfloor zN\rfloor}(\textbf{a}; \textbf{c}; \textbf{b}))}
{\M(R_{\lfloor xN\rfloor,0,\lfloor zN\rfloor}(e_a,o_a;\  e_c, o_c;\  e_b,o_b))}=\lim_{N\to \infty}\frac{\M(S^+)\M(S^-)}{\M(\overline{S}^+)\M(\overline{S}^-)},
\end{align}
Here $S^+$ and $S^-$ are the two  dented semihexagons whose dents obtained by dividing the region $R_{\lfloor xN\rfloor ,0,\lfloor zN\rfloor}(\textbf{a}; \textbf{c}; \textbf{b})$
along the line that our three ferns are resting on ($S^+$ denotes the upper semihexagon, and $S^-$ denotes the lower semihexagon).
Similarly, $\overline{S}^+$ and $\overline{S}^-$ denote the two dented
 semihexagons corresponding to the region $R_{\lfloor xN\rfloor,0,\lfloor zN\rfloor}(e_a,o_a;\  e_c, o_c;\  e_b,o_b)$.

For two ordered sets $E=(s_1,\dotsc, s_m)$ and $E'=(s'_1,\dotsc, s'_n)$, we define the operator $\Delta$ as follows
$\Delta(E)=\prod_{i<j} (s_j-s_i)$, and $\Delta(E,E')=\prod_{i,j}(s'_j-s_i)$. We also use the notation $[a,b]$ to indicate the set of all integers $x$, such that $a\leq x\leq b$, and $y+[a,b]:=[a+y,b+y]$. Finally, we use the notation $\tau_i(\textbf{a})$ for the $i$-th partial sum of the sequence $\textbf{a}=(a_1,a_2,\dotsc,a_m)$, i.e. $\tau_i(\textbf{a})=\sum_{j=1}^{i}a_j$.

 We only need
 to show that the ratio on the right-hand side of (\ref{dualeq2}) tends to the product of the tiling numbers of the six dented semihexagons on the right-hand side of (\ref{dualeq}), as $N$ gets large.
 To do so, we use Cohn--Larsen--Propp's original formula for the number of tilings of a semihexagon with dents as mentioned
in the footnote on page 3. In particular, each semihexagon $S(a_1,a_2,\dots,a_m)$ is obtained from the region $T_{o_a,e_a}\left( \bigcup_{i\geq 1} [\tau_{2i-1}(\textbf{a})+1,\tau_{2i}(\textbf{a})] \right)$ by removing several forced vertical lozenges (see Figure \ref{semihexmultiple2}). Therefore, the two regions have the same number of tilings. Recall that $T_{m,n}(x_1,x_2,\dots,x_n)$ is the region obtained from the semihexagon of side-lengths $m,n,m+n,n$ (clockwise from the top) by removing $n$ up-pointing unit triangles from its bottom that are in the positions $x_1,x_2,\dots,x_n$ as counted from left to right, and that the number of tilings of  $T_{m,n}(x_1,x_2,\dots,x_n)$ is given by the product $\prod_{1\leq i<j \leq n}\frac{x_j-x_i}{j-i}$.

We first consider $S^+$. It has the same number of tilings as the semihexagon $T_{x+e_a+o_b+o_c,z+o_a+e_b+e_c}(A\cup B\cup C)$, where
\[A=\bigcup_{i\geq 1}[\tau_{2i-1}(\textbf{a})+1,\tau_{2i}(\textbf{a})],\]
\[C=\bigcup_{i\geq 1}\left(a+\left\lfloor\frac{\lfloor xN\rfloor+\lfloor zN\rfloor}{2}\right\rfloor\right)+[\tau_{2i-1}(\textbf{c})+1,\tau_{2i}(\textbf{c})],\]
\[B=\bigcup_{i\geq 1}\left(a+c+\lfloor xN\rfloor+\lfloor zN\rfloor\right)+[\tau_{2i-1}(\textbf{b})+1,\tau_{2i}(\textbf{b})].\]

It means that $A,B,C$ are the position sets corresponding to the  up-pointing triangles in the left, the right and the middle ferns, respectively. For convenience, assume that
$\alpha_1,\dotsc, \alpha_{e_a}$ are the positions in set $A$, $(a+c+\lfloor xN\rfloor+\lfloor zN\rfloor)+\beta_{1},\dotsc, (a+c+\lfloor xN\rfloor+\lfloor zN\rfloor)+\beta_{o_b}$ are the positions in $B$, and $(a+\left\lfloor\frac{\lfloor xN\rfloor+\lfloor zN\rfloor}{2}\right\rfloor)+\gamma_{1},\dotsc, (a+\left\lfloor\frac{\lfloor xN\rfloor+\lfloor zN\rfloor}{2}\right\rfloor)+\gamma_{o_c}$ are the positions in $C$. Similarly, we denote by
\[A'=[1,e_a],\]
\[C'=a+[1,o_c],\]
\[B'=a+c+[1,o_b].\]
  the position sets corresponding the semihexagon $\overline{S}^+$.
By Cohn--Larsen--Propp's original formula, the ratio of the tilings number between the above two dented semihexagons can be written as
\begin{align}
\frac{\M(S^+)}{\M(\overline{S}^+)}&=\frac{\Delta(A\cup B \cup C)}{\Delta(A'\cup B'\cup C')}\notag\\
&=\frac{\Delta(A)}{\Delta(A')}\frac{\Delta(B)}{\Delta(B')}\frac{\Delta(C)}{\Delta(C')}\frac{\Delta(A,B)}{\Delta(A',B')}\frac{\Delta(A,C)}{\Delta(A',C')}\frac{\Delta(B,C)}{\Delta(B',C')}.
\end{align}
The first three ratios give us the first, the third and the fifth $s$-functions on the right-hand side of (\ref{dualeq}).
We can write the ratio $\frac{\Delta(A,C)}{\Delta(A',C')}$ as
\begin{equation}
\frac{\Delta(A,C)}{\Delta(A',C')}=\prod_{i,j}\frac{\left(a+\left\lfloor\frac{\lfloor xN\rfloor+\lfloor zN\rfloor}{2}\right\rfloor\right)+\gamma_j-\alpha_i}{\left(a+\left\lfloor\frac{\lfloor xN\rfloor+\lfloor zN\rfloor}{2}\right\rfloor\right)+\gamma'_j-\alpha'_i}.
\end{equation}
For given $i,j$, each factor in the above product tends to $1$, as $N$ gets large (because $|\gamma_j-\alpha_i|, |\gamma'_j-\alpha'_i|\leq a+c$, for any $i,j$). This means that
\begin{equation}
\lim_{N \to \infty} \frac{\Delta(A,C)}{\Delta(A',C')}=1.
\end{equation}
Similarly, we have
\begin{equation}
\lim_{N \to \infty} \frac{\Delta(A,C)}{\Delta(A',C')}=\lim_{N \to \infty} \frac{\Delta(B,C)}{\Delta(B',C')}=1.
\end{equation}
This implies that
\begin{align}
\lim_{N \to \infty}\frac{\M(S^+)}{\M(\overline{S}^+)}=s(a_1,\dotsc,a_{m-1})s(b_1,\dotsc,b_{n-1})s(c_1,\dotsc,c_{k-1}).
\end{align}
Similarly, we get
\begin{align}
\lim_{N \to \infty}\frac{\M(S^-)}{\M(\overline{S}^-)}=s(a_2,\dotsc,a_m)s(b_2,\dotsc,b_{n})s(c_2,\dotsc,c_{k}).
\end{align}
This finishes our proof.
\end{proof}

This theorem implies the dual of MacMahon's theorem introduced by Ciucu (Theorem 1.1 in \cite{Ciu1}) by specializing $\textbf{a}=\textbf{b}=\emptyset$ and $x=z=1$.
\begin{figure}\centering
\includegraphics[width=15cm]{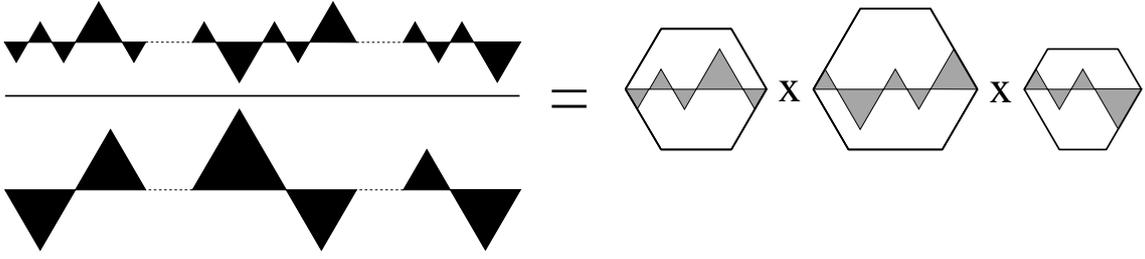}
\caption{The dual of MacMahon's theorem for three ferns.}\label{fig:geointer}
\end{figure}

\begin{figure}\centering
\setlength{\unitlength}{3947sp}%
\begingroup\makeatletter\ifx\SetFigFont\undefined%
\gdef\SetFigFont#1#2#3#4#5{%
  \reset@font\fontsize{#1}{#2pt}%
  \fontfamily{#3}\fontseries{#4}\fontshape{#5}%
  \selectfont}%
\fi\endgroup%
\resizebox{8cm}{!}{
\begin{picture}(0,0)%
\includegraphics{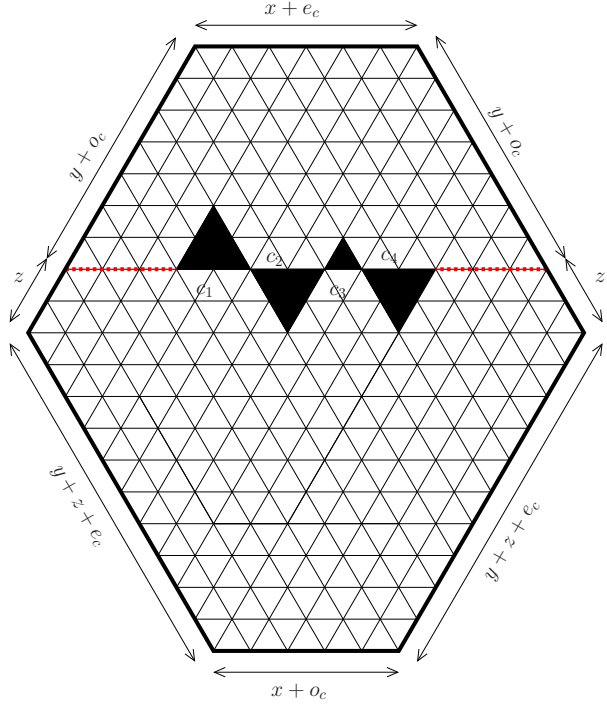}%
\end{picture}%
%
%

\begin{picture}(6658,7899)(3876,-9834)
\put(6777,-9819){\makebox(0,0)[lb]{\smash{{\SetFigFont{14}{16.8}{\rmdefault}{\mddefault}{\updefault}{$x+o_c$}%
}}}}
\put(5949,-5341){\makebox(0,0)[lb]{\smash{{\SetFigFont{14}{16.8}{\rmdefault}{\mddefault}{\updefault}{$c_1$}%
}}}}
\put(6721,-4989){\makebox(0,0)[lb]{\smash{{\SetFigFont{14}{16.8}{\rmdefault}{\mddefault}{\updefault}{$c_2$}%
}}}}
\put(7419,-5349){\makebox(0,0)[lb]{\smash{{\SetFigFont{14}{16.8}{\rmdefault}{\mddefault}{\updefault}{$c_3$}%
}}}}
\put(7989,-4966){\makebox(0,0)[lb]{\smash{{\SetFigFont{14}{16.8}{\rmdefault}{\mddefault}{\updefault}{$c_4$}%
}}}}
\put(10350,-5206){\makebox(0,0)[lb]{\smash{{\SetFigFont{14}{16.8}{\rmdefault}{\mddefault}{\updefault}{$z$}%
}}}}
\put(3928,-5206){\makebox(0,0)[lb]{\smash{{\SetFigFont{14}{16.8}{\rmdefault}{\mddefault}{\updefault}{$z$}%
}}}}
\put(9197,-3292){\rotatebox{300.0}{\makebox(0,0)[lb]{\smash{{\SetFigFont{14}{16.8}{\rmdefault}{\mddefault}{\updefault}{$y+o_c$}%
}}}}}
\put(4629,-4083){\rotatebox{60.0}{\makebox(0,0)[lb]{\smash{{\SetFigFont{14}{16.8}{\rmdefault}{\mddefault}{\updefault}{$y+o_c$}%
}}}}}
\put(6702,-2221){\makebox(0,0)[lb]{\smash{{\SetFigFont{14}{16.8}{\rmdefault}{\mddefault}{\updefault}{$x+e_c$}%
}}}}
\put(9220,-8538){\rotatebox{60.0}{\makebox(0,0)[lb]{\smash{{\SetFigFont{14}{16.8}{\rmdefault}{\mddefault}{\updefault}{$y+z+e_c$}%
}}}}}
\put(4361,-7286){\rotatebox{300.0}{\makebox(0,0)[lb]{\smash{{\SetFigFont{14}{16.8}{\rmdefault}{\mddefault}{\updefault}{$y+z+e_c$}%
}}}}}
\end{picture}}
\caption{A symmetric hexagon with (not necessarily symmetric) fern removed along the symmetric axis.}\label{fig:nosymmetic}
\end{figure}

\medskip

\subsection{Combined theorems and symmetric $F$-cored hexagons}

We start this subsection by noticing that one can combine Theorems  \ref{main1}, \ref{main2}, \ref{mainQ1} and \ref{mainQ2} into a single theorem as follows.

\medskip

 Let $x,z$ be nonnegative integers, and \textbf{a}, \textbf{b}, \textbf{c} be three sequence of nonnegative integers,
such that $a=\sum_i a_i=\sum_j b_j=b$. Consider three ferns whose side-lengths of the triangles are the terms of the sequences \textbf{a}, \textbf{b}, \textbf{c}. We now do \emph{not} have any requirement on the orientations of the first triangle of the three ferns as in the definition of the $R$- and $Q$-families before.

 Consider a symmetric hexagon of side-lengths $x+d_a+d_b+d_c,z+u_a+u_b+u_c,z+d_a+d_b+d_c,x+u_a+u_b+u_c,z+d_a+d_b+d_c,z+u_a+u_b+u_c$, where $u_a$ and $d_a$ denote the sums of the side-lengths of all up-pointing triangles and down-pointing triangles in the $a$-fern, and $u_b,d_b,u_c,d_c$ are defined similarly.
 On the lattice line containing the west and the east vertices of the hexagon, we remove three ferns such that the sequences of side-lengths of the left, right and middle ferns are \textbf{a}, \textbf{b}, \textbf{c}, and that the leftmost of the left fern
  is touching the west vertex of the hexagon, the rightmost of the  right fern is touching the east vertex of the hexagon. For the middle fern, we place it evenly between the left and the right ferns in the sense that the distance between the  left fern and the middle fern is $\lfloor \frac{x+z}{2} \rfloor$ and the distance between the middle fern and the
  right fern is $\lceil \frac{x+z}{2} \rceil$. Denote by $H_{x,z}(\textbf{a};\textbf{c};\textbf{b})$ the resulting region. 
   
\begin{thm}[Combination of Theorems \ref{main1}, \ref{main2}, \ref{mainQ1} and \ref{mainQ2}]\label{combinethm1}
Assume that $x,z$ are nonnegative integers  and that $\textbf{a}=(a_1,\dotsc,a_m),$  $\textbf{b}=(b_1,\dotsc,b_n)$, $\textbf{c}=(c_1,\dotsc,c_k)$ are sequences of nonnegative integers, such that $a=b$.
Assume in addition that $m,n,k$ are all even (as the case when at least one of them are odd can be reduced to the even case by appending a $0$-triangle to the corresponding ferns). Then
\begin{align}\label{combineeq1}
\M&(H_{x,z}(\textbf{a};\textbf{c};\textbf{b}))=\M(C_{x,z+2a,z}(c))\M(S^+)\M(S^-)\notag\\
&\times\frac{\Hf(c+\left\lfloor\frac{x+z}{2}\right\rfloor)}{\Hf(c)\Hf(\left\lfloor\frac{x+z}{2}\right\rfloor)}\frac{\Hf(a+\left\lfloor\frac{x+z}{2}\right\rfloor)}{\Hf(a+c+\left\lfloor\frac{x+z}{2}\right\rfloor)}\notag\\
&\times \frac{\Hf(a+z)\Hf(a+c+z)}{\Hf(u_a+u_b+u_c+z)\Hf(d_a+d_b+d_c+z)}\notag\\
&\times \frac{\Hf(u_a+u_b+u_c)\Hf(d_a+d_b+d_c)}{\Hf(a)^2},
\end{align}
where $S^+$ and $S^-$ are the two semihexagons with dents obtained by dividing the region along the line $\ell$ (the lattice line containing all bases of triangles of the three ferns); the dents of $S^+$ and $S^-$ are defined by the configurations of the three ferns.
\end{thm}

One readily sees that after removing the forced lozenges in $H_{x,z}(\textbf{a};\textbf{c};\textbf{b})$, the remaining region is an $Q^{\odot}$- or $Q^{\leftarrow}$-type region if the left and right ferns both have the first triangles up-pointing, and we obtain an $R^{\odot}$- or $R^{\leftarrow}$-type region when the left fern starts by a down-pointing triangle and the right fern starts by an up-pointing triangle. This means that Theorem \ref{combinethm1} implies all four Theorems \ref{main1}, \ref{main2}, \ref{mainQ1} and \ref{mainQ2}, say after some appropriate changes of variables.

One can obtain similarly a combination of Theorems  \ref{main3}, \ref{main4}, \ref{mainQ3} and \ref{mainQ4}.


\medskip

Next, we consider an interesting special case of the $Q^{\odot}$-type region when $\textbf{a}=\textbf{b}=\emptyset$.
\begin{thm}\label{symmetricthm}
Let $x,y,z$ be nonnegative integers and let $\textbf{c}=(c_1,c_2,\dotsc,c_k)$ be a sequence of nonnegative integers. Assume in addition that $x$ and $y$ have the same parity. Let $B_{x,y,z}(c_1,c_2,\dotsc,c_k)$ be the region obtained from the symmetric hexagon of side-lengths $x+e_c,y+z+o_c,y+z+e_c,x+o_c,y+z+e_c,y+z+o_c$ by removing a fern consisting triangles of side-lengths $c_1,c_2,\dotsc,c_k$ at the level $z$ above the west vertex of the hexagon, so that the distances between two endpoints of the fern and the northwest and northeast sides of the hexagon are the same. The number of tilings of $B_{x,y,z}(c_1,c_2,\dotsc,c_k)$ is given by
\begin{align}
\M(B_{x,y,z}(c_1,c_2,\dotsc,c_k))=&\M(C_{x,y+2z,y}(c))\notag\\
&\times s\left(c_1,\dotsc,c_{k-1}\right)\cdot s\left(z,c_1+\frac{x+y}{2},\dotsc,c_{k},\frac{x+y}{2},z\right)\notag\\
&\times \frac{\Hf(c+\frac{x+y}{2})}{\Hf(c)\Hf(\frac{x+y}{2})}\frac{\Hf(z+\frac{x+y}{2})}{\Hf(z+c+\frac{x+y}{2})}\frac{\Hf(y+z)\Hf(c+y+z)}{\Hf(o_c+y)\Hf(e_c+y+2z)}\frac{\Hf(o_c)\Hf(e_c+2z)}{\Hf(z)^2}.
\end{align}
\end{thm}
Recall that, in general, if we move the removed fern in a $F$-cored hexagon away from the center, the tiling number is not given by a simple product anymore. However, this theorem says that in the case of symmetric hexagons, we can remove a fern at \emph{any} positions perpendicularly to the symmetry axis and still get a simple product formula. Interestingly, the fern does \emph{not} need to be symmetric\footnote{This phenomenon was first observed by Ciucu (private communication).}. This theorem generalizes the author's previous work in \cite{Halfhex2} where we required in additional that the fern is also symmetric.

\section{Combined proof of Theorems \ref{main1}--\ref{mainQ4}}\label{sec:proof1}

\subsection{Organization of the proof}\label{subsec:organize}

Recall that our 8 regions, $R^{\odot}_{x,y,z}(\textbf{a};\ \textbf{c};\ \textbf{b})$, $R^{\leftarrow}_{x,y,z}(\textbf{a};\ \textbf{c};\ \textbf{b})$, $R^{\swarrow}_{x,y,z}(\textbf{a};\ \textbf{c};\ \textbf{b})$,  $R^{\nwarrow}_{x,y,z}(\textbf{a};\ \textbf{c};\ \textbf{b})$, $Q^{\odot}_{x,y,z}(\textbf{a};\ \textbf{c};\ \textbf{b})$,  $Q^{\leftarrow}_{x,y,z}(\textbf{a};\ \textbf{c};\ \textbf{b})$, $Q^{\nwarrow}_{x,y,z}(\textbf{a};\ \textbf{c};\ \textbf{b})$  and $Q^{\nearrow}_{x,y,z}(\textbf{a};\ \textbf{c};\ \textbf{b})$,  are all obtained from a  certain base hexagon $H$ by removing three ferns along a common lattice line $\ell$. The base hexagons of the regions $R^{\odot}_{x,y,z}(\textbf{a};\ \textbf{c};\ \textbf{b})$  and $R^{\leftarrow}_{x,y,z}(\textbf{a};\ \textbf{c};\ \textbf{b})$ are both of side-lengths $x+o_a+e_b+e_c,$  $2y+z+e_a+o_b+e_c+ |a-b|$,  $z+o_a+e_b+e_c,$ $x+e_a+o_b+e_c$, $2y+z+o_a+e_b+e_c+ |a-b|,$ $z+e_a+o_b+e_c$; while the base hexagons of the regions $R^{\swarrow}_{x,y,z}(\textbf{a};\ \textbf{c};\ \textbf{b})$  and $R^{\nwarrow}_{x,y,z}(\textbf{a};\ \textbf{c};\ \textbf{b})$ are of side-lengths $x+o_a+e_b+e_c,$  $2y+z+e_a+o_b+e_c+ |a-b|+1$,  $z+o_a+e_b+e_c,$ $x+e_a+o_b+e_c$, $2y+z+o_a+e_b+e_c+ |a-b|+1,$ $z+e_a+o_b+e_c$.  The perimeter of the base hexagon is then $2x+4y+4z+3a+3b+2|a-b|$ or $2x+4y+4z+3a+3b+2|a-b|+2$, respectively. Similarly, one readily sees that the perimeter of the base hexagons of the regions $Q^{\odot}_{x,y,z}(\textbf{a};\ \textbf{c};\ \textbf{b})$  and $Q^{\leftarrow}_{x,y,z}(\textbf{a};\ \textbf{c};\ \textbf{b})$ always have perimeter equal to $2x+4y+4z+3a+3b+2|a-b|$, and the perimeters of the base hexagons of the regions $Q^{\nwarrow}_{x,y,z}(\textbf{a};\ \textbf{c};\ \textbf{b})$  and $Q^{\nearrow}_{x,y,z}(\textbf{a};\ \textbf{c};\ \textbf{b})$ are both $2x+4y+4z+3a+3b+2|a-b|+2$. We call the perimeter  of the base hexagon the \emph{quasi-perimeter} of our regions, denoted by $p$ in the rest of the proof.

One readily sees that
\begin{claim}\label{claimp}
\[p \geq 2x+4z.\]
\end{claim}
\begin{proof}
If $y\geq 0$, then by the explicit formula of the quasi-perimeter above, we have $p\geq 2x+4z$. We only need to consider the case $y=-1$. However, $y=-1$ only happens in the $R^{\nwarrow}$-, $R^{\swarrow}$-, $Q^{\nwarrow}$-, and $Q^{\nearrow}$-type regions with $|a-b|\geq 1$. In these cases,  we have
\begin{equation}
p=2x+4y+4z+3a+3b+2|a-b|+2\geq 2x-4+4z+2|a-b|+2\geq 2x+4z.
\end{equation}
\end{proof}

We aim to prove \emph{all} eight Theorems \ref{main1}--\ref{mainQ4} at once by induction on $h:=p+x+z$, where $p$ is the quasi-perimeter of the region. Our proof is organized as follows. In Subsection 3.2, we quote the particular versions the Kuo condensation that will be employed in our proofs. Next, in Subsections 3.3--3.10, we will present carefully 18 recurrences for our 8 families of regions obtained by applying Kuo condensation. Each family of regions will have two or three different recurrences, depending on whether $a> b$, $a=b$, or $a> b$.  We would like to emphasize that, due to the difference in the structures of our regions, the universal recurrence seems \emph{not} to exist. Subsection 3.11 is devoted to the main arguments of the inductive proof. Finally, in Subsection 3.12, we handle the algebraic verification that completes our main proof.

\subsection{Kuo condensation and other preliminary results}\label{subsec:kuo}
In general, the tilings of a region $R$ can carry `weights'. In the weighted case, the notation $\M(R)$ stands for the sum of the weights of all tilings of the region $R$, where the \emph{weight} of a tiling is the product of weights of its lozenges.

A \emph{forced lozenge} in a region $R$ on the triangular lattice is a lozenge contained in any tilings of $R$. Assume that we remove several forced lozenges $l_1,l_2\dotsc,l_n$ from the region $R$ and get a new region $R'$. Then
\begin{equation}\label{forcedeq}
\M(R)=\M(R')\prod_{i=1}^{n}wt(l_i),
\end{equation}
where $wt(l_i)$ denotes the weight of the lozenge $l_i$.

A region on the triangular lattice is said to be \emph{balanced} if it has the same number of up- and down-pointing unit triangles. The following useful lemma allows us to decompose a large region into several smaller ones.

\begin{lem}[Region-splitting Lemma \cite{Tri1, Tri2}]\label{RS}
Let $R$ be a balanced region on the triangular lattice. Assume that a sub-region $Q$ of $R$ satisfies the following two conditions:
\begin{enumerate}
\item[(i)] \text{\rm{(Separating Condition)}} There is only one type of unit triangles (up-pointing or down-pointing) running along each side of the border between $Q$ and $R-Q$.

\item[(ii)] \text{\rm{(Balancing Condition)}} $Q$ is balanced.
\end{enumerate}
Then
\begin{equation}
\M(R)=\M(Q)\, \M(R-Q).
\end{equation}
\end{lem}

Let $G$ be a finite simple graph without loops. A \emph{perfect matching} of $G$ is a collection of disjoint edges covering all vertices of $G$. The \emph{(planar) dual graph} of a region $R$ on the triangular lattice  is the graph whose vertices are unit triangles in $R$ and whose edges connect precisely two unit triangles sharing an edge. In the weighted case, the edges of the dual graph carry the same weights as the corresponding lozenges. We can identify the tilings of a region and perfect matchings of its dual graph.  In this point of view, we use the notation $\M(G)$ for the sum of the weights of all perfect matchings in $G$, where the weight of a perfect matching is the product of weights of its constituent edges. In the unweighted case, i.e. when all edges of the graph have weight 1, $\M(G)$ is exactly number the perfect matchings of the graph $G$. 

The following two theorems of  Kuo are the keys of our proofs in this paper.
\begin{thm}[Theorem 5.1 \cite{Kuo}]\label{kuothm1}
Let $G=(V_1,V_2,E)$ be a (weighted) bipartite planar graph in which $|V_1|=|V_2|$. Assume that  $u, v, w, s$ are four vertices appearing in a cyclic order on a face of $G$ so that $u,w \in V_1$ and $v,s \in V_2$. Then
\begin{equation}\label{kuoeq1}
\M(G)\M(G-\{u, v, w, s\})=\M(G-\{u, v\})\M(G-\{ w, s\})+\M(G-\{u, s\})\M(G-\{v, w\}).
\end{equation}
\end{thm}

\begin{thm}[Theorem 5.2 \cite{Kuo}]\label{kuothm2}
Let $G=(V_1,V_2,E)$ be a (weighted) bipartite planar graph in which $|V_1|=|V_2|$. Assume that  $u, v, w, s$ are four vertices appearing in a cyclic order on a face of $G$ so that $u,v \in V_1$ and $w,s \in V_2$. Then
\begin{equation}\label{kuoeq2}
\M(G-\{u, s\})\M(G-\{v, w\})=\M(G)\M(G-\{u, v, w, s\})+\M(G-\{u,w\})\M(G-\{v, s\}).
\end{equation}
\end{thm}

Theorems \ref{kuothm1} and \ref{kuothm2} are usually mentioned as two variants of  \emph{Kuo condensation}. Kuo condensation (or \emph{graphical condensation} as called in \cite{Kuo}) can be considered as a combinatorial interpretation of the well-known \emph{Dodgson condensation} in linear algebra (which is based on the Jacobi--Desnanot identity, see e.g. \cite{Abeles}, \cite{Dod} and \cite{Mui}, pp. 136--148, and \cite{Zeil} for a bijective proof). The Dodgson condensation was named after Charles Lutwidge Dodgson (1832--1898), better known by his pen name Lewis Carroll, an English writer, mathematician, and photographer. 

The preliminary version of Kuo condensation (when the for vertices $u,v,w,s$ in Theorem \ref{kuothm1} form a $4$-cycle in the graph $G$) was originally conjectured by Alexandru Ionescu in context of Aztec diamond graphs, and was proved by Propp in 1993 (see e.g. \cite{Propp2}). Eric H. Kuo introduced Kuo condensation  in his 2004 paper \cite{Kuo} with four different versions, two of them are Theorems \ref{kuothm1} and \ref{kuothm2} stated above. Kuo condensation has become a powerful tool in the enumeration of tilings with a number of applications. We refer the reader to \cite{Ciucu, Ful, Knuth, Kuo06, speyer, YYZ, YZ} for various aspects and generalizations of Kuo condensation, and e.g. \cite{CF, CK, CL, KW, LMNT, Lai15a, Tri1, Tri2, Halfhex1, Halfhex2, LM, LR, Ranjan1, Ranjan2} for recent applications of the method.

\begin{figure}\centering
\setlength{\unitlength}{3947sp}%
\begingroup\makeatletter\ifx\SetFigFont\undefined%
\gdef\SetFigFont#1#2#3#4#5{%
  \reset@font\fontsize{#1}{#2pt}%
  \fontfamily{#3}\fontseries{#4}\fontshape{#5}%
  \selectfont}%
\fi\endgroup%
\resizebox{15cm}{!}{
\begin{picture}(0,0)%
\includegraphics{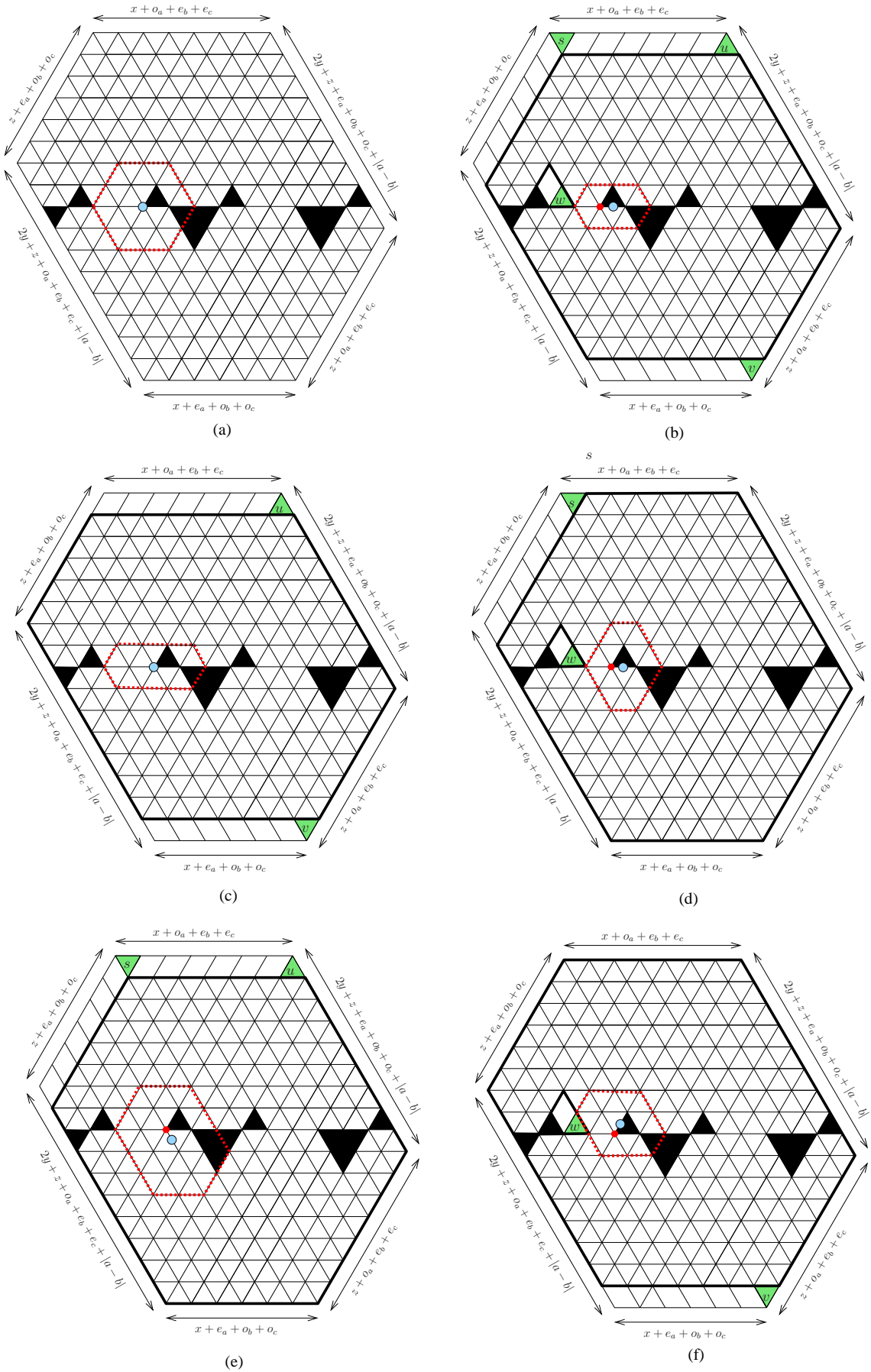}%
\end{picture}%
%
%

\begin{picture}(14027,22390)(1794,-22426)
\put(14383,-8474){\rotatebox{300.0}{\makebox(0,0)[lb]{\smash{{\SetFigFont{12}{14.4}{\rmdefault}{\mddefault}{\itdefault}{$2y+z+e_a+o_b+o_c+|a-b|$}%
}}}}}
\put(11436,-7754){\makebox(0,0)[lb]{\smash{{\SetFigFont{12}{14.4}{\rmdefault}{\mddefault}{\itdefault}{$x+o_a+e_b+e_c$}%
}}}}
\put(6216,-8406){\makebox(0,0)[lb]{\smash{{\SetFigFont{14}{16.8}{\rmdefault}{\mddefault}{\itdefault}{$u$}%
}}}}
\put(2158,-9584){\rotatebox{60.0}{\makebox(0,0)[lb]{\smash{{\SetFigFont{12}{14.4}{\rmdefault}{\mddefault}{\itdefault}{$z+e_a+o_b+o_c$}%
}}}}}
\put(2235,-11309){\rotatebox{300.0}{\makebox(0,0)[lb]{\smash{{\SetFigFont{12}{14.4}{\rmdefault}{\mddefault}{\itdefault}{$2y+z+o_a+e_b+e_c+|a-b|$}%
}}}}}
\put(4738,-14226){\makebox(0,0)[lb]{\smash{{\SetFigFont{12}{14.4}{\rmdefault}{\mddefault}{\itdefault}{$x+e_a+o_b+o_c$}%
}}}}
\put(7338,-13635){\rotatebox{60.0}{\makebox(0,0)[lb]{\smash{{\SetFigFont{12}{14.4}{\rmdefault}{\mddefault}{\itdefault}{$z+o_a+e_b+e_c$}%
}}}}}
\put(14718,-13635){\rotatebox{60.0}{\makebox(0,0)[lb]{\smash{{\SetFigFont{12}{14.4}{\rmdefault}{\mddefault}{\itdefault}{$z+o_a+e_b+e_c$}%
}}}}}
\put(12118,-14226){\makebox(0,0)[lb]{\smash{{\SetFigFont{12}{14.4}{\rmdefault}{\mddefault}{\itdefault}{$x+e_a+o_b+o_c$}%
}}}}
\put(7003,-8474){\rotatebox{300.0}{\makebox(0,0)[lb]{\smash{{\SetFigFont{12}{14.4}{\rmdefault}{\mddefault}{\itdefault}{$2y+z+e_a+o_b+o_c+|a-b|$}%
}}}}}
\put(4056,-7754){\makebox(0,0)[lb]{\smash{{\SetFigFont{12}{14.4}{\rmdefault}{\mddefault}{\itdefault}{$x+o_a+e_b+e_c$}%
}}}}
\put(10972,-18484){\makebox(0,0)[lb]{\smash{{\SetFigFont{14}{16.8}{\rmdefault}{\mddefault}{\itdefault}{$w$}%
}}}}
\put(9594,-17197){\rotatebox{60.0}{\makebox(0,0)[lb]{\smash{{\SetFigFont{12}{14.4}{\rmdefault}{\mddefault}{\itdefault}{$z+e_a+o_b+o_c$}%
}}}}}
\put(9671,-18922){\rotatebox{300.0}{\makebox(0,0)[lb]{\smash{{\SetFigFont{12}{14.4}{\rmdefault}{\mddefault}{\itdefault}{$2y+z+o_a+e_b+e_c+|a-b|$}%
}}}}}
\put(12174,-21839){\makebox(0,0)[lb]{\smash{{\SetFigFont{12}{14.4}{\rmdefault}{\mddefault}{\itdefault}{$x+e_a+o_b+o_c$}%
}}}}
\put(14774,-21248){\rotatebox{60.0}{\makebox(0,0)[lb]{\smash{{\SetFigFont{12}{14.4}{\rmdefault}{\mddefault}{\itdefault}{$z+o_a+e_b+e_c$}%
}}}}}
\put(14439,-16087){\rotatebox{300.0}{\makebox(0,0)[lb]{\smash{{\SetFigFont{12}{14.4}{\rmdefault}{\mddefault}{\itdefault}{$2y+z+e_a+o_b+o_c+|a-b|$}%
}}}}}
\put(11492,-15367){\makebox(0,0)[lb]{\smash{{\SetFigFont{12}{14.4}{\rmdefault}{\mddefault}{\itdefault}{$x+o_a+e_b+e_c$}%
}}}}
\put(6400,-15945){\makebox(0,0)[lb]{\smash{{\SetFigFont{14}{16.8}{\rmdefault}{\mddefault}{\itdefault}{$u$}%
}}}}
\put(2342,-17130){\rotatebox{60.0}{\makebox(0,0)[lb]{\smash{{\SetFigFont{12}{14.4}{\rmdefault}{\mddefault}{\itdefault}{$z+e_a+o_b+o_c$}%
}}}}}
\put(2419,-18855){\rotatebox{300.0}{\makebox(0,0)[lb]{\smash{{\SetFigFont{12}{14.4}{\rmdefault}{\mddefault}{\itdefault}{$2y+z+o_a+e_b+e_c+|a-b|$}%
}}}}}
\put(4922,-21772){\makebox(0,0)[lb]{\smash{{\SetFigFont{12}{14.4}{\rmdefault}{\mddefault}{\itdefault}{$x+e_a+o_b+o_c$}%
}}}}
\put(9615,-11309){\rotatebox{300.0}{\makebox(0,0)[lb]{\smash{{\SetFigFont{12}{14.4}{\rmdefault}{\mddefault}{\itdefault}{$2y+z+o_a+e_b+e_c+|a-b|$}%
}}}}}
\put(9538,-9584){\rotatebox{60.0}{\makebox(0,0)[lb]{\smash{{\SetFigFont{12}{14.4}{\rmdefault}{\mddefault}{\itdefault}{$z+e_a+o_b+o_c$}%
}}}}}
\put(7522,-21181){\rotatebox{60.0}{\makebox(0,0)[lb]{\smash{{\SetFigFont{12}{14.4}{\rmdefault}{\mddefault}{\itdefault}{$z+o_a+e_b+e_c$}%
}}}}}
\put(10951,-8311){\makebox(0,0)[lb]{\smash{{\SetFigFont{14}{16.8}{\rmdefault}{\mddefault}{\itdefault}{$s$}%
}}}}
\put(7187,-16020){\rotatebox{300.0}{\makebox(0,0)[lb]{\smash{{\SetFigFont{12}{14.4}{\rmdefault}{\mddefault}{\itdefault}{$2y+z+e_a+o_b+o_c+|a-b|$}%
}}}}}
\put(4240,-15300){\makebox(0,0)[lb]{\smash{{\SetFigFont{12}{14.4}{\rmdefault}{\mddefault}{\itdefault}{$x+o_a+e_b+e_c$}%
}}}}
\put(9357,-2083){\rotatebox{60.0}{\makebox(0,0)[lb]{\smash{{\SetFigFont{12}{14.4}{\rmdefault}{\mddefault}{\itdefault}{$z+e_a+o_b+o_c$}%
}}}}}
\put(9434,-3808){\rotatebox{300.0}{\makebox(0,0)[lb]{\smash{{\SetFigFont{12}{14.4}{\rmdefault}{\mddefault}{\itdefault}{$2y+z+o_a+e_b+e_c+|a-b|$}%
}}}}}
\put(11937,-6725){\makebox(0,0)[lb]{\smash{{\SetFigFont{12}{14.4}{\rmdefault}{\mddefault}{\itdefault}{$x+e_a+o_b+o_c$}%
}}}}
\put(14537,-6134){\rotatebox{60.0}{\makebox(0,0)[lb]{\smash{{\SetFigFont{12}{14.4}{\rmdefault}{\mddefault}{\itdefault}{$z+o_a+e_b+e_c$}%
}}}}}
\put(6646,-13606){\makebox(0,0)[lb]{\smash{{\SetFigFont{14}{16.8}{\rmdefault}{\mddefault}{\itdefault}{$v$}%
}}}}
\put(10916,-10863){\makebox(0,0)[lb]{\smash{{\SetFigFont{14}{16.8}{\rmdefault}{\mddefault}{\itdefault}{$w$}%
}}}}
\put(14202,-973){\rotatebox{300.0}{\makebox(0,0)[lb]{\smash{{\SetFigFont{12}{14.4}{\rmdefault}{\mddefault}{\itdefault}{$2y+z+e_a+o_b+o_c+|a-b|$}%
}}}}}
\put(11255,-253){\makebox(0,0)[lb]{\smash{{\SetFigFont{12}{14.4}{\rmdefault}{\mddefault}{\itdefault}{$x+o_a+e_b+e_c$}%
}}}}
\put(11236,-7531){\makebox(0,0)[lb]{\smash{{\SetFigFont{14}{16.8}{\rmdefault}{\mddefault}{\itdefault}{$s$}%
}}}}
\put(10737,-3373){\makebox(0,0)[lb]{\smash{{\SetFigFont{14}{16.8}{\rmdefault}{\mddefault}{\itdefault}{$w$}%
}}}}
\put(13861,-6106){\makebox(0,0)[lb]{\smash{{\SetFigFont{14}{16.8}{\rmdefault}{\mddefault}{\itdefault}{$v$}%
}}}}
\put(13415,-898){\makebox(0,0)[lb]{\smash{{\SetFigFont{14}{16.8}{\rmdefault}{\mddefault}{\itdefault}{$u$}%
}}}}
\put(1981,-2079){\rotatebox{60.0}{\makebox(0,0)[lb]{\smash{{\SetFigFont{12}{14.4}{\rmdefault}{\mddefault}{\itdefault}{$z+e_a+o_b+o_c$}%
}}}}}
\put(2058,-3804){\rotatebox{300.0}{\makebox(0,0)[lb]{\smash{{\SetFigFont{12}{14.4}{\rmdefault}{\mddefault}{\itdefault}{$2y+z+o_a+e_b+e_c+|a-b|$}%
}}}}}
\put(4561,-6721){\makebox(0,0)[lb]{\smash{{\SetFigFont{12}{14.4}{\rmdefault}{\mddefault}{\itdefault}{$x+e_a+o_b+o_c$}%
}}}}
\put(10786,-766){\makebox(0,0)[lb]{\smash{{\SetFigFont{14}{16.8}{\rmdefault}{\mddefault}{\itdefault}{$s$}%
}}}}
\put(3765,-15833){\makebox(0,0)[lb]{\smash{{\SetFigFont{14}{16.8}{\rmdefault}{\mddefault}{\itdefault}{$s$}%
}}}}
\put(7161,-6130){\rotatebox{60.0}{\makebox(0,0)[lb]{\smash{{\SetFigFont{12}{14.4}{\rmdefault}{\mddefault}{\itdefault}{$z+o_a+e_b+e_c$}%
}}}}}
\put(6826,-969){\rotatebox{300.0}{\makebox(0,0)[lb]{\smash{{\SetFigFont{12}{14.4}{\rmdefault}{\mddefault}{\itdefault}{$2y+z+e_a+o_b+o_c+|a-b|$}%
}}}}}
\put(3879,-249){\makebox(0,0)[lb]{\smash{{\SetFigFont{12}{14.4}{\rmdefault}{\mddefault}{\itdefault}{$x+o_a+e_b+e_c$}%
}}}}
\put(14093,-21263){\makebox(0,0)[lb]{\smash{{\SetFigFont{14}{16.8}{\rmdefault}{\mddefault}{\itdefault}{$v$}%
}}}}
\end{picture}%

}
\caption{Obtaining the recurrence for the regions $R^{\odot}_{x,y,z}(\textbf{a};\textbf{c};\textbf{b})$, when $a<b$. Kuo condensation is applied to the region $R^{\odot}_{2,1,2}(1,1 ;\ 1,2,1 ;\ 1,2)$ (picture (a)) as shown on the picture (b).}\label{fig:kuocenter1}
\end{figure}

\subsection{Recurrences for $R^{\odot}$-type regions}\label{subsec:recurR1}

Recall that we are assuming that $x$ and $z$ have the same parity and that the leftmost vertex of the middle fern is exactly at the center of auxiliary hexagon $H_0$ of side-lengths $x,z,z,x,z,z$.

If $a< b$ (i.e. the total length of the left fern is not greater than that of the right fern), we apply Kuo condensation (Theorem \ref{kuothm1}) to the dual graph $G$ of $R^{\odot}_{x,y,z}(\textbf{a}; \textbf{c}; \textbf{b})$ with the four vertices $u,v,w,s$  corresponding to the shaded unit triangles with the same label in Figure \ref{fig:kuocenter1}(b).  In particular, the $u$-triangle is the up-pointing shaded unit triangle on the northeast corner of the region,  the $v$-triangle is the down-pointing shaded unit triangle on the southeast corner, the $w$-triangle is the up-pointing shaded unit triangle attached to the rightmost point of the left fern, and the $s$-triangle is the down-pointing shaded unit triangle on the northwest corner. The six regions in Figure \ref{fig:kuocenter1} correspond to the six terms in identity (\ref{kuoeq1}). Strictly speaking, Figure \ref{fig:kuocenter1} shows the regions corresponding to the graphs in this identity.

Let us consider the region corresponding to the graph $G-\{u,v,w,s\}$ shown in picture (b). The removal of the four unit triangles with labels $u,v,w,s$ gives forced lozenges along the north, the northwest and the south sides of the hexagon, as well as the side of the last triangle of the left fern. By removing these forced lozenges, we get a new region with the same number of tilings (see the region, restricted by the bold contour). This new region is exactly an $R^{\leftarrow}$-type region with the $z$-parameter reduced by $1$ unit, the side-length of the last triangle in the left fern extended by $1$ unit (precisely, if the left fern ends with an up-pointing triangle, then the removal of the forced lozenges extends its side-length by $1$; in the case when the left fern ends with a down-pointing triangle, then the removal of the $w$-triangle forms a new up-pointing triangle of side-length 1 at the end of the left fern. However, in the latter case, we regard the fern as having $m+1$ triangles, the last of which is of side-length $0$). Moreover, the center of the new auxiliary hexagon (now with side-lengths $x,z-1,z-1,x,z-1,z-1$) is $1/2$ unit to the right of that of the original auxiliary center. This means that the leftmost point of the middle fern is now $1/2$ unit to the left of the center of the auxiliary hexagon. That explains why the type of our region was changed.

 For convenience, we denote, from now on, by $\textbf{a}^{+1}$ the sequence obtained from the sequence $\textbf{a}$ by adding 1 to the last term if $\textbf{a}$ has an even number of  terms, otherwise, including a new term $1$ to the end of $\textbf{a}$.  We have just established the identity:
\begin{equation}
\M(G-\{ u,v,w,s\})=\M(R^{\leftarrow}_{x,y,z-1}(\textbf{a}^{+1};\ \textbf{c};\  \textbf{b})).
\end{equation}

Working throughout the next four regions in the Figure \ref{fig:kuocenter1}(c)--(f), we get respectively:
\begin{equation}
\M(G-\{ u,v\})=\M(R^{\odot}_{x+1,y,z-1}(\textbf{a};\ \textbf{c};\ \textbf{b})),
\end{equation}
\begin{equation}
\M(G-\{ w,s\})=M(R^{\leftarrow}_{x-1,y,z}(\textbf{a}^{+1};\ \textbf{c};\ \textbf{b})),
\end{equation}
\begin{equation}
\M(G-\{ u,s\})=\M(R^{\nwarrow}_{x,y-1,z}(\textbf{a};\ \textbf{c};\ \textbf{b})),
\end{equation}
\begin{equation}
\M(G-\{ v,w\})=\M(R^{\swarrow}_{x,y,z-1}(\textbf{a}^{+1};\ \textbf{c};\ \textbf{b})).
\end{equation}

One should note that the change of the parameters $x, y,z,$ and the sequence $\textbf{a}$ leads to the change the position of the center of the auxiliary hexagon, as a consequence, the types of our regions are also changed.

Plugging the above identities into identity (\ref{kuoeq1}) of the Kuo condensation, we get the recurrence:
\begin{align}\label{centerrecur1a}
\M(R^{\odot}_{x,y,z}(\textbf{a};\ \textbf{c};\ \textbf{b})) \M(R^{\leftarrow}_{x,y,z-1}(\textbf{a}^{+1};\ \textbf{c};\ \textbf{b}))&=
\M(R^{\odot}_{x+1,y,z-1}(\textbf{a};\ \textbf{c};\ \textbf{b}) )\M(R^{\leftarrow}_{x-1,y,z}(\textbf{a}^{+1};\ \textbf{c};\ \textbf{b}))\notag\\
&+\M(R^{\nwarrow}_{x,y-1,z}(\textbf{a};\ \textbf{c}; \ \textbf{b}))\M(R^{\swarrow}_{x,y,z-1}(\textbf{a}^{+1};\ \textbf{c};\  \textbf{b})),
\end{align}
when $a< b$.

We also note that the above recurrence also works well in the case $a=0$, by regarding that the sequence $\textbf{a}$ consists of a single triangle of side length $0$.

\begin{figure}\centering
\setlength{\unitlength}{3947sp}%
\begingroup\makeatletter\ifx\SetFigFont\undefined%
\gdef\SetFigFont#1#2#3#4#5{%
  \reset@font\fontsize{#1}{#2pt}%
  \fontfamily{#3}\fontseries{#4}\fontshape{#5}%
  \selectfont}%
\fi\endgroup%
\resizebox{15cm}{!}{
\begin{picture}(0,0)%
\includegraphics{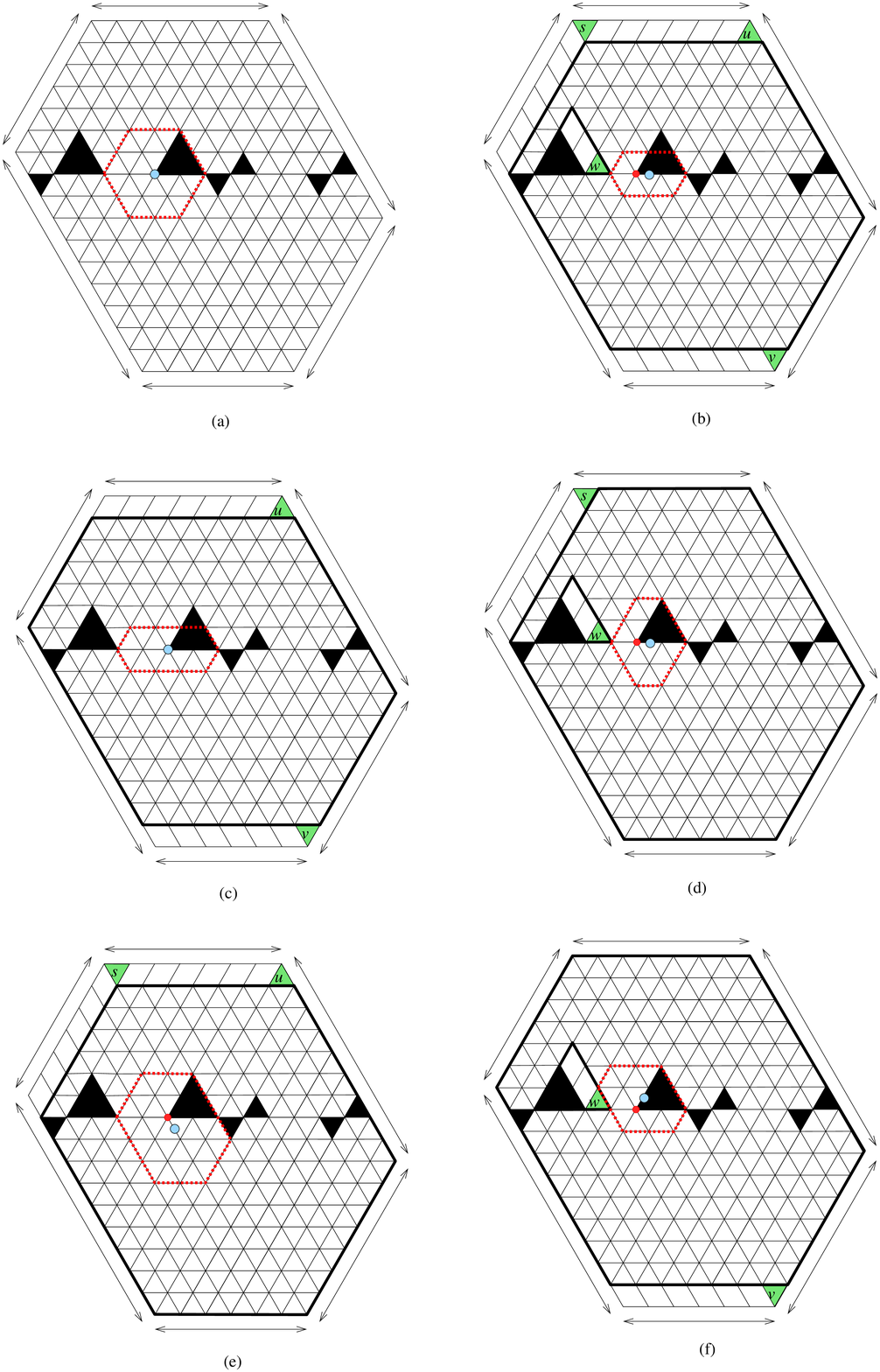}%
\end{picture}%
%
%

\begin{picture}(14268,22409)(1785,-22427)
\put(14602,-16082){\rotatebox{300.0}{\makebox(0,0)[lb]{\smash{{\SetFigFont{12}{14.4}{\rmdefault}{\mddefault}{\itdefault}{$2y+z+e_a+o_b+o_c+|a-b|$}%
}}}}}
\put(11655,-15362){\makebox(0,0)[lb]{\smash{{\SetFigFont{12}{14.4}{\rmdefault}{\mddefault}{\itdefault}{$x+o_a+e_b+e_c$}%
}}}}
\put(2189,-17320){\rotatebox{60.0}{\makebox(0,0)[lb]{\smash{{\SetFigFont{12}{14.4}{\rmdefault}{\mddefault}{\itdefault}{$z+e_a+o_b+o_c$}%
}}}}}
\put(2266,-19045){\rotatebox{300.0}{\makebox(0,0)[lb]{\smash{{\SetFigFont{12}{14.4}{\rmdefault}{\mddefault}{\itdefault}{$2y+z+o_a+e_b+e_c+|a-b|$}%
}}}}}
\put(4769,-21962){\makebox(0,0)[lb]{\smash{{\SetFigFont{12}{14.4}{\rmdefault}{\mddefault}{\itdefault}{$x+e_a+o_b+o_c$}%
}}}}
\put(7369,-21371){\rotatebox{60.0}{\makebox(0,0)[lb]{\smash{{\SetFigFont{12}{14.4}{\rmdefault}{\mddefault}{\itdefault}{$z+o_a+e_b+e_c$}%
}}}}}
\put(7034,-16210){\rotatebox{300.0}{\makebox(0,0)[lb]{\smash{{\SetFigFont{12}{14.4}{\rmdefault}{\mddefault}{\itdefault}{$2y+z+e_a+o_b+o_c+|a-b|$}%
}}}}}
\put(4087,-15490){\makebox(0,0)[lb]{\smash{{\SetFigFont{12}{14.4}{\rmdefault}{\mddefault}{\itdefault}{$x+o_a+e_b+e_c$}%
}}}}
\put(9772,-9640){\rotatebox{60.0}{\makebox(0,0)[lb]{\smash{{\SetFigFont{12}{14.4}{\rmdefault}{\mddefault}{\itdefault}{$z+e_a+o_b+o_c$}%
}}}}}
\put(9849,-11365){\rotatebox{300.0}{\makebox(0,0)[lb]{\smash{{\SetFigFont{12}{14.4}{\rmdefault}{\mddefault}{\itdefault}{$2y+z+o_a+e_b+e_c+|a-b|$}%
}}}}}
\put(12352,-14282){\makebox(0,0)[lb]{\smash{{\SetFigFont{12}{14.4}{\rmdefault}{\mddefault}{\itdefault}{$x+e_a+o_b+o_c$}%
}}}}
\put(14952,-13691){\rotatebox{60.0}{\makebox(0,0)[lb]{\smash{{\SetFigFont{12}{14.4}{\rmdefault}{\mddefault}{\itdefault}{$z+o_a+e_b+e_c$}%
}}}}}
\put(14617,-8530){\rotatebox{300.0}{\makebox(0,0)[lb]{\smash{{\SetFigFont{12}{14.4}{\rmdefault}{\mddefault}{\itdefault}{$2y+z+e_a+o_b+o_c+|a-b|$}%
}}}}}
\put(11670,-7810){\makebox(0,0)[lb]{\smash{{\SetFigFont{12}{14.4}{\rmdefault}{\mddefault}{\itdefault}{$x+o_a+e_b+e_c$}%
}}}}
\put(2197,-9760){\rotatebox{60.0}{\makebox(0,0)[lb]{\smash{{\SetFigFont{12}{14.4}{\rmdefault}{\mddefault}{\itdefault}{$z+e_a+o_b+o_c$}%
}}}}}
\put(2274,-11485){\rotatebox{300.0}{\makebox(0,0)[lb]{\smash{{\SetFigFont{12}{14.4}{\rmdefault}{\mddefault}{\itdefault}{$2y+z+o_a+e_b+e_c+|a-b|$}%
}}}}}
\put(4777,-14402){\makebox(0,0)[lb]{\smash{{\SetFigFont{12}{14.4}{\rmdefault}{\mddefault}{\itdefault}{$x+e_a+o_b+o_c$}%
}}}}
\put(7377,-13811){\rotatebox{60.0}{\makebox(0,0)[lb]{\smash{{\SetFigFont{12}{14.4}{\rmdefault}{\mddefault}{\itdefault}{$z+o_a+e_b+e_c$}%
}}}}}
\put(7042,-8650){\rotatebox{300.0}{\makebox(0,0)[lb]{\smash{{\SetFigFont{12}{14.4}{\rmdefault}{\mddefault}{\itdefault}{$2y+z+e_a+o_b+o_c+|a-b|$}%
}}}}}
\put(4095,-7930){\makebox(0,0)[lb]{\smash{{\SetFigFont{12}{14.4}{\rmdefault}{\mddefault}{\itdefault}{$x+o_a+e_b+e_c$}%
}}}}
\put(9757,-2072){\rotatebox{60.0}{\makebox(0,0)[lb]{\smash{{\SetFigFont{12}{14.4}{\rmdefault}{\mddefault}{\itdefault}{$z+e_a+o_b+o_c$}%
}}}}}
\put(9834,-3797){\rotatebox{300.0}{\makebox(0,0)[lb]{\smash{{\SetFigFont{12}{14.4}{\rmdefault}{\mddefault}{\itdefault}{$2y+z+o_a+e_b+e_c+|a-b|$}%
}}}}}
\put(12337,-6714){\makebox(0,0)[lb]{\smash{{\SetFigFont{12}{14.4}{\rmdefault}{\mddefault}{\itdefault}{$x+e_a+o_b+o_c$}%
}}}}
\put(14937,-6123){\rotatebox{60.0}{\makebox(0,0)[lb]{\smash{{\SetFigFont{12}{14.4}{\rmdefault}{\mddefault}{\itdefault}{$z+o_a+e_b+e_c$}%
}}}}}
\put(14602,-962){\rotatebox{300.0}{\makebox(0,0)[lb]{\smash{{\SetFigFont{12}{14.4}{\rmdefault}{\mddefault}{\itdefault}{$2y+z+e_a+o_b+o_c+|a-b|$}%
}}}}}
\put(11655,-242){\makebox(0,0)[lb]{\smash{{\SetFigFont{12}{14.4}{\rmdefault}{\mddefault}{\itdefault}{$x+o_a+e_b+e_c$}%
}}}}
\put(1981,-2079){\rotatebox{60.0}{\makebox(0,0)[lb]{\smash{{\SetFigFont{12}{14.4}{\rmdefault}{\mddefault}{\itdefault}{$z+e_a+o_b+o_c$}%
}}}}}
\put(2058,-3804){\rotatebox{300.0}{\makebox(0,0)[lb]{\smash{{\SetFigFont{12}{14.4}{\rmdefault}{\mddefault}{\itdefault}{$2y+z+o_a+e_b+e_c+|a-b|$}%
}}}}}
\put(4561,-6721){\makebox(0,0)[lb]{\smash{{\SetFigFont{12}{14.4}{\rmdefault}{\mddefault}{\itdefault}{$x+e_a+o_b+o_c$}%
}}}}
\put(7161,-6130){\rotatebox{60.0}{\makebox(0,0)[lb]{\smash{{\SetFigFont{12}{14.4}{\rmdefault}{\mddefault}{\itdefault}{$z+o_a+e_b+e_c$}%
}}}}}
\put(6826,-969){\rotatebox{300.0}{\makebox(0,0)[lb]{\smash{{\SetFigFont{12}{14.4}{\rmdefault}{\mddefault}{\itdefault}{$2y+z+e_a+o_b+o_c+|a-b|$}%
}}}}}
\put(3879,-249){\makebox(0,0)[lb]{\smash{{\SetFigFont{12}{14.4}{\rmdefault}{\mddefault}{\itdefault}{$x+o_a+e_b+e_c$}%
}}}}
\put(14937,-21243){\rotatebox{60.0}{\makebox(0,0)[lb]{\smash{{\SetFigFont{12}{14.4}{\rmdefault}{\mddefault}{\itdefault}{$z+o_a+e_b+e_c$}%
}}}}}
\put(12337,-21834){\makebox(0,0)[lb]{\smash{{\SetFigFont{12}{14.4}{\rmdefault}{\mddefault}{\itdefault}{$x+e_a+o_b+o_c$}%
}}}}
\put(9834,-18917){\rotatebox{300.0}{\makebox(0,0)[lb]{\smash{{\SetFigFont{12}{14.4}{\rmdefault}{\mddefault}{\itdefault}{$2y+z+o_a+e_b+e_c+|a-b|$}%
}}}}}
\put(9757,-17192){\rotatebox{60.0}{\makebox(0,0)[lb]{\smash{{\SetFigFont{12}{14.4}{\rmdefault}{\mddefault}{\itdefault}{$z+e_a+o_b+o_c$}%
}}}}}
\end{picture}%
}
\caption{Obtaining the recurrence for the region $R^{\odot}_{x,y,z}(\textbf{a};\textbf{c};\textbf{b})$, when $a\geq b$. Kuo condensation is applied to the region $R^{\odot}_{2,1,2}(1,2 ;\ 2,1,1 ;\ 1,1)$  (picture (a)) as shown on the picture (b).}\label{fig:kuocenter2}
\end{figure}

By applying the Kuo condensation with the same choices of the vertices $u,v,w,s$ in the case $a\geq b$, we get a slightly different recurrence (see Figure \ref{fig:kuocenter2}):
\begin{align}\label{centerrecur1b}
\M(R^{\odot}_{x,y,z}(\textbf{a};\ \textbf{c};\ \textbf{b})) \M(R^{\leftarrow}_{x,y-1,z-1}(\textbf{a}^{+1};\ \textbf{c};\  \textbf{b}))&=
\M(R^{\odot}_{x+1,y,z-1}(\textbf{a};\ \textbf{c};\  \textbf{b})) \M(R^{\leftarrow}_{x-1,y-1,z}(\textbf{a}^{+1};\ \textbf{c};\  \textbf{b}))\notag\\
&+\M(R^{\nwarrow}_{x,y-1,z}(\textbf{a};\ \textbf{c};\  \textbf{b}))\M(R^{\swarrow}_{x,y-1,z-1}(\textbf{a}^{+1};\ \textbf{c};\  \textbf{b})),
\end{align}
when $a\geq b$.

The only differences between the above  two recurrences (\ref{centerrecur1a}) and (\ref{centerrecur1b})  are the $y$-parameters in the second, the fourth, and the sixth regions.

It is not hard to verify that the $h$-parameters (the sum of the quasi-perimeter and the $x$- and $z$-parameters) of all the five regions, that are different from $R^{\odot}_{x,y,z}(\textbf{a};\ \textbf{c};\ \textbf{b})$ in the recurrences (\ref{centerrecur1a}) and (\ref{centerrecur1b}), are strictly less than $h$.

Indeed, let $p$ denote the quasi-perimeter of the region $R^{\odot}_{x,y,z}(\textbf{a};\ \textbf{c};\ \textbf{b})$. The quasi-perimeters of the other five regions  in each of the above recurrences are respectively $p-3$, $p-2$, $p-1$, $p-2$, and $p-1$. Moreover, the sum of the $x$- and $z$- perimeters are respectively, $x+z-1$, $x+z$, $x+z-1$, $x+z$, $x+z-1$.

\subsection{Recurrences for $R^{\leftarrow}$-type regions}\label{subsec:recurR2}

We are now obtaining recurrences for the $R^{\leftarrow}$-type regions. We note that the same application of Kuo condensation in Theorem \ref{kuothm1} as in the case of $R^{\odot}$-type regions does \emph{not} work here. The reason is that the removal of the unit triangles $u,v,w,s$ as in Figures  \ref{fig:kuocenter1} and \ref{fig:kuocenter2} may `push' the center of the auxiliary hexagon too far away from the leftmost point of the middle fern, and the forced lozenges removal yields a new region that are not one of the eight types: $R^{\odot}$-, $R^{\leftarrow}$-, $R^{\swarrow}$-, $R^{\nwarrow}$-, $Q^{\odot}$-, $Q^{\leftarrow}$-, $Q^{\nearrow}$-, and $Q^{\nwarrow}$-types.

We now apply Kuo condensation as in Figure  \ref{fig:kuocenter3} instead. The $u$-triangle is still on the northeast corner of the region, however, the positions of the others three unit triangles are changed as shown in Figure  \ref{fig:kuocenter3}(b).

\begin{figure}\centering
\setlength{\unitlength}{3947sp}%
\begingroup\makeatletter\ifx\SetFigFont\undefined%
\gdef\SetFigFont#1#2#3#4#5{%
  \reset@font\fontsize{#1}{#2pt}%
  \fontfamily{#3}\fontseries{#4}\fontshape{#5}%
  \selectfont}%
\fi\endgroup%
\resizebox{15cm}{!}{
\begin{picture}(0,0)%
\includegraphics{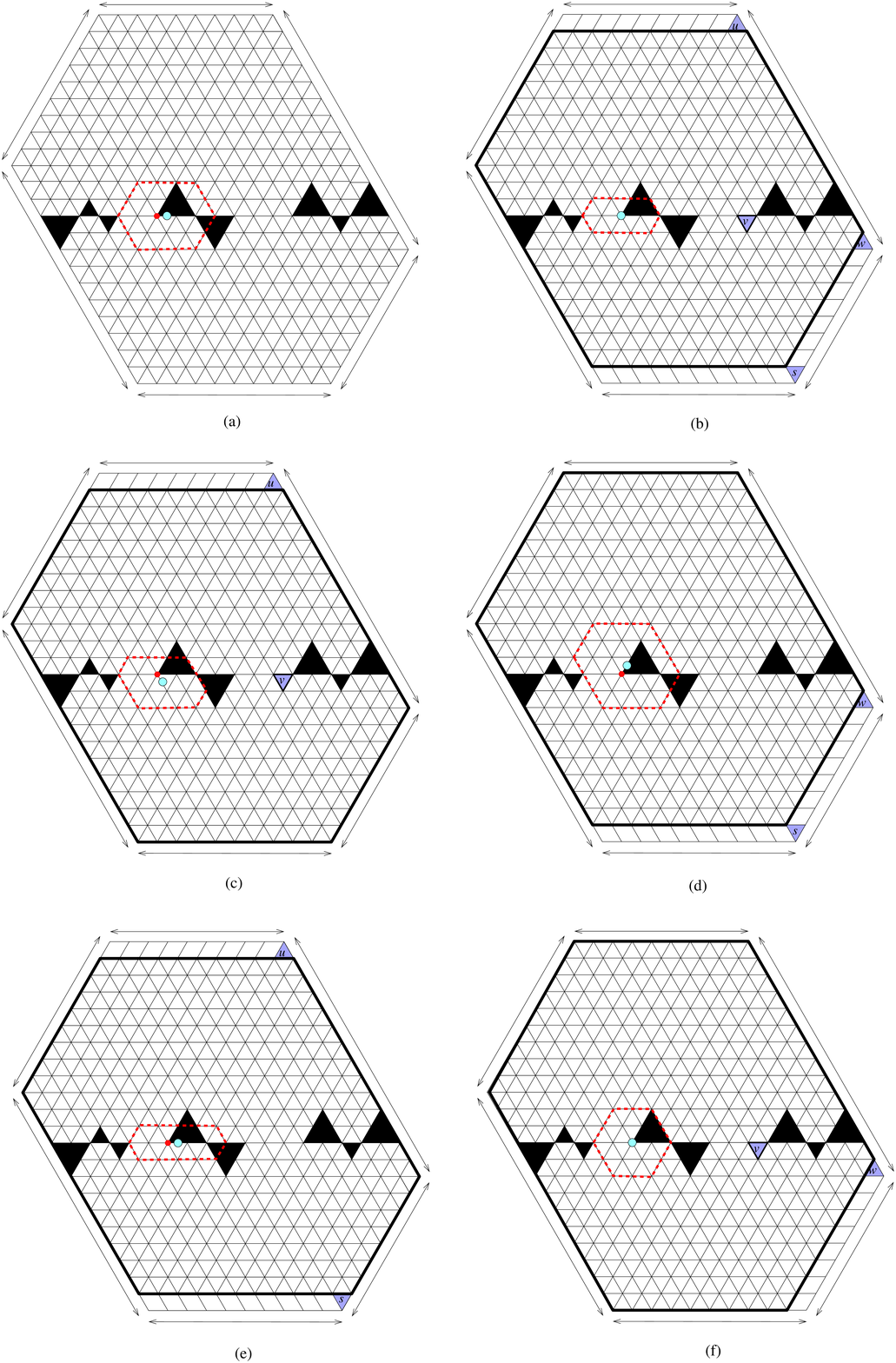}%
\end{picture}%
%
%

\begin{picture}(19001,29131)(1407,-28730)
\put(11715,-22046){\rotatebox{60.0}{\makebox(0,0)[lb]{\smash{{\SetFigFont{16}{16.8}{\rmdefault}{\mddefault}{\updefault}{$z+e_a+o_b+o_c$}%
}}}}}
\put(11841,-24353){\rotatebox{300.0}{\makebox(0,0)[lb]{\smash{{\SetFigFont{16}{16.8}{\rmdefault}{\mddefault}{\updefault}{$2y+z+o_a+e_b+e_c+|a-b|$}%
}}}}}
\put(15577,-28188){\makebox(0,0)[lb]{\smash{{\SetFigFont{16}{16.8}{\rmdefault}{\mddefault}{\updefault}{$x+e_a+o_b+o_c$}%
}}}}
\put(19359,-27254){\rotatebox{60.0}{\makebox(0,0)[lb]{\smash{{\SetFigFont{16}{16.8}{\rmdefault}{\mddefault}{\updefault}{$z+o_a+e_b+e_c$}%
}}}}}
\put(18299,-20844){\rotatebox{300.0}{\makebox(0,0)[lb]{\smash{{\SetFigFont{16}{16.8}{\rmdefault}{\mddefault}{\updefault}{$2y+z+e_a+o_b+o_c+|a-b|$}%
}}}}}
\put(14350,-19477){\makebox(0,0)[lb]{\smash{{\SetFigFont{16}{16.8}{\rmdefault}{\mddefault}{\updefault}{$x+o_a+e_b+e_c$}%
}}}}
\put(1897,-22054){\rotatebox{60.0}{\makebox(0,0)[lb]{\smash{{\SetFigFont{16}{16.8}{\rmdefault}{\mddefault}{\updefault}{$z+e_a+o_b+o_c$}%
}}}}}
\put(2023,-24361){\rotatebox{300.0}{\makebox(0,0)[lb]{\smash{{\SetFigFont{16}{16.8}{\rmdefault}{\mddefault}{\updefault}{$2y+z+o_a+e_b+e_c+|a-b|$}%
}}}}}
\put(5759,-28196){\makebox(0,0)[lb]{\smash{{\SetFigFont{16}{16.8}{\rmdefault}{\mddefault}{\updefault}{$x+e_a+o_b+o_c$}%
}}}}
\put(9541,-27262){\rotatebox{60.0}{\makebox(0,0)[lb]{\smash{{\SetFigFont{16}{16.8}{\rmdefault}{\mddefault}{\updefault}{$z+o_a+e_b+e_c$}%
}}}}}
\put(4297,101){\makebox(0,0)[lb]{\smash{{\SetFigFont{16}{16.8}{\rmdefault}{\mddefault}{\updefault}{$x+o_a+e_b+e_c$}%
}}}}
\put(8246,-1266){\rotatebox{300.0}{\makebox(0,0)[lb]{\smash{{\SetFigFont{16}{16.8}{\rmdefault}{\mddefault}{\updefault}{$2y+z+e_a+o_b+o_c+|a-b|$}%
}}}}}
\put(9306,-7676){\rotatebox{60.0}{\makebox(0,0)[lb]{\smash{{\SetFigFont{16}{16.8}{\rmdefault}{\mddefault}{\updefault}{$z+o_a+e_b+e_c$}%
}}}}}
\put(5524,-8610){\makebox(0,0)[lb]{\smash{{\SetFigFont{16}{16.8}{\rmdefault}{\mddefault}{\updefault}{$x+e_a+o_b+o_c$}%
}}}}
\put(1788,-4775){\rotatebox{300.0}{\makebox(0,0)[lb]{\smash{{\SetFigFont{16}{16.8}{\rmdefault}{\mddefault}{\updefault}{$2y+z+o_a+e_b+e_c+|a-b|$}%
}}}}}
\put(1662,-2468){\rotatebox{60.0}{\makebox(0,0)[lb]{\smash{{\SetFigFont{16}{16.8}{\rmdefault}{\mddefault}{\updefault}{$z+e_a+o_b+o_c$}%
}}}}}
\put(14115,109){\makebox(0,0)[lb]{\smash{{\SetFigFont{16}{16.8}{\rmdefault}{\mddefault}{\updefault}{$x+o_a+e_b+e_c$}%
}}}}
\put(18064,-1258){\rotatebox{300.0}{\makebox(0,0)[lb]{\smash{{\SetFigFont{16}{16.8}{\rmdefault}{\mddefault}{\updefault}{$2y+z+e_a+o_b+o_c+|a-b|$}%
}}}}}
\put(19124,-7668){\rotatebox{60.0}{\makebox(0,0)[lb]{\smash{{\SetFigFont{16}{16.8}{\rmdefault}{\mddefault}{\updefault}{$z+o_a+e_b+e_c$}%
}}}}}
\put(15342,-8602){\makebox(0,0)[lb]{\smash{{\SetFigFont{16}{16.8}{\rmdefault}{\mddefault}{\updefault}{$x+e_a+o_b+o_c$}%
}}}}
\put(11606,-4767){\rotatebox{300.0}{\makebox(0,0)[lb]{\smash{{\SetFigFont{16}{16.8}{\rmdefault}{\mddefault}{\updefault}{$2y+z+o_a+e_b+e_c+|a-b|$}%
}}}}}
\put(11480,-2460){\rotatebox{60.0}{\makebox(0,0)[lb]{\smash{{\SetFigFont{16}{16.8}{\rmdefault}{\mddefault}{\updefault}{$z+e_a+o_b+o_c$}%
}}}}}
\put(4307,-9585){\makebox(0,0)[lb]{\smash{{\SetFigFont{16}{16.8}{\rmdefault}{\mddefault}{\updefault}{$x+o_a+e_b+e_c$}%
}}}}
\put(8256,-10952){\rotatebox{300.0}{\makebox(0,0)[lb]{\smash{{\SetFigFont{16}{16.8}{\rmdefault}{\mddefault}{\updefault}{$2y+z+e_a+o_b+o_c+|a-b|$}%
}}}}}
\put(9316,-17362){\rotatebox{60.0}{\makebox(0,0)[lb]{\smash{{\SetFigFont{14}{16.8}{\rmdefault}{\mddefault}{\updefault}{$z+o_a+e_b+e_c$}%
}}}}}
\put(5534,-18296){\makebox(0,0)[lb]{\smash{{\SetFigFont{16}{16.8}{\rmdefault}{\mddefault}{\updefault}{$x+e_a+o_b+o_c$}%
}}}}
\put(1798,-14461){\rotatebox{300.0}{\makebox(0,0)[lb]{\smash{{\SetFigFont{16}{16.8}{\rmdefault}{\mddefault}{\updefault}{$2y+z+o_a+e_b+e_c+|a-b|$}%
}}}}}
\put(1672,-12154){\rotatebox{60.0}{\makebox(0,0)[lb]{\smash{{\SetFigFont{14}{16.8}{\rmdefault}{\mddefault}{\updefault}{$z+e_a+o_b+o_c$}%
}}}}}
\put(14125,-9577){\makebox(0,0)[lb]{\smash{{\SetFigFont{16}{16.8}{\rmdefault}{\mddefault}{\updefault}{$x+o_a+e_b+e_c$}%
}}}}
\put(18074,-10944){\rotatebox{300.0}{\makebox(0,0)[lb]{\smash{{\SetFigFont{16}{16.8}{\rmdefault}{\mddefault}{\updefault}{$2y+z+e_a+o_b+o_c+|a-b|$}%
}}}}}
\put(19134,-17354){\rotatebox{60.0}{\makebox(0,0)[lb]{\smash{{\SetFigFont{16}{16.8}{\rmdefault}{\mddefault}{\updefault}{$z+o_a+e_b+e_c$}%
}}}}}
\put(15352,-18288){\makebox(0,0)[lb]{\smash{{\SetFigFont{16}{16.8}{\rmdefault}{\mddefault}{\updefault}{$x+e_a+o_b+o_c$}%
}}}}
\put(11616,-14453){\rotatebox{300.0}{\makebox(0,0)[lb]{\smash{{\SetFigFont{16}{16.8}{\rmdefault}{\mddefault}{\updefault}{$2y+z+o_a+e_b+e_c+|a-b|$}%
}}}}}
\put(11490,-12146){\rotatebox{60.0}{\makebox(0,0)[lb]{\smash{{\SetFigFont{16}{16.8}{\rmdefault}{\mddefault}{\updefault}{$z+e_a+o_b+o_c$}%
}}}}}
\put(4532,-19485){\makebox(0,0)[lb]{\smash{{\SetFigFont{16}{16.8}{\rmdefault}{\mddefault}{\updefault}{$x+o_a+e_b+e_c$}%
}}}}
\put(8481,-20852){\rotatebox{300.0}{\makebox(0,0)[lb]{\smash{{\SetFigFont{16}{16.8}{\rmdefault}{\mddefault}{\updefault}{$2y+z+e_a+o_b+o_c+|a-b|$}%
}}}}}
\end{picture}%
}
\caption{Obtaining the recurrence for the region $R^{\leftarrow}_{x,y,z}(\textbf{a};\textbf{c};\textbf{b})$, when $a\leq b$. Kuo condensation is applied to the region $R^{\leftarrow}_{3,2,2}(2,1,1 ;\ 2,2;\ 2,1,2)$ (picture (a)) as shown on the picture (b).}\label{fig:kuocenter3}
\end{figure}

Figure \ref{fig:kuocenter3} tells  us that the product of the numbers of tilings of the two regions in the top row is equal to the product of the tiling numbers of the two regions in the middle row, plus the product of the tiling numbers of the two regions in the bottom row. The figure shows the case when $\textbf{b}$ has an odd number of terms, the removal of the $v$-triangle give a new triangle of side length $1$ to the right fern. In the case $\textbf{b}$ has an even number of terms, this removal increases the side length of the last triangle of the right fern by $1$ unit. By considering forced lozenges as shown in the figure, we get
\begin{align}\label{centerrecur2a}
\M(R^{\leftarrow}_{x,y,z}(\textbf{a};\ \textbf{c};\ \textbf{b})) \M(R^{\odot}_{x,y-1,z-1}(\textbf{a};\ \textbf{c};\ \textbf{b}^{+1}))&=\M(R^{\nwarrow}_{x,y-1,z-1}(\textbf{a};\ \textbf{c};\ \textbf{b}^{+1}))\M(R^{\swarrow}_{x,y-1,z}(\textbf{a};\ \textbf{c};\ \textbf{b}))\notag\\
&+
\M(R^{\leftarrow}_{x+1,y,z-1}(\textbf{a};\ \textbf{c};\ \textbf{b})) \M(R^{\odot}_{x-1,y-1,z}(\textbf{a};\ \textbf{c};\ \textbf{b}^{+1})),
\end{align}
for the case $a\leq b$.

\begin{figure}\centering
\setlength{\unitlength}{3947sp}%
\begingroup\makeatletter\ifx\SetFigFont\undefined%
\gdef\SetFigFont#1#2#3#4#5{%
  \reset@font\fontsize{#1}{#2pt}%
  \fontfamily{#3}\fontseries{#4}\fontshape{#5}%
  \selectfont}%
\fi\endgroup%
\resizebox{15cm}{!}{
\begin{picture}(0,0)%
\includegraphics{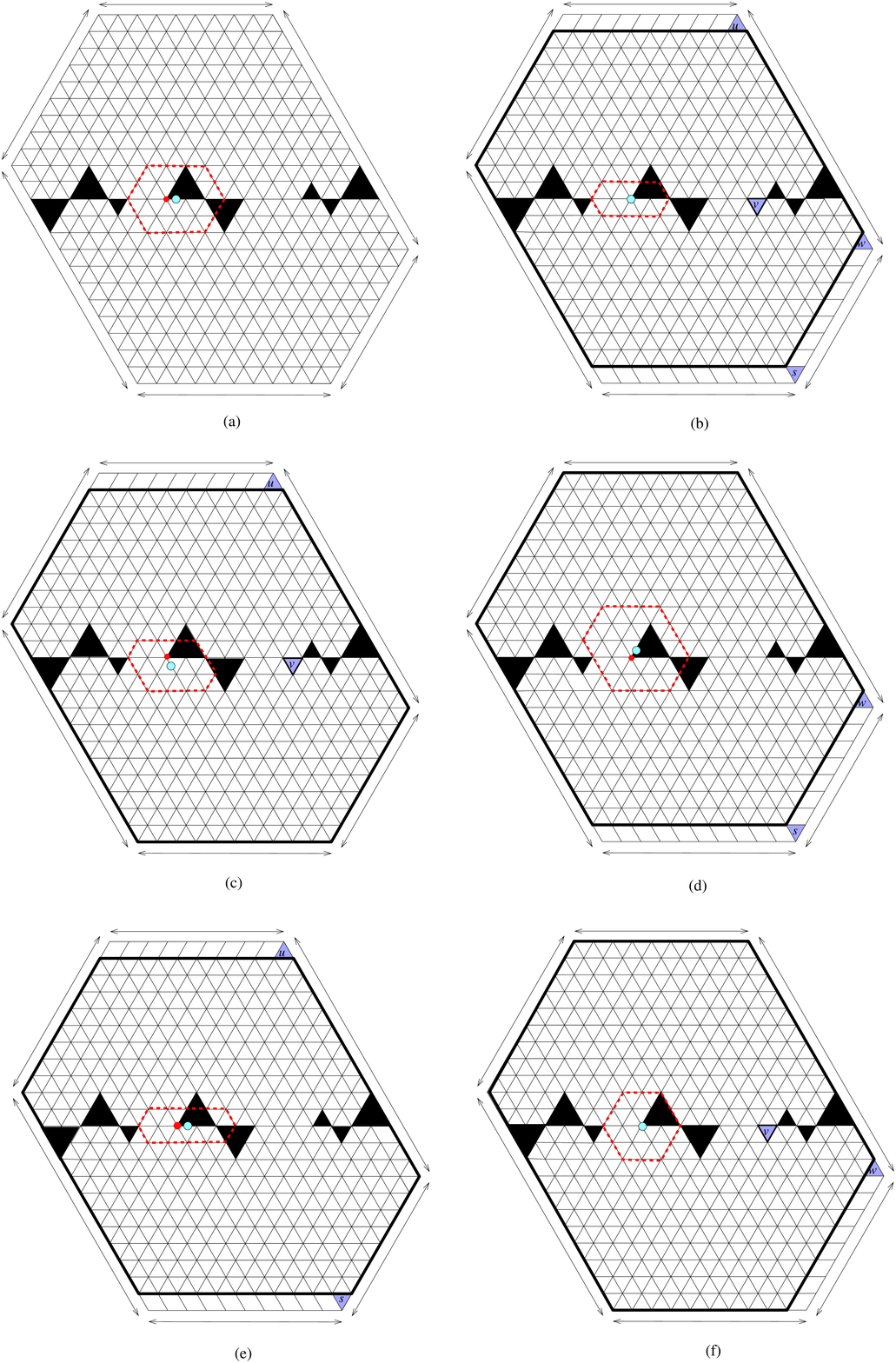}%
\end{picture}%
%
%

\begin{picture}(19035,29106)(1421,-28726)
\put(8256,-10952){\rotatebox{300.0}{\makebox(0,0)[lb]{\smash{{\SetFigFont{14}{16.8}{\rmdefault}{\mddefault}{\updefault}{$2y+z+e_a+o_b+o_c+|a-b|$}%
}}}}}
\put(9316,-17362){\rotatebox{60.0}{\makebox(0,0)[lb]{\smash{{\SetFigFont{14}{16.8}{\rmdefault}{\mddefault}{\updefault}{$z+o_a+e_b+e_c$}%
}}}}}
\put(5534,-18296){\makebox(0,0)[lb]{\smash{{\SetFigFont{14}{16.8}{\rmdefault}{\mddefault}{\updefault}{$x+e_a+o_b+o_c$}%
}}}}
\put(1798,-14461){\rotatebox{300.0}{\makebox(0,0)[lb]{\smash{{\SetFigFont{14}{16.8}{\rmdefault}{\mddefault}{\updefault}{$2y+z+o_a+e_b+e_c+|a-b|$}%
}}}}}
\put(1672,-12154){\rotatebox{60.0}{\makebox(0,0)[lb]{\smash{{\SetFigFont{14}{16.8}{\rmdefault}{\mddefault}{\updefault}{$z+e_a+o_b+o_c$}%
}}}}}
\put(14125,-9577){\makebox(0,0)[lb]{\smash{{\SetFigFont{14}{16.8}{\rmdefault}{\mddefault}{\updefault}{$x+o_a+e_b+e_c$}%
}}}}
\put(18074,-10944){\rotatebox{300.0}{\makebox(0,0)[lb]{\smash{{\SetFigFont{14}{16.8}{\rmdefault}{\mddefault}{\updefault}{$2y+z+e_a+o_b+o_c+|a-b|$}%
}}}}}
\put(19134,-17354){\rotatebox{60.0}{\makebox(0,0)[lb]{\smash{{\SetFigFont{14}{16.8}{\rmdefault}{\mddefault}{\updefault}{$z+o_a+e_b+e_c$}%
}}}}}
\put(15352,-18288){\makebox(0,0)[lb]{\smash{{\SetFigFont{14}{16.8}{\rmdefault}{\mddefault}{\updefault}{$x+e_a+o_b+o_c$}%
}}}}
\put(11616,-14453){\rotatebox{300.0}{\makebox(0,0)[lb]{\smash{{\SetFigFont{14}{16.8}{\rmdefault}{\mddefault}{\updefault}{$2y+z+o_a+e_b+e_c+|a-b|$}%
}}}}}
\put(11490,-12146){\rotatebox{60.0}{\makebox(0,0)[lb]{\smash{{\SetFigFont{14}{16.8}{\rmdefault}{\mddefault}{\updefault}{$z+e_a+o_b+o_c$}%
}}}}}
\put(4532,-19485){\makebox(0,0)[lb]{\smash{{\SetFigFont{14}{16.8}{\rmdefault}{\mddefault}{\updefault}{$x+o_a+e_b+e_c$}%
}}}}
\put(8481,-20852){\rotatebox{300.0}{\makebox(0,0)[lb]{\smash{{\SetFigFont{14}{16.8}{\rmdefault}{\mddefault}{\updefault}{$2y+z+e_a+o_b+o_c+|a-b|$}%
}}}}}
\put(9541,-27262){\rotatebox{60.0}{\makebox(0,0)[lb]{\smash{{\SetFigFont{14}{16.8}{\rmdefault}{\mddefault}{\updefault}{$z+o_a+e_b+e_c$}%
}}}}}
\put(5759,-28196){\makebox(0,0)[lb]{\smash{{\SetFigFont{14}{16.8}{\rmdefault}{\mddefault}{\updefault}{$x+e_a+o_b+o_c$}%
}}}}
\put(2023,-24361){\rotatebox{300.0}{\makebox(0,0)[lb]{\smash{{\SetFigFont{14}{16.8}{\rmdefault}{\mddefault}{\updefault}{$2y+z+o_a+e_b+e_c+|a-b|$}%
}}}}}
\put(1897,-22054){\rotatebox{60.0}{\makebox(0,0)[lb]{\smash{{\SetFigFont{14}{16.8}{\rmdefault}{\mddefault}{\updefault}{$z+e_a+o_b+o_c$}%
}}}}}
\put(14350,-19477){\makebox(0,0)[lb]{\smash{{\SetFigFont{14}{16.8}{\rmdefault}{\mddefault}{\updefault}{$x+o_a+e_b+e_c$}%
}}}}
\put(18299,-20844){\rotatebox{300.0}{\makebox(0,0)[lb]{\smash{{\SetFigFont{14}{16.8}{\rmdefault}{\mddefault}{\updefault}{$2y+z+e_a+o_b+o_c+|a-b|$}%
}}}}}
\put(19359,-27254){\rotatebox{60.0}{\makebox(0,0)[lb]{\smash{{\SetFigFont{14}{16.8}{\rmdefault}{\mddefault}{\updefault}{$z+o_a+e_b+e_c$}%
}}}}}
\put(15577,-28188){\makebox(0,0)[lb]{\smash{{\SetFigFont{14}{16.8}{\rmdefault}{\mddefault}{\updefault}{$x+e_a+o_b+o_c$}%
}}}}
\put(11841,-24353){\rotatebox{300.0}{\makebox(0,0)[lb]{\smash{{\SetFigFont{14}{16.8}{\rmdefault}{\mddefault}{\updefault}{$2y+z+o_a+e_b+e_c+|a-b|$}%
}}}}}
\put(11715,-22046){\rotatebox{60.0}{\makebox(0,0)[lb]{\smash{{\SetFigFont{14}{16.8}{\rmdefault}{\mddefault}{\updefault}{$z+e_a+o_b+o_c$}%
}}}}}
\put(4297,101){\makebox(0,0)[lb]{\smash{{\SetFigFont{14}{16.8}{\rmdefault}{\mddefault}{\updefault}{$x+o_a+e_b+e_c$}%
}}}}
\put(8246,-1266){\rotatebox{300.0}{\makebox(0,0)[lb]{\smash{{\SetFigFont{14}{16.8}{\rmdefault}{\mddefault}{\updefault}{$2y+z+e_a+o_b+o_c+|a-b|$}%
}}}}}
\put(9306,-7676){\rotatebox{60.0}{\makebox(0,0)[lb]{\smash{{\SetFigFont{14}{16.8}{\rmdefault}{\mddefault}{\updefault}{$z+o_a+e_b+e_c$}%
}}}}}
\put(5524,-8610){\makebox(0,0)[lb]{\smash{{\SetFigFont{14}{16.8}{\rmdefault}{\mddefault}{\updefault}{$x+e_a+o_b+o_c$}%
}}}}
\put(1788,-4775){\rotatebox{300.0}{\makebox(0,0)[lb]{\smash{{\SetFigFont{14}{16.8}{\rmdefault}{\mddefault}{\updefault}{$2y+z+o_a+e_b+e_c+|a-b|$}%
}}}}}
\put(1662,-2468){\rotatebox{60.0}{\makebox(0,0)[lb]{\smash{{\SetFigFont{14}{16.8}{\rmdefault}{\mddefault}{\updefault}{$z+e_a+o_b+o_c$}%
}}}}}
\put(14115,109){\makebox(0,0)[lb]{\smash{{\SetFigFont{14}{16.8}{\rmdefault}{\mddefault}{\updefault}{$x+o_a+e_b+e_c$}%
}}}}
\put(18064,-1258){\rotatebox{300.0}{\makebox(0,0)[lb]{\smash{{\SetFigFont{14}{16.8}{\rmdefault}{\mddefault}{\updefault}{$2y+z+e_a+o_b+o_c+|a-b|$}%
}}}}}
\put(19124,-7668){\rotatebox{60.0}{\makebox(0,0)[lb]{\smash{{\SetFigFont{14}{16.8}{\rmdefault}{\mddefault}{\updefault}{$z+o_a+e_b+e_c$}%
}}}}}
\put(15342,-8602){\makebox(0,0)[lb]{\smash{{\SetFigFont{14}{16.8}{\rmdefault}{\mddefault}{\updefault}{$x+e_a+o_b+o_c$}%
}}}}
\put(11606,-4767){\rotatebox{300.0}{\makebox(0,0)[lb]{\smash{{\SetFigFont{14}{16.8}{\rmdefault}{\mddefault}{\updefault}{$2y+z+o_a+e_b+e_c+|a-b|$}%
}}}}}
\put(11480,-2460){\rotatebox{60.0}{\makebox(0,0)[lb]{\smash{{\SetFigFont{14}{16.8}{\rmdefault}{\mddefault}{\updefault}{$z+e_a+o_b+o_c$}%
}}}}}
\put(4307,-9585){\makebox(0,0)[lb]{\smash{{\SetFigFont{14}{16.8}{\rmdefault}{\mddefault}{\updefault}{$x+o_a+e_b+e_c$}%
}}}}
\end{picture}}
\caption{Obtaining the recurrence for the region $R^{\leftarrow}_{x,y,z}(\textbf{a};\ \textbf{c};\ \textbf{b})$, when $a> b$. Kuo condensation is applied to the region $R^{\leftarrow}_{3,2,2}(2,2,1;\ 2,2;\ 2,1,1)$ (picture (a)) as shown on the picture (b).}\label{fig:kuocenter4}
\end{figure}

Similarly, Figure \ref{fig:kuocenter4} tells us that
\begin{align}\label{centerrecur2b}
\M(R^{\leftarrow}_{x,y,z}(\textbf{a};\ \textbf{c};\ \textbf{b}) )\M(R^{\odot}_{x,y,z-1}(\textbf{a};\ \textbf{c};\ \textbf{b}^{+1}))&=\M(R^{\nwarrow}_{x,y,z-1}(\textbf{a};\ \textbf{c};\ \textbf{b}^{+1}))\M(R^{\swarrow}_{x,y-1,z}(\textbf{a};\ \textbf{c};\ \textbf{b}))\notag\\
&+
\M(R^{\leftarrow}_{x+1,y,z-1}(\textbf{a};\ \textbf{c};\ \textbf{b})) \M(R^{\odot}_{x-1,y,z}(\textbf{a};\ \textbf{c};\ \textbf{b}^{+1})),
\end{align}
for the case $a> b$.

Similar to the case of the $R^{\odot}$-type regions, the $h$-parameters of all regions, except the first one, in the above two recurrences are strictly less than the $h$-parameter of the region $R^{\leftarrow}_{x,y,z}(\textbf{a};\ \textbf{c};\ \textbf{b})$.

\subsection{Recurrences for $R^{\swarrow}$-type regions}\label{subsec:recurR3}

\begin{figure}\centering
\setlength{\unitlength}{3947sp}%
\begingroup\makeatletter\ifx\SetFigFont\undefined%
\gdef\SetFigFont#1#2#3#4#5{%
  \reset@font\fontsize{#1}{#2pt}%
  \fontfamily{#3}\fontseries{#4}\fontshape{#5}%
  \selectfont}%
\fi\endgroup%
\resizebox{15cm}{!}{
\begin{picture}(0,0)%
\includegraphics{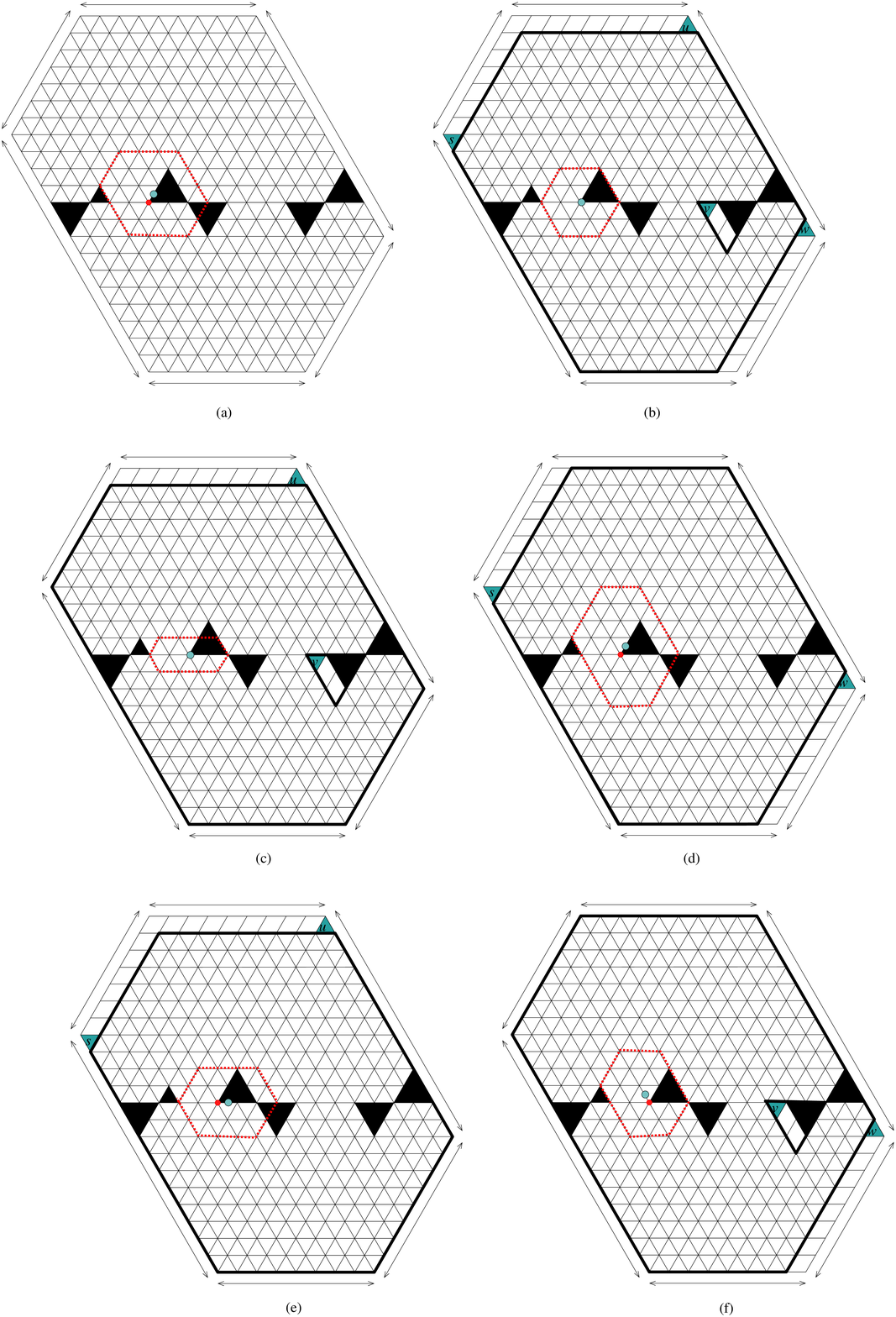}%
\end{picture}%
%
%

\begin{picture}(18668,27745)(1421,-28847)
\put(12508,-24684){\rotatebox{300.0}{\makebox(0,0)[lb]{\smash{{\SetFigFont{12}{14.4}{\rmdefault}{\mddefault}{\updefault}{$2y+z+o_a+e_b+e_c+|a-b|+1$}%
}}}}}
\put(15985,-28363){\makebox(0,0)[lb]{\smash{{\SetFigFont{12}{14.4}{\rmdefault}{\mddefault}{\updefault}{$x+e_a+o_b+o_c$}%
}}}}
\put(19050,-27412){\rotatebox{60.0}{\makebox(0,0)[lb]{\smash{{\SetFigFont{12}{14.4}{\rmdefault}{\mddefault}{\updefault}{$z+o_a+e_b+e_c$}%
}}}}}
\put(18435,-21794){\rotatebox{300.0}{\makebox(0,0)[lb]{\smash{{\SetFigFont{12}{14.4}{\rmdefault}{\mddefault}{\updefault}{$2y+z+e_a+o_b+o_c+|a-b|+1$}%
}}}}}
\put(14640,-20129){\makebox(0,0)[lb]{\smash{{\SetFigFont{12}{14.4}{\rmdefault}{\mddefault}{\updefault}{$x+o_a+e_b+e_c$}%
}}}}
\put(3176,-22130){\rotatebox{60.0}{\makebox(0,0)[lb]{\smash{{\SetFigFont{12}{14.4}{\rmdefault}{\mddefault}{\updefault}{$z+e_a+o_b+o_c$}%
}}}}}
\put(3489,-24690){\rotatebox{300.0}{\makebox(0,0)[lb]{\smash{{\SetFigFont{12}{14.4}{\rmdefault}{\mddefault}{\updefault}{$2y+z+o_a+e_b+e_c+|a-b|+1$}%
}}}}}
\put(6966,-28369){\makebox(0,0)[lb]{\smash{{\SetFigFont{12}{14.4}{\rmdefault}{\mddefault}{\updefault}{$x+e_a+o_b+o_c$}%
}}}}
\put(10031,-27418){\rotatebox{60.0}{\makebox(0,0)[lb]{\smash{{\SetFigFont{12}{14.4}{\rmdefault}{\mddefault}{\updefault}{$z+o_a+e_b+e_c$}%
}}}}}
\put(9416,-21800){\rotatebox{300.0}{\makebox(0,0)[lb]{\smash{{\SetFigFont{12}{14.4}{\rmdefault}{\mddefault}{\updefault}{$2y+z+e_a+o_b+o_c+|a-b|+1$}%
}}}}}
\put(5621,-20135){\makebox(0,0)[lb]{\smash{{\SetFigFont{12}{14.4}{\rmdefault}{\mddefault}{\updefault}{$x+o_a+e_b+e_c$}%
}}}}
\put(11595,-12764){\rotatebox{60.0}{\makebox(0,0)[lb]{\smash{{\SetFigFont{12}{14.4}{\rmdefault}{\mddefault}{\updefault}{$z+e_a+o_b+o_c$}%
}}}}}
\put(11908,-15324){\rotatebox{300.0}{\makebox(0,0)[lb]{\smash{{\SetFigFont{12}{14.4}{\rmdefault}{\mddefault}{\updefault}{$2y+z+o_a+e_b+e_c+|a-b|+1$}%
}}}}}
\put(15385,-19003){\makebox(0,0)[lb]{\smash{{\SetFigFont{12}{14.4}{\rmdefault}{\mddefault}{\updefault}{$x+e_a+o_b+o_c$}%
}}}}
\put(18450,-18052){\rotatebox{60.0}{\makebox(0,0)[lb]{\smash{{\SetFigFont{12}{14.4}{\rmdefault}{\mddefault}{\updefault}{$z+o_a+e_b+e_c$}%
}}}}}
\put(17835,-12434){\rotatebox{300.0}{\makebox(0,0)[lb]{\smash{{\SetFigFont{12}{14.4}{\rmdefault}{\mddefault}{\updefault}{$2y+z+e_a+o_b+o_c+|a-b|+1$}%
}}}}}
\put(4179,-1321){\makebox(0,0)[lb]{\smash{{\SetFigFont{12}{14.4}{\rmdefault}{\mddefault}{\updefault}{$x+o_a+e_b+e_c$}%
}}}}
\put(7974,-2986){\rotatebox{300.0}{\makebox(0,0)[lb]{\smash{{\SetFigFont{12}{14.4}{\rmdefault}{\mddefault}{\updefault}{$2y+z+e_a+o_b+o_c+|a-b|+1$}%
}}}}}
\put(8589,-8604){\rotatebox{60.0}{\makebox(0,0)[lb]{\smash{{\SetFigFont{12}{14.4}{\rmdefault}{\mddefault}{\updefault}{$z+o_a+e_b+e_c$}%
}}}}}
\put(5524,-9555){\makebox(0,0)[lb]{\smash{{\SetFigFont{12}{14.4}{\rmdefault}{\mddefault}{\updefault}{$x+e_a+o_b+o_c$}%
}}}}
\put(2047,-5876){\rotatebox{300.0}{\makebox(0,0)[lb]{\smash{{\SetFigFont{12}{14.4}{\rmdefault}{\mddefault}{\updefault}{$2y+z+o_a+e_b+e_c+|a-b|+1$}%
}}}}}
\put(1734,-3316){\rotatebox{60.0}{\makebox(0,0)[lb]{\smash{{\SetFigFont{12}{14.4}{\rmdefault}{\mddefault}{\updefault}{$z+e_a+o_b+o_c$}%
}}}}}
\put(13198,-1315){\makebox(0,0)[lb]{\smash{{\SetFigFont{12}{14.4}{\rmdefault}{\mddefault}{\updefault}{$x+o_a+e_b+e_c$}%
}}}}
\put(16993,-2980){\rotatebox{300.0}{\makebox(0,0)[lb]{\smash{{\SetFigFont{12}{14.4}{\rmdefault}{\mddefault}{\updefault}{$2y+z+e_a+o_b+o_c+|a-b|+1$}%
}}}}}
\put(17608,-8598){\rotatebox{60.0}{\makebox(0,0)[lb]{\smash{{\SetFigFont{12}{14.4}{\rmdefault}{\mddefault}{\updefault}{$z+o_a+e_b+e_c$}%
}}}}}
\put(14543,-9549){\makebox(0,0)[lb]{\smash{{\SetFigFont{12}{14.4}{\rmdefault}{\mddefault}{\updefault}{$x+e_a+o_b+o_c$}%
}}}}
\put(11066,-5870){\rotatebox{300.0}{\makebox(0,0)[lb]{\smash{{\SetFigFont{12}{14.4}{\rmdefault}{\mddefault}{\updefault}{$2y+z+o_a+e_b+e_c+|a-b|+1$}%
}}}}}
\put(10753,-3310){\rotatebox{60.0}{\makebox(0,0)[lb]{\smash{{\SetFigFont{12}{14.4}{\rmdefault}{\mddefault}{\updefault}{$z+e_a+o_b+o_c$}%
}}}}}
\put(5021,-10775){\makebox(0,0)[lb]{\smash{{\SetFigFont{12}{14.4}{\rmdefault}{\mddefault}{\updefault}{$x+o_a+e_b+e_c$}%
}}}}
\put(8816,-12440){\rotatebox{300.0}{\makebox(0,0)[lb]{\smash{{\SetFigFont{12}{14.4}{\rmdefault}{\mddefault}{\updefault}{$2y+z+e_a+o_b+o_c+|a-b|+1$}%
}}}}}
\put(9431,-18058){\rotatebox{60.0}{\makebox(0,0)[lb]{\smash{{\SetFigFont{12}{14.4}{\rmdefault}{\mddefault}{\updefault}{$z+o_a+e_b+e_c$}%
}}}}}
\put(6366,-19009){\makebox(0,0)[lb]{\smash{{\SetFigFont{12}{14.4}{\rmdefault}{\mddefault}{\updefault}{$x+e_a+o_b+o_c$}%
}}}}
\put(2889,-15330){\rotatebox{300.0}{\makebox(0,0)[lb]{\smash{{\SetFigFont{12}{14.4}{\rmdefault}{\mddefault}{\updefault}{$2y+z+o_a+e_b+e_c+|a-b|+1$}%
}}}}}
\put(2576,-12770){\rotatebox{60.0}{\makebox(0,0)[lb]{\smash{{\SetFigFont{12}{14.4}{\rmdefault}{\mddefault}{\updefault}{$z+e_a+o_b+o_c$}%
}}}}}
\put(14040,-10769){\makebox(0,0)[lb]{\smash{{\SetFigFont{12}{14.4}{\rmdefault}{\mddefault}{\updefault}{$x+o_a+e_b+e_c$}%
}}}}
\put(12195,-22124){\rotatebox{60.0}{\makebox(0,0)[lb]{\smash{{\SetFigFont{12}{14.4}{\rmdefault}{\mddefault}{\updefault}{$z+e_a+o_b+o_c$}%
}}}}}
\end{picture}%
}
\caption{Obtaining the recurrence for the region $R^{\swarrow}_{x,y,z}(\textbf{a};\ \textbf{c};\ \textbf{b})$, when $a\leq b$. Kuo condensation is applied to the region $R^{\swarrow}_{3,2,2}(2,1 ;\ 2,2 ;\ 2,2)$ (picture (a)) as shown on the picture (b).}\label{fig:kuocenter5}
\end{figure}

We apply Kuo condensation in Theorem \ref{kuothm1} to the dual graph $G$ of the region $R^{\swarrow}_{x,y,z}(\textbf{a};\ \textbf{c};\ \textbf{b})$ with the four vertices $u,v,w,s$ chosen as shown in Figure \ref{fig:kuocenter5}(b) in the case $a\leq b$. In particular, the $u$-triangle corresponding to the vertex $u$ is now on the northwest corner of the region, the $v$-, $w$-, $s$-triangles are the shaded ones appearing on the boundary of the region as we go in the clockwise order from the $u$-triangle. By removing forced lozenges in the regions corresponding to the graphs $G-\{u,v,w,s\}$, $G-\{u,v\}$, $G-\{w,s\}$, and $G-\{u,s\}$ (as shown in Figures \ref{fig:kuocenter5}(b)--(e), respectively), we get the regions $R^{\odot}_{x-1,y-1,z}(\textbf{a};\ \textbf{c};\ \textbf{b}^{+1})$, $R^{\odot}_{x,y,z-1}(\textbf{a};\ \textbf{c};\  \textbf{b}^{+1})$, $R^{\swarrow}_{x-1,y-1,z+1}(\textbf{a};\ \textbf{c};\  \textbf{b})$, and $R^{\leftarrow}_{x,y,z}(\textbf{a};\ \textbf{c};\  \textbf{b})$, respectively.

 Unlike the situations in the $R^{\odot}$- and $R^{\leftarrow}$-type regions considered above, after removing forced lozenges from the region corresponding to the graph $G-\{v,w\}$, we do \emph{not} get any regions of one of the four $R$-types (see Figure \ref{fig:kuocenter5}(f)). We next rotate this resulting region by $180^{\circ}$, we get the region  $R^{\nwarrow}_{x-1,y-1,z}(\textbf{b}^{+1};\ \overline{\textbf{c}};\  \textbf{a})$. Here $\overline{\textbf{c}}$ denotes the sequence obtained from the sequence $\textbf{c}$ by reverting the order of the terms if we have an even number of terms, otherwise we revert the sequence and include a new $0$ term in front of the sequence. This way, we get the following recurrence for the $R^{\swarrow}$-type regions, when $a\leq b$:
\begin{align}\label{centerrecur3a}
\M(R^{\swarrow}_{x,y,z}(\textbf{a};\ \textbf{c};\ \textbf{b})) \M(R^{\odot}_{x-1,y-1,z}(\textbf{a};\ \textbf{c};\  \textbf{b}^{+1}))&=\M(R^{\odot}_{x,y,z-1}(\textbf{a};\ \textbf{c};\  \textbf{b}^{+1}))\M(R^{\swarrow}_{x-1,y-1,z+1}(\textbf{a};\ \textbf{c};\  \textbf{b}))\notag\\
&+
\M(R^{\leftarrow}_{x,y,z}(\textbf{a};\ \textbf{c};\  \textbf{b})) \M(R^{\nwarrow}_{x-1,y-1,z}(\textbf{b}^{+1};\ \overline{\textbf{c}};\  \textbf{a})).
\end{align}

\begin{figure}\centering
\setlength{\unitlength}{3947sp}%
\begingroup\makeatletter\ifx\SetFigFont\undefined%
\gdef\SetFigFont#1#2#3#4#5{%
  \reset@font\fontsize{#1}{#2pt}%
  \fontfamily{#3}\fontseries{#4}\fontshape{#5}%
  \selectfont}%
\fi\endgroup%
\resizebox{15cm}{!}{
\begin{picture}(0,0)%
\includegraphics{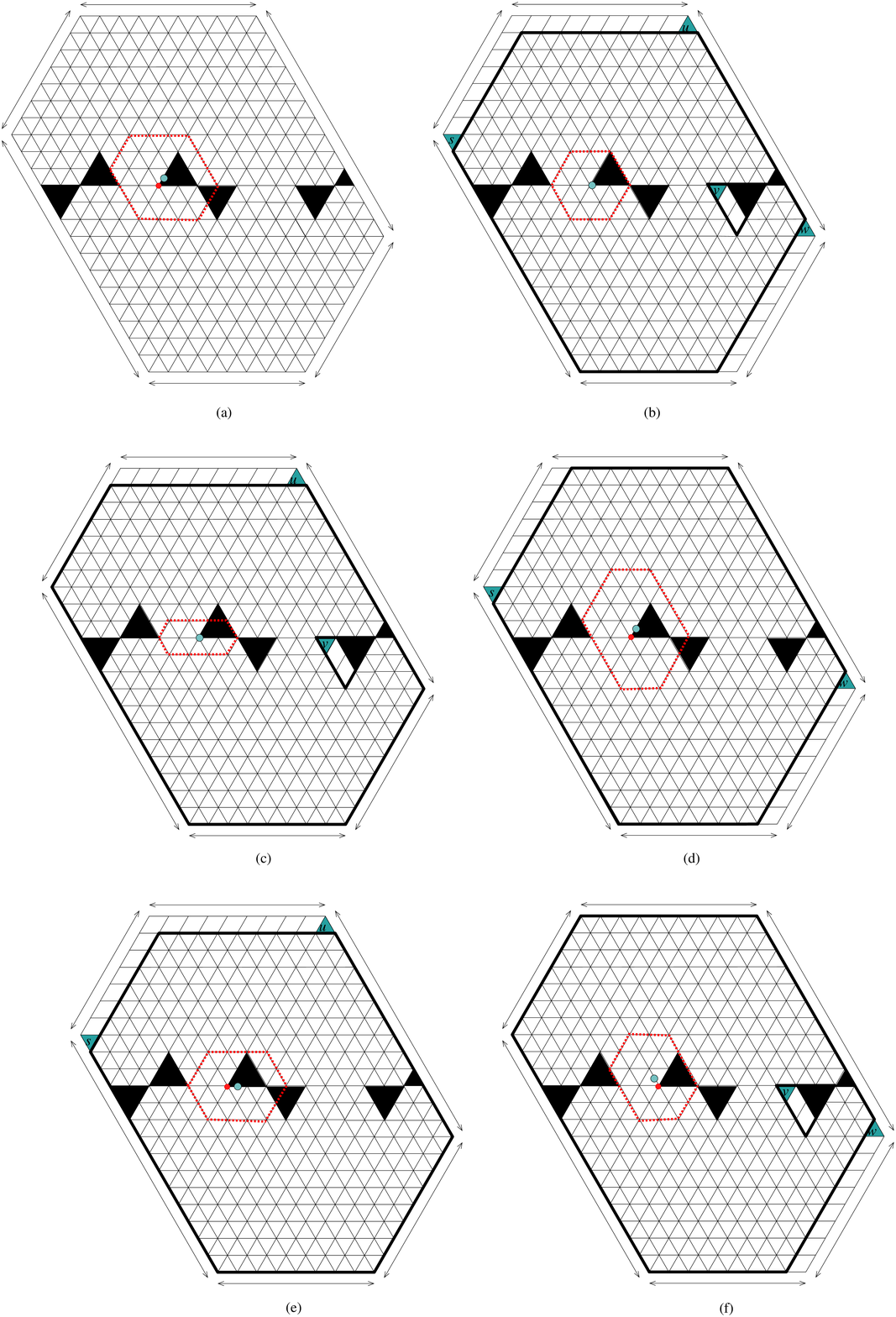}%
\end{picture}%
%
%

\begin{picture}(18668,27751)(1421,-28847)
\put(12508,-24684){\rotatebox{300.0}{\makebox(0,0)[lb]{\smash{{\SetFigFont{12}{14.4}{\rmdefault}{\mddefault}{\updefault}{$2y+z+o_a+e_b+e_c+|a-b|+1$}%
}}}}}
\put(15985,-28363){\makebox(0,0)[lb]{\smash{{\SetFigFont{12}{14.4}{\rmdefault}{\mddefault}{\updefault}{$x+e_a+o_b+o_c$}%
}}}}
\put(19050,-27412){\rotatebox{60.0}{\makebox(0,0)[lb]{\smash{{\SetFigFont{12}{14.4}{\rmdefault}{\mddefault}{\updefault}{$z+o_a+e_b+e_c$}%
}}}}}
\put(18435,-21794){\rotatebox{300.0}{\makebox(0,0)[lb]{\smash{{\SetFigFont{12}{14.4}{\rmdefault}{\mddefault}{\updefault}{$2y+z+e_a+o_b+o_c+|a-b|+1$}%
}}}}}
\put(14640,-20129){\makebox(0,0)[lb]{\smash{{\SetFigFont{12}{14.4}{\rmdefault}{\mddefault}{\updefault}{$x+o_a+e_b+e_c$}%
}}}}
\put(3176,-22130){\rotatebox{60.0}{\makebox(0,0)[lb]{\smash{{\SetFigFont{12}{14.4}{\rmdefault}{\mddefault}{\updefault}{$z+e_a+o_b+o_c$}%
}}}}}
\put(3489,-24690){\rotatebox{300.0}{\makebox(0,0)[lb]{\smash{{\SetFigFont{12}{14.4}{\rmdefault}{\mddefault}{\updefault}{$2y+z+o_a+e_b+e_c+|a-b|+1$}%
}}}}}
\put(6966,-28369){\makebox(0,0)[lb]{\smash{{\SetFigFont{12}{14.4}{\rmdefault}{\mddefault}{\updefault}{$x+e_a+o_b+o_c$}%
}}}}
\put(10031,-27418){\rotatebox{60.0}{\makebox(0,0)[lb]{\smash{{\SetFigFont{12}{14.4}{\rmdefault}{\mddefault}{\updefault}{$z+o_a+e_b+e_c$}%
}}}}}
\put(9416,-21800){\rotatebox{300.0}{\makebox(0,0)[lb]{\smash{{\SetFigFont{12}{14.4}{\rmdefault}{\mddefault}{\updefault}{$2y+z+e_a+o_b+o_c+|a-b|+1$}%
}}}}}
\put(5621,-20135){\makebox(0,0)[lb]{\smash{{\SetFigFont{12}{14.4}{\rmdefault}{\mddefault}{\updefault}{$x+o_a+e_b+e_c$}%
}}}}
\put(11595,-12764){\rotatebox{60.0}{\makebox(0,0)[lb]{\smash{{\SetFigFont{12}{14.4}{\rmdefault}{\mddefault}{\updefault}{$z+e_a+o_b+o_c$}%
}}}}}
\put(11908,-15324){\rotatebox{300.0}{\makebox(0,0)[lb]{\smash{{\SetFigFont{12}{14.4}{\rmdefault}{\mddefault}{\updefault}{$2y+z+o_a+e_b+e_c+|a-b|+1$}%
}}}}}
\put(15385,-19003){\makebox(0,0)[lb]{\smash{{\SetFigFont{12}{14.4}{\rmdefault}{\mddefault}{\updefault}{$x+e_a+o_b+o_c$}%
}}}}
\put(18450,-18052){\rotatebox{60.0}{\makebox(0,0)[lb]{\smash{{\SetFigFont{12}{14.4}{\rmdefault}{\mddefault}{\updefault}{$z+o_a+e_b+e_c$}%
}}}}}
\put(17835,-12434){\rotatebox{300.0}{\makebox(0,0)[lb]{\smash{{\SetFigFont{12}{14.4}{\rmdefault}{\mddefault}{\updefault}{$2y+z+e_a+o_b+o_c+|a-b|+1$}%
}}}}}
\put(4179,-1321){\makebox(0,0)[lb]{\smash{{\SetFigFont{12}{14.4}{\rmdefault}{\mddefault}{\updefault}{$x+o_a+e_b+e_c$}%
}}}}
\put(7974,-2986){\rotatebox{300.0}{\makebox(0,0)[lb]{\smash{{\SetFigFont{12}{14.4}{\rmdefault}{\mddefault}{\updefault}{$2y+z+e_a+o_b+o_c+|a-b|+1$}%
}}}}}
\put(8589,-8604){\rotatebox{60.0}{\makebox(0,0)[lb]{\smash{{\SetFigFont{12}{14.4}{\rmdefault}{\mddefault}{\updefault}{$z+o_a+e_b+e_c$}%
}}}}}
\put(5524,-9555){\makebox(0,0)[lb]{\smash{{\SetFigFont{12}{14.4}{\rmdefault}{\mddefault}{\updefault}{$x+e_a+o_b+o_c$}%
}}}}
\put(2047,-5876){\rotatebox{300.0}{\makebox(0,0)[lb]{\smash{{\SetFigFont{12}{14.4}{\rmdefault}{\mddefault}{\updefault}{$2y+z+o_a+e_b+e_c+|a-b|+1$}%
}}}}}
\put(1734,-3316){\rotatebox{60.0}{\makebox(0,0)[lb]{\smash{{\SetFigFont{12}{14.4}{\rmdefault}{\mddefault}{\updefault}{$z+e_a+o_b+o_c$}%
}}}}}
\put(13198,-1315){\makebox(0,0)[lb]{\smash{{\SetFigFont{12}{14.4}{\rmdefault}{\mddefault}{\updefault}{$x+o_a+e_b+e_c$}%
}}}}
\put(16993,-2980){\rotatebox{300.0}{\makebox(0,0)[lb]{\smash{{\SetFigFont{12}{14.4}{\rmdefault}{\mddefault}{\updefault}{$2y+z+e_a+o_b+o_c+|a-b|+1$}%
}}}}}
\put(17608,-8598){\rotatebox{60.0}{\makebox(0,0)[lb]{\smash{{\SetFigFont{12}{14.4}{\rmdefault}{\mddefault}{\updefault}{$z+o_a+e_b+e_c$}%
}}}}}
\put(14543,-9549){\makebox(0,0)[lb]{\smash{{\SetFigFont{12}{14.4}{\rmdefault}{\mddefault}{\updefault}{$x+e_a+o_b+o_c$}%
}}}}
\put(11066,-5870){\rotatebox{300.0}{\makebox(0,0)[lb]{\smash{{\SetFigFont{12}{14.4}{\rmdefault}{\mddefault}{\updefault}{$2y+z+o_a+e_b+e_c+|a-b|+1$}%
}}}}}
\put(10753,-3310){\rotatebox{60.0}{\makebox(0,0)[lb]{\smash{{\SetFigFont{12}{14.4}{\rmdefault}{\mddefault}{\updefault}{$z+e_a+o_b+o_c$}%
}}}}}
\put(5021,-10775){\makebox(0,0)[lb]{\smash{{\SetFigFont{12}{14.4}{\rmdefault}{\mddefault}{\updefault}{$x+o_a+e_b+e_c$}%
}}}}
\put(8816,-12440){\rotatebox{300.0}{\makebox(0,0)[lb]{\smash{{\SetFigFont{12}{14.4}{\rmdefault}{\mddefault}{\updefault}{$2y+z+e_a+o_b+o_c+|a-b|+1$}%
}}}}}
\put(9431,-18058){\rotatebox{60.0}{\makebox(0,0)[lb]{\smash{{\SetFigFont{12}{14.4}{\rmdefault}{\mddefault}{\updefault}{$z+o_a+e_b+e_c$}%
}}}}}
\put(6366,-19009){\makebox(0,0)[lb]{\smash{{\SetFigFont{12}{14.4}{\rmdefault}{\mddefault}{\updefault}{$x+e_a+o_b+o_c$}%
}}}}
\put(2889,-15330){\rotatebox{300.0}{\makebox(0,0)[lb]{\smash{{\SetFigFont{12}{14.4}{\rmdefault}{\mddefault}{\updefault}{$2y+z+o_a+e_b+e_c+|a-b|+1$}%
}}}}}
\put(2576,-12770){\rotatebox{60.0}{\makebox(0,0)[lb]{\smash{{\SetFigFont{12}{14.4}{\rmdefault}{\mddefault}{\updefault}{$z+e_a+o_b+o_c$}%
}}}}}
\put(14040,-10769){\makebox(0,0)[lb]{\smash{{\SetFigFont{12}{14.4}{\rmdefault}{\mddefault}{\updefault}{$x+o_a+e_b+e_c$}%
}}}}
\put(12195,-22124){\rotatebox{60.0}{\makebox(0,0)[lb]{\smash{{\SetFigFont{12}{14.4}{\rmdefault}{\mddefault}{\updefault}{$z+e_a+o_b+o_c$}%
}}}}}
\end{picture}%
}
\caption{Obtaining the recurrence for the region $R^{\swarrow}_{x,y,z}(\textbf{a};\textbf{c};\textbf{b})$, when $a> b$. Kuo condensation is applied to the region $R^{\swarrow}_{3,2,2}(2,2;\ 2,2 ;\ 1,2)$ (picture (a)) as shown on the picture (b).}\label{fig:kuocenter6}
\end{figure}

Similarly, when $a> b$, we apply Kuo condensation to the dual graph $G$ of the region $R^{\swarrow}_{x,y,z}(\textbf{a};\textbf{c};\textbf{b})$ in the same way as shown in Figure \ref{fig:kuocenter6}. The removal of forced lozenges yields a slightly different recurrence from that in the case $a\leq b$ above:
\begin{align}\label{centerrecur3b}
\M(R^{\swarrow}_{x,y,z}(\textbf{a};\ \textbf{c};\ \textbf{b})) \M(R^{\odot}_{x-1,y,z}(\textbf{a};\ \textbf{c}; \ \textbf{b}^{+1}))&=\M(R^{\odot}_{x,y+1,z-1}(\textbf{a};\ \textbf{c};\  \textbf{b}^{+1}))\M(R^{\swarrow}_{x-1,y-1,z+1}(\textbf{a};\ \textbf{c};\  \textbf{b}))\notag\\
&+
\M(R^{\leftarrow}_{x,y,z}(\textbf{a};\ \textbf{c};\  \textbf{b}) )\M(R^{\nwarrow}_{x-1,y,z}(\textbf{b}^{+1};\ \overline{\textbf{c}}; \ \textbf{a})).
\end{align}
Here the second factor of the second term on the right-hand side is also obtained by rotating   $180^{\circ}$ the region restricted in the bold contour in Figure \ref{fig:kuocenter6}(f).

\subsection{Recurrences for $R^{\nwarrow}$-type regions}\label{subsec:recurR4}

We now consider the recurrences for the last $R$-type regions, the region $R^{\nwarrow}_{x,y,z}(\textbf{a};\textbf{c};\textbf{b})$. We apply Kuo's Theorem \ref{kuothm1} to the dual graph $G$ of the region for the case $a<b$. The four vertices $u,v,w,s$ correspond to the four shaded unit triangles of the same labels as illustrated in Figure \ref{fig:kuocenter7}(b). The difference from the cases treated above is that only two of these four unit triangles are on the boundary of the base hexagon; the other two are at the ends of the left and right ferns. By considering forced lozenges arising from the  removal of the four shaded triangles, we get
\begin{equation}
\M(G-\{u,v,w,s\})= \M(R^{\odot}_{x-1,y,z-1}(\textbf{a}^{+1};\ \textbf{c};\ \textbf{b}^{+1})),
\end{equation}
\begin{equation}
\M(G-\{w,s\})= \M(R^{\leftarrow}_{x,y+1,z-1}(\textbf{a}^{+1};\ \textbf{c};\  \textbf{b})),
\end{equation}
\begin{equation}
\M(G-\{u,s\})=\M(R^{\nwarrow}_{x-1,y,z-1}(\textbf{a}^{+1};\ \textbf{c};\  \textbf{b}^{+1})),
\end{equation}
\begin{equation}
\M(G-\{v,w\})=\M(R^{\odot}_{x,y,z}(\textbf{a};\ \textbf{c};\ \textbf{b})),
\end{equation}
(see Figures \ref{fig:kuocenter7}(b), (d), (e), and (f), respectively).
For the region corresponding to $G-\{u,v\}$, we rotate $180^{\circ}$ the leftover region after removing the forced lozenges to obtain the region  $R^{\swarrow}_{x-1,y-1,z}(\textbf{b}^{+1};\ \overline{\textbf{c}}; \ \textbf{a})$ (see Figure \ref{fig:kuocenter7}(c)). This means that we get the following recurrence for $a<b$:

\begin{figure}\centering
\setlength{\unitlength}{3947sp}%
\begingroup\makeatletter\ifx\SetFigFont\undefined%
\gdef\SetFigFont#1#2#3#4#5{%
  \reset@font\fontsize{#1}{#2pt}%
  \fontfamily{#3}\fontseries{#4}\fontshape{#5}%
  \selectfont}%
\fi\endgroup%
\resizebox{15cm}{!}{
\begin{picture}(0,0)%
\includegraphics{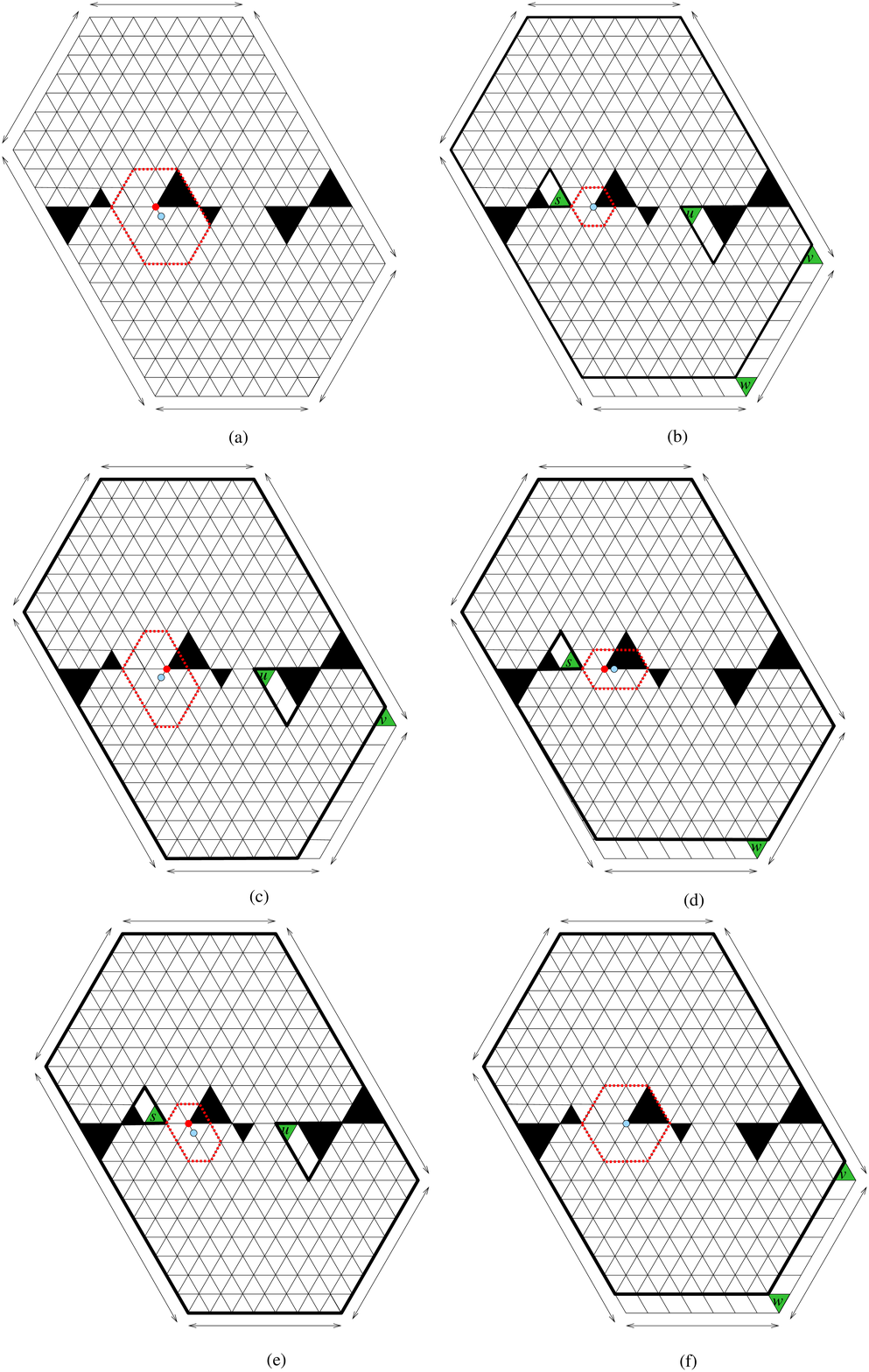}%
\end{picture}%
%
%

\begin{picture}(16191,25832)(3056,-26739)
\put(9479,-8143){\rotatebox{60.0}{\makebox(0,0)[lb]{\smash{{\SetFigFont{12}{14.4}{\rmdefault}{\mddefault}{\updefault}{$z+o_a+e_b+e_c$}%
}}}}}
\put(6757,-8970){\makebox(0,0)[lb]{\smash{{\SetFigFont{12}{14.4}{\rmdefault}{\mddefault}{\updefault}{$x+e_a+o_b+o_c$}%
}}}}
\put(3279,-5073){\rotatebox{300.0}{\makebox(0,0)[lb]{\smash{{\SetFigFont{12}{14.4}{\rmdefault}{\mddefault}{\updefault}{$2y+z+o_a+e_b+e_c+|a-b|+1$}%
}}}}}
\put(3240,-3065){\rotatebox{60.0}{\makebox(0,0)[lb]{\smash{{\SetFigFont{12}{14.4}{\rmdefault}{\mddefault}{\updefault}{$z+e_a+o_b+o_c$}%
}}}}}
\put(14123,-9753){\makebox(0,0)[lb]{\smash{{\SetFigFont{12}{14.4}{\rmdefault}{\mddefault}{\updefault}{$x+o_a+e_b+e_c$}%
}}}}
\put(16987,-10970){\rotatebox{300.0}{\makebox(0,0)[lb]{\smash{{\SetFigFont{12}{14.4}{\rmdefault}{\mddefault}{\updefault}{$2y+z+e_a+o_b+o_c+|a-b|+1$}%
}}}}}
\put(17868,-16780){\rotatebox{60.0}{\makebox(0,0)[lb]{\smash{{\SetFigFont{12}{14.4}{\rmdefault}{\mddefault}{\updefault}{$z+o_a+e_b+e_c$}%
}}}}}
\put(15146,-17607){\makebox(0,0)[lb]{\smash{{\SetFigFont{12}{14.4}{\rmdefault}{\mddefault}{\updefault}{$x+e_a+o_b+o_c$}%
}}}}
\put(11668,-13710){\rotatebox{300.0}{\makebox(0,0)[lb]{\smash{{\SetFigFont{12}{14.4}{\rmdefault}{\mddefault}{\updefault}{$2y+z+o_a+e_b+e_c+|a-b|+1$}%
}}}}}
\put(11629,-11702){\rotatebox{60.0}{\makebox(0,0)[lb]{\smash{{\SetFigFont{12}{14.4}{\rmdefault}{\mddefault}{\updefault}{$z+e_a+o_b+o_c$}%
}}}}}
\put(5941,-9752){\makebox(0,0)[lb]{\smash{{\SetFigFont{12}{14.4}{\rmdefault}{\mddefault}{\updefault}{$x+o_a+e_b+e_c$}%
}}}}
\put(8805,-10969){\rotatebox{300.0}{\makebox(0,0)[lb]{\smash{{\SetFigFont{12}{14.4}{\rmdefault}{\mddefault}{\updefault}{$2y+z+e_a+o_b+o_c+|a-b|+1$}%
}}}}}
\put(9686,-16779){\rotatebox{60.0}{\makebox(0,0)[lb]{\smash{{\SetFigFont{12}{14.4}{\rmdefault}{\mddefault}{\updefault}{$z+o_a+e_b+e_c$}%
}}}}}
\put(6964,-17606){\makebox(0,0)[lb]{\smash{{\SetFigFont{12}{14.4}{\rmdefault}{\mddefault}{\updefault}{$x+e_a+o_b+o_c$}%
}}}}
\put(3486,-13709){\rotatebox{300.0}{\makebox(0,0)[lb]{\smash{{\SetFigFont{12}{14.4}{\rmdefault}{\mddefault}{\updefault}{$2y+z+o_a+e_b+e_c+|a-b|+1$}%
}}}}}
\put(3447,-11701){\rotatebox{60.0}{\makebox(0,0)[lb]{\smash{{\SetFigFont{12}{14.4}{\rmdefault}{\mddefault}{\updefault}{$z+e_a+o_b+o_c$}%
}}}}}
\put(14530,-18245){\makebox(0,0)[lb]{\smash{{\SetFigFont{12}{14.4}{\rmdefault}{\mddefault}{\updefault}{$x+o_a+e_b+e_c$}%
}}}}
\put(17394,-19462){\rotatebox{300.0}{\makebox(0,0)[lb]{\smash{{\SetFigFont{12}{14.4}{\rmdefault}{\mddefault}{\updefault}{$2y+z+e_a+o_b+o_c+|a-b|+1$}%
}}}}}
\put(18275,-25272){\rotatebox{60.0}{\makebox(0,0)[lb]{\smash{{\SetFigFont{12}{14.4}{\rmdefault}{\mddefault}{\updefault}{$z+o_a+e_b+e_c$}%
}}}}}
\put(15553,-26099){\makebox(0,0)[lb]{\smash{{\SetFigFont{12}{14.4}{\rmdefault}{\mddefault}{\updefault}{$x+e_a+o_b+o_c$}%
}}}}
\put(12075,-22202){\rotatebox{300.0}{\makebox(0,0)[lb]{\smash{{\SetFigFont{12}{14.4}{\rmdefault}{\mddefault}{\updefault}{$2y+z+o_a+e_b+e_c+|a-b|+1$}%
}}}}}
\put(12036,-20194){\rotatebox{60.0}{\makebox(0,0)[lb]{\smash{{\SetFigFont{12}{14.4}{\rmdefault}{\mddefault}{\updefault}{$z+e_a+o_b+o_c$}%
}}}}}
\put(6348,-18244){\makebox(0,0)[lb]{\smash{{\SetFigFont{12}{14.4}{\rmdefault}{\mddefault}{\updefault}{$x+o_a+e_b+e_c$}%
}}}}
\put(9212,-19461){\rotatebox{300.0}{\makebox(0,0)[lb]{\smash{{\SetFigFont{12}{14.4}{\rmdefault}{\mddefault}{\updefault}{$2y+z+e_a+o_b+o_c+|a-b|+1$}%
}}}}}
\put(10093,-25271){\rotatebox{60.0}{\makebox(0,0)[lb]{\smash{{\SetFigFont{12}{14.4}{\rmdefault}{\mddefault}{\updefault}{$z+o_a+e_b+e_c$}%
}}}}}
\put(7371,-26098){\makebox(0,0)[lb]{\smash{{\SetFigFont{12}{14.4}{\rmdefault}{\mddefault}{\updefault}{$x+e_a+o_b+o_c$}%
}}}}
\put(3893,-22201){\rotatebox{300.0}{\makebox(0,0)[lb]{\smash{{\SetFigFont{12}{14.4}{\rmdefault}{\mddefault}{\updefault}{$2y+z+o_a+e_b+e_c+|a-b|+1$}%
}}}}}
\put(3854,-20193){\rotatebox{60.0}{\makebox(0,0)[lb]{\smash{{\SetFigFont{12}{14.4}{\rmdefault}{\mddefault}{\updefault}{$z+e_a+o_b+o_c$}%
}}}}}
\put(13916,-1117){\makebox(0,0)[lb]{\smash{{\SetFigFont{12}{14.4}{\rmdefault}{\mddefault}{\updefault}{$x+o_a+e_b+e_c$}%
}}}}
\put(16780,-2334){\rotatebox{300.0}{\makebox(0,0)[lb]{\smash{{\SetFigFont{12}{14.4}{\rmdefault}{\mddefault}{\updefault}{$2y+z+e_a+o_b+o_c+|a-b|+1$}%
}}}}}
\put(17661,-8144){\rotatebox{60.0}{\makebox(0,0)[lb]{\smash{{\SetFigFont{12}{14.4}{\rmdefault}{\mddefault}{\updefault}{$z+o_a+e_b+e_c$}%
}}}}}
\put(14939,-8971){\makebox(0,0)[lb]{\smash{{\SetFigFont{12}{14.4}{\rmdefault}{\mddefault}{\updefault}{$x+e_a+o_b+o_c$}%
}}}}
\put(11461,-5074){\rotatebox{300.0}{\makebox(0,0)[lb]{\smash{{\SetFigFont{12}{14.4}{\rmdefault}{\mddefault}{\updefault}{$2y+z+o_a+e_b+e_c+|a-b|+1$}%
}}}}}
\put(11422,-3066){\rotatebox{60.0}{\makebox(0,0)[lb]{\smash{{\SetFigFont{12}{14.4}{\rmdefault}{\mddefault}{\updefault}{$z+e_a+o_b+o_c$}%
}}}}}
\put(5734,-1116){\makebox(0,0)[lb]{\smash{{\SetFigFont{12}{14.4}{\rmdefault}{\mddefault}{\updefault}{$x+o_a+e_b+e_c$}%
}}}}
\put(8598,-2333){\rotatebox{300.0}{\makebox(0,0)[lb]{\smash{{\SetFigFont{12}{14.4}{\rmdefault}{\mddefault}{\updefault}{$2y+z+e_a+o_b+o_c+|a-b|+1$}%
}}}}}
\end{picture}%
}
\caption{Obtaining the recurrence for the region $R^{\nwarrow}_{x,y,z}(\textbf{a};\ \textbf{c};\ \textbf{b})$, when $a< b$. Kuo condensation is applied to the region $R^{\nwarrow}_{2,2,2}(2,1 ;\ 2,1 ;\ 2,2)$ (picture (a)) as shown on the picture (b).}\label{fig:kuocenter7}
\end{figure}

\begin{align}\label{centerrecur4a}
\M(R^{\nwarrow}_{x,y,z}(\textbf{a};\ \textbf{c};\ \textbf{b})) \M(R^{\odot}_{x-1,y,z-1}(\textbf{a}^{+1};\ \textbf{c};\ \textbf{b}^{+1}))&=\M(R^{\swarrow}_{x-1,y-1,z}(\textbf{b}^{+1};\ \overline{\textbf{c}}; \ \textbf{a}))\M(R^{\leftarrow}_{x,y+1,z-1}(\textbf{a}^{+1};\ \textbf{c};\  \textbf{b}))\notag\\
&+
\M(R^{\nwarrow}_{x-1,y,z-1}(\textbf{a}^{+1};\ \textbf{c};\  \textbf{b}^{+1})) \M(R^{\odot}_{x,y,z}(\textbf{a};\ \textbf{c};\ \textbf{b})).
\end{align}

\begin{figure}\centering
\setlength{\unitlength}{3947sp}%
\begingroup\makeatletter\ifx\SetFigFont\undefined%
\gdef\SetFigFont#1#2#3#4#5{%
  \reset@font\fontsize{#1}{#2pt}%
  \fontfamily{#3}\fontseries{#4}\fontshape{#5}%
  \selectfont}%
\fi\endgroup%
\resizebox{15cm}{!}{
\begin{picture}(0,0)%
\includegraphics{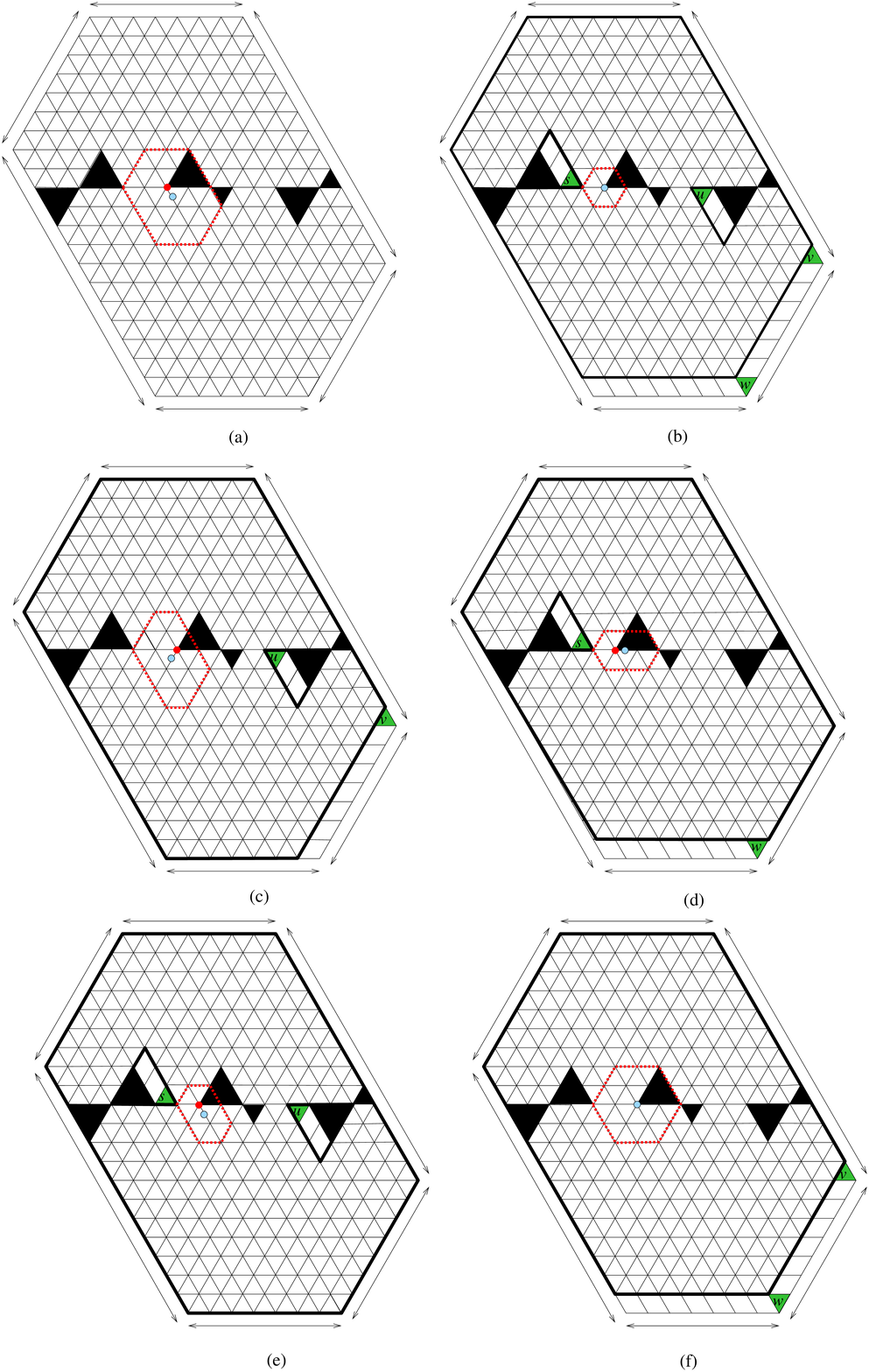}%
\end{picture}%
%
%

\begin{picture}(16194,25834)(3053,-26737)
\put(9479,-8143){\rotatebox{60.0}{\makebox(0,0)[lb]{\smash{{\SetFigFont{12}{14.4}{\rmdefault}{\mddefault}{\updefault}{$z+o_a+e_b+e_c$}%
}}}}}
\put(6757,-8970){\makebox(0,0)[lb]{\smash{{\SetFigFont{12}{14.4}{\rmdefault}{\mddefault}{\updefault}{$x+e_a+o_b+o_c$}%
}}}}
\put(3279,-5073){\rotatebox{300.0}{\makebox(0,0)[lb]{\smash{{\SetFigFont{12}{14.4}{\rmdefault}{\mddefault}{\updefault}{$2y+z+o_a+e_b+e_c+|a-b|+1$}%
}}}}}
\put(3240,-3065){\rotatebox{60.0}{\makebox(0,0)[lb]{\smash{{\SetFigFont{12}{14.4}{\rmdefault}{\mddefault}{\updefault}{$z+e_a+o_b+o_c$}%
}}}}}
\put(14123,-9753){\makebox(0,0)[lb]{\smash{{\SetFigFont{12}{14.4}{\rmdefault}{\mddefault}{\updefault}{$x+o_a+e_b+e_c$}%
}}}}
\put(16987,-10970){\rotatebox{300.0}{\makebox(0,0)[lb]{\smash{{\SetFigFont{12}{14.4}{\rmdefault}{\mddefault}{\updefault}{$2y+z+e_a+o_b+o_c+|a-b|+1$}%
}}}}}
\put(17868,-16780){\rotatebox{60.0}{\makebox(0,0)[lb]{\smash{{\SetFigFont{12}{14.4}{\rmdefault}{\mddefault}{\updefault}{$z+o_a+e_b+e_c$}%
}}}}}
\put(15146,-17607){\makebox(0,0)[lb]{\smash{{\SetFigFont{12}{14.4}{\rmdefault}{\mddefault}{\updefault}{$x+e_a+o_b+o_c$}%
}}}}
\put(11668,-13710){\rotatebox{300.0}{\makebox(0,0)[lb]{\smash{{\SetFigFont{12}{14.4}{\rmdefault}{\mddefault}{\updefault}{$2y+z+o_a+e_b+e_c+|a-b|+1$}%
}}}}}
\put(11629,-11702){\rotatebox{60.0}{\makebox(0,0)[lb]{\smash{{\SetFigFont{12}{14.4}{\rmdefault}{\mddefault}{\updefault}{$z+e_a+o_b+o_c$}%
}}}}}
\put(5941,-9752){\makebox(0,0)[lb]{\smash{{\SetFigFont{12}{14.4}{\rmdefault}{\mddefault}{\updefault}{$x+o+a+e_b+e_c$}%
}}}}
\put(8805,-10969){\rotatebox{300.0}{\makebox(0,0)[lb]{\smash{{\SetFigFont{12}{14.4}{\rmdefault}{\mddefault}{\updefault}{$2y+z+e_a+o_b+o_c+|a-b|+1$}%
}}}}}
\put(9686,-16779){\rotatebox{60.0}{\makebox(0,0)[lb]{\smash{{\SetFigFont{12}{14.4}{\rmdefault}{\mddefault}{\updefault}{$z+o_a+e_b+e_c$}%
}}}}}
\put(6964,-17606){\makebox(0,0)[lb]{\smash{{\SetFigFont{12}{14.4}{\rmdefault}{\mddefault}{\updefault}{$x+e_a+o_b+o_c$}%
}}}}
\put(3486,-13709){\rotatebox{300.0}{\makebox(0,0)[lb]{\smash{{\SetFigFont{12}{14.4}{\rmdefault}{\mddefault}{\updefault}{$2y+z+o_a+e_b+e_c+|a-b|+1$}%
}}}}}
\put(3447,-11701){\rotatebox{60.0}{\makebox(0,0)[lb]{\smash{{\SetFigFont{12}{14.4}{\rmdefault}{\mddefault}{\updefault}{$z+e_a+o_b+o_c$}%
}}}}}
\put(14530,-18245){\makebox(0,0)[lb]{\smash{{\SetFigFont{12}{14.4}{\rmdefault}{\mddefault}{\updefault}{$x+o_a+e_b+e_c$}%
}}}}
\put(17394,-19462){\rotatebox{300.0}{\makebox(0,0)[lb]{\smash{{\SetFigFont{12}{14.4}{\rmdefault}{\mddefault}{\updefault}{$2y+z+e_a+o_b+o_c+|a-b|+1$}%
}}}}}
\put(18275,-25272){\rotatebox{60.0}{\makebox(0,0)[lb]{\smash{{\SetFigFont{12}{14.4}{\rmdefault}{\mddefault}{\updefault}{$z+o_a+e_b+e_c$}%
}}}}}
\put(15553,-26099){\makebox(0,0)[lb]{\smash{{\SetFigFont{12}{14.4}{\rmdefault}{\mddefault}{\updefault}{$x+e_a+o_b+o_c$}%
}}}}
\put(12075,-22202){\rotatebox{300.0}{\makebox(0,0)[lb]{\smash{{\SetFigFont{12}{14.4}{\rmdefault}{\mddefault}{\updefault}{$2y+z+o_a+e_b+e_c+|a-b|+1$}%
}}}}}
\put(12036,-20194){\rotatebox{60.0}{\makebox(0,0)[lb]{\smash{{\SetFigFont{12}{14.4}{\rmdefault}{\mddefault}{\updefault}{$z+e_a+o_b+o_c$}%
}}}}}
\put(6348,-18244){\makebox(0,0)[lb]{\smash{{\SetFigFont{12}{14.4}{\rmdefault}{\mddefault}{\updefault}{$x+o_a+e_b+e_c$}%
}}}}
\put(9212,-19461){\rotatebox{300.0}{\makebox(0,0)[lb]{\smash{{\SetFigFont{12}{14.4}{\rmdefault}{\mddefault}{\updefault}{$2y+z+e_a+o_b+o_c+|a-b|+1$}%
}}}}}
\put(10093,-25271){\rotatebox{60.0}{\makebox(0,0)[lb]{\smash{{\SetFigFont{12}{14.4}{\rmdefault}{\mddefault}{\updefault}{$z+o_a+e_b+e_c$}%
}}}}}
\put(7371,-26098){\makebox(0,0)[lb]{\smash{{\SetFigFont{12}{14.4}{\rmdefault}{\mddefault}{\updefault}{$x+e_a+o_b+o_c$}%
}}}}
\put(3893,-22201){\rotatebox{300.0}{\makebox(0,0)[lb]{\smash{{\SetFigFont{12}{14.4}{\rmdefault}{\mddefault}{\updefault}{$2y+z+o_a+e_b+e_c+|a-b|+1$}%
}}}}}
\put(3854,-20193){\rotatebox{60.0}{\makebox(0,0)[lb]{\smash{{\SetFigFont{12}{14.4}{\rmdefault}{\mddefault}{\updefault}{$z+e_a+o_b+o_c$}%
}}}}}
\put(13916,-1117){\makebox(0,0)[lb]{\smash{{\SetFigFont{12}{14.4}{\rmdefault}{\mddefault}{\updefault}{$x+o_a+e_b+e_c$}%
}}}}
\put(16780,-2334){\rotatebox{300.0}{\makebox(0,0)[lb]{\smash{{\SetFigFont{12}{14.4}{\rmdefault}{\mddefault}{\updefault}{$2y+z+e_a+o_b+o_c+|a-b|+1$}%
}}}}}
\put(17661,-8144){\rotatebox{60.0}{\makebox(0,0)[lb]{\smash{{\SetFigFont{12}{14.4}{\rmdefault}{\mddefault}{\updefault}{$z+o_a+e_b+e_c$}%
}}}}}
\put(14939,-8971){\makebox(0,0)[lb]{\smash{{\SetFigFont{12}{14.4}{\rmdefault}{\mddefault}{\updefault}{$x+e_a+o_b+o_c$}%
}}}}
\put(11461,-5074){\rotatebox{300.0}{\makebox(0,0)[lb]{\smash{{\SetFigFont{12}{14.4}{\rmdefault}{\mddefault}{\updefault}{$2y+z+o_a+e_b+e_c+|a-b|+1$}%
}}}}}
\put(11422,-3066){\rotatebox{60.0}{\makebox(0,0)[lb]{\smash{{\SetFigFont{12}{14.4}{\rmdefault}{\mddefault}{\updefault}{$z+e_a+o_b+o_c$}%
}}}}}
\put(5734,-1116){\makebox(0,0)[lb]{\smash{{\SetFigFont{12}{14.4}{\rmdefault}{\mddefault}{\updefault}{$x+o_a+e_b+e_c$}%
}}}}
\put(8598,-2333){\rotatebox{300.0}{\makebox(0,0)[lb]{\smash{{\SetFigFont{12}{14.4}{\rmdefault}{\mddefault}{\updefault}{$2y+z+e_a+o_b+o_c+|a-b|+1$}%
}}}}}
\end{picture}%
}
\caption{Obtaining the recurrence for the region $R^{\nwarrow}_{x,y,z}(\textbf{a};\textbf{c};\textbf{b})$, when $a> b$. Kuo condensation is applied to the region $R^{\nwarrow}_{2,2,2}(2,2 ;\ 2,1 ;\ 1,2)$ (picture (a)) as shown on the picture (b).}\label{fig:kuocenter8}
\end{figure}

By working similarly as in Figure \ref{fig:kuocenter8} in the case when $a> b$, we get
\begin{align}\label{centerrecur4b}
\M(R^{\nwarrow}_{x,y,z}(\textbf{a};\ \textbf{c};\ \textbf{b})) \M(R^{\odot}_{x-1,y,z-1}(\textbf{a}^{+1};\ \textbf{c};\ \textbf{b}^{+1}))&=\M(R^{\swarrow}_{x-1,y,z}(\textbf{b}^{+1};\ \overline{\textbf{c}}; \ \textbf{a}))\M(R^{\leftarrow}_{x,y,z-1}(\textbf{a}^{+1};\ \textbf{c};\  \textbf{b}))\notag\\
&+
\M(R^{\nwarrow}_{x-1,y,z-1}(\textbf{a}^{+1};\ \textbf{c};\  \textbf{b}^{+1})) \M(R^{\odot}_{x,y,z}(\textbf{a};\ \textbf{c};\ \textbf{b})).
\end{align}

Finally, our choice of the four vertices $u,v,w,s$ causes an additional case when $a=b$ (the corresponding  picture for Kuo condensation is not shown here). Processing similarly to the two cases treated above gives us:
\begin{align}\label{centerrecur4c}
\M(R^{\nwarrow}_{x,y,z}(\textbf{a};\ \textbf{c};\ \textbf{b})) \M(R^{\odot}_{x-1,y,z-1}(\textbf{a}^{+1};\ \textbf{c};\ \textbf{b}^{+1}))&=\M(R^{\swarrow}_{x-1,y-1,z}(\textbf{b}^{+1};\ \overline{\textbf{c}}; \ \textbf{a}))\M(R^{\leftarrow}_{x,y,z-1}(\textbf{a}^{+1};\ \textbf{c};\  \textbf{b}))\notag\\
&+
\M(R^{\nwarrow}_{x-1,y,z-1}(\textbf{a}^{+1};\ \textbf{c};\  \textbf{b}^{+1})) \M(R^{\odot}_{x,y,z}(\textbf{a};\ \textbf{c};\ \textbf{b}))
\end{align}
when $a=b$.

\subsection{Recurrences for $Q^{\odot}$-type regions}\label{subsec:recurQ1}
We now setup recurrences for the $Q^{\odot}$-type regions.
\begin{figure}\centering
\setlength{\unitlength}{3947sp}%
\begingroup\makeatletter\ifx\SetFigFont\undefined%
\gdef\SetFigFont#1#2#3#4#5{%
  \reset@font\fontsize{#1}{#2pt}%
  \fontfamily{#3}\fontseries{#4}\fontshape{#5}%
  \selectfont}%
\fi\endgroup%
\resizebox{15cm}{!}{
\begin{picture}(0,0)%
\includegraphics{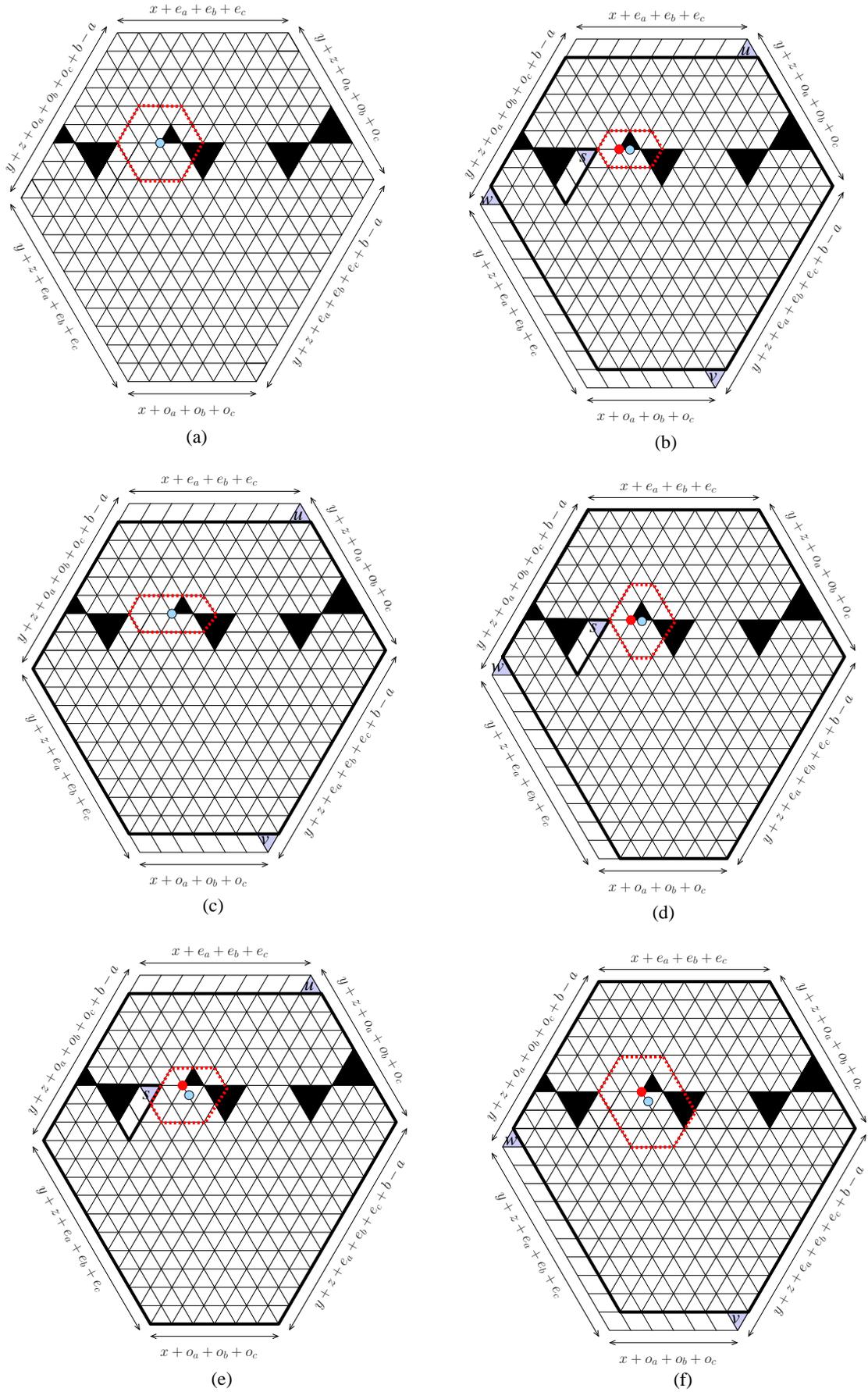}%
\end{picture}%
%
%

\begin{picture}(16984,26942)(1630,-28821)
\put(16354,-18240){\rotatebox{60.0}{\makebox(0,0)[lb]{\smash{{\SetFigFont{17}{20.4}{\rmdefault}{\mddefault}{\itdefault}{$y+z+e_a+e_b+e_c+b-a$}%
}}}}}
\put(13294,-19150){\makebox(0,0)[lb]{\smash{{\SetFigFont{17}{20.4}{\rmdefault}{\mddefault}{\itdefault}{$x+o_a+o_b+o_c$}%
}}}}
\put(10954,-15910){\rotatebox{300.0}{\makebox(0,0)[lb]{\smash{{\SetFigFont{17}{20.4}{\rmdefault}{\mddefault}{\itdefault}{$y+z+e_a+e_b+e_c$}%
}}}}}
\put(10864,-14490){\rotatebox{60.0}{\makebox(0,0)[lb]{\smash{{\SetFigFont{17}{20.4}{\rmdefault}{\mddefault}{\itdefault}{$y+z+o_a+o_b+o_c+b-a$}%
}}}}}
\put(4930,-20392){\makebox(0,0)[lb]{\smash{{\SetFigFont{17}{20.4}{\rmdefault}{\mddefault}{\itdefault}{$x+e_a+e_b+e_c$}%
}}}}
\put(8130,-20852){\rotatebox{300.0}{\makebox(0,0)[lb]{\smash{{\SetFigFont{17}{20.4}{\rmdefault}{\mddefault}{\itdefault}{$y+z+o_a+o_b+o_c$}%
}}}}}
\put(7770,-27222){\rotatebox{60.0}{\makebox(0,0)[lb]{\smash{{\SetFigFont{17}{20.4}{\rmdefault}{\mddefault}{\itdefault}{$y+z+e_a+e_b+e_c+b-a$}%
}}}}}
\put(4710,-28132){\makebox(0,0)[lb]{\smash{{\SetFigFont{17}{20.4}{\rmdefault}{\mddefault}{\itdefault}{$x+o_a+o_b+o_c$}%
}}}}
\put(2370,-24892){\rotatebox{300.0}{\makebox(0,0)[lb]{\smash{{\SetFigFont{17}{20.4}{\rmdefault}{\mddefault}{\itdefault}{$y+z+e_a+e_b+e_c$}%
}}}}}
\put(2280,-23472){\rotatebox{60.0}{\makebox(0,0)[lb]{\smash{{\SetFigFont{17}{20.4}{\rmdefault}{\mddefault}{\itdefault}{$y+z+o_a+o_b+o_c+b-a$}%
}}}}}
\put(13718,-20515){\makebox(0,0)[lb]{\smash{{\SetFigFont{17}{20.4}{\rmdefault}{\mddefault}{\itdefault}{$x+e_a+e_b+e_c$}%
}}}}
\put(16918,-20975){\rotatebox{300.0}{\makebox(0,0)[lb]{\smash{{\SetFigFont{17}{20.4}{\rmdefault}{\mddefault}{\itdefault}{$y+z+o_a+o_b+o_c$}%
}}}}}
\put(16558,-27345){\rotatebox{60.0}{\makebox(0,0)[lb]{\smash{{\SetFigFont{17}{20.4}{\rmdefault}{\mddefault}{\itdefault}{$y+z+e_a+e_b+e_c+b-a$}%
}}}}}
\put(13498,-28255){\makebox(0,0)[lb]{\smash{{\SetFigFont{17}{20.4}{\rmdefault}{\mddefault}{\itdefault}{$x+o_a+o_b+o_c$}%
}}}}
\put(11158,-25015){\rotatebox{300.0}{\makebox(0,0)[lb]{\smash{{\SetFigFont{17}{20.4}{\rmdefault}{\mddefault}{\itdefault}{$y+z+e_a+e_b+e_c$}%
}}}}}
\put(11068,-23595){\rotatebox{60.0}{\makebox(0,0)[lb]{\smash{{\SetFigFont{17}{20.4}{\rmdefault}{\mddefault}{\itdefault}{$y+z+o_a+o_b+o_c+b-a$}%
}}}}}
\put(4501,-2201){\makebox(0,0)[lb]{\smash{{\SetFigFont{17}{20.4}{\rmdefault}{\mddefault}{\itdefault}{$x+e_a+e_b+e_c$}%
}}}}
\put(7701,-2661){\rotatebox{300.0}{\makebox(0,0)[lb]{\smash{{\SetFigFont{17}{20.4}{\rmdefault}{\mddefault}{\itdefault}{$y+z+o_a+o_b+o_c$}%
}}}}}
\put(7341,-9031){\rotatebox{60.0}{\makebox(0,0)[lb]{\smash{{\SetFigFont{17}{20.4}{\rmdefault}{\mddefault}{\itdefault}{$y+z+e_a+e_b+e_c+b-a$}%
}}}}}
\put(4281,-9941){\makebox(0,0)[lb]{\smash{{\SetFigFont{17}{20.4}{\rmdefault}{\mddefault}{\itdefault}{$x+o_a+o_b+o_c$}%
}}}}
\put(1941,-6701){\rotatebox{300.0}{\makebox(0,0)[lb]{\smash{{\SetFigFont{17}{20.4}{\rmdefault}{\mddefault}{\itdefault}{$y+z+e_a+e_b+e_c$}%
}}}}}
\put(1851,-5281){\rotatebox{60.0}{\makebox(0,0)[lb]{\smash{{\SetFigFont{17}{20.4}{\rmdefault}{\mddefault}{\itdefault}{$y+z+o_a+o_b+o_c+b-a$}%
}}}}}
\put(13289,-2324){\makebox(0,0)[lb]{\smash{{\SetFigFont{17}{20.4}{\rmdefault}{\mddefault}{\itdefault}{$x+e_a+e_b+e_c$}%
}}}}
\put(16489,-2784){\rotatebox{300.0}{\makebox(0,0)[lb]{\smash{{\SetFigFont{17}{20.4}{\rmdefault}{\mddefault}{\itdefault}{$y+z+o_a+o_b+o_c$}%
}}}}}
\put(16129,-9154){\rotatebox{60.0}{\makebox(0,0)[lb]{\smash{{\SetFigFont{17}{20.4}{\rmdefault}{\mddefault}{\itdefault}{$y+z+e_a+e_b+e_c+b-a$}%
}}}}}
\put(13069,-10064){\makebox(0,0)[lb]{\smash{{\SetFigFont{17}{20.4}{\rmdefault}{\mddefault}{\itdefault}{$x+o_a+o_b+o_c$}%
}}}}
\put(10729,-6824){\rotatebox{300.0}{\makebox(0,0)[lb]{\smash{{\SetFigFont{17}{20.4}{\rmdefault}{\mddefault}{\itdefault}{$y+z+e_a+e_b+e_c$}%
}}}}}
\put(10639,-5404){\rotatebox{60.0}{\makebox(0,0)[lb]{\smash{{\SetFigFont{17}{20.4}{\rmdefault}{\mddefault}{\itdefault}{$y+z+o_a+o_b+o_c+b-a$}%
}}}}}
\put(4726,-11287){\makebox(0,0)[lb]{\smash{{\SetFigFont{17}{20.4}{\rmdefault}{\mddefault}{\itdefault}{$x+e_a+e_b+e_c$}%
}}}}
\put(7926,-11747){\rotatebox{300.0}{\makebox(0,0)[lb]{\smash{{\SetFigFont{17}{20.4}{\rmdefault}{\mddefault}{\itdefault}{$y+z+o_a+o_b+o_c$}%
}}}}}
\put(7566,-18117){\rotatebox{60.0}{\makebox(0,0)[lb]{\smash{{\SetFigFont{17}{20.4}{\rmdefault}{\mddefault}{\itdefault}{$y+z+e_a+e_b+e_c+b-a$}%
}}}}}
\put(4506,-19027){\makebox(0,0)[lb]{\smash{{\SetFigFont{17}{20.4}{\rmdefault}{\mddefault}{\itdefault}{$x+o_a+o_b+o_c$}%
}}}}
\put(2166,-15787){\rotatebox{300.0}{\makebox(0,0)[lb]{\smash{{\SetFigFont{17}{20.4}{\rmdefault}{\mddefault}{\itdefault}{$y+z+e_a+e_b+e_c$}%
}}}}}
\put(2076,-14367){\rotatebox{60.0}{\makebox(0,0)[lb]{\smash{{\SetFigFont{17}{20.4}{\rmdefault}{\mddefault}{\itdefault}{$y+z+o_a+o_b+o_c+b-a$}%
}}}}}
\put(13514,-11410){\makebox(0,0)[lb]{\smash{{\SetFigFont{17}{20.4}{\rmdefault}{\mddefault}{\itdefault}{$x+e_a+e_b+e_c$}%
}}}}
\put(16714,-11870){\rotatebox{300.0}{\makebox(0,0)[lb]{\smash{{\SetFigFont{17}{20.4}{\rmdefault}{\mddefault}{\itdefault}{$y+z+o_a+o_b+o_c$}%
}}}}}
\end{picture}%
}
\caption{Obtaining the recurrence for the region $Q^{\odot}_{x,y,z}(\textbf{a};\textbf{c};\textbf{b})$, when $a< b$. Kuo condensation is applied to the region $Q^{\odot}_{2,2,2}(1,2 ;\ 1,2 ;\ 2,2)$ (picture (a)) as shown on the picture (b).}\label{fig:kuocenter9}
\end{figure}

We apply again Kuo's Theorem \ref{kuothm1} to the dual graph $G$ of the region $Q^{\odot}_{x,y,z}(\textbf{a};\ \textbf{c};\ \textbf{b})$ as in Figure \ref{fig:kuocenter9} with the four vertices $u,v,w,s$ chosen as in picture (b). The six regions in Figure  \ref{fig:kuocenter9}  correspond to the six terms in the equation of Theorem \ref{kuothm1}. Again, the figure says that the  product of tiling numbers of the two regions in the top row equals the product of the tiling numbers of the two regions in the middle row, plus the product of tiling numbers of the two regions in the bottom row. By considering forced lozenges as shown in the figure, the above identity is converted into the following recurrence for $Q^{\odot}$-regions:

\begin{align}\label{centerrecur5a}
\M(Q^{\odot}_{x,y,z}(\textbf{a};\ \textbf{c};\ \textbf{b})) \M(Q^{\leftarrow}_{x,y,z-1}(\textbf{a}^{+1};\ \textbf{c};\  \textbf{b}))&=\M(Q^{\odot}_{x+1,y,z-1}(\textbf{a};\ \textbf{c};\ \textbf{b})) \M(Q^{\leftarrow}_{x-1,y,z}(\textbf{a}^{+1};\ \textbf{c};\ \textbf{b}))\notag\\
&+
\M(Q^{\nearrow}_{x,y,z-1}(\textbf{b};\ \textbf{c}^{\leftrightarrow}; \ \textbf{a}^{+1}))\M(Q^{\nwarrow}_{x,y-1,z}(\textbf{a};\ \textbf{c};\ \textbf{b})),
\end{align}
for the case $a< b$.  Strictly speaking, the region obtained by removing forced lozenges from the region corresponding to the graph $G-\{u,s\}$ is \emph{not} an $Q^{\odot}$-, $Q^{\leftarrow}$-, $Q^{\nwarrow}$-, or $Q^{\nearrow}$-type region. We need to reflect this region over a vertical line to get back the region  $Q^{\nearrow}_{x,y,z-1}(\textbf{b};\ \textbf{c}^{\leftrightarrow}; \textbf{a}^{+1})$,  where $\textbf{c}^{\leftrightarrow}$ denotes the sequence obtained by reserving the sequence $\textbf{c}$ if $\textbf{c}$ has an odd number of terms, and by reversing and adding a $0$ term in the beginning of $\textbf{c}$ in the case of even number of terms. The reader should distinguish the sequence   $\textbf{c}^{\leftrightarrow}$ from it `dual', $\overline{\textbf{c}}$, in the recurrences for the regions $R^{\nwarrow}_{x,y,z}(\textbf{a};\ \textbf{c};\ \textbf{b})$ and $R^{\swarrow}_{x,y,z}(\textbf{a};\ \textbf{c};\ \textbf{b})$ above.

Working in the same way as in the case $a< b$
, we obtain:
\begin{align}\label{centerrecur5b}
\M(Q^{\odot}_{x,y,z}(\textbf{a};\ \textbf{c};\ \textbf{b})) \M(Q^{\leftarrow}_{x,y-1,z-1}(\textbf{a}^{+1};\ \textbf{c};\ \textbf{b}))&=\M(Q^{\odot}_{x+1,y,z-1}(\textbf{a};\ \textbf{c};\ \textbf{b})) \M(Q^{\leftarrow}_{x-1,y-1,z}(\textbf{a}^{+1};\ \textbf{c};\ \textbf{b}))\notag\\
&+
\M(Q^{\nearrow}_{x,y-1,z-1}(\textbf{b};\ \textbf{c}^{\leftrightarrow};\  \textbf{a}^{+1}))\M(Q^{\nwarrow}_{x,y-1,z}(\textbf{a};\ \textbf{c};\ \textbf{b}))
\end{align}
for $a\geq b$.

One may realize that the application of Kuo condensation to the $Q^{\odot}$-type regions is similar to that in the case of $R^{\odot}$-type regions treated before. However, the resulting recurrences in the two cases are \emph{not} the same.

\subsection{Recurrences for $Q^{\leftarrow}$-type regions}\label{subsec:recurQ2}

\begin{figure}\centering
\setlength{\unitlength}{3947sp}%
\begingroup\makeatletter\ifx\SetFigFont\undefined%
\gdef\SetFigFont#1#2#3#4#5{%
  \reset@font\fontsize{#1}{#2pt}%
  \fontfamily{#3}\fontseries{#4}\fontshape{#5}%
  \selectfont}%
\fi\endgroup%
\resizebox{15cm}{!}{
\begin{picture}(0,0)%
\includegraphics{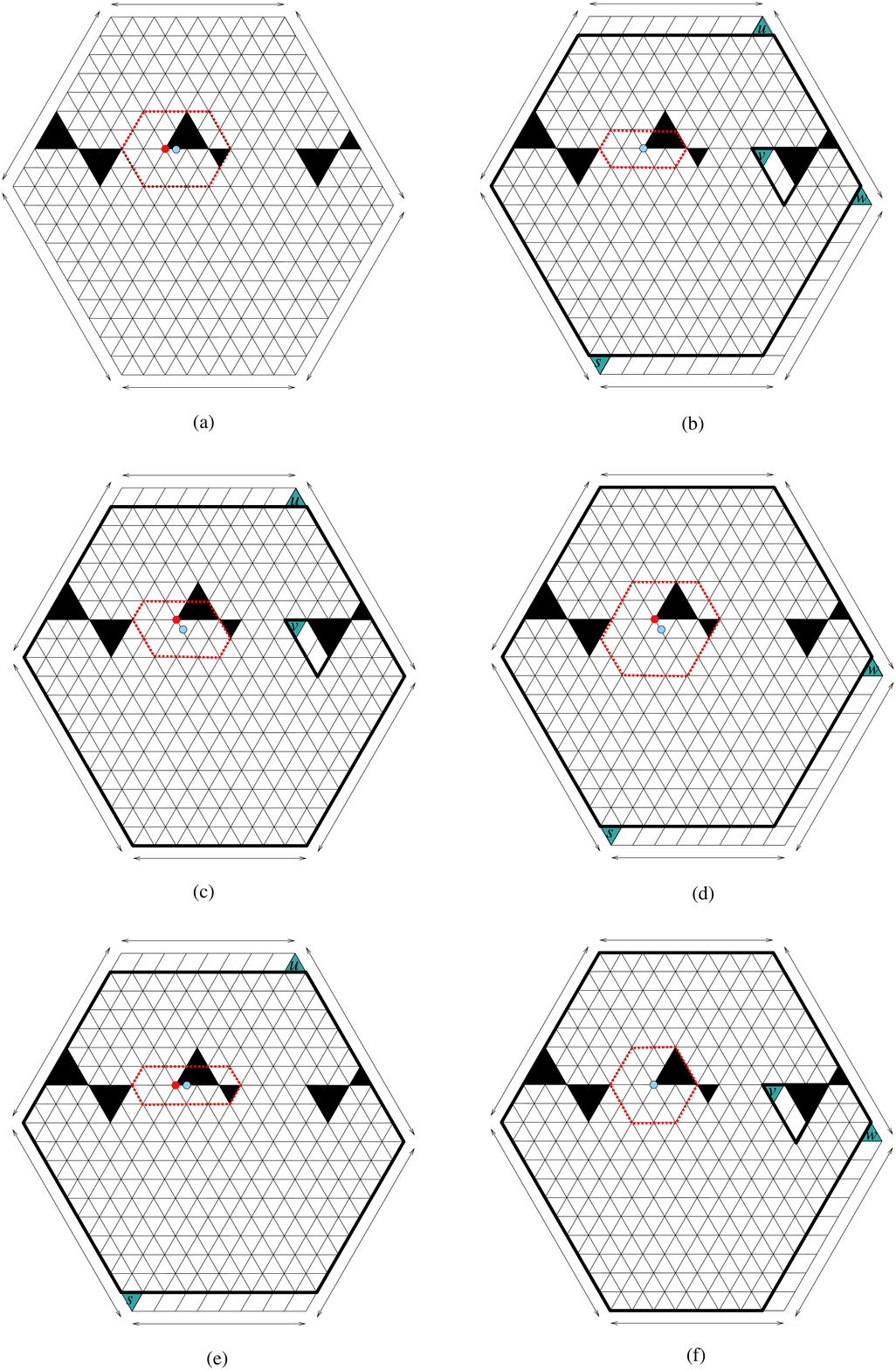}%
\end{picture}%
%
%

\begin{picture}(17318,26173)(2183,-27696)
\put(5015,-9464){\makebox(0,0)[lb]{\smash{{\SetFigFont{17}{20.4}{\rmdefault}{\mddefault}{\itdefault}{$x+o_a+o_b+o_c$}%
}}}}
\put(4938,-1845){\makebox(0,0)[lb]{\smash{{\SetFigFont{17}{20.4}{\rmdefault}{\mddefault}{\itdefault}{$x+e_a+e_b+e_c$}%
}}}}
\put(2404,-4949){\rotatebox{60.0}{\makebox(0,0)[lb]{\smash{{\SetFigFont{17}{20.4}{\rmdefault}{\mddefault}{\itdefault}{$y+z+o_a+o_b+o_c$}%
}}}}}
\put(2354,-6040){\rotatebox{300.0}{\makebox(0,0)[lb]{\smash{{\SetFigFont{17}{20.4}{\rmdefault}{\mddefault}{\itdefault}{$y+z+e_a+e_b+e_c+a-b$}%
}}}}}
\put(8580,-9158){\rotatebox{60.0}{\makebox(0,0)[lb]{\smash{{\SetFigFont{17}{20.4}{\rmdefault}{\mddefault}{\itdefault}{$y+z+e_a+e_b+e_c$}%
}}}}}
\put(8184,-2260){\rotatebox{300.0}{\makebox(0,0)[lb]{\smash{{\SetFigFont{17}{20.4}{\rmdefault}{\mddefault}{\itdefault}{$y+z+o_a+o_b+o_c+a-b$}%
}}}}}
\put(13938,-1834){\makebox(0,0)[lb]{\smash{{\SetFigFont{17}{20.4}{\rmdefault}{\mddefault}{\itdefault}{$x+e_a+e_b+e_c$}%
}}}}
\put(11404,-4938){\rotatebox{60.0}{\makebox(0,0)[lb]{\smash{{\SetFigFont{17}{20.4}{\rmdefault}{\mddefault}{\itdefault}{$y+z+o_a+o_b+o_c$}%
}}}}}
\put(11354,-6029){\rotatebox{300.0}{\makebox(0,0)[lb]{\smash{{\SetFigFont{17}{20.4}{\rmdefault}{\mddefault}{\itdefault}{$y+z+e_a+e_b+e_c+a-b$}%
}}}}}
\put(14015,-9453){\makebox(0,0)[lb]{\smash{{\SetFigFont{17}{20.4}{\rmdefault}{\mddefault}{\itdefault}{$x+o_a+o_b+o_c$}%
}}}}
\put(17580,-9147){\rotatebox{60.0}{\makebox(0,0)[lb]{\smash{{\SetFigFont{17}{20.4}{\rmdefault}{\mddefault}{\itdefault}{$y+z+e_a+e_b+e_c$}%
}}}}}
\put(17184,-2249){\rotatebox{300.0}{\makebox(0,0)[lb]{\smash{{\SetFigFont{17}{20.4}{\rmdefault}{\mddefault}{\itdefault}{$y+z+o_a+o_b+o_c+a-b$}%
}}}}}
\put(5148,-10701){\makebox(0,0)[lb]{\smash{{\SetFigFont{17}{20.4}{\rmdefault}{\mddefault}{\itdefault}{$x+e_a+e_b+e_c$}%
}}}}
\put(2614,-13805){\rotatebox{60.0}{\makebox(0,0)[lb]{\smash{{\SetFigFont{17}{20.4}{\rmdefault}{\mddefault}{\itdefault}{$y+z+o_a+o_b+o_c$}%
}}}}}
\put(2564,-14896){\rotatebox{300.0}{\makebox(0,0)[lb]{\smash{{\SetFigFont{17}{20.4}{\rmdefault}{\mddefault}{\itdefault}{$y+z+e_a+e_b+e_c+a-b$}%
}}}}}
\put(5225,-18320){\makebox(0,0)[lb]{\smash{{\SetFigFont{17}{20.4}{\rmdefault}{\mddefault}{\itdefault}{$x+o_a+o_b+o_c$}%
}}}}
\put(8790,-18014){\rotatebox{60.0}{\makebox(0,0)[lb]{\smash{{\SetFigFont{17}{20.4}{\rmdefault}{\mddefault}{\itdefault}{$y+z+e_a+e_b+e_c$}%
}}}}}
\put(8394,-11116){\rotatebox{300.0}{\makebox(0,0)[lb]{\smash{{\SetFigFont{17}{20.4}{\rmdefault}{\mddefault}{\itdefault}{$y+z+o_a+o_b+o_c+a-b$}%
}}}}}
\put(14148,-10690){\makebox(0,0)[lb]{\smash{{\SetFigFont{17}{20.4}{\rmdefault}{\mddefault}{\itdefault}{$x+e_a+e_b+e_c$}%
}}}}
\put(11614,-13794){\rotatebox{60.0}{\makebox(0,0)[lb]{\smash{{\SetFigFont{17}{20.4}{\rmdefault}{\mddefault}{\itdefault}{$y+z+o_a+o_b+o_c$}%
}}}}}
\put(11564,-14885){\rotatebox{300.0}{\makebox(0,0)[lb]{\smash{{\SetFigFont{17}{20.4}{\rmdefault}{\mddefault}{\itdefault}{$y+z+e_a+e_b+e_c+a-b$}%
}}}}}
\put(14225,-18309){\makebox(0,0)[lb]{\smash{{\SetFigFont{17}{20.4}{\rmdefault}{\mddefault}{\itdefault}{$x+o_a+o_b+o_c$}%
}}}}
\put(17790,-18003){\rotatebox{60.0}{\makebox(0,0)[lb]{\smash{{\SetFigFont{17}{20.4}{\rmdefault}{\mddefault}{\itdefault}{$y+z+e_a+e_b+e_c$}%
}}}}}
\put(17394,-11105){\rotatebox{300.0}{\makebox(0,0)[lb]{\smash{{\SetFigFont{17}{20.4}{\rmdefault}{\mddefault}{\itdefault}{$y+z+o_a+o_b+o_c+a-b$}%
}}}}}
\put(5132,-19455){\makebox(0,0)[lb]{\smash{{\SetFigFont{17}{20.4}{\rmdefault}{\mddefault}{\itdefault}{$x+e_a+e_b+e_c$}%
}}}}
\put(2598,-22559){\rotatebox{60.0}{\makebox(0,0)[lb]{\smash{{\SetFigFont{17}{20.4}{\rmdefault}{\mddefault}{\itdefault}{$y+z+o_a+o_b+o_c$}%
}}}}}
\put(2548,-23650){\rotatebox{300.0}{\makebox(0,0)[lb]{\smash{{\SetFigFont{17}{20.4}{\rmdefault}{\mddefault}{\itdefault}{$y+z+e_a+e_b+e_c+a-b$}%
}}}}}
\put(5209,-27074){\makebox(0,0)[lb]{\smash{{\SetFigFont{17}{20.4}{\rmdefault}{\mddefault}{\itdefault}{$x+o_a+o_b+o_c$}%
}}}}
\put(8774,-26768){\rotatebox{60.0}{\makebox(0,0)[lb]{\smash{{\SetFigFont{17}{20.4}{\rmdefault}{\mddefault}{\itdefault}{$y+z+e_a+e_b+e_c$}%
}}}}}
\put(8378,-19870){\rotatebox{300.0}{\makebox(0,0)[lb]{\smash{{\SetFigFont{17}{20.4}{\rmdefault}{\mddefault}{\itdefault}{$y+z+o_a+o_b+o_c+a-b$}%
}}}}}
\put(14132,-19444){\makebox(0,0)[lb]{\smash{{\SetFigFont{17}{20.4}{\rmdefault}{\mddefault}{\itdefault}{$x+e_a+e_b+e_c$}%
}}}}
\put(11598,-22548){\rotatebox{60.0}{\makebox(0,0)[lb]{\smash{{\SetFigFont{17}{20.4}{\rmdefault}{\mddefault}{\itdefault}{$y+z+o_a+o_b+o_c$}%
}}}}}
\put(11548,-23639){\rotatebox{300.0}{\makebox(0,0)[lb]{\smash{{\SetFigFont{17}{20.4}{\rmdefault}{\mddefault}{\itdefault}{$y+z+e_a+e_b+e_c+a-b$}%
}}}}}
\put(14209,-27063){\makebox(0,0)[lb]{\smash{{\SetFigFont{17}{20.4}{\rmdefault}{\mddefault}{\itdefault}{$x+o_a+o_b+o_c$}%
}}}}
\put(17774,-26757){\rotatebox{60.0}{\makebox(0,0)[lb]{\smash{{\SetFigFont{17}{20.4}{\rmdefault}{\mddefault}{\itdefault}{$y+z+e_a+e_b+e_c$}%
}}}}}
\put(17378,-19859){\rotatebox{300.0}{\makebox(0,0)[lb]{\smash{{\SetFigFont{17}{20.4}{\rmdefault}{\mddefault}{\itdefault}{$y+z+o_a+o_b+o_c+a-b$}%
}}}}}
\end{picture}%
}
\caption{Obtaining the recurrence for the region $Q^{\leftarrow}_{x,y,z}(\textbf{a};\textbf{c};\textbf{b})$, when $a\geq b$.
Kuo condensation is applied to the region $Q^{\leftarrow}_{3,2,2}(2,2 ;\ 2,1 ;\ 1,2)$ (picture (a)) as shown on the picture (b).}\label{fig:kuocenter10}
\end{figure}

We now apply Kuo condensation  (Theorem \ref{kuothm1}) to the dual graph $G$ of the region $Q^{\leftarrow}_{x,y,z}(\textbf{a};\textbf{c};\textbf{b})$ with the choice of the four vertices $u,v,w,s$ similar to that in the case of $R^{\leftarrow}$-type regions (illustrated in Figure \ref{fig:kuocenter10}(b)). Again, we do not show these vertices directly, and we show here the unit triangles corresponding to them instead. The removal of the $u$-, $v$-, $w$-, $s$-triangles yields several forced lozenges along the boundary of the region and at the end of the left fern (see Figure \ref{fig:kuocenter10} for the case $a>b$;  the case $a\leq b$ can be treated in the same manner). In all cases, after removing the forced lozenges, we recover a new region of the $Q^{\odot}$-, $Q^{\leftarrow}$, or $Q^{\nwarrow}$-type, except for the case of $G-\{w,s\}$. After removing the forced lozenges from the region corresponding to $G-\{w,s\}$, we need to reflect the resulting region along a vertical line to get the region $Q^{\nearrow}_{x,y-1,z}(\textbf{b};\ \textbf{c}^{\leftrightarrow}; \textbf{a})$. In particular, we obtain the following recurrences:

\begin{align}\label{centerrecur6a}
\M(Q^{\leftarrow}_{x,y,z}(\textbf{a};\ \textbf{c};\ \textbf{b})) \M(Q^{\odot}_{x,y-1,z-1}(\textbf{a};\ \textbf{c}; \textbf{b}^{+1}))&=\M(Q^{\nwarrow}_{x,y-1,z-1}(\textbf{a};\ \textbf{c}; \textbf{b}^{+1}))\M(Q^{\nearrow}_{x,y-1,z}(\textbf{b};\ \textbf{c}^{\leftrightarrow}; \textbf{a}))\notag\\
&+
\M(Q^{\leftarrow}_{x+1,y,z-1}(\textbf{a};\ \textbf{c}; \textbf{b})) \M(Q^{\odot}_{x-1,y-1,z}(\textbf{a};\ \textbf{c};\ \textbf{b}^{+1})),
\end{align}
for the case $a\leq b$, and


\begin{align}\label{centerrecur6b}
\M(Q^{\leftarrow}_{x,y,z}(\textbf{a};\ \textbf{c};\ \textbf{b})) \M(Q^{\odot}_{x,y,z-1}(\textbf{a};\ \textbf{c}; \textbf{b}^{+1}))&=\M(Q^{\nwarrow}_{x,y,z-1}(\textbf{a};\ \textbf{c}; \textbf{b}^{+1}))\M(Q^{\nearrow}_{x,y-1,z}(\textbf{b};\ \textbf{c}^{\leftrightarrow}; \textbf{a}))\notag\\
&+
\M(Q^{\leftarrow}_{x+1,y,z-1}(\textbf{a};\ \textbf{c}; \textbf{b})) \M(Q^{\odot}_{x-1,y,z}(\textbf{a};\ \textbf{c};\ \textbf{b}^{+1})),
\end{align}
for the case $a> b$.

\subsection{Recurrences for $Q^{\nwarrow}$-type regions}\label{subsec:recurQ3}

\begin{figure}\centering
\setlength{\unitlength}{3947sp}%
\begingroup\makeatletter\ifx\SetFigFont\undefined%
\gdef\SetFigFont#1#2#3#4#5{%
  \reset@font\fontsize{#1}{#2pt}%
  \fontfamily{#3}\fontseries{#4}\fontshape{#5}%
  \selectfont}%
\fi\endgroup%
\resizebox{15cm}{!}{
\begin{picture}(0,0)%
\includegraphics{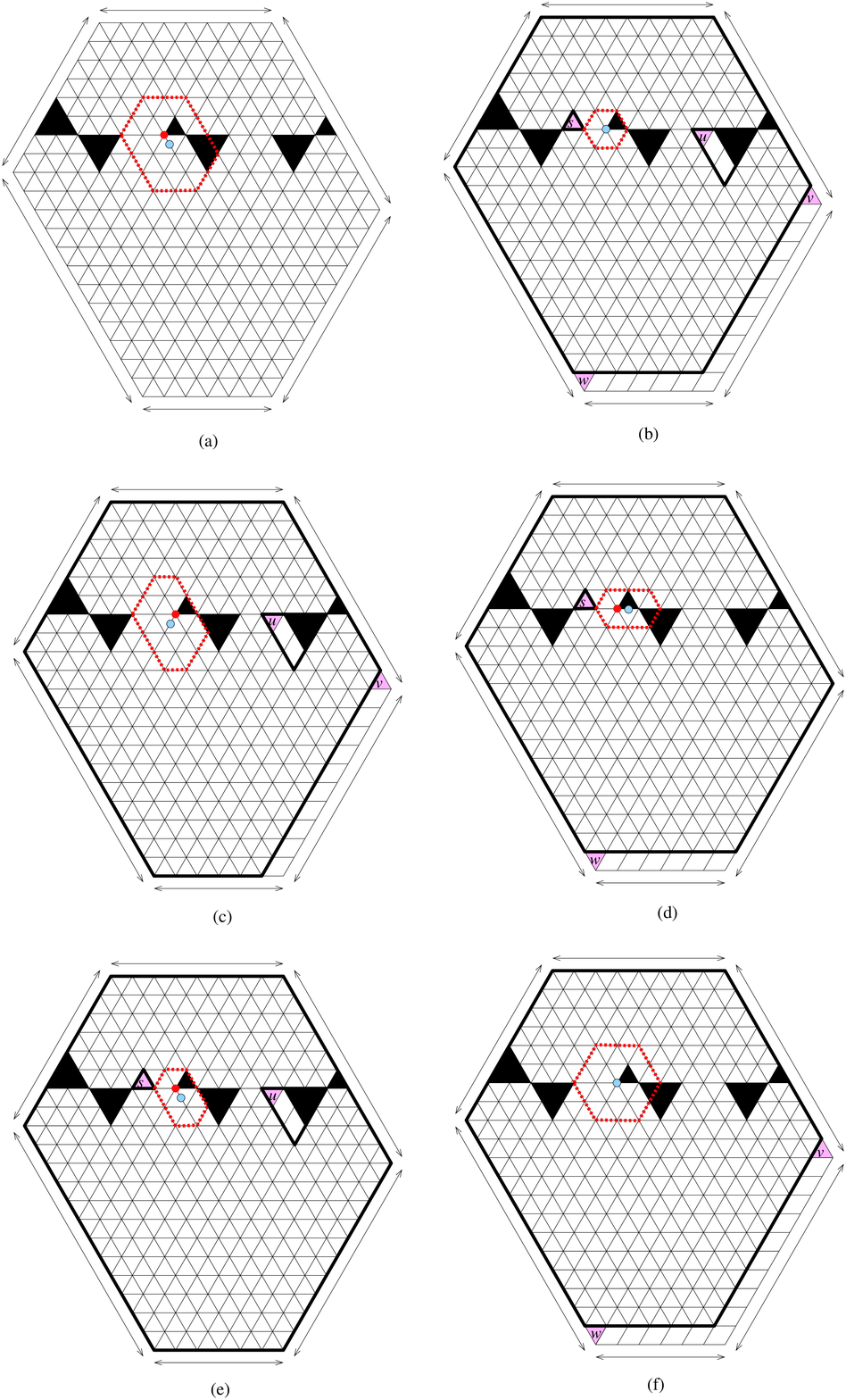}%
\end{picture}%
%
%

\begin{picture}(16325,26809)(1830,-28106)
\put(7780,-20359){\rotatebox{300.0}{\makebox(0,0)[lb]{\smash{{\SetFigFont{14}{16.8}{\rmdefault}{\mddefault}{\updefault}{$y+z+o_a+o_b+o_c+(a-b)+1$}%
}}}}}
\put(4838,-19759){\makebox(0,0)[lb]{\smash{{\SetFigFont{14}{16.8}{\rmdefault}{\mddefault}{\updefault}{$x+e_a+e_b+e_c$}%
}}}}
\put(8094,-26619){\rotatebox{60.0}{\makebox(0,0)[lb]{\smash{{\SetFigFont{14}{16.8}{\rmdefault}{\mddefault}{\updefault}{$y+z+e_a+e_b+e_c$}%
}}}}}
\put(5288,-27719){\makebox(0,0)[lb]{\smash{{\SetFigFont{14}{16.8}{\rmdefault}{\mddefault}{\updefault}{$x+o_a+o_b+o_c$}%
}}}}
\put(2248,-23939){\rotatebox{300.0}{\makebox(0,0)[lb]{\smash{{\SetFigFont{14}{16.8}{\rmdefault}{\mddefault}{\updefault}{$y+z+e_a+e_b+e_c+(a-b)+1$}%
}}}}}
\put(2418,-21899){\rotatebox{60.0}{\makebox(0,0)[lb]{\smash{{\SetFigFont{14}{16.8}{\rmdefault}{\mddefault}{\updefault}{$y+z+o_a+o_b+o_c$}%
}}}}}
\put(13216,-19654){\makebox(0,0)[lb]{\smash{{\SetFigFont{14}{16.8}{\rmdefault}{\mddefault}{\updefault}{$x+e_a+e_b+e_c$}%
}}}}
\put(16158,-20254){\rotatebox{300.0}{\makebox(0,0)[lb]{\smash{{\SetFigFont{14}{16.8}{\rmdefault}{\mddefault}{\updefault}{$y+z+o_a+o_b+o_c+(a-b)+1$}%
}}}}}
\put(16472,-26514){\rotatebox{60.0}{\makebox(0,0)[lb]{\smash{{\SetFigFont{14}{16.8}{\rmdefault}{\mddefault}{\updefault}{$y+z+e_a+e_b+e_c$}%
}}}}}
\put(13666,-27614){\makebox(0,0)[lb]{\smash{{\SetFigFont{14}{16.8}{\rmdefault}{\mddefault}{\updefault}{$x+o_a+o_b+o_c$}%
}}}}
\put(10626,-23834){\rotatebox{300.0}{\makebox(0,0)[lb]{\smash{{\SetFigFont{14}{16.8}{\rmdefault}{\mddefault}{\updefault}{$y+z+e_a+e_b+e_c+(a-b)+1$}%
}}}}}
\put(10796,-12809){\rotatebox{60.0}{\makebox(0,0)[lb]{\smash{{\SetFigFont{14}{16.8}{\rmdefault}{\mddefault}{\updefault}{$y+z+o_a+o_b+o_c$}%
}}}}}
\put(10626,-14849){\rotatebox{300.0}{\makebox(0,0)[lb]{\smash{{\SetFigFont{14}{16.8}{\rmdefault}{\mddefault}{\updefault}{$y+z+e_a+e_b+e_c+(a-b)+1$}%
}}}}}
\put(13666,-18629){\makebox(0,0)[lb]{\smash{{\SetFigFont{14}{16.8}{\rmdefault}{\mddefault}{\updefault}{$x+o_a+o_b+o_c$}%
}}}}
\put(16472,-17529){\rotatebox{60.0}{\makebox(0,0)[lb]{\smash{{\SetFigFont{14}{16.8}{\rmdefault}{\mddefault}{\updefault}{$y+z+e_a+e_b+e_c$}%
}}}}}
\put(16158,-11269){\rotatebox{300.0}{\makebox(0,0)[lb]{\smash{{\SetFigFont{14}{16.8}{\rmdefault}{\mddefault}{\updefault}{$y+z+o_a+o_b+o_c+(a-b)+1$}%
}}}}}
\put(4621,-1691){\makebox(0,0)[lb]{\smash{{\SetFigFont{14}{16.8}{\rmdefault}{\mddefault}{\updefault}{$x+e_a+e_b+e_c$}%
}}}}
\put(7563,-2291){\rotatebox{300.0}{\makebox(0,0)[lb]{\smash{{\SetFigFont{14}{16.8}{\rmdefault}{\mddefault}{\updefault}{$y+z+o_a+o_b+o_c+(a-b)+1$}%
}}}}}
\put(7877,-8551){\rotatebox{60.0}{\makebox(0,0)[lb]{\smash{{\SetFigFont{14}{16.8}{\rmdefault}{\mddefault}{\updefault}{$y+z+e_a+e_b+e_c$}%
}}}}}
\put(5071,-9651){\makebox(0,0)[lb]{\smash{{\SetFigFont{14}{16.8}{\rmdefault}{\mddefault}{\updefault}{$x+o_a+o_b+o_c$}%
}}}}
\put(2031,-5871){\rotatebox{300.0}{\makebox(0,0)[lb]{\smash{{\SetFigFont{14}{16.8}{\rmdefault}{\mddefault}{\updefault}{$y+z+e_a+e_b+e_c+(a-b)+1$}%
}}}}}
\put(2201,-3831){\rotatebox{60.0}{\makebox(0,0)[lb]{\smash{{\SetFigFont{14}{16.8}{\rmdefault}{\mddefault}{\updefault}{$y+z+o_a+o_b+o_c$}%
}}}}}
\put(12999,-1586){\makebox(0,0)[lb]{\smash{{\SetFigFont{14}{16.8}{\rmdefault}{\mddefault}{\updefault}{$x+e_a+e_b+e_c$}%
}}}}
\put(15941,-2186){\rotatebox{300.0}{\makebox(0,0)[lb]{\smash{{\SetFigFont{14}{16.8}{\rmdefault}{\mddefault}{\updefault}{$y+z+o_a+o_b+o_c+(a-b)+1$}%
}}}}}
\put(16255,-8446){\rotatebox{60.0}{\makebox(0,0)[lb]{\smash{{\SetFigFont{14}{16.8}{\rmdefault}{\mddefault}{\updefault}{$y+z+e_a+e_b+e_c$}%
}}}}}
\put(13449,-9546){\makebox(0,0)[lb]{\smash{{\SetFigFont{14}{16.8}{\rmdefault}{\mddefault}{\updefault}{$x+o_a+o_b+o_c$}%
}}}}
\put(10409,-5766){\rotatebox{300.0}{\makebox(0,0)[lb]{\smash{{\SetFigFont{14}{16.8}{\rmdefault}{\mddefault}{\updefault}{$y+z+e_a+e_b+e_c+(a-b)+1$}%
}}}}}
\put(10579,-3726){\rotatebox{60.0}{\makebox(0,0)[lb]{\smash{{\SetFigFont{14}{16.8}{\rmdefault}{\mddefault}{\updefault}{$y+z+o_a+o_b+o_c$}%
}}}}}
\put(4838,-10774){\makebox(0,0)[lb]{\smash{{\SetFigFont{14}{16.8}{\rmdefault}{\mddefault}{\updefault}{$x+e_a+e_b+e_c$}%
}}}}
\put(7780,-11374){\rotatebox{300.0}{\makebox(0,0)[lb]{\smash{{\SetFigFont{14}{16.8}{\rmdefault}{\mddefault}{\updefault}{$y+z+o_a+o_b+o_c+(a-b)+1$}%
}}}}}
\put(8094,-17634){\rotatebox{60.0}{\makebox(0,0)[lb]{\smash{{\SetFigFont{14}{16.8}{\rmdefault}{\mddefault}{\updefault}{$y+z+e_a+e_b+e_c$}%
}}}}}
\put(5288,-18734){\makebox(0,0)[lb]{\smash{{\SetFigFont{14}{16.8}{\rmdefault}{\mddefault}{\updefault}{$x+o_a+o_b+o_c$}%
}}}}
\put(2248,-14954){\rotatebox{300.0}{\makebox(0,0)[lb]{\smash{{\SetFigFont{14}{16.8}{\rmdefault}{\mddefault}{\updefault}{$y+z+e_a+e_b+e_c+(a-b)+1$}%
}}}}}
\put(2418,-12914){\rotatebox{60.0}{\makebox(0,0)[lb]{\smash{{\SetFigFont{14}{16.8}{\rmdefault}{\mddefault}{\updefault}{$y+z+o_a+o_b+o_c$}%
}}}}}
\put(13216,-10669){\makebox(0,0)[lb]{\smash{{\SetFigFont{14}{16.8}{\rmdefault}{\mddefault}{\updefault}{$x+e_a+e_b+e_c$}%
}}}}
\put(10796,-21794){\rotatebox{60.0}{\makebox(0,0)[lb]{\smash{{\SetFigFont{14}{16.8}{\rmdefault}{\mddefault}{\updefault}{$y+z+o_a+o_b+o_c$}%
}}}}}
\end{picture}%
}
\caption{Obtaining the recurrence for the region $Q^{\nwarrow}_{x,y,z}(\textbf{a};\textbf{c};\textbf{b})$, when $a> b$. Kuo condensation is applied to the region $Q^{\nwarrow}_{2,2,2}(2,2 ;\ 1,2 ;\ 1,2)$ (picture (a)) as shown on the picture (b).}\label{fig:kuocenter11}
\end{figure}

Like the cases of the $Q^{\odot}$- and $Q^{\leftarrow}$-type regions treated above, the application of Kuo condensation to the $Q^{\nwarrow}$-type regions is similar  to its $R$-counterpart, the $Q^{\nwarrow}$-type regions. In particular, we apply Kuo's Theorem \ref{kuothm1} to the dual graph $G$ of the region $Q^{\nwarrow}_{x,y,z}(\textbf{a};\ \textbf{c};\ \textbf{b})$ as shown in Figure \ref{fig:kuocenter11} for $a>b$ (the cases $a> b$ and $a=b$ are similar). By considering forced lozenges, we get
\begin{align}\label{centerrecur8a}
\M(Q^{\nwarrow}_{x,y,z}(\textbf{a};\ \textbf{c};\ \textbf{b})) \M(Q^{\odot}_{x-1,y,z-1}(\textbf{a}^{+1};\ \textbf{c}; \ \textbf{b}^{+1}))&=\M(Q^{\nearrow}_{x-1,y-1,z}(\textbf{a};\ \textbf{c}; \ \textbf{b}^{+1}))\M(Q^{\leftarrow}_{x,y+1,z-1}(\textbf{a}^{+1};\ \textbf{c};\ \textbf{b}))\notag\\
&+
\M(Q^{\nwarrow}_{x-1,y,z-1}(\textbf{a}^{+1};\ \textbf{c}; \  \textbf{b}^{+1})) \M(Q^{\odot}_{x,y,z}(\textbf{a};\ \textbf{c};\ \textbf{b})),
\end{align}
for the case $a< b$,


\begin{align}\label{centerrecur8b}
\M(Q^{\nwarrow}_{x,y,z}(\textbf{a};\ \textbf{c};\ \textbf{b})) \M(Q^{\odot}_{x-1,y,z-1}(\textbf{a}^{+1};\ \textbf{c}; \ \textbf{b}^{+1}))&=\M(Q^{\nearrow}_{x-1,y,z}(\textbf{a};\ \textbf{c};\  \textbf{b}^{+1}))\M(Q^{\leftarrow}_{x,y,z-1}(\textbf{a}^{+1};\ \textbf{c};\ \textbf{b}))\notag\\
&+
\M(Q^{\nwarrow}_{x-1,y,z-1}(\textbf{a}^{+1};\ \textbf{c};\ \textbf{b}^{+1})) \M(Q^{\odot}_{x,y,z}(\textbf{a};\ \textbf{c};\ \textbf{b})),
\end{align}
for the case $a> b$, and
\begin{align}\label{centerrecur8c}
\M(Q^{\nwarrow}_{x,y,z}(\textbf{a};\ \textbf{c};\ \textbf{b})) \M(Q^{\odot}_{x-1,y,z-1}(\textbf{a}^{+1};\ \textbf{c}; \ \textbf{b}^{+1}))&=\M(Q^{\nearrow}_{x-1,y-1,z}(\textbf{a};\ \textbf{c};\  \textbf{b}^{+1}))\M(Q^{\leftarrow}_{x,y,z-1}(\textbf{a}^{+1};\ \textbf{c};\ \textbf{b}))\notag\\
&+
\M(Q^{\nwarrow}_{x-1,y,z-1}(\textbf{a}^{+1};\ \textbf{c};\ \textbf{b}^{+1})) \M(Q^{\odot}_{x,y,z}(\textbf{a};\ \textbf{c};\ \textbf{b})),
\end{align}
when $a=b$.

\subsection{Recurrences for $Q^{\nearrow}$-type regions}\label{subsec:recurQ4}

\begin{figure}\centering
\setlength{\unitlength}{3947sp}%
\begingroup\makeatletter\ifx\SetFigFont\undefined%
\gdef\SetFigFont#1#2#3#4#5{%
  \reset@font\fontsize{#1}{#2pt}%
  \fontfamily{#3}\fontseries{#4}\fontshape{#5}%
  \selectfont}%
\fi\endgroup%
\resizebox{15cm}{!}{
\begin{picture}(0,0)%
\includegraphics{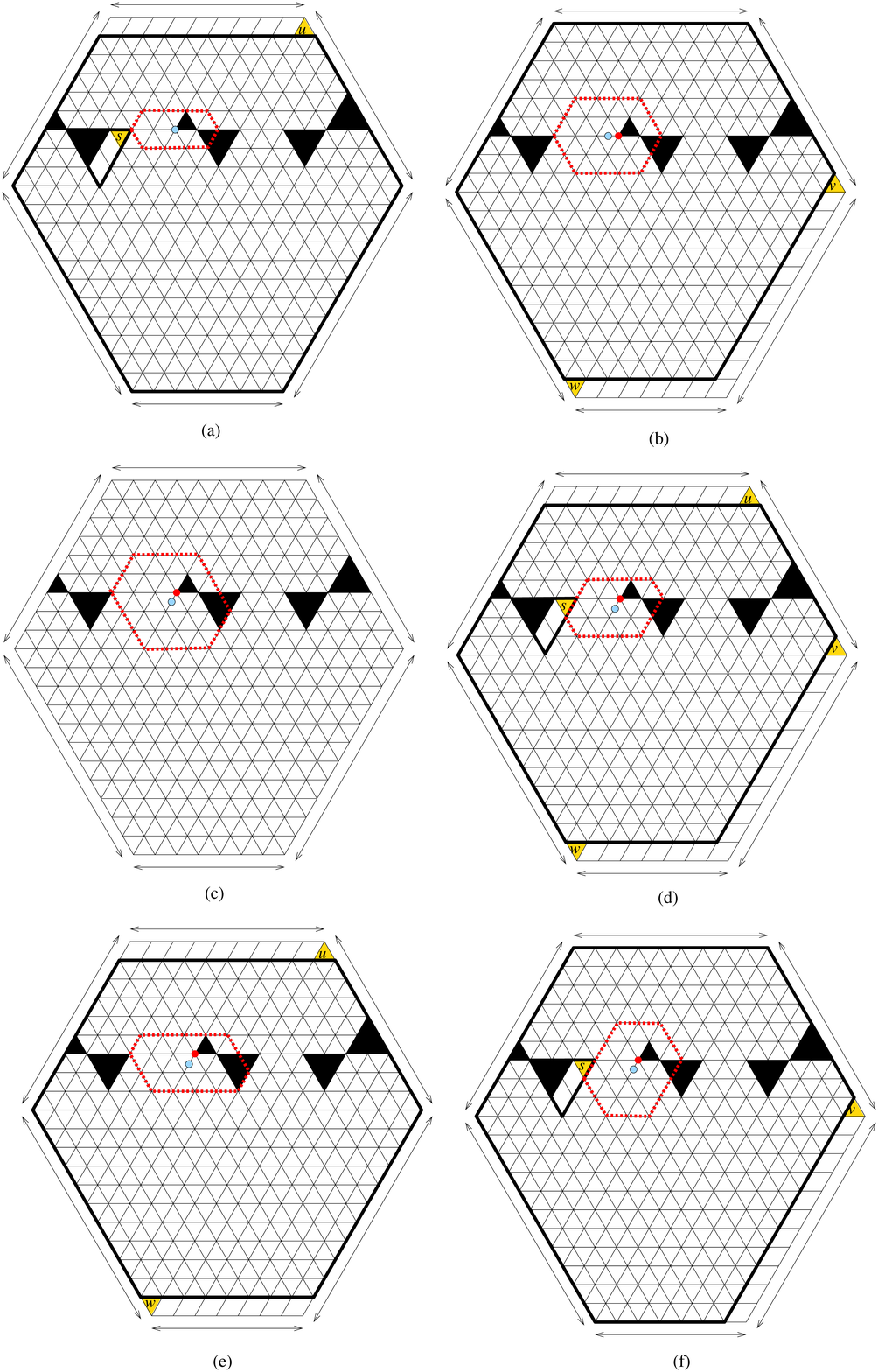}%
\end{picture}%
%
%

\begin{picture}(16901,26307)(554,-28957)
\put(12141,-3076){\makebox(0,0)[lb]{\smash{{\SetFigFont{14}{16.8}{\familydefault}{\mddefault}{\updefault}{$x+e_a+e_b+e_c$}%
}}}}
\put(4075,-28407){\makebox(0,0)[lb]{\smash{{\SetFigFont{14}{16.8}{\familydefault}{\mddefault}{\updefault}{$x+o_a+o_b+o_c$}%
}}}}
\put(1185,-25037){\rotatebox{300.0}{\makebox(0,0)[lb]{\smash{{\SetFigFont{14}{16.8}{\familydefault}{\mddefault}{\updefault}{$y+z+e_a+e_b+e_c+1$}%
}}}}}
\put(1135,-23207){\rotatebox{60.0}{\makebox(0,0)[lb]{\smash{{\SetFigFont{14}{16.8}{\familydefault}{\mddefault}{\updefault}{$y+z+o_a+o_b+o_c+(b-a)$}%
}}}}}
\put(15854,-21295){\rotatebox{300.0}{\makebox(0,0)[lb]{\smash{{\SetFigFont{14}{16.8}{\familydefault}{\mddefault}{\updefault}{$y+z+o_a+o_b+o_c+1$}%
}}}}}
\put(15554,-27655){\rotatebox{60.0}{\makebox(0,0)[lb]{\smash{{\SetFigFont{14}{16.8}{\familydefault}{\mddefault}{\updefault}{$y+z+e_a+e_b+e_c+(b-a)$}%
}}}}}
\put(12474,-28525){\makebox(0,0)[lb]{\smash{{\SetFigFont{14}{16.8}{\familydefault}{\mddefault}{\updefault}{$x+o_a+o_b+o_c$}%
}}}}
\put(9584,-25155){\rotatebox{300.0}{\makebox(0,0)[lb]{\smash{{\SetFigFont{14}{16.8}{\familydefault}{\mddefault}{\updefault}{$y+z+e_a+e_b+e_c+1$}%
}}}}}
\put(9534,-23325){\rotatebox{60.0}{\makebox(0,0)[lb]{\smash{{\SetFigFont{14}{16.8}{\familydefault}{\mddefault}{\updefault}{$y+z+o_a+o_b+o_c+(b-a)$}%
}}}}}
\put(12611,-20596){\makebox(0,0)[lb]{\smash{{\SetFigFont{14}{16.8}{\familydefault}{\mddefault}{\updefault}{$x+e_a+e_b+e_c$}%
}}}}
\put(12321,-11846){\makebox(0,0)[lb]{\smash{{\SetFigFont{14}{16.8}{\familydefault}{\mddefault}{\updefault}{$x+e_a+e_b+e_c$}%
}}}}
\put(7155,-27537){\rotatebox{60.0}{\makebox(0,0)[lb]{\smash{{\SetFigFont{14}{16.8}{\familydefault}{\mddefault}{\updefault}{$y+z+e_a+e_b+e_c+(b-a)$}%
}}}}}
\put(7455,-21177){\rotatebox{300.0}{\makebox(0,0)[lb]{\smash{{\SetFigFont{14}{16.8}{\familydefault}{\mddefault}{\updefault}{$y+z+o_a+o_b+o_c+1$}%
}}}}}
\put(4055,-20437){\makebox(0,0)[lb]{\smash{{\SetFigFont{14}{16.8}{\familydefault}{\mddefault}{\updefault}{$x+e_a+e_b+e_c$}%
}}}}
\put(9189,-14595){\rotatebox{60.0}{\makebox(0,0)[lb]{\smash{{\SetFigFont{14}{16.8}{\familydefault}{\mddefault}{\updefault}{$y+z+o_a+o_b+o_c+(b-a)$}%
}}}}}
\put(9239,-16425){\rotatebox{300.0}{\makebox(0,0)[lb]{\smash{{\SetFigFont{14}{16.8}{\familydefault}{\mddefault}{\updefault}{$y+z+e_a+e_b+e_c+1$}%
}}}}}
\put(12129,-19795){\makebox(0,0)[lb]{\smash{{\SetFigFont{14}{16.8}{\familydefault}{\mddefault}{\updefault}{$x+o_a+o_b+o_c$}%
}}}}
\put(3681,-2941){\makebox(0,0)[lb]{\smash{{\SetFigFont{14}{16.8}{\familydefault}{\mddefault}{\updefault}{$x+e_a+e_b+e_c$}%
}}}}
\put(7081,-3681){\rotatebox{300.0}{\makebox(0,0)[lb]{\smash{{\SetFigFont{14}{16.8}{\familydefault}{\mddefault}{\updefault}{$y+z+o_a+o_b+o_c+1$}%
}}}}}
\put(6781,-10041){\rotatebox{60.0}{\makebox(0,0)[lb]{\smash{{\SetFigFont{14}{16.8}{\familydefault}{\mddefault}{\updefault}{$y+z+e_a+e_b+e_c+(b-a)$}%
}}}}}
\put(3701,-10911){\makebox(0,0)[lb]{\smash{{\SetFigFont{14}{16.8}{\familydefault}{\mddefault}{\updefault}{$x+o_a+o_b+o_c$}%
}}}}
\put(811,-7541){\rotatebox{300.0}{\makebox(0,0)[lb]{\smash{{\SetFigFont{14}{16.8}{\familydefault}{\mddefault}{\updefault}{$y+z+e_a+e_b+e_c+1$}%
}}}}}
\put(761,-5711){\rotatebox{60.0}{\makebox(0,0)[lb]{\smash{{\SetFigFont{14}{16.8}{\familydefault}{\mddefault}{\updefault}{$y+z+o_a+o_b+o_c+(b-a)$}%
}}}}}
\put(15480,-3799){\rotatebox{300.0}{\makebox(0,0)[lb]{\smash{{\SetFigFont{14}{16.8}{\familydefault}{\mddefault}{\updefault}{$y+z+o_a+o_b+o_c+1$}%
}}}}}
\put(15180,-10159){\rotatebox{60.0}{\makebox(0,0)[lb]{\smash{{\SetFigFont{14}{16.8}{\familydefault}{\mddefault}{\updefault}{$y+z+e_a+e_b+e_c+(b-a)$}%
}}}}}
\put(12100,-11029){\makebox(0,0)[lb]{\smash{{\SetFigFont{14}{16.8}{\familydefault}{\mddefault}{\updefault}{$x+o_a+o_b+o_c$}%
}}}}
\put(9210,-7659){\rotatebox{300.0}{\makebox(0,0)[lb]{\smash{{\SetFigFont{14}{16.8}{\familydefault}{\mddefault}{\updefault}{$y+z+e_a+e_b+e_c+1$}%
}}}}}
\put(9160,-5829){\rotatebox{60.0}{\makebox(0,0)[lb]{\smash{{\SetFigFont{14}{16.8}{\familydefault}{\mddefault}{\updefault}{$y+z+o_a+o_b+o_c+(b-a)$}%
}}}}}
\put(3710,-11707){\makebox(0,0)[lb]{\smash{{\SetFigFont{14}{16.8}{\familydefault}{\mddefault}{\updefault}{$x+e_a+e_b+e_c$}%
}}}}
\put(7110,-12447){\rotatebox{300.0}{\makebox(0,0)[lb]{\smash{{\SetFigFont{14}{16.8}{\familydefault}{\mddefault}{\updefault}{$y+z+o_a+o_b+o_c+1$}%
}}}}}
\put(6810,-18807){\rotatebox{60.0}{\makebox(0,0)[lb]{\smash{{\SetFigFont{14}{16.8}{\familydefault}{\mddefault}{\updefault}{$y+z+e_a+e_b+e_c+(b-a)$}%
}}}}}
\put(3730,-19677){\makebox(0,0)[lb]{\smash{{\SetFigFont{14}{16.8}{\familydefault}{\mddefault}{\updefault}{$x+o_a+o_b+o_c$}%
}}}}
\put(840,-16307){\rotatebox{300.0}{\makebox(0,0)[lb]{\smash{{\SetFigFont{14}{16.8}{\familydefault}{\mddefault}{\updefault}{$y+z+e_a+e_b+e_c+1$}%
}}}}}
\put(790,-14477){\rotatebox{60.0}{\makebox(0,0)[lb]{\smash{{\SetFigFont{14}{16.8}{\familydefault}{\mddefault}{\updefault}{$y+z+o_a+o_b+o_c+(b-a)$}%
}}}}}
\put(15509,-12565){\rotatebox{300.0}{\makebox(0,0)[lb]{\smash{{\SetFigFont{14}{16.8}{\familydefault}{\mddefault}{\updefault}{$y+z+o_a+o_b+o_c+1$}%
}}}}}
\put(15209,-18925){\rotatebox{60.0}{\makebox(0,0)[lb]{\smash{{\SetFigFont{14}{16.8}{\familydefault}{\mddefault}{\updefault}{$y+z+e_a+e_b+e_c+(b-a)$}%
}}}}}
\end{picture}%
}
\caption{Obtaining the recurrence for the region $Q^{\nearrow}_{x,y,z}(\textbf{a};\textbf{c};\textbf{b})$, when $a< b$. Kuo condensation is applied to the region $Q^{\nearrow}_{3,2,2}(1,2 ;\ 1,2;\ 2,2)$ (picture (c)) as shown on the picture (d).}\label{fig:kuocenter12}
\end{figure}

We now need to use a different Kuo condensation from that in the previous cases. In particular,  we apply here Theorem \ref{kuothm2} (as opposed to Theorem \ref{kuothm1} as in the previous cases) with the four vertices selected as in Figure \ref{fig:kuocenter12}(d). The regions in Figures \ref{fig:kuocenter12}(a)--(f) correspond to the terms in the equation of Theorem \ref{kuothm2}.

We first consider the case $a< b$. The removals of forced lozenges in the regions corresponding to $G-\{u,s\}$ and $G-\{u,w\}$ give us respectively the regions $Q^{\odot}_{x,y+1,z-1}(\textbf{a}^{+1};\ \textbf{c};\ \textbf{b})$ and $Q^{\nearrow}_{x+1,y,z-1}(\textbf{a};\ \textbf{c}\; \textbf{b})$. However, for the regions corresponding to $G-\{v,w\}$, $G-\{u,v,w,s\}$, $G-\{v,s\}$, we do not end up with a $Q$-region after removing forced lozenges. We need to use the symmetry of $Q$-regions by reflecting the leftover regions over a vertical line to get the regions $Q^{\leftarrow}_{x,y,z}(\textbf{b};\ \textbf{c}^{\leftrightarrow};\  \textbf{a})$, $Q^{\nwarrow}_{x,y,z-1}(\textbf{b};\ \textbf{c}^{\leftrightarrow};\ \textbf{a}^{+1})$, and  $Q^{\nwarrow}_{x-1,y,z}(\textbf{b};\ \textbf{c}^{\leftrightarrow};\ \textbf{a}^{+1})$, respectively (see Figure \ref{fig:kuocenter12}). This way, we get the recurrence
\begin{align}\label{centerrecur7a}
\M(Q^{\odot}_{x,y+1,z-1}(\textbf{a}^{+1};\ \textbf{c};\ \textbf{b}))\M(Q^{\leftarrow}_{x,y,z}(\textbf{b};\ \textbf{c}^{\leftrightarrow};\  \textbf{a}))&=\M(Q^{\nearrow}_{x,y,z}(\textbf{a};\ \textbf{c};\ \textbf{b})) \M(Q^{\nwarrow}_{x,y,z-1}(\textbf{b};\ \textbf{c}^{\leftrightarrow};\ \textbf{a}^{+1}))\notag\\
&+
\M(Q^{\nearrow}_{x+1,y,z-1}(\textbf{a};\ \textbf{c};\  \textbf{b})) \M(Q^{\nwarrow}_{x-1,y,z}(\textbf{b};\ \textbf{c}^{\leftrightarrow};\ \textbf{a}^{+1})),
\end{align}
for the case $a< b$.


Working similarly for the case $a\geq b$, we have
\begin{align}\label{centerrecur7b}
\M(Q^{\odot}_{x,y,z-1}(\textbf{a}^{+1};\ \textbf{c};\ \textbf{b}))\M(Q^{\leftarrow}_{x,y,z}(\textbf{b};\ \textbf{c}^{\leftrightarrow}; \ \textbf{a}))&=\M(Q^{\nearrow}_{x,y,z}(\textbf{a};\ \textbf{c};\ \textbf{b})) \M(Q^{\nwarrow}_{x,y-1,z-1}(\textbf{b};\ \textbf{c}^{\leftrightarrow};\ \textbf{a}^{+1}))\notag\\
&+
\M(Q^{\nearrow}_{x+1,y,z-1}(\textbf{a};\ \textbf{c};\ \textbf{b})) \M(Q^{\nwarrow}_{x-1,y-1,z}(\textbf{b};\ \textbf{c}^{\leftrightarrow};\ \textbf{a}^{+1})).
\end{align}

\subsection{Two extremal cases}

In this subsection, we deal with two extremal cases when certain parameters of our 8 families of regions achieve their minimal values.

\begin{figure}\centering
\setlength{\unitlength}{3947sp}%
\begingroup\makeatletter\ifx\SetFigFont\undefined%
\gdef\SetFigFont#1#2#3#4#5{%
  \reset@font\fontsize{#1}{#2pt}%
  \fontfamily{#3}\fontseries{#4}\fontshape{#5}%
  \selectfont}%
\fi\endgroup%
\resizebox{15cm}{!}{
\begin{picture}(0,0)%
\includegraphics{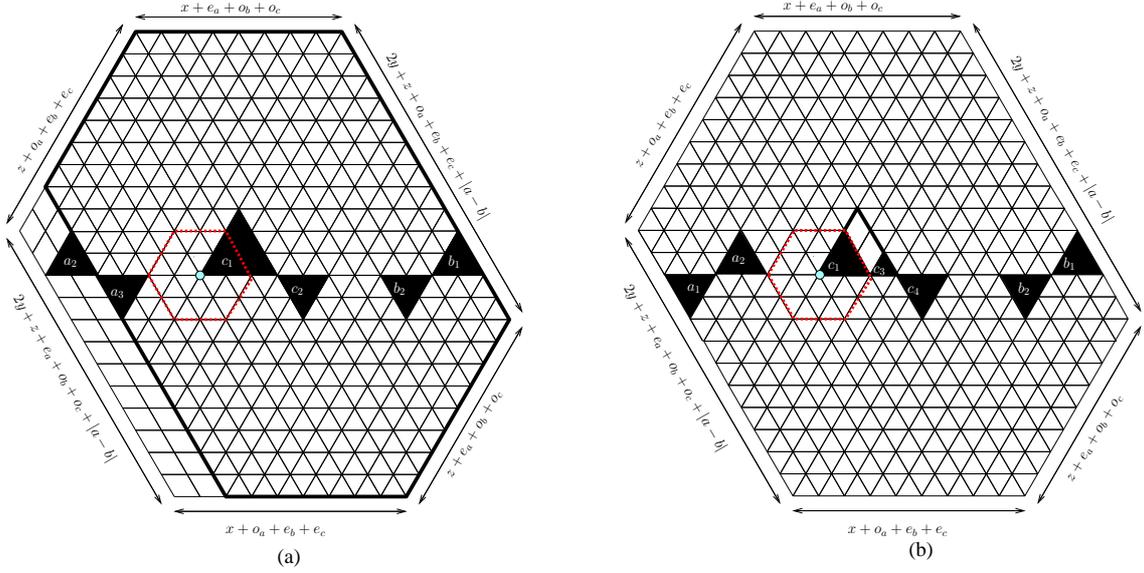}%
\end{picture}%
%
%

\begin{picture}(18044,9252)(1830,-10258)
\put(16141,-5821){\makebox(0,0)[lb]{\smash{{\SetFigFont{14}{16.8}{\rmdefault}{\mddefault}{\updefault}{\color[rgb]{1,1,1}$c_4$}%
}}}}
\put(17896,-5821){\makebox(0,0)[lb]{\smash{{\SetFigFont{14}{16.8}{\rmdefault}{\mddefault}{\updefault}{\color[rgb]{1,1,1}$b_2$}%
}}}}
\put(18618,-5414){\makebox(0,0)[lb]{\smash{{\SetFigFont{14}{16.8}{\rmdefault}{\mddefault}{\updefault}{\color[rgb]{1,1,1}$b_1$}%
}}}}
\put(4604,-1311){\makebox(0,0)[lb]{\smash{{\SetFigFont{14}{16.8}{\familydefault}{\mddefault}{\updefault}{\color[rgb]{0,0,0}$x+e_a+o_b+o_c$}%
}}}}
\put(7861,-2031){\rotatebox{300.0}{\makebox(0,0)[lb]{\smash{{\SetFigFont{14}{16.8}{\familydefault}{\mddefault}{\updefault}{\color[rgb]{0,0,0}$2y+z+o_a+e_b+e_c+|a-b|$}%
}}}}}
\put(8911,-8781){\rotatebox{60.0}{\makebox(0,0)[lb]{\smash{{\SetFigFont{14}{16.8}{\familydefault}{\mddefault}{\updefault}{\color[rgb]{0,0,0}$z+e_a+o_b+o_c$}%
}}}}}
\put(5320,-9673){\makebox(0,0)[lb]{\smash{{\SetFigFont{14}{16.8}{\familydefault}{\mddefault}{\updefault}{\color[rgb]{0,0,0}$x+o_a+e_b+e_c$}%
}}}}
\put(1944,-5905){\rotatebox{300.0}{\makebox(0,0)[lb]{\smash{{\SetFigFont{14}{16.8}{\familydefault}{\mddefault}{\updefault}{\color[rgb]{0,0,0}$2y+z+e_a+o_b+o_c+|a-b|$}%
}}}}}
\put(2091,-3911){\rotatebox{60.0}{\makebox(0,0)[lb]{\smash{{\SetFigFont{14}{16.8}{\familydefault}{\mddefault}{\updefault}{\color[rgb]{0,0,0}$z+o_a+e_b+e_c$}%
}}}}}
\put(14167,-1287){\makebox(0,0)[lb]{\smash{{\SetFigFont{14}{16.8}{\familydefault}{\mddefault}{\updefault}{\color[rgb]{0,0,0}$x+e_a+o_b+o_c$}%
}}}}
\put(15601,-5501){\makebox(0,0)[lb]{\smash{{\SetFigFont{14}{16.8}{\rmdefault}{\mddefault}{\updefault}{\color[rgb]{1,1,1}$c_3$}%
}}}}
\put(15195,-9673){\makebox(0,0)[lb]{\smash{{\SetFigFont{14}{16.8}{\familydefault}{\mddefault}{\updefault}{\color[rgb]{0,0,0}$x+o_a+e_b+e_c$}%
}}}}
\put(11633,-5735){\rotatebox{300.0}{\makebox(0,0)[lb]{\smash{{\SetFigFont{14}{16.8}{\familydefault}{\mddefault}{\updefault}{\color[rgb]{0,0,0}$2y+z+e_a+o_b+o_c+|a-b|$}%
}}}}}
\put(17775,-2081){\rotatebox{300.0}{\makebox(0,0)[lb]{\smash{{\SetFigFont{14}{16.8}{\familydefault}{\mddefault}{\updefault}{\color[rgb]{0,0,0}$2y+z+o_a+e_b+e_c+|a-b|$}%
}}}}}
\put(11906,-3829){\rotatebox{60.0}{\makebox(0,0)[lb]{\smash{{\SetFigFont{14}{16.8}{\familydefault}{\mddefault}{\updefault}{\color[rgb]{0,0,0}$z+o_a+e_b+e_c$}%
}}}}}
\put(18757,-8852){\rotatebox{60.0}{\makebox(0,0)[lb]{\smash{{\SetFigFont{14}{16.8}{\familydefault}{\mddefault}{\updefault}{\color[rgb]{0,0,0}$z+e_a+o_b+o_c$}%
}}}}}
\put(2746,-5371){\makebox(0,0)[lb]{\smash{{\SetFigFont{14}{16.8}{\rmdefault}{\mddefault}{\updefault}{\color[rgb]{1,1,1}$a_2$}%
}}}}
\put(3491,-5871){\makebox(0,0)[lb]{\smash{{\SetFigFont{14}{16.8}{\rmdefault}{\mddefault}{\updefault}{\color[rgb]{1,1,1}$a_3$}%
}}}}
\put(5251,-5371){\makebox(0,0)[lb]{\smash{{\SetFigFont{14}{16.8}{\rmdefault}{\mddefault}{\updefault}{\color[rgb]{1,1,1}$c_1$}%
}}}}
\put(6361,-5811){\makebox(0,0)[lb]{\smash{{\SetFigFont{14}{16.8}{\rmdefault}{\mddefault}{\updefault}{\color[rgb]{1,1,1}$c_2$}%
}}}}
\put(8851,-5391){\makebox(0,0)[lb]{\smash{{\SetFigFont{14}{16.8}{\rmdefault}{\mddefault}{\updefault}{\color[rgb]{1,1,1}$b_1$}%
}}}}
\put(7981,-5821){\makebox(0,0)[lb]{\smash{{\SetFigFont{14}{16.8}{\rmdefault}{\mddefault}{\updefault}{\color[rgb]{1,1,1}$b_2$}%
}}}}
\put(12646,-5791){\makebox(0,0)[lb]{\smash{{\SetFigFont{14}{16.8}{\rmdefault}{\mddefault}{\updefault}{\color[rgb]{1,1,1}$a_1$}%
}}}}
\put(13351,-5326){\makebox(0,0)[lb]{\smash{{\SetFigFont{14}{16.8}{\rmdefault}{\mddefault}{\updefault}{\color[rgb]{1,1,1}$a_2$}%
}}}}
\put(14881,-5431){\makebox(0,0)[lb]{\smash{{\SetFigFont{14}{16.8}{\rmdefault}{\mddefault}{\updefault}{\color[rgb]{1,1,1}$c_1$}%
}}}}
\end{picture}%
}
\caption{Eliminating triangles of side-length $0$  from the ferns.}\label{fig:Special4}
\end{figure}

We first consider the case when one or more triangles in one of the three ferns have side-length $0$. The following lemma intuitively says that we can simply skip this case when working on our inductive proof on $h:=p+x+z$.

\begin{lem}\label{lem1}
For any regions of one of the eight types: $R^{\odot},$ $ R^{\leftarrow},$ $ R^{\nwarrow},$ $ R^{\swarrow},$ $Q^{\odot},$ $Q^{\leftarrow},$ $ Q^{\nwarrow}$, and $Q^{\nearrow}$. We can find a new region of the same type (1) whose number of tilings is the same,
(2) whose $h$-parameter is not greater, (3) whose left and right ferns consist of all triangles with positive side-lengths, and (4) whose middle fern contains
perhaps the first triangle of side-length $0$.
\end{lem}
\begin{proof}
We only consider the $R:=R^{\odot}_{x,y,z}(\textbf{a}; \textbf{c}; \textbf{b})$ region, the other $7$ regions can be treated similarly.

We will show how to eliminate $0$-triangles in the three ferns without changing the tiling number or increasing the $h$-parameter. We consider the following three $0$-eliminating procedures for the left fern:

(1) If $a_1=a_2=\dotsc=a_{2i}=0$, for some $i\geq 1$, we can simply truncate the first $2i$ zero terms in the sequence $\textbf{a}$. The new
region is `exactly' the old one, however, strictly speaking, it has less $0$-triangles in the left fern.

(2) If $a_1=0$ and $a_2>0$, then we can remove forced lozenges along the southwest side of the region $R$ and obtain the region
 $R^{\odot}_{x,y,z}(a_3,\dots,a_m;\ \textbf{c};\ \textbf{b})$ (see Figure \ref{fig:Special4}(a)).  The new region has the same number of tilings as the original one, the
 $h$-parameter $a_1$-unit less than $h$, and less $0$-triangles in the left fern.

(3) If $a_i=0$, for some $i>1$, then we can eliminate this $0$-triangle by combining the $(i-1)$-th and the $(i+1)$-th triangles in the fern (as shown
in Figure \ref{fig:Special4}(b)).

Repeating these three procedures if needed, one can eliminate all $0$-triangles from the left fern. Working similarly for the right fern, we obtain a region with no $0$-triangle in the left and right ferns. For the middle fern, we apply the procedure (3) to eliminate all $0$-triangles, except for the possible first $0$-triangle. This finishes our proof.
\end{proof}

The next lemma helps us handle the extremal case with respect to the $y$-parameter of our main proof provided in the next section.
\begin{lem}\label{lem2}
For any regions of one of the eight types with the  minimal $y$-parameter (i.e. $y=0$ for the $R^{\odot}$-, $R^{\leftarrow}$-, $Q^{\odot}$, $Q^{\leftarrow}$-type regions; $y=0$ or $-1$ for the other four types of regions), we can find other regions of one of the eight types, whose number of tilings are the same and whose $h$-parameter is strictly smaller.
\end{lem}

\begin{figure}
\setlength{\unitlength}{3947sp}%
\begingroup\makeatletter\ifx\SetFigFont\undefined%
\gdef\SetFigFont#1#2#3#4#5{%
  \reset@font\fontsize{#1}{#2pt}%
  \fontfamily{#3}\fontseries{#4}\fontshape{#5}%
  \selectfont}%
\fi\endgroup%
\resizebox{15cm}{!}{
\begin{picture}(0,0)%
\includegraphics{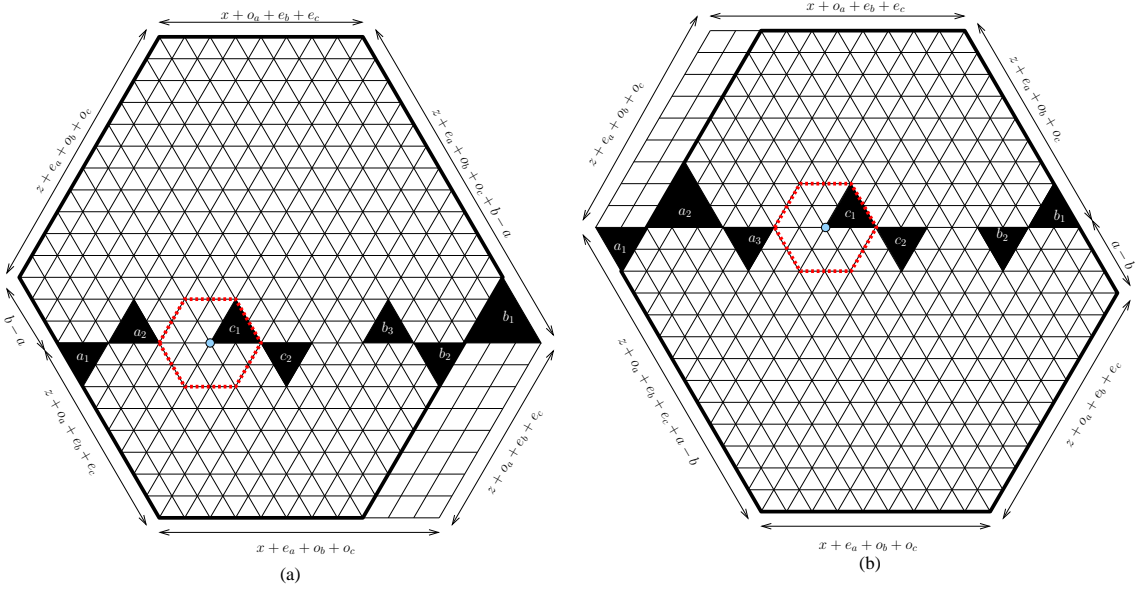}%
\end{picture}%
%
%

\begin{picture}(18272,9633)(1376,-9602)
\put(2558,-5927){\makebox(0,0)[lb]{\smash{{\SetFigFont{14}{16.8}{\rmdefault}{\mddefault}{\updefault}{\color[rgb]{1,1,1}$a_1$}%
}}}}
\put(3479,-5539){\makebox(0,0)[lb]{\smash{{\SetFigFont{14}{16.8}{\rmdefault}{\mddefault}{\updefault}{\color[rgb]{1,1,1}$a_2$}%
}}}}
\put(5013,-5480){\makebox(0,0)[lb]{\smash{{\SetFigFont{14}{16.8}{\rmdefault}{\mddefault}{\updefault}{\color[rgb]{1,1,1}$c_1$}%
}}}}
\put(5831,-5927){\makebox(0,0)[lb]{\smash{{\SetFigFont{14}{16.8}{\rmdefault}{\mddefault}{\updefault}{\color[rgb]{1,1,1}$c_2$}%
}}}}
\put(9411,-5303){\makebox(0,0)[lb]{\smash{{\SetFigFont{14}{16.8}{\rmdefault}{\mddefault}{\updefault}{\color[rgb]{1,1,1}$b_1$}%
}}}}
\put(8388,-5893){\makebox(0,0)[lb]{\smash{{\SetFigFont{14}{16.8}{\rmdefault}{\mddefault}{\updefault}{\color[rgb]{1,1,1}$b_2$}%
}}}}
\put(7468,-5480){\makebox(0,0)[lb]{\smash{{\SetFigFont{14}{16.8}{\rmdefault}{\mddefault}{\updefault}{\color[rgb]{1,1,1}$b_3$}%
}}}}
\put(4881,-401){\makebox(0,0)[lb]{\smash{{\SetFigFont{14}{16.8}{\rmdefault}{\mddefault}{\updefault}{\color[rgb]{0,0,0}$x+o_a+e_b+e_c$}%
}}}}
\put(8271,-1901){\rotatebox{300.0}{\makebox(0,0)[lb]{\smash{{\SetFigFont{14}{16.8}{\rmdefault}{\mddefault}{\updefault}{\color[rgb]{0,0,0}$z+e_a+o_b+o_c+b-a$}%
}}}}}
\put(9251,-8121){\rotatebox{60.0}{\makebox(0,0)[lb]{\smash{{\SetFigFont{14}{16.8}{\rmdefault}{\mddefault}{\updefault}{\color[rgb]{0,0,0}$z+o_a+e_b+e_c$}%
}}}}}
\put(5451,-9041){\makebox(0,0)[lb]{\smash{{\SetFigFont{14}{16.8}{\rmdefault}{\mddefault}{\updefault}{\color[rgb]{0,0,0}$x+e_a+o_b+o_c$}%
}}}}
\put(2041,-6441){\rotatebox{300.0}{\makebox(0,0)[lb]{\smash{{\SetFigFont{14}{16.8}{\rmdefault}{\mddefault}{\updefault}{\color[rgb]{0,0,0}$z+o_a+e_b+e_c$}%
}}}}}
\put(1391,-5261){\rotatebox{300.0}{\makebox(0,0)[lb]{\smash{{\SetFigFont{14}{16.8}{\rmdefault}{\mddefault}{\updefault}{\color[rgb]{0,0,0}$b-a$}%
}}}}}
\put(1991,-3241){\rotatebox{60.0}{\makebox(0,0)[lb]{\smash{{\SetFigFont{14}{16.8}{\rmdefault}{\mddefault}{\updefault}{\color[rgb]{0,0,0}$z+e_a+o_b+o_c$}%
}}}}}
\put(14906,-3608){\makebox(0,0)[lb]{\smash{{\SetFigFont{14}{16.8}{\rmdefault}{\mddefault}{\updefault}{\color[rgb]{1,1,1}$c_1$}%
}}}}
\put(15711,-4055){\makebox(0,0)[lb]{\smash{{\SetFigFont{14}{16.8}{\rmdefault}{\mddefault}{\updefault}{\color[rgb]{1,1,1}$c_2$}%
}}}}
\put(14292,-258){\makebox(0,0)[lb]{\smash{{\SetFigFont{14}{16.8}{\rmdefault}{\mddefault}{\updefault}{\color[rgb]{0,0,0}$x+o_a+e_b+e_c$}%
}}}}
\put(17551,-1040){\rotatebox{300.0}{\makebox(0,0)[lb]{\smash{{\SetFigFont{14}{16.8}{\rmdefault}{\mddefault}{\updefault}{\color[rgb]{0,0,0}$z+e_a+o_b+o_c$}%
}}}}}
\put(18547,-7351){\rotatebox{60.0}{\makebox(0,0)[lb]{\smash{{\SetFigFont{14}{16.8}{\rmdefault}{\mddefault}{\updefault}{\color[rgb]{0,0,0}$z+o_a+e_b+e_c$}%
}}}}}
\put(14491,-8998){\makebox(0,0)[lb]{\smash{{\SetFigFont{14}{16.8}{\rmdefault}{\mddefault}{\updefault}{\color[rgb]{0,0,0}$x+e_a+o_b+o_c$}%
}}}}
\put(11257,-5610){\rotatebox{300.0}{\makebox(0,0)[lb]{\smash{{\SetFigFont{14}{16.8}{\rmdefault}{\mddefault}{\updefault}{\color[rgb]{0,0,0}$z+o_a+e_b+e_c+a-b$}%
}}}}}
\put(19189,-4038){\rotatebox{300.0}{\makebox(0,0)[lb]{\smash{{\SetFigFont{14}{16.8}{\rmdefault}{\mddefault}{\updefault}{\color[rgb]{0,0,0}$a-b$}%
}}}}}
\put(10873,-2807){\rotatebox{60.0}{\makebox(0,0)[lb]{\smash{{\SetFigFont{14}{16.8}{\rmdefault}{\mddefault}{\updefault}{\color[rgb]{0,0,0}$z+e_a+o_b+o_c$}%
}}}}}
\put(11195,-4155){\makebox(0,0)[lb]{\smash{{\SetFigFont{14}{16.8}{\rmdefault}{\mddefault}{\updefault}{\color[rgb]{1,1,1}$a_1$}%
}}}}
\put(12233,-3583){\makebox(0,0)[lb]{\smash{{\SetFigFont{14}{16.8}{\rmdefault}{\mddefault}{\updefault}{\color[rgb]{1,1,1}$a_2$}%
}}}}
\put(13358,-4021){\makebox(0,0)[lb]{\smash{{\SetFigFont{14}{16.8}{\rmdefault}{\mddefault}{\updefault}{\color[rgb]{1,1,1}$a_3$}%
}}}}
\put(18268,-3667){\makebox(0,0)[lb]{\smash{{\SetFigFont{14}{16.8}{\rmdefault}{\mddefault}{\updefault}{\color[rgb]{1,1,1}$b_1$}%
}}}}
\put(17348,-3962){\makebox(0,0)[lb]{\smash{{\SetFigFont{14}{16.8}{\rmdefault}{\mddefault}{\updefault}{\color[rgb]{1,1,1}$b_2$}%
}}}}
\end{picture}%
}
\caption{Obtaining a $Q^{\odot}$-type (resp., $Q^{\leftarrow}$-type) region from a $R^{\odot}$-type (resp., $R^{\leftarrow}$-type) region by removing forced lozenges.}\label{SpecialR1}
\end{figure}

\begin{proof}
We first recall that the $y$-parameter can only obtains the value  $-1$ in the following four cases: (1) the case of $R^{\nwarrow}$-type regions with $a<b$, (2) the case of $R^{\swarrow}$-type regions with $a>b$, (3) the case of $Q^{\nwarrow}$-type regions with $a<b$, and (4) the case of $Q^{\nearrow}$-type regions with $a<b$.

By Lemma \ref{lem1}, we can assume, without loss of generality, that all $a_i$'s, $b_j$'s and $c_t$'s are all positive for $i\geq 1, j\geq1, t\geq 2$.

If our region is of type $R^{\odot}$ or type $R^{\leftarrow}$ and having the left fern not longer than the right fern (i.e., $a\leq b$) and $y=0$, then there are several forced  lozenges along the southeast side,
by removing these lozenges, we get an upside down $Q^{\odot}$-type region or an $Q^{\leftarrow}$- type region with $h$-parameter $1$-unit less than $h$. In particular, we have:
\begin{align}\label{specialeq1}
\M(R^{\odot}_{x,y,z}(\textbf{a}; \textbf{c}; \textbf{b}))&=\M(Q^{\odot}_{x, \min(b_1,b-a),z}(\textbf{a};\  {}^0\textbf{c};\ b_2,\dotsc,b_n));\\
\M(R^{\leftarrow}_{x,y,z}(\textbf{a}; \textbf{c}; \textbf{b}))&=\M(Q^{\leftarrow}_{x, \min(b_1,b-a),z}(\textbf{a};\  {}^0\textbf{c};\ b_2,\dotsc,b_n))
\end{align}
 (see Figure \ref{SpecialR1}(a) for the case of $R^{\odot}$-type regions; the case of $R^{\leftarrow}$-type regions is analogous). Recall that ${}^0\textbf{c}$ denotes the sequence obtained by including a new $0$ term in front of the sequence $\textbf{c}$.
Similarly, if $a\geq b$ and $y=0$, then the removal of forced lozenges along the northwest side of the region $R^{\odot}_{x,y,z}(\textbf{a}; \textbf{c}; \textbf{b})$
(resp., $R^{\leftarrow}_{x,y,z}(\textbf{a}; \textbf{c}; \textbf{b})$)  gives the region $Q^{\odot}_{x, \min(a_1,a-b),z}(a_2,\dotsc,a_m;\textbf{c}; \textbf{b})$ (resp., $Q^{\leftarrow}_{x, \min(a_1,a-b),z}(a_2,\dotsc,a_m;\textbf{c}; \textbf{b})$).
 See Figure  \ref{SpecialR1}(b) for the case of $R^{\odot}$-type regions; the case of $R^{\leftarrow}$-type regions is similar.

Next, let us consider the case of the $R^{\swarrow}$-type regions. If $a\leq b$ and $y=0$, then after removing forced lozenges as in the cases of the $R^{\odot}$- and $R^{\leftarrow}$-type regions above, we obtain
\begin{align}\label{specialeq2}
\M(R^{\swarrow}_{x,y,z}(\textbf{a}; \textbf{c}; \textbf{b}))=Q^{\nearrow}_{x, \min(b_1,b-a),z}(\textbf{a};\  {}^0\textbf{c};\ b_2,\dotsc,b_n).
\end{align}
If $a>b$ and $y=-1$, then we have forced lozenges along the northwest side of the region $R^{\swarrow}_{x,-1,z}(\textbf{a};\textbf{c};\textbf{b})$. By removing these forced lozenges, we get the region $Q^{\nearrow}_{x,\min(a_1, a-b)-1,z}(\textbf{b};\ \textbf{c}^{\leftrightarrow}; \ a_2,\dotsc,a_m)$ (see Figure \ref{fig:SpecialR4}(a) for an example). Recall that $\textbf{c}^{\leftrightarrow}$ is the sequence obtained from $\textbf{c}$ by reverting its order if the number of term is odd, otherwise, it is obtained by reverting order and including a $0$ term in front of the resulting sequence.

\begin{figure}\centering
\setlength{\unitlength}{3947sp}%
\begingroup\makeatletter\ifx\SetFigFont\undefined%
\gdef\SetFigFont#1#2#3#4#5{%
  \reset@font\fontsize{#1}{#2pt}%
  \fontfamily{#3}\fontseries{#4}\fontshape{#5}%
  \selectfont}%
\fi\endgroup%
\resizebox{15cm}{!}{
\begin{picture}(0,0)%
\includegraphics{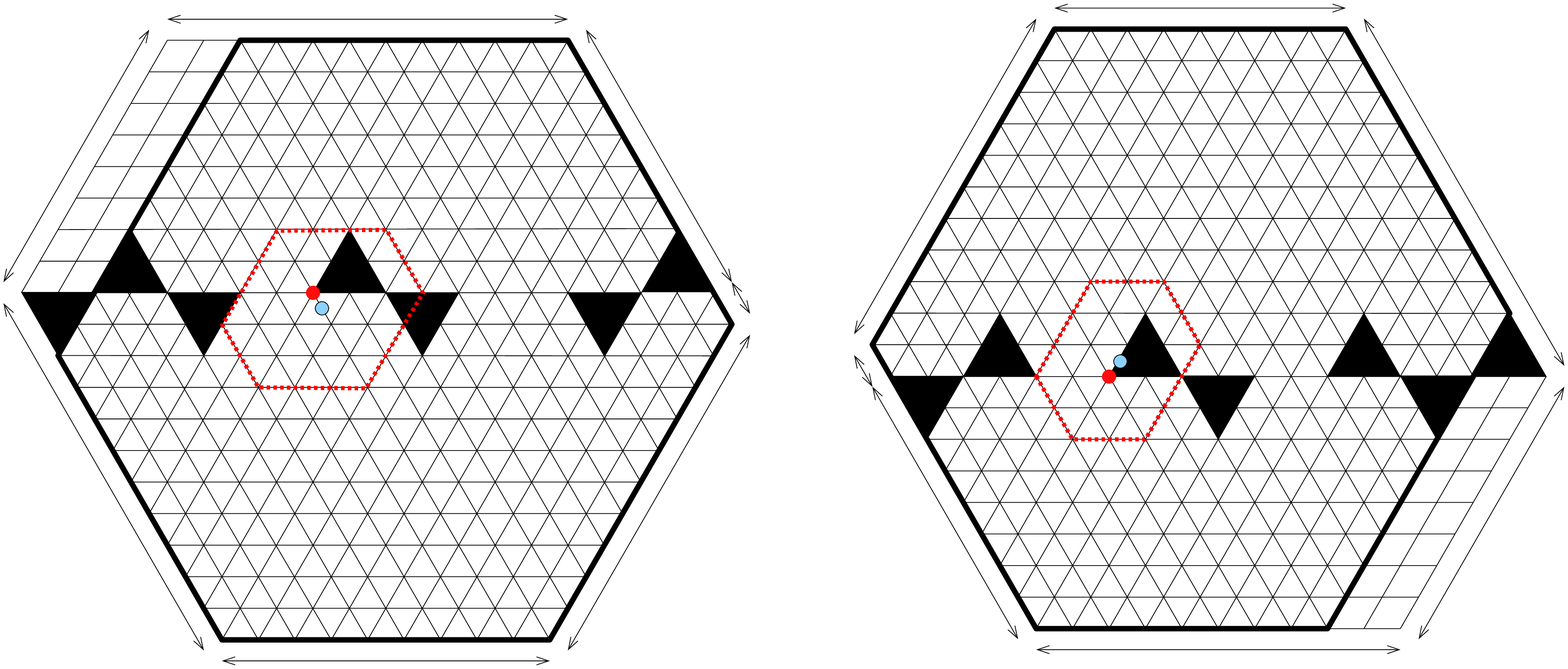}%
\end{picture}%
%
%

\begin{picture}(18070,8619)(1404,-10647)
\put(4808,-2374){\makebox(0,0)[lb]{\smash{{\SetFigFont{14}{16.8}{\rmdefault}{\mddefault}{\updefault}{\color[rgb]{0,0,0}$x+o_a+e_b+e_c$}%
}}}}
\put(8695,-3189){\rotatebox{300.0}{\makebox(0,0)[lb]{\smash{{\SetFigFont{14}{16.8}{\rmdefault}{\mddefault}{\updefault}{\color[rgb]{0,0,0}$z+e_a+o_b+o_c$}%
}}}}}
\put(9923,-5433){\rotatebox{300.0}{\makebox(0,0)[lb]{\smash{{\SetFigFont{14}{16.8}{\rmdefault}{\mddefault}{\updefault}{\color[rgb]{0,0,0}$a-b-1$}%
}}}}}
\put(8695,-8905){\rotatebox{60.0}{\makebox(0,0)[lb]{\smash{{\SetFigFont{14}{16.8}{\rmdefault}{\mddefault}{\updefault}{\color[rgb]{0,0,0}$z+o_a+e_b+e_c$}%
}}}}}
\put(1638,-4712){\rotatebox{60.0}{\makebox(0,0)[lb]{\smash{{\SetFigFont{14}{16.8}{\rmdefault}{\mddefault}{\updefault}{\color[rgb]{0,0,0}$z+e_a+o_b+o_c$}%
}}}}}
\put(1842,-6602){\rotatebox{300.0}{\makebox(0,0)[lb]{\smash{{\SetFigFont{14}{16.8}{\rmdefault}{\mddefault}{\updefault}{\color[rgb]{0,0,0}$z+o_a+e_b+e_c+b-a-1$}%
}}}}}
\put(4808,-10027){\makebox(0,0)[lb]{\smash{{\SetFigFont{14}{16.8}{\rmdefault}{\mddefault}{\updefault}{\color[rgb]{0,0,0}$x+e_a+o_b+o_c$}%
}}}}
\put(14012,-2295){\makebox(0,0)[lb]{\smash{{\SetFigFont{14}{16.8}{\rmdefault}{\mddefault}{\updefault}{\color[rgb]{0,0,0}$x+o_a+e_b+e_c$}%
}}}}
\put(11437,-5135){\rotatebox{60.0}{\makebox(0,0)[lb]{\smash{{\SetFigFont{14}{16.8}{\rmdefault}{\mddefault}{\updefault}{\color[rgb]{0,0,0}$z+e_a+o_b+o_c$}%
}}}}}
\put(17658,-3519){\rotatebox{300.0}{\makebox(0,0)[lb]{\smash{{\SetFigFont{14}{16.8}{\rmdefault}{\mddefault}{\updefault}{\color[rgb]{0,0,0}$z+e_a+o_b+o_c+b-a-1$}%
}}}}}
\put(11400,-7427){\rotatebox{300.0}{\makebox(0,0)[lb]{\smash{{\SetFigFont{14}{16.8}{\rmdefault}{\mddefault}{\updefault}{\color[rgb]{0,0,0}$z+o_a+e_b+e_c$}%
}}}}}
\put(10843,-6471){\rotatebox{300.0}{\makebox(0,0)[lb]{\smash{{\SetFigFont{14}{16.8}{\rmdefault}{\mddefault}{\updefault}{\color[rgb]{0,0,0}$b-a-1$}%
}}}}}
\put(14136,-10027){\makebox(0,0)[lb]{\smash{{\SetFigFont{14}{16.8}{\rmdefault}{\mddefault}{\updefault}{\color[rgb]{0,0,0}$x+e_a+o_b+o_c$}%
}}}}
\put(18101,-9295){\rotatebox{60.0}{\makebox(0,0)[lb]{\smash{{\SetFigFont{14}{16.8}{\rmdefault}{\mddefault}{\updefault}{\color[rgb]{0,0,0}$z+o_a+e_b+e_c$}%
}}}}}
\put(3729,-5933){\makebox(0,0)[lb]{\smash{{\SetFigFont{14}{16.8}{\rmdefault}{\mddefault}{\updefault}{\color[rgb]{1,1,1}$a_3$}%
}}}}
\put(5360,-5569){\makebox(0,0)[lb]{\smash{{\SetFigFont{14}{16.8}{\rmdefault}{\mddefault}{\updefault}{\color[rgb]{1,1,1}$c_1$}%
}}}}
\put(6234,-5962){\makebox(0,0)[lb]{\smash{{\SetFigFont{14}{16.8}{\rmdefault}{\mddefault}{\updefault}{\color[rgb]{1,1,1}$c_2$}%
}}}}
\put(9147,-5540){\makebox(0,0)[lb]{\smash{{\SetFigFont{14}{16.8}{\rmdefault}{\mddefault}{\updefault}{\color[rgb]{1,1,1}$b_1$}%
}}}}
\put(8258,-5977){\makebox(0,0)[lb]{\smash{{\SetFigFont{14}{16.8}{\rmdefault}{\mddefault}{\updefault}{\color[rgb]{1,1,1}$b_2$}%
}}}}
\put(11899,-6865){\makebox(0,0)[lb]{\smash{{\SetFigFont{14}{16.8}{\rmdefault}{\mddefault}{\updefault}{\color[rgb]{1,1,1}$a_1$}%
}}}}
\put(12715,-6443){\makebox(0,0)[lb]{\smash{{\SetFigFont{14}{16.8}{\rmdefault}{\mddefault}{\updefault}{\color[rgb]{1,1,1}$a_2$}%
}}}}
\put(14360,-6501){\makebox(0,0)[lb]{\smash{{\SetFigFont{14}{16.8}{\rmdefault}{\mddefault}{\updefault}{\color[rgb]{1,1,1}$c_1$}%
}}}}
\put(15161,-6894){\makebox(0,0)[lb]{\smash{{\SetFigFont{14}{16.8}{\rmdefault}{\mddefault}{\updefault}{\color[rgb]{1,1,1}$c_2$}%
}}}}
\put(18351,-6501){\makebox(0,0)[lb]{\smash{{\SetFigFont{14}{16.8}{\rmdefault}{\mddefault}{\updefault}{\color[rgb]{1,1,1}$b_1$}%
}}}}
\put(17608,-6909){\makebox(0,0)[lb]{\smash{{\SetFigFont{14}{16.8}{\rmdefault}{\mddefault}{\updefault}{\color[rgb]{1,1,1}$b_2$}%
}}}}
\put(16792,-6472){\makebox(0,0)[lb]{\smash{{\SetFigFont{14}{16.8}{\rmdefault}{\mddefault}{\updefault}{\color[rgb]{1,1,1}$b_3$}%
}}}}
\put(5579,-10491){\makebox(0,0)[lb]{\smash{{\SetFigFont{20}{24.0}{\rmdefault}{\mddefault}{\updefault}{\color[rgb]{0,0,0}(a)}%
}}}}
\put(15001,-10535){\makebox(0,0)[lb]{\smash{{\SetFigFont{20}{24.0}{\rmdefault}{\mddefault}{\updefault}{\color[rgb]{0,0,0}(b)}%
}}}}
\put(2113,-5940){\makebox(0,0)[lb]{\smash{{\SetFigFont{14}{16.8}{\rmdefault}{\mddefault}{\updefault}{\color[rgb]{1,1,1}$a_1$}%
}}}}
\put(2943,-5518){\makebox(0,0)[lb]{\smash{{\SetFigFont{14}{16.8}{\rmdefault}{\mddefault}{\updefault}{\color[rgb]{1,1,1}$a_2$}%
}}}}
\end{picture}
}
\caption{(a) Obtaining a $Q^{\nearrow}$-type region from the region $R^{\swarrow}_{x,-1,z}(\textbf{a};\textbf{c};\textbf{b})$ when $a>b$. (b) Obtaining a $Q^{\nwarrow}$-type region from the region
$R^{\nwarrow}_{x,-1,z}(\textbf{a};\textbf{c};\textbf{b})$ when $a<b$.}\label{fig:SpecialR4}
\end{figure}

The case of the $R^{\nwarrow}$-type regions can be treated similarly to the case of the $R^{\swarrow}$-type regions above. If $a\geq b$ and $y=0$, then, by removing forced lozenges along the northwest side as in the case of $R^{\odot}$-
and $R^{\leftarrow}$-type regions, we get
\begin{align}\label{specialeq2}
\M(R^{\nwarrow}_{x,y,z}(\textbf{a}; \textbf{c}; \textbf{b}))=Q^{\nwarrow}_{x, \min(a_1,a-b),z}(a_2,\dotsc,a_m;\textbf{c}; \textbf{b}).
\end{align}
If $a< b$ and $y=-1$, after removing forced lozenges along the southeast side of the region, we get the region $Q^{\nwarrow}_{x,\min (b_1,b-a)-1, z}(\textbf{a};\ {}^0\textbf{c};\ b_2,\dotsc,b_n)$ (shown in Figure \ref{fig:SpecialR4}(b)).

\begin{figure}\centering
\setlength{\unitlength}{3947sp}%
\begingroup\makeatletter\ifx\SetFigFont\undefined%
\gdef\SetFigFont#1#2#3#4#5{%
  \reset@font\fontsize{#1}{#2pt}%
  \fontfamily{#3}\fontseries{#4}\fontshape{#5}%
  \selectfont}%
\fi\endgroup%
\resizebox{15cm}{!}{
\begin{picture}(0,0)%
\includegraphics{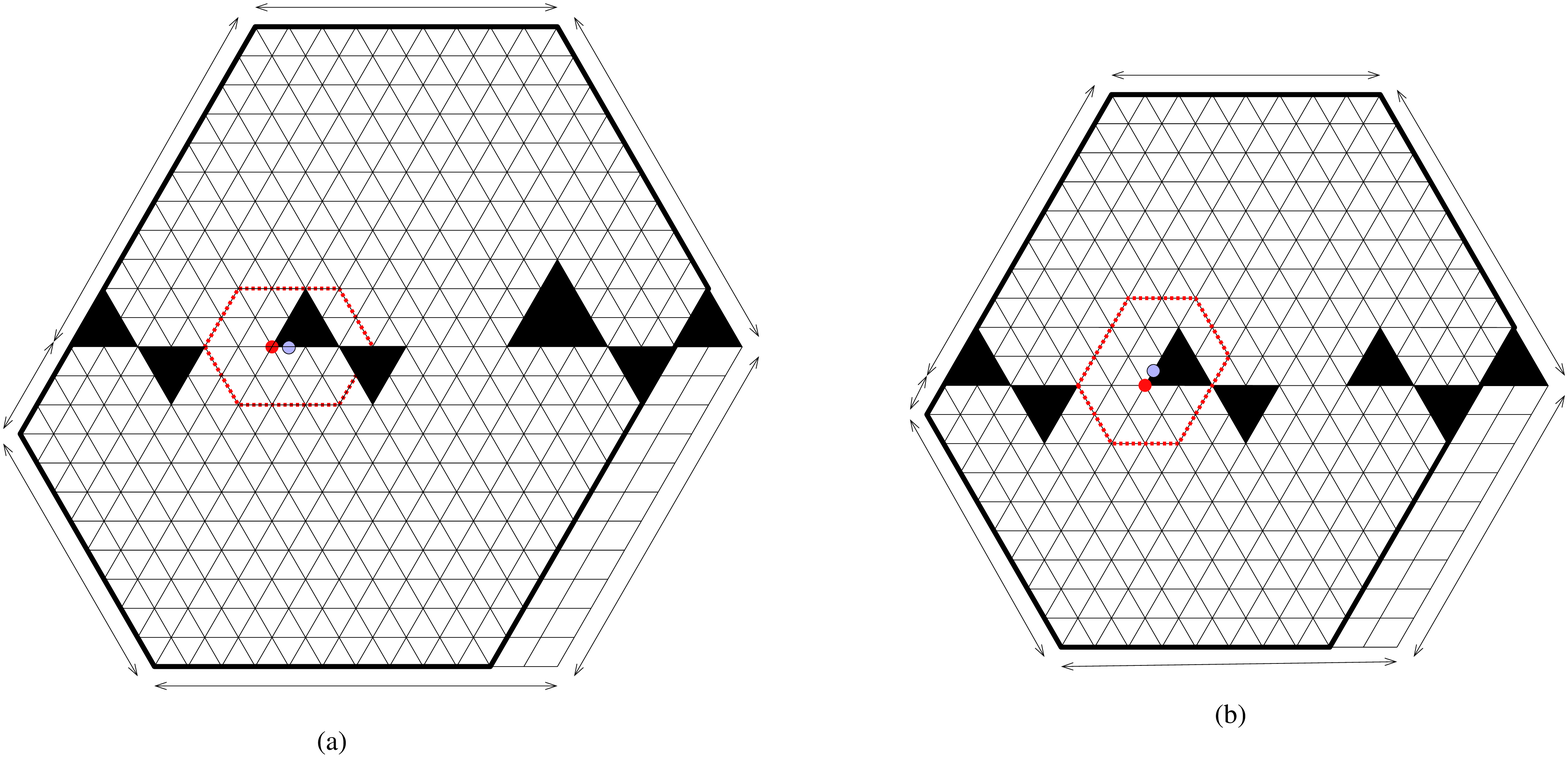}%
\end{picture}%
%
%

\begin{picture}(19531,9611)(1207,-9736)
\put(5422,-401){\makebox(0,0)[lb]{\smash{{\SetFigFont{14}{16.8}{\rmdefault}{\mddefault}{\itdefault}{\color[rgb]{0,0,0}$x+e_a+e_b+e_c$}%
}}}}
\put(9002,-1405){\rotatebox{300.0}{\makebox(0,0)[lb]{\smash{{\SetFigFont{14}{16.8}{\rmdefault}{\mddefault}{\itdefault}{\color[rgb]{0,0,0}$z+o_a+o_b+o_c$}%
}}}}}
\put(9227,-7901){\rotatebox{60.0}{\makebox(0,0)[lb]{\smash{{\SetFigFont{14}{16.8}{\rmdefault}{\mddefault}{\itdefault}{\color[rgb]{0,0,0}$z+e_a+e_b+e_c+b-a$}%
}}}}}
\put(4297,-9082){\makebox(0,0)[lb]{\smash{{\SetFigFont{14}{16.8}{\rmdefault}{\mddefault}{\itdefault}{\color[rgb]{0,0,0}$x+o_a+o_b+o_c$}%
}}}}
\put(1529,-6699){\rotatebox{300.0}{\makebox(0,0)[lb]{\smash{{\SetFigFont{14}{16.8}{\rmdefault}{\mddefault}{\itdefault}{\color[rgb]{0,0,0}$z+e_a+e_b+e_c$}%
}}}}}
\put(2165,-4121){\rotatebox{60.0}{\makebox(0,0)[lb]{\smash{{\SetFigFont{14}{16.8}{\rmdefault}{\mddefault}{\itdefault}{\color[rgb]{0,0,0}$z+o_a+o_b+o_c$}%
}}}}}
\put(2456,-4598){\makebox(0,0)[lb]{\smash{{\SetFigFont{14}{16.8}{\rmdefault}{\mddefault}{\itdefault}{\color[rgb]{1,1,1}$a_1$}%
}}}}
\put(3321,-4981){\makebox(0,0)[lb]{\smash{{\SetFigFont{14}{16.8}{\rmdefault}{\mddefault}{\itdefault}{\color[rgb]{1,1,1}$a_2$}%
}}}}
\put(4941,-4571){\makebox(0,0)[lb]{\smash{{\SetFigFont{14}{16.8}{\rmdefault}{\mddefault}{\itdefault}{\color[rgb]{1,1,1}$c_1$}%
}}}}
\put(5841,-5051){\makebox(0,0)[lb]{\smash{{\SetFigFont{14}{16.8}{\rmdefault}{\mddefault}{\itdefault}{\color[rgb]{1,1,1}$c_2$}%
}}}}
\put(8081,-4481){\makebox(0,0)[lb]{\smash{{\SetFigFont{14}{16.8}{\rmdefault}{\mddefault}{\itdefault}{\color[rgb]{1,1,1}$b_3$}%
}}}}
\put(9081,-5021){\makebox(0,0)[lb]{\smash{{\SetFigFont{14}{16.8}{\rmdefault}{\mddefault}{\itdefault}{\color[rgb]{1,1,1}$b_2$}%
}}}}
\put(9821,-4598){\makebox(0,0)[lb]{\smash{{\SetFigFont{14}{16.8}{\rmdefault}{\mddefault}{\itdefault}{\color[rgb]{1,1,1}$b_1$}%
}}}}
\put(13094,-5071){\makebox(0,0)[lb]{\smash{{\SetFigFont{14}{16.8}{\rmdefault}{\mddefault}{\itdefault}{\color[rgb]{1,1,1}$a_1$}%
}}}}
\put(13959,-5454){\makebox(0,0)[lb]{\smash{{\SetFigFont{14}{16.8}{\rmdefault}{\mddefault}{\itdefault}{\color[rgb]{1,1,1}$a_2$}%
}}}}
\put(15612,-5044){\makebox(0,0)[lb]{\smash{{\SetFigFont{14}{16.8}{\rmdefault}{\mddefault}{\itdefault}{\color[rgb]{1,1,1}$c_1$}%
}}}}
\put(16469,-5539){\makebox(0,0)[lb]{\smash{{\SetFigFont{14}{16.8}{\rmdefault}{\mddefault}{\itdefault}{\color[rgb]{1,1,1}$c_2$}%
}}}}
\put(18004,-5071){\makebox(0,0)[lb]{\smash{{\SetFigFont{14}{16.8}{\rmdefault}{\mddefault}{\itdefault}{\color[rgb]{1,1,1}$b_3$}%
}}}}
\put(18822,-5439){\makebox(0,0)[lb]{\smash{{\SetFigFont{14}{16.8}{\rmdefault}{\mddefault}{\itdefault}{\color[rgb]{1,1,1}$b_2$}%
}}}}
\put(19641,-5070){\makebox(0,0)[lb]{\smash{{\SetFigFont{14}{16.8}{\rmdefault}{\mddefault}{\itdefault}{\color[rgb]{1,1,1}$b_1$}%
}}}}
\put(15512,-1241){\makebox(0,0)[lb]{\smash{{\SetFigFont{14}{16.8}{\rmdefault}{\mddefault}{\itdefault}{\color[rgb]{0,0,0}$x+e_a+e_b+e_c$}%
}}}}
\put(18936,-1981){\rotatebox{300.0}{\makebox(0,0)[lb]{\smash{{\SetFigFont{14}{16.8}{\rmdefault}{\mddefault}{\itdefault}{\color[rgb]{0,0,0}$z+o_a+o_b+o_c$}%
}}}}}
\put(19096,-8397){\rotatebox{60.0}{\makebox(0,0)[lb]{\smash{{\SetFigFont{14}{16.8}{\rmdefault}{\mddefault}{\itdefault}{\color[rgb]{0,0,0}$z+e_a+e_b+e_c+b-a-1$}%
}}}}}
\put(12948,-4164){\rotatebox{60.0}{\makebox(0,0)[lb]{\smash{{\SetFigFont{14}{16.8}{\rmdefault}{\mddefault}{\itdefault}{\color[rgb]{0,0,0}$z+o_a+o_b+o_c$}%
}}}}}
\put(1449,-5421){\rotatebox{60.0}{\makebox(0,0)[lb]{\smash{{\SetFigFont{14}{16.8}{\rmdefault}{\mddefault}{\itdefault}{\color[rgb]{0,0,0}$b-a$}%
}}}}}
\put(12294,-5478){\rotatebox{60.0}{\makebox(0,0)[lb]{\smash{{\SetFigFont{14}{16.8}{\rmdefault}{\mddefault}{\itdefault}{\color[rgb]{0,0,0}$b-a-1$}%
}}}}}
\put(12493,-6366){\rotatebox{300.0}{\makebox(0,0)[lb]{\smash{{\SetFigFont{14}{16.8}{\rmdefault}{\mddefault}{\itdefault}{\color[rgb]{0,0,0}$z+e_a+e_b+e_c$}%
}}}}}
\put(15388,-8925){\makebox(0,0)[lb]{\smash{{\SetFigFont{14}{16.8}{\rmdefault}{\mddefault}{\itdefault}{\color[rgb]{0,0,0}$x+o_a+o_b+o_c$}%
}}}}
\end{picture}%
}
\caption{(a) Obtaining a $R^{\leftarrow}$-type region from the region $Q^{\leftarrow}_{x,0,z}(\textbf{a}; \textbf{c}; \textbf{b})$, when $a\leq b$. (b) Obtaining a $R^{\nwarrow}$-type region from the region $Q^{\nwarrow}_{x,-1,z}(\textbf{a}; \textbf{c}; \textbf{b})$ when $a<b$.}\label{fig:SpecialQ2}
\end{figure}

Next, we consider the four $Q$-regions. The region $Q^{\odot}_{x,0,z}(\textbf{a};\ \textbf{c};\ \textbf{b})$ has forced lozenges along its southeast (resp., northwest) side when $a\leq b$ (resp., $a\geq b$).
 By removing these forced lozenges, we get the region $R^{\odot}_{x,\min (b_1, b-a), z} (\textbf{a};\ {}^{0}\textbf{c};\ b_2,\dotsc,b_n)$
 (resp., $R^{\odot}_{x,\min (a_1, a-b), z} (a_2,\dotsc,a_m;\ \textbf{c}; \ \textbf{b})$). Similarly, the removal of forced lozenges in the region $Q^{\leftarrow}_{x,0,z}(\textbf{a};\ \textbf{c};\ \textbf{b})$
 gives us the region $R^{\leftarrow}_{x,\min(b_1, b-a), z}(\textbf{a};\ {}^{0}\textbf{c};\  b_2,\dotsc,b_n)$ (up to a reflection) if $a\leq b$ (see Figure \ref{fig:SpecialQ2}(a)), or
 $R^{\leftarrow}_{x,\min(a_1, a-b), z}(a_2,\dotsc,a_m;\ \textbf{c};\ \textbf{b})$ if $a\geq b$. Moreover, this lozenge-removal always reduces the $h$-parameter of the region.

If $a\geq b$, then the same thing happens for the regions $Q^{\nwarrow}_{x,0,z}(\textbf{a};\ \textbf{c};\ \textbf{b})$ and  $Q^{\nearrow}_{x,0,z}(\textbf{a};\ \textbf{c};\ \textbf{b})$. In particular, we have
\begin{align}\label{specialeq3}
\M(Q^{\nwarrow}_{x,y,z}(\textbf{a}; \textbf{c}; \textbf{b}))&=\M(R^{\nwarrow}_{x, \min(a_1,a-b),z}(a_2,\dotsc,a_m;\ \textbf{c};\ \textbf{b}));\\
\M(Q^{\nearrow}_{x,y,z}(\textbf{a}; \textbf{c}; \textbf{b}))&=\M(R^{\swarrow}_{x, \min(a_1,a-b),z}(a_2,\dotsc,a_m;\ \textbf{c};\ \textbf{b})).
\end{align}


Finally, if $a<b$ and $y=-1$, we get the region $R^{\nwarrow}_{x,\min(b_1,b-a)-1,z}(\textbf{a};\ {}^{0}\textbf{c};\ b_2,\dotsc,b_n)$ and  \\ $R^{\nearrow}_{x,\min(b_1,b-a)-1,z}(b_2,\dotsc,b_n;\ \textbf{c}^{\leftrightarrow};\ \textbf{a})$ from the regions $Q^{\nwarrow}_{x,-1,z}(\textbf{a};\ \textbf{c};\ \textbf{b})$ and  $Q^{\nearrow}_{x,-1,z}(\textbf{a};\ \textbf{c};\ \textbf{b})$ by removing forced lozenges, respectively (see Figure \ref{fig:SpecialQ2}(b) for an example of $Q^{\nwarrow}_{x,-1,z}(\textbf{a};\ \textbf{c};\ \textbf{b})$ region).
\end{proof}

\subsection{The main proof of Theorems \ref{main1}--\ref{mainQ4}}

We are now ready to prove our theorems by induction on $h:=p+x+z$. Recall that $p$ denotes the quasi-perimeter of the region
 (the perimeter of the base hexagon that our region is obtained by removing three ferns).

By Lemma \ref{lem1}, we can assume, without loss of generality, that $a_i$'s, $b_j$'s, and $c_t$'s are all positive (for $i\geq 1, j\geq 1,t\geq 2$) in the rest of the proof.

The base cases are the situations when at least one of the parameters $x,z$ is equal to $0$, and the case $p<6$.

We consider the first base case when $z=0$. We can divide the region $R^{\odot}_{x,y,0}(\textbf{a}; \textbf{c}; \textbf{b})$ into two balanced subregions satisfying the conditions in Regions-Splitting Lemma \ref{RS} by cutting along the lattice $\ell$ that the three ferns are resting on. The upper and lower halves are respectively the dented semihexagons corresponding to the two $s$-terms (with $z=0$) in the formula of Theorem \ref{main1}. This means that Theorem \ref{main1} follows from Cohn--Larsen--Propp's formula (\ref{semieq}). Similarly, we can verify the tiling formulas (in the case $z=0$) for all other $7$ regions in Theorems \ref{main2}--\ref{mainQ4}.

If $x=0$, we also apply Region--Splitting Lemma \ref{RS} to our eight regions by cutting along the lattice line $\ell$ containing the bases on triangles in the three ferns. The only difference is that we now add two `bumps' of lengths $\left\lfloor \frac{z}{2} \right\rfloor$ and $\left\lceil \frac{z}{2} \right\rceil$ to the cut at the positions of the `gaps' between two consecutive ferns. The upper part is a dented semihexagon, while the lower part is also isomorphic to a dented semihexagon after removing several vertical forced lozenges at the places corresponding to the bumps above. Again, by Cohn--Larsen--Propp's formula (\ref{semieq}), we can verify our theorems for the case $x=0$.

If $p<6$, by Claim \ref{claimp}, we have $2z+4z<6$. It means that at least one of $x$ and $z$ is $0$, this is reduced to the base cases treated above.

\begin{figure}\centering
\setlength{\unitlength}{3947sp}%
\begingroup\makeatletter\ifx\SetFigFont\undefined%
\gdef\SetFigFont#1#2#3#4#5{%
  \reset@font\fontsize{#1}{#2pt}%
  \fontfamily{#3}\fontseries{#4}\fontshape{#5}%
  \selectfont}%
\fi\endgroup%
\resizebox{10cm}{!}{
\begin{picture}(0,0)%
\includegraphics{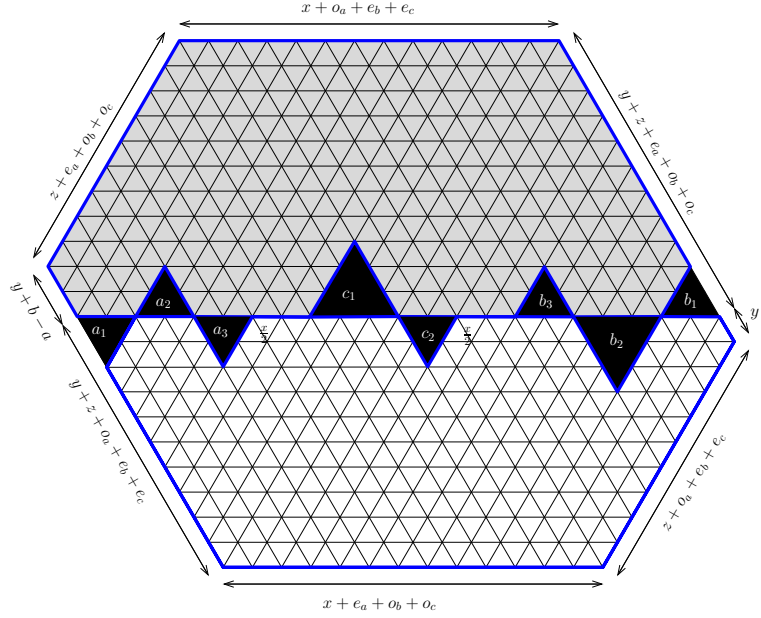}%
\end{picture}%
%
%

\begin{picture}(10593,8716)(1466,-10616)
\put(2661,-6731){\makebox(0,0)[lb]{\smash{{\SetFigFont{14}{16.8}{\rmdefault}{\mddefault}{\updefault}{\color[rgb]{1,1,1}$a_1$}%
}}}}
\put(3561,-6331){\makebox(0,0)[lb]{\smash{{\SetFigFont{14}{16.8}{\rmdefault}{\mddefault}{\updefault}{\color[rgb]{1,1,1}$a_2$}%
}}}}
\put(4361,-6731){\makebox(0,0)[lb]{\smash{{\SetFigFont{14}{16.8}{\rmdefault}{\mddefault}{\updefault}{\color[rgb]{1,1,1}$a_3$}%
}}}}
\put(6171,-6221){\makebox(0,0)[lb]{\smash{{\SetFigFont{14}{16.8}{\rmdefault}{\mddefault}{\updefault}{\color[rgb]{1,1,1}$c_1$}%
}}}}
\put(7281,-6791){\makebox(0,0)[lb]{\smash{{\SetFigFont{14}{16.8}{\rmdefault}{\mddefault}{\updefault}{\color[rgb]{1,1,1}$c_2$}%
}}}}
\put(10961,-6321){\makebox(0,0)[lb]{\smash{{\SetFigFont{14}{16.8}{\rmdefault}{\mddefault}{\updefault}{\color[rgb]{1,1,1}$b_1$}%
}}}}
\put(9921,-6891){\makebox(0,0)[lb]{\smash{{\SetFigFont{14}{16.8}{\rmdefault}{\mddefault}{\updefault}{\color[rgb]{1,1,1}$b_2$}%
}}}}
\put(8931,-6331){\makebox(0,0)[lb]{\smash{{\SetFigFont{14}{16.8}{\rmdefault}{\mddefault}{\updefault}{\color[rgb]{1,1,1}$b_3$}%
}}}}
\put(5011,-6781){\makebox(0,0)[lb]{\smash{{\SetFigFont{14}{16.8}{\rmdefault}{\mddefault}{\updefault}{\color[rgb]{0,0,0}$\frac{x}{2}$}%
}}}}
\put(7856,-6801){\makebox(0,0)[lb]{\smash{{\SetFigFont{14}{16.8}{\rmdefault}{\mddefault}{\updefault}{\color[rgb]{0,0,0}$\frac{x}{2}$}%
}}}}
\put(5601,-2181){\makebox(0,0)[lb]{\smash{{\SetFigFont{14}{16.8}{\rmdefault}{\mddefault}{\updefault}{\color[rgb]{0,0,0}$x+o_a+e_b+e_c$}%
}}}}
\put(10071,-3331){\rotatebox{300.0}{\makebox(0,0)[lb]{\smash{{\SetFigFont{14}{16.8}{\rmdefault}{\mddefault}{\updefault}{\color[rgb]{0,0,0}$y+z+e_a+o_b+o_c$}%
}}}}}
\put(11891,-6471){\makebox(0,0)[lb]{\smash{{\SetFigFont{14}{16.8}{\rmdefault}{\mddefault}{\updefault}{\color[rgb]{0,0,0}$y$}%
}}}}
\put(10741,-9531){\rotatebox{60.0}{\makebox(0,0)[lb]{\smash{{\SetFigFont{14}{16.8}{\rmdefault}{\mddefault}{\updefault}{\color[rgb]{0,0,0}$z+o_a+e_b+e_c$}%
}}}}}
\put(5901,-10601){\makebox(0,0)[lb]{\smash{{\SetFigFont{14}{16.8}{\rmdefault}{\mddefault}{\updefault}{\color[rgb]{0,0,0}$x+e_a+o_b+o_c$}%
}}}}
\put(2371,-7411){\rotatebox{300.0}{\makebox(0,0)[lb]{\smash{{\SetFigFont{14}{16.8}{\rmdefault}{\mddefault}{\updefault}{\color[rgb]{0,0,0}$y+z+o_a+e_b+e_c$}%
}}}}}
\put(1551,-6041){\rotatebox{300.0}{\makebox(0,0)[lb]{\smash{{\SetFigFont{14}{16.8}{\rmdefault}{\mddefault}{\updefault}{\color[rgb]{0,0,0}$y+b-a$}%
}}}}}
\put(2151,-4851){\rotatebox{60.0}{\makebox(0,0)[lb]{\smash{{\SetFigFont{14}{16.8}{\rmdefault}{\mddefault}{\updefault}{\color[rgb]{0,0,0}$z+e_a+o_b+o_c$}%
}}}}}
\end{picture}%
}
\caption{Partitioning the region $R^{\odot}_{x,y,0}(\textbf{a};\textbf{c}; \textbf{b})$ into two dented semihexagons.}\label{Basecasethreefern}
\end{figure}

\begin{figure}\centering
\setlength{\unitlength}{3947sp}%
\begingroup\makeatletter\ifx\SetFigFont\undefined%
\gdef\SetFigFont#1#2#3#4#5{%
  \reset@font\fontsize{#1}{#2pt}%
  \fontfamily{#3}\fontseries{#4}\fontshape{#5}%
  \selectfont}%
\fi\endgroup%
\resizebox{10cm}{!}{
\begin{picture}(0,0)%
\includegraphics{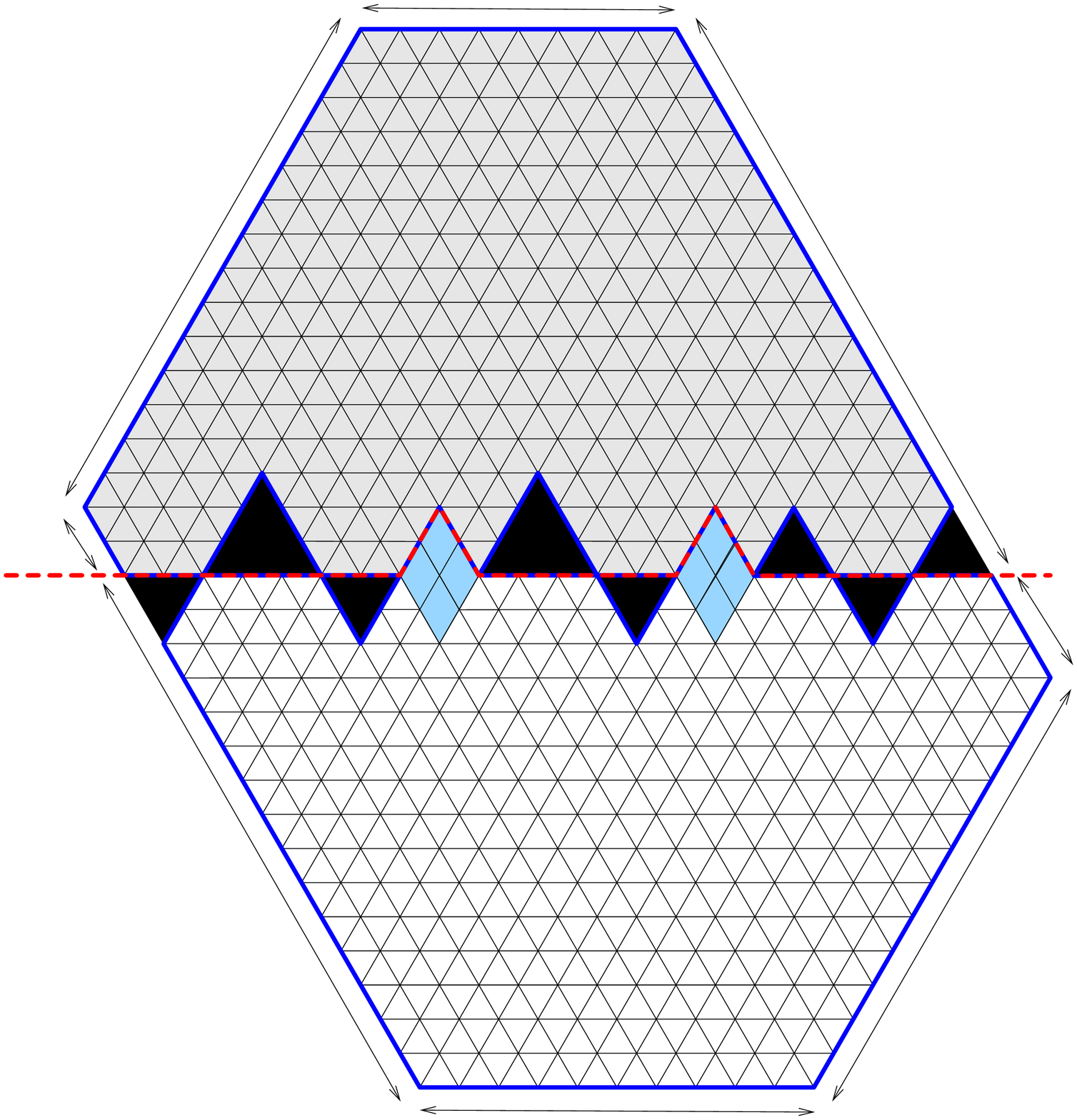}%
\end{picture}%
%
%

\begin{picture}(11516,12159)(980,-15096)
\put(5291,-3211){\makebox(0,0)[lb]{\smash{{\SetFigFont{14}{16.8}{\rmdefault}{\mddefault}{\updefault}{\color[rgb]{0,0,0}$x+o_a+e_b+e_c$}%
}}}}
\put(9371,-4801){\rotatebox{300.0}{\makebox(0,0)[lb]{\smash{{\SetFigFont{14}{16.8}{\rmdefault}{\mddefault}{\updefault}{\color[rgb]{0,0,0}$y+z+e_a+o_b+o_c$}%
}}}}}
\put(11871,-9441){\rotatebox{300.0}{\makebox(0,0)[lb]{\smash{{\SetFigFont{14}{16.8}{\rmdefault}{\mddefault}{\updefault}{\color[rgb]{0,0,0}$y+a-b$}%
}}}}}
\put(10421,-13781){\rotatebox{60.0}{\makebox(0,0)[lb]{\smash{{\SetFigFont{14}{16.8}{\rmdefault}{\mddefault}{\updefault}{\color[rgb]{0,0,0}$z+o_a+e_b+e_c$}%
}}}}}
\put(6281,-15081){\makebox(0,0)[lb]{\smash{{\SetFigFont{14}{16.8}{\rmdefault}{\mddefault}{\updefault}{\color[rgb]{0,0,0}$x+e_a+o_b+o_c$}%
}}}}
\put(2803,-11211){\rotatebox{300.0}{\makebox(0,0)[lb]{\smash{{\SetFigFont{14}{16.8}{\rmdefault}{\mddefault}{\updefault}{\color[rgb]{0,0,0}$y+z+o_a+e_b+e_c$}%
}}}}}
\put(1509,-8971){\makebox(0,0)[lb]{\smash{{\SetFigFont{14}{16.8}{\rmdefault}{\mddefault}{\updefault}{\color[rgb]{0,0,0}$y$}%
}}}}
\put(2501,-6611){\rotatebox{60.0}{\makebox(0,0)[lb]{\smash{{\SetFigFont{14}{16.8}{\rmdefault}{\mddefault}{\updefault}{\color[rgb]{0,0,0}$z+e_a+o_b+o_c$}%
}}}}}
\put(4891,-8821){\makebox(0,0)[lb]{\smash{{\SetFigFont{14}{16.8}{\rmdefault}{\mddefault}{\updefault}{\color[rgb]{0,0,0}$\frac{z}{2}$}%
}}}}
\put(7756,-8814){\makebox(0,0)[lb]{\smash{{\SetFigFont{14}{16.8}{\rmdefault}{\mddefault}{\updefault}{\color[rgb]{0,0,0}$\frac{z}{2}$}%
}}}}
\put(7456,-9526){\makebox(0,0)[lb]{\smash{{\SetFigFont{14}{16.8}{\rmdefault}{\mddefault}{\updefault}{\color[rgb]{1,1,1}$c_2$}%
}}}}
\put(9922,-9555){\makebox(0,0)[lb]{\smash{{\SetFigFont{14}{16.8}{\rmdefault}{\mddefault}{\updefault}{\color[rgb]{1,1,1}$b_2$}%
}}}}
\put(10639,-9086){\makebox(0,0)[lb]{\smash{{\SetFigFont{14}{16.8}{\rmdefault}{\mddefault}{\updefault}{\color[rgb]{1,1,1}$b_1$}%
}}}}
\put(9095,-9082){\makebox(0,0)[lb]{\smash{{\SetFigFont{14}{16.8}{\rmdefault}{\mddefault}{\updefault}{\color[rgb]{1,1,1}$b_3$}%
}}}}
\put(2599,-9437){\makebox(0,0)[lb]{\smash{{\SetFigFont{14}{16.8}{\rmdefault}{\mddefault}{\updefault}{\color[rgb]{1,1,1}$a_1$}%
}}}}
\put(3544,-8846){\makebox(0,0)[lb]{\smash{{\SetFigFont{14}{16.8}{\rmdefault}{\mddefault}{\updefault}{\color[rgb]{1,1,1}$a_2$}%
}}}}
\put(4607,-9437){\makebox(0,0)[lb]{\smash{{\SetFigFont{14}{16.8}{\rmdefault}{\mddefault}{\updefault}{\color[rgb]{1,1,1}$a_3$}%
}}}}
\put(6497,-8964){\makebox(0,0)[lb]{\smash{{\SetFigFont{14}{16.8}{\rmdefault}{\mddefault}{\updefault}{\color[rgb]{1,1,1}$c_1$}%
}}}}
\end{picture}%
}
\caption{Dividing the region $R^{\odot}_{0,y,z}(\textbf{a}; \textbf{c}; \textbf{b})$ into two regions.}\label{Basecasethreefern2}
\end{figure}



\medskip

For the induction step, we assume that $x$ and $z$ are positive, that $p\geq 6$, and that Theorems \ref{main1}--\ref{mainQ4}
all hold for any $R^{\odot}$-, $R^{\leftarrow}$-, $R^{\swarrow}$-, $R^{\nwarrow}$-,  $Q^{\odot}$-, $Q^{\leftarrow}$-, $Q^{\nearrow}$-, and $Q^{\nwarrow}$-type regions whose $h$-parameter is strictly less than $h=p+x+z$.

If $y$ achieves its minimal values  (which is $0$ or $-1$), then by  Lemma \ref{lem2} our region have the same tiling number as another region whose $h$ parameter is strictly less than $h$. Then our theorem follows from the induction hypothesis.

We know assume that $y$ does not achieve its minimal value. In this case, all of our 18 recurrences in Sections 3.3--3.10 work. 
Let $\mathcal{R}$ be either the region $R^{\odot}_{x,y,z}(\textbf{a};\ \textbf{c};\ \textbf{b})$, $R^{\leftarrow}_{x,y,z}(\textbf{a};\ \textbf{c};\ \textbf{b})$, $R^{\swarrow}_{x,y,z}(\textbf{a};\ \textbf{c};\ \textbf{b})$, $R^{\nwarrow}_{x,y,z}(\textbf{a};\ \textbf{c};\ \textbf{b})$,  $Q^{\odot}_{x,y,z}(\textbf{a};\ \textbf{c};\ \textbf{b})$, $Q^{\leftarrow}_{x,y,z}(\textbf{a};\ \textbf{c};\ \textbf{b})$, $Q^{\nearrow}_{x,y,z}(\textbf{a};\ \textbf{c};\ \textbf{b})$, or $Q^{\nwarrow}_{x,y,z}(\textbf{a};\ \textbf{c};\ \textbf{b})$ in the 18 recurrences. We also denote $\mathcal{R}_1, \mathcal{R}_2,\dotsc, \mathcal{R}_5$ the \emph{other} five regions appearing in the recurrences corresponding to $\mathcal{R}$, from left to right. In the next two paragraphs, we show that $\mathcal{R}_i$'s have $h$-parameter strictly smaller than $h=p+x+z$.

In particular, in each of the recurrences, the sum of the $x$- and $z$-parameters of $\mathcal{R}_i$'s
are always less than or equal to $x+z$. Moreover, the quasi perimeters of $\mathcal{R}_i$'s are $p$,
$p-1$, $p-2$, or $p-3$ (depending of how many of the four triangles $u,v,w,s$ are on the boundary of the base hexagon).
In particular, if the lozenge-removal pattern along one side of the base hexagon is not overlapping with any other, the portion of the old boundary that adjacent to the forced lozenges is replaced by an 1-unit shorter portion, this reduces the length of the boundary of the base hexagon by $1$
(see the pictures in the first row of Figure \ref{Boundaryreduce}).
 In the case when two lozenge-removal patterns along two consecutive sides of the base hexagon are overlapping,
 the portion of the old boundary corresponding to the combined pattern is replaced by a $2$-unit shorter one indicated
 by the dotted line (see the two examples in the lower row of Figure \ref{Boundaryreduce}). This means that the quasi-parameter of $\mathcal{R}_i$ is
 $p-k$, where $k$ is the number of triangles $u,v,w,s$ which are on the boundary of the base hexagon.

This means that, if at least one of the removed unit triangles $u,v,w,s$ lies on the boundary, then the corresponding $\mathcal{R}_i$ region has the
$h$-parameter strictly less than $h=p+x+z$. The other case only happens when the region $\mathcal{R}_i$ corresponds to the graph $G$ with two removed unit triangles appended to the end of the left and to the right ferns
 (as in the second region in the recurrences (\ref{centerrecur4a}), (\ref{centerrecur4b}) and (\ref{centerrecur4c}) of the $R^{\nwarrow}$-type regions,
 or recurrences (\ref{centerrecur8a}), (\ref{centerrecur8b}) and (\ref{centerrecur8c}) of the $Q^{\nwarrow}$-type regions), then
the quasi-parameter of  $\mathcal{R}_i$ is exactly $p$. However, in this case, the sum of $x$- and $z$-parameters of $\mathcal{R}_i$
 is always $x+z-2$. This implies that its $h$-parameter is $p+x+z-2=h-2$, which is still less than $h$.

\begin{figure}\centering
\includegraphics[width=13cm]{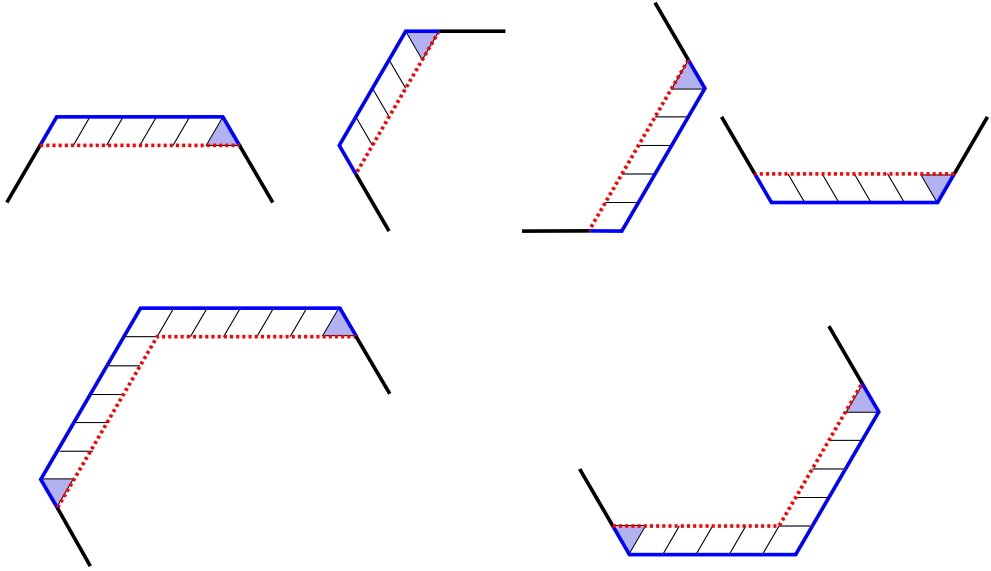}
\caption{Reduction of the length of the quasi-boundary after the removal of forced lozenges.}\label{Boundaryreduce}
\end{figure}

\medskip

In total, we can always write the number of tilings of our region in terms of tiling numbers of other regions
 with $h$-parameter strictly less than, and the latter regions have tiling numbers given by explicit product formulas
 by the induction hypothesis. Our final work is now to verify that the tiling formulas in Theorems \ref{main1}--\ref{mainQ4}
 satisfy the same recurrences. This verification will be left to the next section.

\subsection{Verifying the formulas in Theorems \ref{main1}--\ref{mainQ4} satisfy the recurrences (\ref{centerrecur1a})--(\ref{centerrecur8b})}

We only show here the verification for the recurrences for $R^{\odot}$-type regions, as other 16 recurrences can be treated in the same manner. Without loss of generality, we can assume that each of the three ferns consists of an even number of triangles, i.e. $m,n,k$ are all even.

Let us denote by $g_{x,y,z}^{\odot}(\textbf{a}; \textbf{c}; \textbf{b}),$ $g_{x,y,z}^{\leftarrow}(\textbf{a}; \textbf{c}; \textbf{b}),$ $g_{x,y,z}^{\nwarrow}(\textbf{a}; \textbf{c}; \textbf{b}),$ and $g_{x,y,z}^{\swarrow}(\textbf{a}; \textbf{c}; \textbf{b}),$ the tiling formulas in Theorems \ref{main1}--\ref{main4} (for Theorem \ref{main1}, we consider here the combined formula (\ref{maineq1c})).

We first work on the case when $a<b$.  We need to verify that
\begin{align}\label{verify1a}
g^{\odot}_{x,y,z}(\textbf{a};\ \textbf{c};\ \textbf{b}) g^{\leftarrow}_{x,y,z-1}(\textbf{a}^{+1};\ \textbf{c};\ \textbf{b})&=
g^{\odot}_{x+1,y,z-1}(\textbf{a};\ \textbf{c};\ \textbf{b}) g^{\leftarrow}_{x-1,y,z}(\textbf{a}^{+1};\ \textbf{c};\ \textbf{b})\notag\\
&+g^{\nwarrow}_{x,y-1,z}(\textbf{a};\ \textbf{c}; \ \textbf{b})g^{\swarrow}_{x,y,z-1}(\textbf{a}^{+1};\ \textbf{c};\  \textbf{b}).
\end{align}
Equivalently, we need to show that
\begin{align}\label{verify1b}
\frac{g^{\odot}_{x+1,y,z-1}(\textbf{a};\ \textbf{c};\ \textbf{b})}{g^{\odot}_{x,y,z}(\textbf{a};\ \textbf{c};\ \textbf{b})} \frac{g^{\leftarrow}_{x-1,y,z}(\textbf{a}^{+1};\ \textbf{c};\ \textbf{b})}{g^{\leftarrow}_{x,y,z-1}(\textbf{a}^{+1};\ \textbf{c};\ \textbf{b})}
+\frac{g^{\nwarrow}_{x,y-1,z}(\textbf{a};\ \textbf{c}; \ \textbf{b})}{g^{\odot}_{x,y,z}(\textbf{a};\ \textbf{c};\ \textbf{b}) }\frac{g^{\swarrow}_{x,y,z-1}(\textbf{a}^{+1};\ \textbf{c};\  \textbf{b})}{g^{\leftarrow}_{x,y,z-1}(\textbf{a}^{+1};\ \textbf{c};\ \textbf{b})}=1.
\end{align}
We first simplify the first fraction of the first term on the left-hand side of (\ref{verify1b}) as
\begin{align}\label{verify1c}
\frac{g^{\odot}_{x+1,y,z-1}(\textbf{a};\ \textbf{c};\ \textbf{b})}{g^{\odot}_{x,y,z}(\textbf{a};\ \textbf{c};\ \textbf{b})} &=\frac{\M(C_{x+1,2y+z+2b-1,z-1}(c))}{\M(C_{x,2y+z+2b,z}(c))}\notag\\
&\times \frac{\Hf(b+y+z-1)}{\Hf(b+y+z)}\frac{\Hf(b+c+y+z-1)}{\Hf(b+c+y+z)}\notag\\
&\times \frac{\Hf(b-o_a+o_b+o_c+y+z)}{\Hf(b-o_a+o_b+o_c+y+z-1)}\frac{\Hf(b+o_a-o_b+e_c+y+z)}{\Hf(b+o_a-o_b+e_c+y+z-1)}.
\end{align}
Similarly, we get
\begin{align}\label{verify1d}
\frac{g^{\leftarrow}_{x-1,y,z}(\textbf{a}^{+1};\ \textbf{c};\ \textbf{b})}{g^{\leftarrow}_{x,y,z-1}(\textbf{a}^{+1};\ \textbf{c};\ \textbf{b})} &=\frac{\M(C_{x-1,2y+z+2b,z}(c))}{\M(C_{x,2y+z+2b-1,z-1}(c))}\notag\\
&\times \frac{\Hf(b+y+z)}{\Hf(b+y+z-1)}\frac{\Hf(b+c+y+z)}{\Hf(b+c+y+z-1)}\notag\\
&\times \frac{\Hf(b-o_a+o_b+o_c+y+z-1)}{\Hf(b-o_a+o_b+o_c+y+z)}\frac{\Hf(b+o_a-o_b+e_c+y+z-1)}{\Hf(b+o_a-o_b+e_c+y+z)}.
\end{align}
By (\ref{verify1c}) and (\ref{verify1d}), we have the first term on the left-hand side of (\ref{verify1b}) simplified as
\begin{align}\label{verify1e}
\frac{g^{\odot}_{x+1,y,z-1}(\textbf{a};\ \textbf{c};\ \textbf{b})}{g^{\odot}_{x,y,z}(\textbf{a};\ \textbf{c};\ \textbf{b})} \frac{g^{\leftarrow}_{x-1,y,z}(\textbf{a}^{+1};\ \textbf{c};\ \textbf{b})}{g^{\leftarrow}_{x,y,z-1}(\textbf{a}^{+1};\ \textbf{c};\ \textbf{b})}=\frac{\M(C_{x+1,2y+z+2b-1,z-1}(c))}{\M(C_{x,2y+z+2b,z}(c))}\frac{\M(C_{x-1,2y+z+2b,z}(c))}{\M(C_{x,2y+z+2b-1,z-1}(c))}.
\end{align}

We now work on the second term on the left-hand side of (\ref{verify1a}). By definition, the first fraction here can be  written as
\begin{align}\label{verify1f}
\frac{g^{\nwarrow}_{x,y-1,z}(\textbf{a};\ \textbf{c}; \ \textbf{b})}{g^{\odot}_{x,y,z}(\textbf{a};\ \textbf{c};\ \textbf{b}) } &=\frac{\M(C_{x,2y+z+2b-1,z}(c))}{\M(C_{x,2y+z+2b,z}(c))}\notag\\
&\times \frac{s\left(y+b-a-1,a_1,\dotsc, a_{m},\frac{x+z}{2},c_1,\dotsc,c_{k}+\frac{x+z}{2}+b_n,b_{n-1},\dotsc,b_1\right)}{s\left(y+b-a,a_1,\dotsc, a_{m},\frac{x+z}{2},c_1,\dotsc,c_{k}+\frac{x+z}{2}+b_n,b_{n-1},\dotsc,b_1\right)}\notag\\
&\times \frac{(b+y-1)!}{(b+c+y+z-1)!}\frac{(b+c+y+\frac{x+z}{2}-1)!}{(b+y+\frac{x+z}{2}-1)!} \frac{(b-o_a+o_b+o_c+y+z-1)!}{(b-o_a+o_b+o_c+y-1)!}.
\end{align}
Similarly, we have
\begin{align}\label{verify1g}
\frac{g^{\swarrow}_{x,y,z-1}(\textbf{a}^{+1};\ \textbf{c};\  \textbf{b})}{g^{\leftarrow}_{x,y,z-1}(\textbf{a}^{+1};\ \textbf{c};\ \textbf{b})}&=\frac{\M(C_{x,2y+z+2b,z-1}(c))}{\M(C_{x,2y+z+2b-1,z-1}(c))}\notag\\
&\times \frac{s\left(y+b-\min(a,b),a_1,\dotsc, a_{m}+1,\frac{x+z}{2}-1,c_1,\dotsc,c_{k}+\frac{x+z}{2}+b_n,b_{n-1},\dotsc,b_1\right)}{s\left(y+b-\min(a,b)-1,a_1,\dotsc, a_{m}+1,\frac{x+z}{2}-1,c_1,\dotsc,c_{k}+\frac{x+z}{2}+b_n,b_{n-1},\dotsc,b_1\right)}\notag\\
&\times \frac{(b+c+y+z-1)!}{(b+y)!} \frac{(b+y+\frac{x+z}{2}-1)!}{(b+c+y+\frac{x+z}{2}-1)!} \frac{(b-o_a+o_b+o_c+y)!}{(b-o_a+o_b+o_c+y+z-1)!}.
\end{align}
By (\ref{verify1f}) and (\ref{verify1g}), we have the second term on the left-hand side of (\ref{verify1b}) simplified as
\begin{align}\label{verify1h}
&\frac{g^{\nwarrow}_{x,y-1,z}(\textbf{a};\ \textbf{c}; \ \textbf{b})}{g^{\odot}_{x,y,z}(\textbf{a};\ \textbf{c};\ \textbf{b}) } \frac{g^{\swarrow}_{x,y,z-1}(\textbf{a}^{+1};\ \textbf{c};\  \textbf{b})}{g^{\leftarrow}_{x,y,z-1}(\textbf{a}^{+1};\ \textbf{c};\ \textbf{b})} =\frac{\M(C_{x,2y+z+2b-1,z}(c))}{\M(C_{x,2y+z+2b,z}(c))}\frac{\M(C_{x,2y+z+2b,z-1}(c))}{\M(C_{x,2y+z+2b-1,z-1}(c))}\notag\\
&\times \frac{s\left(y+b-a-1,a_1,\dotsc, a_{m},\frac{x+z}{2},c_1,\dotsc,c_{k}+\frac{x+z}{2}+b_n,b_{n-1},\dotsc,b_1\right)}{s\left(y+b-a,a_1,\dotsc, a_{m},\frac{x+z}{2},c_1,\dotsc,c_{k}+\frac{x+z}{2}+b_n,b_{n-1},\dotsc,b_1\right)}\notag\\
&\times  \frac{s\left(y+b-a,a_1,\dotsc, a_{m}+1,\frac{x+z}{2}-1,c_1,\dotsc,c_{k}+\frac{x+z}{2}+b_n,b_{n-1},\dotsc,b_1\right)}{s\left(y+b-a-1,a_1,\dotsc, a_{m}+1,\frac{x+z}{2}-1,c_1,\dotsc,c_{k}+\frac{x+z}{2}+b_n,b_{n-1},\dotsc,b_1\right)}\notag\\
&\times \frac{b-o_a+o_b+o_c+y}{b+y}.
\end{align}
We have the following claim as a direct consequence of Cohn--Larsen--Propp's formula (\ref{semieq}):
\begin{claim}\label{claimS}
Let $t_1,t_2,\dotsc,t_{2l}$ are non-negative integers. Then
\begin{align}
\dfrac{\dfrac{s(t_1,t_2,\dotsc,t_{2n-1},t_{2n},t_{2n+1},t_{2n+2},\dotsc, t_{2l})}{s(t_1,t_2,\dotsc,t_{2n-1},t_{2n}+1,t_{2n+1}-1,t_{2n+2},\dotsc, t_{2l})}}{ \dfrac{s(t_1-1,t_2,\dotsc,t_{2n-1},t_{2n},t_{2n+1},t_{2n+2},\dotsc, t_{2l})}{s(t_1-1,t_2,\dotsc,t_{2n-1}, t_{2n}+1,t_{2n+1}-1,t_{2n+2},\dotsc, t_{2l})} }=\dfrac{t_1+t_2+\dotsc+t_{2n}}{o_t-1}.
\end{align}
\end{claim}

Apply the claim to the $s$-terms on the right-hand side of (\ref{verify1h}), we get
\begin{align}\label{verify1i}
\frac{g^{\nwarrow}_{x,y-1,z}(\textbf{a};\ \textbf{c}; \ \textbf{b})}{g^{\odot}_{x,y,z}(\textbf{a};\ \textbf{c};\ \textbf{b}) } \frac{g^{\swarrow}_{x,y,z-1}(\textbf{a}^{+1};\ \textbf{c};\  \textbf{b})}{g^{\leftarrow}_{x,y,z-1}(\textbf{a}^{+1};\ \textbf{c};\ \textbf{b})} &=\frac{\M(C_{x,2y+z+2b-1,z}(c))}{\M(C_{x,2y+z+2b,z}(c))}\frac{\M(C_{x,2y+z+2b,z-1}(c))}{\M(C_{x,2y+z+2b-1,z-1}(c))}.
\end{align}
This means that we now only need to show that
\begin{align}\label{verify1k}
\frac{\M(C_{x+1,2y+z+2b-1,z-1}(c))}{\M(C_{x,2y+z+2b,z}(c))}&\frac{\M(C_{x-1,2y+z+2b,z}(c))}{\M(C_{x,2y+z+2b-1,z-1}(c))}\notag\\
&+\frac{\M(C_{x,2y+z+2b-1,z}(c))}{\M(C_{x,2y+z+2b,z}(c))}\frac{\M(C_{x,2y+z+2b,z-1}(c))}{\M(C_{x,2y+z+2b-1,z-1}(c))}=1,
\end{align}
or equivalently
\begin{align}\label{verify1l}
\M(C_{x,2y+z+2b,z}(c))&\M(C_{x,2y+z+2b-1,z-1}(c))=\notag\\
&\M(C_{x+1,2y+z+2b-1,z-1}(c))\M(C_{x-1,2y+z+2b,z}(c))\notag\\
&+\M(C_{x,2y+z+2b-1,z}(c))\M(C_{x,2y+z+2b,z-1}(c)).
\end{align}
This is straightforward from the tiling formulas of the cored hexagons in \cite{CEKZ}.
However, one can verify (\ref{verify1l}) \emph{without} using tiling formulas of cored hexagons by observing that it is actually a consequence of the recurrence (\ref{centerrecur1a}) as follows.

Apply recurrence (\ref{centerrecur1a}) to the region $R^{\odot}_{x+1,y+b-1,z}((0,0); (c); (0,1))$, we get
\begin{align}
\M(R^{\odot}_{x+1,y+b-1,z}&((0,0);\ (c);\ (0,1)) \M(R^{\leftarrow}_{x+1,y+b-1,z-1}((0,1);\ (c);\ (0,1)))=\notag\\
&\M(R^{\odot}_{x+2,y+b-1,z-1}((0,0);\ (c);\ (0,1))) \M(R^{\leftarrow}_{x,y+b-1,z}((0,1);\ (c);\ (0,1)))\notag\\
&+\M(R^{\nwarrow}_{x+1,y+b-2,z}((0,0);\ (c);\ (0,1)))\M(R^{\swarrow}_{x+1,y+b-1,z-1}((0,1);\ (c);\ (0,1))).
\end{align}
After removing forced lozenges along the northeast side from each region in the above recurrence, and removing forced lozenges along the southwest side of the the regions whose left fern corresponds to the sequence $\textbf{a}=(0,1)$, we get back the cored hexagons in (\ref{verify1l}). This finishes our verification for (\ref{verify1a}).

Similarly, we can verify that our tiling formulas satisfy the recurrence (\ref{centerrecur1b}) for $R^{\odot}$-type regions when $a\geq b$. It means that we need to show
\begin{align}\label{verify1m}
g^{\odot}_{x,y,z}(\textbf{a};\ \textbf{c};\ \textbf{b}) g^{\leftarrow}_{x,y-1,z-1}(\textbf{a}^{+1};\ \textbf{c};\ \textbf{b})&=
g^{\odot}_{x+1,y,z-1}(\textbf{a};\ \textbf{c};\ \textbf{b}) g^{\leftarrow}_{x-1,y-1,z}(\textbf{a}^{+1};\ \textbf{c};\ \textbf{b})\notag\\
&+g^{\nwarrow}_{x,y-1,z}(\textbf{a};\ \textbf{c}; \ \textbf{b})g^{\swarrow}_{x,y-1,z-1}(\textbf{a}^{+1};\ \textbf{c};\  \textbf{b}).
\end{align}
for $a\geq b$. However, this verification is essentially the same as that of the case $a<b$ treated above and is omitted.

\section{Concluding Remarks}\label{sec:remark}
As pointed out by Fulmek in \cite{Ful},  Kuo's graphical condensation is simply a special case of the determinant-permanent method. However, this paper shows a particular advantage of Kuo's method comparing with the classical  determinant-permanent method when dealing with regions with complicated structures.

The main results in the series of papers \cite{Halfhex1, Halfhex2, Halfhex2} imply an explicit enumeration of a reflectively symmetric tilings of the $Q^{\odot}$-type regions (i.e. tilings which are invariant under a refection over a vertical symmetry axis). This results also extends Proctor's enumeration of transpose-complementary plane partitions \cite{Proc} and the related work of Ciucu \cite{Ciucu1} and Rohatgi \cite{Ranjan1}. We are also interested in the centrally symmetric tilings of the  $R^{\odot}$-type regions (i.e. tilings which are invariant under $180^{\circ}$ rotations). Our data suggests that the number of these symmetric tilings would have nice prime factorizations. 
\begin{con}
The number of centrally symmetric tilings of the region $R^{\odot}_{x,yz}(\textbf{a};\ \textbf{c};\ \textbf{b})$ is always given by a simple product formula. 
\end{con}
This (if verified) gives a generalization of Stanley's enumeration of self-complementary plane partitions \cite[Eq. (3a)--(3c)]{Stanley}.

In \cite{CEKZ}, Ciucu, Eisenk\"{o}lbl, Krattenthaler and Zare  posed two striking conjectures for the tiling formulas of a hexagon when the triangular hole is $1$ or $3/2$ unit off the center. The conjectures were recently proved by Rosengren \cite{Rosen} using lattice path combinatorics and Selberg integral. In the sequel of the paper, we will enumerate extensively $30$ different hexagons with three ferns removed, in which the middle fern is slightly off the center.  Two of our enumerations have Ciucu--Eisenk\"{o}lbl--Krattenthaler--Zare's conjectures as two very-special cases (when the two side ferns are empty and the middle fern consists of a single triangle). This provides new proofs of the conjectures.

Intuitively, our main theorems (Theorems 2.3--2.9) say that the number of tilings of a hexagon with three ferns removed can be always factorized into the tiling number of a cored-hexagon and a simple multiplicative factor. One would ask for a similar factorization for the general case of an arbitrary number of collinear ferns removed. Such a factorization seems to exist and will be investigated in a separate paper.

It would be interesting to investigate whether Rosengren's weighted formula in \cite{Rosen} can be generalized to our hexagons with three ferns removed. This and several new duals and $q$-duals of MacMahon's theorem will also be considered in a separate paper.

\section*{Acknowledgement}

The author would like to thank Dennis Stanton and Hjalmar Rosengren for pointing out to him paper \cite{Rosen}.


\end{document}